\documentclass[a4paper,11pt]{article}

\usepackage{macros} %
\usepackage[normalem]{ulem} %
\usepackage{nameref}
\usepackage{zref-xr}
\usepackage{tabu} %
\zxrsetup{toltxlabel}
\zexternaldocument*[bhhms2:]{BHHMS2} %
\zexternaldocument*[bhhms4:]{BHHMS4}%
\title{On the constituents of the mod $p$ cohomology of Shimura curves}
\definecolor{olive}{rgb}{0.5, 0.5, 0.0}

\newcounter{mar-counter}
\renewcommand{\mar}[1]{\stepcounter{mar-counter}\marginpar{\#\arabic{mar-counter}\ \raggedright\tiny #1}}

\usepackage{bigfoot}

\DeclareNewFootnote{AAffil}[alph]

\usepackage{etoolbox}
\makeatletter
\patchcmd\maketitle{\def\@makefnmark{\rlap{\@textsuperscript{\normalfont\@thefnmark}}}}{}{}{}
\makeatother

\makeatletter
\def\thanksAAffil#1{%
  \footnotemarkAAffil\protected@xdef\@thanks{\@thanks%
        \protect\footnotetextAAffil[\the \c@footnoteAAffil]{#1}}%
}
\def\thanksANote#1{%
  \footnotemarkANote%
  \protected@xdef\@thanks{\@thanks%
        \protect\footnotetextANote[\the \c@footnoteANote]{#1}}%
}
\makeatother

\author{Christophe Breuil\thanksAAffil{CNRS, B\^atiment 307, Facult\'e d'Orsay, Universit\'e Paris-Saclay, 91405 Orsay Cedex, France}\\
\and
Florian Herzig\thanksAAffil{Dept.\ of Math., Univ.\ of Toronto, 40 St.\ George St., BA6290, Toronto, ON M5S 2E4, Canada;\newline\indent\hspace{1.5mm} Korea Institute for Advanced Study, 85 Hoegi-ro, Dongdaemun-gu, Seoul 02455, Republic of Korea}\\
\and
Yongquan Hu\thanksAAffil{Morningside Center of Mathematics, Academy of Mathematics and Systems Science, Chinese Academy  of \newline
\indent\hspace{1.5mm} Sciences, Beijing 100190, China; University of the Chinese Academy of Sciences, Beijing 100049, China}\\
\and
Stefano Morra\thanksAAffil{Lab.\ d'Analyse, G\'eom\'etrie, Alg\`ebre, 99
  Av.\ Jean Baptiste Cl\'ement, 93430 Villetaneuse, France }$^{,}$\footnotemarkAAffil[6]\\
\and
Benjamin Schraen\thanksAAffil{Institut Camille Jordan, Universit\'e Claude Bernard Lyon I, 69622 Villeurbanne, France}$^{,}$\thanksAAffil{Institut
  Universitaire de France (IUF)}}

\date{ }

\begin{document}

\maketitle

\begin{abstract}
  Let $p$ be a prime number and $K$ a finite unramified extension of $\Qp$.
  When $p$ is large enough with respect to $[K:\Qp]$ and under mild genericity assumptions, we proved in our previous work that the admissible smooth representations $\pi$ of $\GL_2(K)$ that occur in Hecke eigenspaces of the mod $p$ cohomology of Shimura curves are of finite length.
  In this paper we obtain various refined results about the structure of subquotients of $\pi$, such as their Iwahori-socle filtrations and $K_1$-invariants, where $K_1$ is the principal congruence subgroup of $\GL_2(\cO_K)$.
  We also determine the Hilbert series of $\pi$ as Iwahori-representation under these conditions.
\end{abstract}

\tableofcontents

\section{Introduction}
\label{sec:introduction}

\subsection{The main results}
\label{sec:results}

Let $p$ be a prime number, $F$ a totally real number field and $D$ a quaternion algebra of center $F$ which is split at all $p$-adic places and at exactly one infinite place.
In order to simplify this introduction we assume that $p$ is inert in $F$ (in the text we only need $p$ unramified in $F$) and denote by $v$ the unique $p$-adic place of $F$.
Let ${\mathbb A}_F^{\infty,v}$ denote the ring of finite prime-to-$v$ ad\`eles of $F$ and $\F$ a sufficiently large finite extension of $\Fp$.
To any absolutely irreducible continuous representation $\rbar : {\rm Gal}(\overline F/F)\rightarrow \GL_2(\F)$ and $V^v$ a compact open subgroup of $(D\otimes_F{\mathbb A}_F^{\infty,v})^\times$, we associate the admissible smooth representation of $\GL_2(F_v)$ over $\F$: 
\begin{equation}\label{eq:goal}
\pi\defeq \varinjlim_{V_v}\Hom_{{\rm Gal}(\overline F/F)}\!\big(\rbar,H^1_{{\rm \acute et}}(X_{V^vV_v} \times_F \overline F, \F)\big),
\end{equation}
where the inductive limit runs over compact open subgroups $V_v$ of $(D\otimes_FF_v)^\times\cong \GL_2(F_v)$ and $X_{V^vV_v}$ is the smooth projective Shimura curve over $F$ associated to $D$ and $V^vV_v$.
Throughout this introduction we fix $\pi$ as in (\ref{eq:goal}) such that $\pi\ne 0$.
We assume moreover that $\rbar$ is sufficiently generic (\emph{strongly generic} in the terminology of this paper, see below) and that a standard multiplicity one assumption, defined below, holds (this multiplicity one assumption is commonly referred to as ``the minimal case'').

In our previous work we established that $\pi$ is of finite length.
More precisely we showed that $\pi$ is irreducible if $\rbar\vert_{\Gal(\o F_v/F_v)}$ is irreducible \cite[Thm.~3.105(i)]{BHHMS2}
and that $\pi$ is of length at least 3 (if $F \ne \Qp$) and at most $[F_v:\Qp]+1$ if $\pi$ is reducible \cite[Thm.\ 1.1.1]{BHHMS4}.
We moreover showed in \emph{loc.~cit.}~that $\pi$ is uniserial with distinct irreducible constituents if $\rbar\vert_{\Gal(\o F_v/F_v)}$ is nonsplit reducible.
The goal of this paper, which is a continuation of our previous aforementioned paper, is to investigate the irreducible constituents of $\pi$ when $\rbar\vert_{\Gal(\o F_v/F_v)}$ is reducible, especially in the more difficult case when $\rbar\vert_{\Gal(\o F_v/F_v)}$ is nonsplit.
(We remark that all but two of the irreducible constituents are supersingular.)
In particular, for any subquotient $\pi'$ of $\pi$ we determine its Iwahori-socle filtration, its invariants under the first principal congruence subgroup, and the dimension of its associated cyclotomic $(\vp,\Gamma)$-module, showing in each case that the answer is \emph{local}, i.e.\ only depends on $\rbar\vert_{\Gal(\o F_v/F_v)}$ (as expected).

Let us describe our most important results in more detail. 

We set $K\defeq F_v$, $f\defeq [K:\Qp]$ and $q\defeq p^f$. We denote by $\omega$ the mod $p$ cyclotomic character of $\Gal(\o K/K)$ (that we consider as a character of $K^\times$ via local class field theory, where uniformizers correspond to geometric Frobenius elements), and by $\omega_{f}$, $\omega_{2f}$ Serre's fundamental characters of the inertia subgroup $I_K$ of $\Gal(\o K/K)$ of level $f$, $2f$ respectively. In this introduction, we say that $\rbar$ is \emph{generic} if the following conditions are satisfied for $N \defeq \max\{12,2f+7\}$:
\begin{enumerate}
\item $\rbar\vert_{\Gal(\overline F/F(\sqrt[p]{1}))}$ is absolutely irreducible;
\item\label{it:wramifies} for $w\!\nmid\! p$ such that either $D$ or $\rbar$ ramifies at $w$, the framed deformation ring of $\rbar\vert_{{\rm Gal}(\overline F_w/F_w)}$ over the Witt vectors $W(\F)$ is formally smooth;
\item\label{it:wisv}$\rbar\vert_{I_{K}}$ is up to twist of form
\begin{equation*}
\begin{pmatrix}\omega_f^{\sum_{j=0}^{f-1} (r_j+1) p^j}&*\\0&1\end{pmatrix}\text{\ with\ }N \le r_j \le p-3-N
\end{equation*}
or
\begin{equation*}
\left(\begin{matrix}\omega_{2f}^{\sum_{j=0}^{f-1} (r_j+1) p^j} & \\ & \omega_{2f}^{q(\text{same})}\end{matrix}\right) \text{\ with\ }
  \begin{cases}
    N \le r_j \le p-3-N, \!\!&\!\! j>0 \\ N+1 \le r_0 \le p-2-N.&
  \end{cases}
\end{equation*}
\end{enumerate}
Note that \ref{it:wisv} implies $p\geq \max\{27,4f+9\}$ and that \ref{it:wramifies} can be made explicit (\cite{Shotton}, \cite[Rk.\ 8.1.1]{BHHMS1}).
{We say that $\rbar$ is \emph{strongly generic} if the above conditions are satisfied with $N \defeq \max\{12,4f+1\}$.}
By \cite[Thm.\ 1.9]{BHHMS1} (for $\rbar\vert_{{\rm Gal}(\o K/K)}$ semisimple) and \cite[Thm.\ 6.3(ii)]{YitongWangGKD} (for $\rbar\vert_{{\rm Gal}(\o K/K)}$ non-semisimple) for $\rbar$ generic there is a unique integer $r\geq 1$ (the ``multiplicity'') such that, for any (absolutely) irreducible representation $\sigma$ of $\GL_2(\cO_K)$ over $\F$, we have $\dim_{\F}\Hom_{\GL_2(\cO_K)}(\sigma, \pi)\in \{0, r\}$ (the notation $\o r$ and $r$ is somewhat unfortunate but is consistent with \cite[\S~8]{BHHMS1}).

In the sequel we let $\rhobar\defeq \rbar^\vee\vert_{{\rm Gal}(\o K/K)}$, where $\rbar^\vee$ is the dual of $\rbar$. 

Let $I$ (resp.~$I_1$) be the subgroup of $\GL_2(\cO_K)$ of matrices which are upper triangular modulo $p$ (resp.~upper unipotent modulo $p$) and $K_1\defeq 1+p{\rm M}_2(\cO_K) \subset I_1$. %
Let $Z_1\cong 1+p\cO_K$ be the center of $I_1$ (or $K_1$).
For any admissible smooth representation $\pi'$ of $\GL_2(K)$, we consider $\pi'^{I_1}$ (resp.\ $\pi'^{K_1}$) as finite-dimensional representation of $I/I_1 \cong \F_q\s \times \F_q\s$ (resp.\ $\GL_2(\cO_K)/K_1 \cong \GL_2(\F_q)$).

Suppose from now on that $\rbar$ is generic, that $r=1$ and that $\rhobar$ is nonsplit reducible.
{(Even though not explicitly written in the literature, representations $\rbar$ satisfying these conditions always exist, using suitable globalizations arguments (e.g.~\cite[Cor.\ A.3]{gee-kisin} and \cite[Appendix B]{EGS}) and refining them to minimal level using \cite[\S~3.3]{BD}.)}
Then, we recall that $\pi^{K_1} \cong D_0(\brho)$ \cite{LMS, HuWang, DanWild}, where the $\GL_2(\cO_K)$-representation $D_0(\brho)$ was defined in \cite[Thm.~13.8]{BP}.
As we recalled above, $\pi$ is uniserial with distinct irreducible constituents, so any subquotient $\pi'$ is uniquely a quotient $\pi_1'/\pi_1$ for some subrepresentations $\pi_1 \subset \pi_1' \subset \pi$.
There is a strictly increasing filtration $D_0(\brho)_{\le i}$ ($-1 \le i \le f$) of $D_0(\brho) = \pi^{K_1}$ defined in \cite[Prop.\ 5.2]{Hu-SMF}, and by \cite[Thm.\ \ref{bhhms4:thm:conj2}]{BHHMS4} there exist \emph{unique} integers $-1 \le i_0 \le i_0' \le f$ such that $\pi_1^{K_1} = D_0(\brho)_{\le i_0}$ and $\pi_1'^{K_1} = D_0(\brho)_{\le i_0'}$.
Let $D_0(\brho)_i \defeq D_0(\brho)_{\leq i}/D_0(\brho)_{\leq i-1}$ for $0 \le i \le f$.
If $D_0(\brho^\ss)$ denotes the analog of $D_0(\brho)$ for $\brho^\ss$, then there exists a decomposition $D_0(\brho^\ss) = \bigoplus_{i=0}^f D_0(\brho^\ss)_i$ \cite[Thm.\ 15.4]{BP} such that $D_0(\brho)_i \subset D_0(\brho^\ss)_i$ for all $i$.
The following is one of our main results.

\begin{thm}[Corollary~\ref{cor:K1-subquot}]\label{thm:intro-pi'-K1}
  Assume that $\rbar$ is generic, that $r=1$ and that $\rhobar$ is nonsplit reducible.
  Then for any nonzero subquotient $\pi'$ of $\pi$ we have
  $$\pi'^{K_1} \cong D_0(\brho^\ss)_{i_0+1} \oplus_{D_0(\brho)_{i_0+1}} (D_0(\brho)_{\le i_0'}/D_0(\brho)_{\le i_0})$$
  as $\GL_2(\cO_K)$-representations.
\end{thm}

Note that if $\brho$ is split reducible, we prove a stronger result in Proposition \ref{prop:K1-square-invariants}.

Using the theorem it is not hard to determine the $I_1$-invariants and the $\GL_2(\cO_K)$-socle of any subquotient $\pi'$, as described in Theorem~\ref{thm:intro-structure-pi'} below.
In fact we do not know how to prove Theorem~\ref{thm:intro-pi'-K1} directly but rather deduce it with the help of Theorem~\ref{thm:intro-structure-pi'}.

To state Theorem~\ref{thm:intro-structure-pi'}, we recall some more standard notation (for more details, see \S~\ref{sec:notation}).
The set $\P$ parametrizes $\JH(D_0(\brho)^{I_1})$ and likewise $\P^\ss \supset \P$ parametrizes $\JH(D_0(\brho^\ss)^{I_1})$, where $\JH(\cdot)$ denotes the set of Jordan--H\"older factors (which are $1$-dimensional here since $I/I_1$ is commutative).
Given $\lambda \in \P^\ss$ let $\chi_\lambda : I/I_1 \to \F\s$ denote the corresponding character and let $\ell(\lambda) \in \{0,1,\dots,f\}$ be the unique integer $i$ such that $\chi_\lambda \in \JH(D_0(\brho^\ss)_i)$.
Let $W(\brho)$ denote the set of Serre weights of $\brho$ (cf.~\cite{BDJ}), i.e.\ the irreducible subrepresentations of $\pi|_{\GL_2(\cO_K)}$ or equivalently of $D_0(\brho)$ by its construction (\cite[\S~12]{BP}), and similarly define $W(\brho^\ss)$ using $D_0(\brho^\ss)$ (so that $W(\brho^\ss)$ contains $W(\brho)$). 
In other words, $W(\brho) = \JH(\soc_{\GL_2(\cO_K)}(D_0(\brho)))$, where $\soc_{\GL_2(\cO_K)}(\cdot)$ denotes the $\GL_2(\cO_K)$-socle.
For $\sigma \in W(\brho^\ss)$ we let $\ell(\sigma) \defeq \ell(\lambda)$, where $\lambda \in \P^\ss$ parametrizes $\sigma^{I_1} \subset D_0(\brho^\ss)^{I_1}$.

\begin{thm}[Corollaries~\ref{cor:subquot-I1-invt} and~\ref{cor:subquot-K-soc}]\label{thm:intro-structure-pi'}
  Assume that $\rbar$ is generic, that $r=1$ and that $\rhobar$ is nonsplit reducible.
  Then for any nonzero subquotient $\pi'$ of $\pi$ we have:
  \begin{enumerate}
  \item $\JH(\pi'^{I_1}) = \big\{ \chi_\lambda : \text{$\lambda \in \P,\; i_0 < \ell(\lambda) \le i'_0$ or $\lambda \in \P^\ss \setminus \P,\; \ell(\lambda) = i_0+1$} \big\}$;
  \item $\soc_{\GL_2(\cO_K)}(\pi') \cong \Big(\bigoplus_{\sigma \in W(\brho), i_0 < \ell(\sigma) \le i'_0} \sigma\Big) \oplus \Big(\bigoplus_{\sigma \in W(\brho^\ss) \setminus W(\brho), \ell(\sigma) = i_0+1} \sigma\Big)$.
  \end{enumerate}
\end{thm}

(Here part (i) is proved in Corollary~\ref{cor:subquot-I1-invt} and part (ii) in Corollary~\ref{cor:subquot-K-soc}.)

By Theorem~\ref{thm:intro-pi'-K1} we can relate the rank of the $(\varphi,\Gamma)$-module of a subquotient $\pi'$ to the $K_1$-invariants $\pi'^{K_1}$  (in a way which is compatible with \cite[\S\ 2.4.3\ Example\ 1, Conj.~1.4]{BHHMS2}), generalizing a result of Yitong Wang \cite[Thm.\ 1.2]{YitongWang} from subrepresentations to subquotients:
\begin{cor}\label{thm:intro-yitong-subquot}
  Assume that $\rbar$ is generic and that $r=1$.
  Then for any subquotient $\pi'$ of $\pi$ we have
  \begin{equation*}
    \dim_{\F\ppar{X}} D_{\xi}^{\vee}(\pi')=|\JH(\pi'^{K_1}) \cap W(\rhobar^\ss)|, %
  \end{equation*}
  where $D_\xi^\vee({\pi'})$ is the cyclotomic $(\varphi,\Gamma)$-module associated to ${\pi'}$ in \cite[\S~\ref{bhhms2:covariant}]{BHHMS2}.
\end{cor}
This is proved in Corollary \ref{cor:yitong-subquot-ss} if $\brho$ is semisimple and Corollary \ref{cor:yitong-subquot} otherwise.

The key in proving Theorem~\ref{thm:intro-structure-pi'} (and hence Theorem~\ref{thm:intro-pi'-K1}) is the following result which determines the (dual of the) socle filtration of $\pi'$ as an $I$-representation.
To explain, let $\Lambda\defeq \F\bbra{I_1/Z_1}$ denote the Iwasawa algebra of $I_1/Z_1$, which is a (noncommutative) noetherian local ring of Krull dimension $3f$.
We denote by $\m$ its maximal ideal.
Since $\pi$ has a central character, any subquotient $\pi'$ of $\pi$ is an admissible smooth representation of $\GL_2(K)/Z_1$ and hence its linear dual $\pi'^\vee \defeq \Hom_\F(\pi',\F)$ is a finitely generated $\Lambda$-module.
For any $\lambda \in \P$ there are explicit graded ideals
\begin{equation*}
  \fa(\lambda) = \fa_1^{f}(\lambda) \subset \fa_1^{f-1}(\lambda) \subset \cdots \subset \fa_1^{0}(\lambda) \subset \fa_1^{-1}(\lambda) = \gr_\m(\Lambda)
\end{equation*}
of $\gr_\m(\Lambda)$ (with commutative quotient rings of dimension $f$), cf.\ \cite[eq.~\eqref{bhhms4:eq:fa-1}]{BHHMS4}.
If $M$ is a graded module and $k\in\Z$, we let $M(k)$ denote $M$ with shifted grading $M(k)_n\defeq M_{n+k}$ for all $n\in \Z$.
(With our conventions, note further that $\gr_\m(\Lambda)$ and $\gr_\m(\pi'^\vee)$ are supported in \emph{non-positive} degrees, i.e.\ the degree $d$ part of $\gr_\m(\pi'^\vee)$ equals $\m^{-d} \pi'^\vee/\m^{-d+1}\pi'^\vee$.)

\begin{thm}[Corollary \ref{cor:gr-pi'}]\label{thm:intro-gr-pi'}
  Assume that $\rbar$ is %
  \emph{strongly} generic, that $r=1$ and that $\rhobar$ is nonsplit reducible.
  Then for any subquotient $\pi'$ of $\pi$ we have an isomorphism of graded $\gr_\m(\Lambda)$-modules with compatible $I/I_1$-actions,
  \begin{equation}
  \label{eq:intro:2}
    \gr_\m(\pi'^\vee) \cong \bigoplus_{\lambda\in\mathscr{P}}\chi_{\lambda}^{-1}\otimes \frac{\fa_1^{i_0}(\lambda)}{\fa_1^{i_0'}(\lambda)}(-d_\lambda),
  \end{equation}
  where $d_\lambda \defeq  \max\{i_0+1-\ell(\lambda),0\}$.
\end{thm}

For $\brho$ semisimple the analogous result is \cite[Cor.~\ref{bhhms4:cor:finite-length}(ii)]{BHHMS4}.
{We remark that $\rbar$ generic, rather than strongly generic, is sufficient in case $\pi'$ is a quotient of $\pi$.}

Theorem~\ref{thm:intro-gr-pi'} should be compared with \cite[Cor.\ \ref{bhhms4:cor:subquot}]{BHHMS4}, which shows that
\begin{equation}\label{eq:gr-subquot}
  \gr_F(\pi'^\vee) \cong \bigoplus_{\lambda\in\mathscr{P}}\chi_{\lambda}^{-1}\otimes \frac{\mathfrak{a}_1^{i_0}(\lambda)}{\mathfrak{a}_1^{i_0'}(\lambda)},
\end{equation}
where $F$ denotes the subquotient filtration induced by the $\m$-adic filtration on $\pi^\vee$.
It also generalizes \cite[Thm.~\ref{bhhms4:thm:CMC}]{BHHMS4} (when $\pi' = \pi$) and \cite[Cor.~\ref{bhhms4:cor:gr-pi1}]{BHHMS4} (when $\pi' \subset \pi$), though under stronger genericity assumptions.

{We point out the following interesting consequence of Theorem~\ref{thm:intro-gr-pi'} when $f = 2$.
Let $\pi_s$ be defined analogously to $\pi$ by using a global Galois representation $\rbar_s$ such that $\rhobar^\ss \cong \rbar_s^\vee\vert_{{\rm Gal}(\o K/K)}$.
As $f=2$ we know that $\pi$ is uniserial of the form $\pi_0 \-- \pi_1 \-- \pi_2$ by \cite[Thm.~10.37]{HuWang2} and that $\pi_s \cong \pi_0 \oplus \pi_1' \oplus \pi_2$ by \cite[Thm.~3.105(ii)]{BHHMS2} for explicit principal series $\pi_0$, $\pi_2$ and some irreducible supersingular representations $\pi_1$, $\pi_1'$.
Optimistically one may hope that $\pi_1 \cong \pi_1'$.
By comparing Theorem~\ref{thm:intro-gr-pi'} and \cite[Cor.~\ref{bhhms4:cor:finite-length}(ii)]{BHHMS4} we can provide the nontrivial evidence that $\gr_\m(\pi_1^\vee) \cong \gr_\m(\pi_1'^\vee)$, cf.\ Remark~\ref{rem:gr-m-semisimple-nonsplit} (which also gives a weaker result for $f > 2$).}
Moreover, $\pi_1^{K_1} \cong \pi_1'^{K_1}$ as $\GL_2(\cO_K)$-representations by comparing Theorem~\ref{thm:intro-pi'-K1} with (the $K_1$-invariants in) Proposition~\ref{prop:K1-square-invariants}.

In another direction, we determine the $\m_{K_1}^2$-invariants of subquotients in case $\brho$ is split reducible, where $\fm_{K_1}$ denotes the maximal ideal of the local ring $\F\bbra{K_1/Z_1}$.
We find, in particular, some weak evidence for the hope that $\pi$ is semisimple in this case:
\begin{prop}[Proposition \ref{prop:K1-square-invariants}]\label{prop:intro-K1-square-invariants}
  Assume that $\rbar$ is generic, that $r=1$ and that $\rhobar$ is split reducible.
  For any subrepresentations $\pi_1\subset\pi_2$ of $\pi$ the induced sequence of $\GL_2(K)$-representations
  \[0\ra \pi_1[\m_{K_1}^2]\ra \pi_2[\m_{K_1}^2]\ra (\pi_2/\pi_1)[\m_{K_1}^2]\ra0\]
  is split exact.
\end{prop}

Finally, we determine the Hilbert series of the associated graded module $\gr_\m(\pi^\vee)$, namely the series $h_\pi(t) \defeq \sum_{n\geq 0} \dim_{\F}(\m^n\pi^\vee/\m^{n+1}\pi^\vee)t^n\in \Z\bbra{t}$.
If $\brho$ is nonsplit reducible let $d_{\brho} \in \{0,1,\dots,f-1\}$, so that $2^{d_{\brho}} = |W(\brho)|$.

\begin{thm}[Theorem \ref{thm:Hilbert}]\label{thm:intro-Hilbert}
  Assume that $\rbar$ is generic and that $r=1$.
  \begin{enumerate}
  \item If $\brho$ is irreducible, then $\displaystyle h_{{\pi}}(t)=\frac{(3+t)^f}{(1-t)^f}-1$. %
  \item If $\brho$ is split reducible, then $\displaystyle h_{{\pi}}(t)=\frac{(3+t)^f}{(1-t)^f}+1$.
  \item If $\brho$ is nonsplit reducible, then $\displaystyle h_{{\pi}}(t)=2^{f-d_{\brho}}\cdot \frac{(1+t)^{f-d_{\brho}}(3+t)^{d_{\brho}}}{(1-t)^f}$.
  \end{enumerate}
\end{thm}

This follows from the special case of Theorem~\ref{thm:intro-gr-pi'} when $\pi' = \pi$, which we established earlier \cite[Thm.~\ref{bhhms4:thm:CMC}]{BHHMS4}.
We also determine the Hilbert series of $h_{\pi'}(t)$ for subquotients $\pi'$ of $\pi$, in case $\brho$ is split reducible.
(It is possible to determine $h_{\pi'}(t)$ for nonsplit $\brho$, but we did not find nice formulas in general.)

In fact, all of our results do not just apply to the global representation $\pi$ defined in~\eqref{eq:goal}, but to an arbitrary smooth representation of $\GL_2(K)$ that satisfies axioms \ref{it:assum-i}--\ref{it:assum-v} in section~\ref{subsec:assumptions}.
In section~\ref{sec:global-arg} we verify that a globally defined representation $\pi(\brho)$ satisfies all of these axioms (and actually relaxing condition~\ref{it:wramifies} at the beginning of the introduction, see section~\ref{sec:global:setting} for a detailed explanation). 

\subsection{Sketch proof of Theorem~\ref{thm:intro-gr-pi'}}\label{sec:sketchintro}

In order to sketch the proof of the key Theorem~\ref{thm:intro-gr-pi'} we assume for simplicity that $\pi' \defeq \pi_2 = \pi/\pi_1$ is a quotient of $\pi$, which is where the main difficulty lies.
Let $N_2'$ denote the graded $\gr_\m(\Lambda)$-module on the right-hand side of the theorem, i.e.\ $N_2' \defeq \bigoplus_{\lambda\in\mathscr{P}}\chi_{\lambda}^{-1}\otimes \frac{\fa_1^{i_0}(\lambda)}{\fa(\lambda)}(-d_\lambda)$ (as $i_0' = f$).
Let $\o\m$ denote the unique maximal graded ideal of $\gr_\m(\Lambda)$.
The proof of Theorem~\ref{thm:intro-gr-pi'} breaks into three steps:
\begin{enumerate}[label=(\alph*)]
\item Show that there exists a surjection $\gr_\m(\pi_2^\vee) \onto N_2'/\o\m^3$.
\item Show that $\gr_\m(\pi_2^\vee)/\o\m^3 \cong N_2'/\o\m^3$.
\item Lift the isomorphism in (b) to an isomorphism $\gr_\m(\pi_2^\vee) \cong N_2'$.
\end{enumerate}

For part (a), we let $\Theta_n \defeq \pi[\m^n]/\pi_1[\m^n] \subset \pi_2[\m^n] \subset \pi_2$ for some integer $n \ge 1$.
Hence $\pi_2^\vee \onto \Theta_n^\vee$ as $\Lambda$-modules and so $\gr_\m(\pi_2^\vee) \onto \gr_\m(\Theta_n^\vee)$.
By using our previous work \cite[Lemma~\ref{bhhms4:lem:isom-modcI}]{BHHMS4} we can determine $\pi[\m^n]$, and hence $\Theta_n$, completely explicitly as an $I$-representation, provided $\brho$ is sufficiently generic relative to $n$.
For $n$ sufficiently large (in fact, $n = i_0+4 \le f+4$ suffices) a computation shows that $\gr_\m(\Theta_n^\vee)/\o\m^3 \cong N_2'/\o\m^3$ and (a) follows.

For part (b) we use some filtered and graded techniques.
We first have an exact sequence of filtered $\Lambda$-modules,
\begin{equation}\label{eq:1st-seq}
  0 \to C \to \pi_2^\vee/\m^3 \to \pi^\vee/\m^3 \to \pi_1^\vee/\m^3 \to 0,
\end{equation}
where $C = \coker(\Tor_1^\Lambda(\Lambda/\m^3,\pi^\vee) \to \Tor_1^\Lambda(\Lambda/\m^3,\pi_1^\vee))$.
On the other hand, the exact sequence $0 \to \gr_{F}(\pi_2^\vee) \to \gr_{\m}(\pi^\vee) \to \gr_{\m}(\pi_1^\vee) \to 0$ of graded $\gr_\m(\Lambda)$-modules, where $F$ denotes again the induced filtration on $\pi_2^\vee$, %
gives rise to the following exact sequence:
\begin{equation}\label{eq:2nd-seq}
  0 \to C' \to \gr_{F}(\pi_2^\vee)/\o\m^3 \to \gr_{\m}(\pi^\vee)/\o\m^3 \to \gr_{\m}(\pi_1^\vee)/\o\m^3 \to 0,
\end{equation}
where $C'$ is an analogous cokernel of graded modules, cf.\ \eqref{eq:coker-Tor1-gr}.

We now compare dimensions of corresponding terms in the two exact sequences.
By a subtle spectral sequence argument we see that $\gr(C)$ is a subquotient of $C'$ (for a suitable filtration on $C$), so
\begin{equation}\label{eq:1st-term}
  \dim_\F(C) \le \dim_\F(C').
\end{equation}
On the other hand,
\begin{equation}\label{eq:2nd-term}
  \dim_\F(\pi_2^\vee/\m^3) = \dim_\F(\gr_\m(\pi_2^\vee)/\o\m^3) \ge \dim_\F(N_2'/\o\m^3) = \dim_\F(\gr_{F}(\pi_2^\vee)/\o\m^3),
\end{equation}
where the inequality results from (a) and the last equality from~\eqref{eq:gr-subquot}.
As the third (resp.\ fourth) nonzero terms in~\eqref{eq:1st-seq} and \eqref{eq:2nd-seq} evidently have the same dimensions, we deduce that $\dim_\F(C) - \dim_\F(\pi_2^\vee/\m^3) = \dim_\F(C') - \dim_\F(\gr_{F}(\pi_2^\vee)/\o\m^3)$, hence equality holds in \eqref{eq:1st-term} and \eqref{eq:2nd-term}, so (b) follows (using (a)).

For part (c), we start with the map of graded modules $f : N_2' \onto \gr_\m(\pi_2^\vee)/\o\m^3$ from (b).
By showing that $N_2'$ admits a presentation of the form $\gr(\Lambda)(1)^{\oplus i} \oplus \gr(\Lambda)(2)^{\oplus j} \to \gr(\Lambda)^{\oplus k} \to N_2' \to 0$ for some integers $i,j,k \ge 0$ we can lift $f$ to $\wt f : N_2' \to \gr_\m(\pi_2^\vee)$.
The map $\wt f$ is surjective by Nakayama's lemma and injective by a computation of cycles, using that $N_2'$ is Cohen--Macaulay (and computing cycles using~\eqref{eq:gr-subquot}).

\subsection{Sketch proof of Theorem~\ref{thm:intro-pi'-K1} and Theorem~\ref{thm:intro-structure-pi'}}\label{sec:sketchintro2}

Part (i) of Theorem~\ref{thm:intro-structure-pi'} follows directly from Theorem~\ref{thm:intro-gr-pi'}, by evaluating both sides in degree 0.
Part (ii) follows relatively easily by building on part (i), using that every irreducible $\GL_2(\cO_K)$-representation has nonzero $I_1$-invariants.

The proof of Theorem~\ref{thm:intro-pi'-K1} in subsection~\ref{sec:k1-quotient} is technically the most involved argument of this paper.
Again the essential difficulty is when $\pi' \defeq \pi_2$ is a quotient of $\pi$, which we assume from now on.

It is relatively straightforward to understand the right-hand side of Theorem~\ref{thm:intro-pi'-K1} as $\GL_2(\cO_K)$-representation, namely $D_{i_0} \defeq D_0(\brho^\ss)_{i_0+1} \oplus_{D_0(\brho)_{i_0+1}} (D_0(\brho)/D_0(\brho)_{\le i_0})$ (as $i_0' = f$), and to relate it to $\pi_2^{K_1}$:
\begin{enumerate}[label=(\alph*)]
\item $D_{i_0}$ is multiplicity free;
\item $D_{i_0} \into \pi_2^{K_1}$ as $\GL_2(\cO_K)$-representations;
\item the embedding in (b) induces isomorphisms on $I_1$-invariants and $\GL_2(\cO_K)$-socles.
\end{enumerate}
(Here we use Theorem~\ref{thm:intro-structure-pi'}(i) and (ii) for (c).)

The main thrust for showing that the embedding in (b) is an isomorphism is the fact that $\pi_2[\m^3]$ is multiplicity free (which follows from Theorem~\ref{thm:intro-gr-pi'}).

Suppose that the embedding in (b) is not an isomorphism.
Take a minimal subrepresentation $V'$ of $\pi_2^{K_1}$ not contained in $D_{i_0}$.
Then $\tau'\defeq V'/(V'\cap D_{i_0})$ is the irreducible cosocle of $V'$, and one can show that $\tau' \in W(\brho^\ss)$. 
Note that $[V':\tau'] \in \{1,2\}$ by (a) and that $V' \not\cong \tau'$ by (c).

Suppose first that $[V':\tau'] = 1$, and to simplify notation we also assume that $f = 1$ (the argument easily generalizes to $f>1$).
In a first step we show that the radical $\rad(V')$ of $V'$ is semisimple (Lemma~\ref{lem:D=A+B}(iii)).
By known $\Ext^1$ results there are $3$ possibilities: (i) $\rad(V') \cong \mu_0^-(\tau')$ or (ii) $\rad(V') \cong \mu_0^+(\tau')$ or (iii) $\rad(V') \cong \mu_0^-(\tau') \oplus \mu_0^+(\tau')$ for certain Serre weights $\mu_0^{\pm}(\tau')$ associated to the Serre weight $\tau'$ as defined in \S~\ref{sec:more-gamma-repr}.
Case (i) is ruled out by the equality of $I_1$-invariants in (c) above.
In cases (ii) and (iii) we enlarge $V'$ slightly to $\wt{V'}$, where $V' \subset \wt{V'} \subset \pi_2^{K_1}$.
In either case $\wt{V'}$ (but not $V'$!)\ is a quotient of $\Ind_I^{\GL_2(\cO_K)} \cW$, where $\cW$ is the $I$-representation
\begin{equation*}
  \cW \ \cong\ \left( \vcenter{\xymatrix@R-2pc@C-0.85pc{
        \chi_{\mu_0^-(\tau')} \ar@{-}[rd]\\
        & \chi_{\tau'} \\
        \chi_{\mu_0^+(\tau')} \ar@{-}[ru]}}
  \right),
\end{equation*}
where we write $\chi_\sigma \defeq \sigma^{I_1}$ for any Serre weight $\sigma$.
By Frobenius reciprocity we get a nonzero map $\cW \to \pi_2[\m^2]$ whose image is not contained in $D_{i_0}$.
On the other hand, we show that $\tau' \in \JH(D_{i_0})$ and by following the same reasoning as above we obtain another nonzero map $\cW \to \pi_2[\m^2]$ whose image is contained in $D_{i_0}$.
This shows that $[\pi_2[\m^2]:\chi_{\tau'}] \ge 2$, contradicting the multiplicity freeness of $\pi_2[\m^2]$.

The case where $[V':\tau'] = 2$ is similar but more involved, using multiplicity freeness of $\pi_2[\m^3]$.

\color{black}

\textbf{Acknowledgements}: %
We thank the referee for a very careful reading, which in particular led us to correct a number of typos and small mistakes, as well as improve the exposition. 
Y.\;H.\ is \ partially  supported by CAS Project for Young Scientists in Basic Research, Grant No.~YSBR-033, National Natural Science Foundation of China Grants 12288201 and 12425103, National Center for Mathematics and Interdisciplinary Sciences and Hua Loo-Keng Key Laboratory of Mathematics, Chinese Academy of Sciences. F.\;H.\ is partially supported by an NSERC grant, a grant from Simons Foundation International [SFI-MPS-SFM-00006210, FH], and a Visiting Professorship at Korea Institute for Advanced Study. S.\;M.\ and B.\;S.\ are partially supported by the Institut Universitaire de France.
F.\;H.\ thanks the Korea Institute for Advanced Study for the excellent working conditions provided.

\subsection{Notation and preliminaries}
\label{sec:notation}
\label{sec:preliminaries}

We normalize local class field theory so that uniformizers correspond to geometric Frobenius elements. We fix an embedding $\kappa_0:\F_q\into \F$ and let $\kappa_j\defeq \kappa_0\circ\varphi^j$, where $\varphi$ is the arithmetic Frobenius on $\F_q$. Given $J\subset\{0,\dots,f-1\}$ we define $J^c\defeq \{0,1,\dots,f-1\} \setminus J$. We let $I\defeq \begin{pmatrix}\cO_K^\times&\cO_K\\p\cO_K&\cO_K^\times\end{pmatrix} \subset \GL_2(\cO_K)$ denote the (upper) Iwahori subgroup of $\GL_2(K)$, $I_1$ the pro-$p$ radical of $I$, $Z_1$ the center of $I_1$, and $K_1\defeq 1+p{\rm M}_2(\cO_K)\subset I_1$. We let $\Gamma\defeq \GL_2(\F_q)\cong \GL_2(\cO_K)/K_1$.

Let $\brho:\Gal(\o K/K)\ra \GL_2(\F)$ be a continuous representation. We will say that $\rhobar$ is \emph{$n$-generic} for some integer $n \ge 0$ if, up to twist, $\rhobar|_{I_K}^{\ss}\not\cong \omega^{\pm 1}\oplus 1$ and either (using the notation of \S~\ref{sec:results})
\begin{equation}\label{eq:rhobar-generic-red}
  \rhobar|_{I_K} \cong \left(\begin{matrix}\omega_f^{\sum_{j=0}^{f-1} (r_j+1) p^j} & * \\ & 1\end{matrix}\right) \qquad \text{with $n \leq r_j \leq p-3-n$ for all $0 \le j \le f-1$}
\end{equation}
or 
\begin{equation}\label{eq:rhobar-generic-irred}
  \rhobar|_{I_K} \cong \left(\begin{matrix}\omega_{2f}^{\sum_{j=0}^{f-1} (r_j+1) p^j} & \\ & \omega_{2f}^{p^f(\text{same})}\end{matrix}\right) \qquad \text{with\ }
  \begin{cases}
    n \leq r_j \leq p-3-n &\text{for $0 < j \le f-1$,} \\ n+1 \leq r_0 \leq p-2-n & \text{for $j = 0$.}
  \end{cases}
\end{equation}
In particular, if $\rhobar$ is $n$-generic then it is  $n$-generic in the sense of \cite[Def.\ 2.3.4]{BHHMS1}, and $\rhobar$ is $0$-generic precisely when $\rhobar$ is generic in the sense of \cite[Def.~11.7]{BP} (note that the condition $\rhobar|_{I_K}^{\ss}\not\cong \omega^{\pm 1}\oplus 1$ precisely rules out that $(r_0,\dots,r_{f-1})\in\{(0,\dots,0),(p-3,\dots,p-3)\}$ when $\rhobar$ is reducible).

Attached to a $0$-generic $\rhobar$ we have a set $W(\rhobar)$ of Serre weights, i.e.~irreducible representations of $\Gamma$ over $\F$, defined in \cite[\S~3]{BDJ}, and a  multiplicity-free finite length $\Gamma$-representation $D_0(\rhobar)$ over $\F$, defined in \cite[\S~13]{BP}, which is of the form $D_0(\rhobar)=\bigoplus_{\tau\in W(\rhobar)}D_{0,\tau}(\rhobar)$, where each $D_{0,\tau}(\rhobar)$ is indecomposable (and multiplicity free) with socle the Serre weight $\tau$ (\cite[\S~15]{BP}).

Suppose that $\rhobar$ is 0-generic. Recall the set $\P$ parametrizing $\JH(D_0(\rhobar)^{I_1})$, see \cite[\S~4]{breuil-buzzati} (and denoted there by $\mathscr{PD}$, resp.\ $\mathscr{PID}$, if $\rhobar$ is reducible, resp.\ irreducible). Recall also the subset $\D \subset \P$ parametrizing (the $I_1$-invariants of) the set of Serre weights in $W(\brho)$ (denoted in \emph{loc.~cit.} by $\D$ or $\mathscr{ID}$ if $\rhobar$ is reducible or irreducible respectively).
We let $\D^\ss\subset \P^\ss$ denote the corresponding sets for the semisimplification $\rhobar^\ss$ of $\rhobar$, so $\P \subset \P^\ss$ and $\D \subset \D^\ss$.

Since we will be using this many times, we recall more precisely that if $\brho$ is reducible, $\P^\ss$ denotes the set of $f$-tuples $(\lambda_0(x_0),\dots,\lambda_{f-1}(x_{f-1}))$ such that:
\begin{enumerate}
\item $\lambda_j(x_j) \in \{x_j,x_j+1,x_j+2,p-3-x_j,p-2-x_j,p-1-x_j\}$;
\item if $\lambda_j(x_j) \in \{x_j,x_j+1,x_j+2\}$, then $\lambda_{j+1}(x_{j+1}) \in \{x_{j+1},x_{j+1}+2,p-2-x_{j+1}\}$;
\item if $\lambda_j(x_j) \in \{p-3-x_j,p-2-x_j,p-1-x_j\}$, then $\lambda_{j+1}(x_{j+1}) \in \{x_{j+1}+1,p-3-x_{j+1},p-1-x_{j+1}\}$
\end{enumerate}
and $\D^\ss$ is the subset such that $\lambda_j(x_j) \in \{x_j,x_j+1,p-3-x_j,p-2-x_j\}$.
Moreover, there exists a unique subset $J_{\rhobar} \subset \{0,\dots,f-1\}$ such that
\begin{align}
  \D &= \Big\{\lambda \in \D^\ss : \lambda_j(x_j) \in \{x_j+1,p-3-x_j\} \Rightarrow j \in J_{\rhobar}\Big\},\notag \\
  \P &= \Big\{\lambda \in \P^\ss : \lambda_j(x_j) \in \{x_j+2,p-3-x_j\} \Rightarrow j \in J_{\rhobar}\Big\}.\label{eq:P}
\end{align}
In particular, $|W(\brho)| = 2^{|J_{\brho}|}$.

For $\lambda\in\mathscr{P}$ we denote by $\chi_\lambda$ the character of $H$ corresponding to $\lambda$.
(More precisely, in \cite[\S~4]{breuil-buzzati} a Serre weight $\sigma_\lambda$ is associated to $\lambda \in \P$ and $\chi_\lambda$ is the action of $H=I/I_1$ on the $1$-dimensional subspace $\sigma_\lambda^{I_1}$.)
Set
\begin{equation}\label{eq:J-lambda}J_{\lambda}\defeq \{j\in\{0,\dots,f-1\}: \lambda_j(x_j)\in \{x_j+1,x_j+2,p-3-x_j\}\}\end{equation}
and let $\ell(\lambda) \defeq  |J_\lambda|.$
By \cite[\S~11]{BP} the map $\lambda \mapsto J_\lambda$ induces a bijection between $\mathscr{D}^\ss$ and the set of subsets of $\{0,\dots,f-1\}$.
Sometimes we will abuse notation and write $J_\tau \defeq  J_\lambda$ and $\ell(\tau) \defeq  \ell(\lambda)$ if $\tau \in W(\brho^\ss)$ is parametrized by $\lambda \in \D^\ss$.
Given $\lambda\in \D^\ss$ with corresponding subset $J=J_\lambda\subset\{0,\dots,f-1\}$ we write $\delta(\lambda)\in \D^\ss$ for the $f$-tuple defined by $\delta(\lambda)_j\defeq \lambda_{j+1}$ for all $j\in\{0,\dots,f-1\}$, and $\delta(J)\subset\{0,\dots,f-1\}$ for the subset corresponding to $\delta(\lambda)$.

As in \cite[\S~1]{BP}, given $f$ integers $r_0,\dots,r_{f-1}\in\{0,\dots,p-1\}$ we denote by $(r_0,\dots,r_{f-1})$ the Serre weight
\begin{equation*}
\mathrm{Sym}^{r_0}\F^2\otimes_{\F}(\mathrm{Sym}^{r_1}\F^2)^{\rm Fr}\otimes\cdots\otimes_{\F}(\mathrm{Sym}^{r_{f-1}}\F^2)^{{\rm Fr}^{f-1}},
\end{equation*}
where $\GL_2(\Fq)$ acts on $(\mathrm{Sym}^{r_j}\F^2)^{{\rm Fr}^j}$ via $\kappa_j:\Fq\hookrightarrow \F$.
Following \cite[\S~2]{HuWang2}, we say that a Serre weight is \emph{$m$-generic} for some integer $m\geq 0$ if, up to twist, $\sigma\cong (r_0,\dots,r_{f-1})$, where $m\leq r_j\leq p-2-m$ for all $j\in \{0,\dots,f-1\}$.
We say that an $\F$-valued character $\chi$ of $I$ is \emph{$m$-generic} if $\chi=\sigma^{I_1}$ for some $m$-generic Serre weight $\sigma$. 
For any smooth character $\chi : I \to \F\s$ we define $\chi^s\defeq \chi(\Pi (\cdot) \Pi^{-1})$ with $\Pi\defeq \begin{pmatrix}0&1\\ p&0\end{pmatrix}$. 
If $\sigma$ is a Serre weight, we write $\chi_\sigma$ for the character of $I/I_1$ on $\sigma^{I_1}$ and $\sigma^{[s]}$ for the unique Serre weight distinct from $\sigma$ such that $\chi_{\sigma^{[s]}}=\chi_{\sigma}^s$.
Finally, if $\chi,\chi': I\ra\F^\times$ are smooth characters such that $\Ext^1_{I/Z_1}(\chi',\chi)\neq 0$ we let $E_{\chi,\chi'}$ denote the unique nonsplit extension of $\chi'$ by $\chi$, i.e.\
$0\ra\chi\ra E_{\chi,\chi'}\ra\chi'\ra 0$.
(The uniqueness follows from \cite[Lemme~2.4]{yongquan-algebra}.)

Let $\fm_{K_1}$ denote the maximal ideal of {the Iwasawa algebra} $\F\bbra{K_1/Z_1}$ and let \[\wt{\Gamma}\defeq \F\bbra{\GL_2(\cO_K)/Z_1}/\fm_{K_1}^2\] (a finite-dimensional $\F$-algebra). 
We use the terminology ``$\wt{\Gamma}$-representations'' and ``$\wt{\Gamma}$-modules'' interchangeably.

We write $\wt{D}_{0}(\brho)$ for the finite-dimensional $\wt{\Gamma}$-representation over $\F$ constructed in  \cite[Prop.~4.3]{HuWang2}. 
It is the unique (up to isomorphism) $\wt{\Gamma}$-representation which is maximal with respect to the following two properties:
\begin{enumerate}
\item $\soc_{\wt{\Gamma}}\wt{D}_0(\brho)\cong\bigoplus_{\sigma\in W(\brho)}\sigma$;
\item any Serre weight of $W(\brho)$ occurs in $\wt{D}_0(\brho)$ with multiplicity one.
\end{enumerate}

Let $\Lambda \defeq  \F\bbra{I_1/Z_1}$, a complete noetherian local ring with maximal ideal $\m \defeq \m_{I_1/Z_1}$, and let $\gr(\Lambda) \defeq \gr_\m(\Lambda) = \bigoplus_{n\geq 0}\m^n/\m^{n+1}$ be the associated graded ring of $\Lambda$ with respect to the $\fm$-adic filtration on $\Lambda$. The rings $\Lambda$ and $\gr(\Lambda)$ are Auslander regular (see \cite[Thm.~5.3.4]{BHHMS1} with \cite[Thm.~III.2.2.5]{LiOy}). Recall (\cite[\S~\ref{bhhms2:phigamma}]{BHHMS2}) that we have an isomorphism of (noncommutative) algebras
\begin{equation}\label{eq:grlambda}
\gr(\Lambda)\cong \bigotimes_{j\in\{0,\dots,f-1\}}\F\langle y_j,z_j,h_j\rangle
\end{equation}
with relations $[y_j,z_j]=h_j$, $[h_j,z_i]=[y_i,h_j]=0$ for all $i,j\in\{0,\dots,f-1\}$.
We use increasing filtrations throughout, i.e.\ $F_n \Lambda = \m^{-n}$ for $n \le 0$, and the degrees of $y_j$ and $z_j$ (resp.\ $h_j$) are $-1$ (resp.\ $-2$).
Define the graded ideal $J \defeq  (h_j,y_jz_j : 0 \le j \le f-1)$ of $\gr(\Lambda)$.
As in \cite[\S~\ref{bhhms2:gl2}]{BHHMS2} we define
\[R \defeq  \gr(\Lambda)/(h_j : 0 \le j \le f-1) \cong \F[y_j,z_j : 0 \le j \le f-1]\]
which is the largest commutative quotient of $\gr(\Lambda)$. We also define the following quotient of $R$:
\[\o R \defeq  \gr(\Lambda)/J \cong R/(y_jz_j : 0 \le j \le f-1).\]

We recall from \cite[Def.~3.57]{BHHMS2} that given $\lambda\in\mathscr{P}$ we have an associated ideal $\mathfrak{a}(\lambda)=(t_0,\dots,t_{f-1})$ of $R$, where the $t_j = t_j(\lambda)$ are defined as follows:
\begin{equation}
\label{eq:id:al}
t_j\defeq \left\{\begin{array}{llll}z_j& {\rm if} &\lambda_j(x_j)\in\{x_j,p-3-x_j\} \ \mathrm{and}\ j\in J_{\brho}\\
y_j& {\rm if} &\lambda_j(x_j)\in\{x_j+2,p-1-x_j\} \ \mathrm{and}\ j\in J_{\brho}\\
y_jz_j & {\rm if} &\lambda_j(x_j)\in \{x_j,p-1-x_j\} \ \mathrm{and}\ j\notin J_{\brho}\\
y_jz_j& {\rm if} &\lambda_j(x_j)\in \{x_j+1,p-2-x_j\}.\end{array}\right.
\end{equation}
Note that $(y_jz_j : 0 \le j \le f-1) \subset \fa(\lambda)$, so we often think of $\fa(\lambda)$ as ideal of $\o R$.

Let $H\defeq \begin{pmatrix}\F_q^\times&0\\0&\F_q^\times\end{pmatrix}\cong I/I_1$.
We write $\alpha_j:H\ra\F^\times$ for the character defined by $\begin{pmatrix}a&0\\0&d\end{pmatrix}\mapsto \kappa_j(ad^{-1})$.
We recall that for any $j\in\{0,\dots,f-1\}$ the element $y_j$ (resp.~$z_j$, resp.~$h_j$) is an $H$-eigenvector with associated eigencharacter $\alpha_j$ (resp.~$\alpha_j^{-1}$, resp.~the trivial character).

Suppose that $H'$ is a compact $p$-adic analytic group and that $\pi_1$, $\pi_2$ are smooth representations of $H'$ over $\F$.
We write $\Ext^i_{H'}(\pi_1,\pi_2)$ for the $i$-th Ext group computed in the category of smooth representations of $H'$ over $\F$.
Dually, the functors $\Tor_i^{\F\bbra{H'}}(\pi_1^\vee,\pi_2^\vee)$ and $\Ext^i_{\F\bbra{H'}}(\pi_1^\vee,\pi_2^\vee)$ are computed in the abelian category of pseudocompact $\F\bbra{H'}$-modules.
(See for example \cite[\S~2]{emerton-ordII}.)
If $\sigma$ has finite length, we write $\JH(\sigma)$ for its set of irreducible constituents up to isomorphism.

Throughout this paper, if $R$ is a filtered (resp.\ graded) ring, a morphism of filtered (resp.\ graded) $R$-modules $f:M\ra N$ will always be a \emph{filtered (resp.\ graded) morphism of degree zero}, i.e.~satisfying $f(M_i)\subset N_i$ for all $i\in \Z$. 
If $R$ is any ring and $M$ any left $R$-module, we recall that $\Ext^{i}_R(M,R)$ for $i\in \Z_{\ge 0}$ is a right $R$-module (for $i=0$ the right $R$-action is given by $(f r)(m)\defeq f(m)r$ for $r\in R$, $f\in \Hom_R(M,R)$ and $m\in M$) and we use the notation $\EE^i_R(M)\defeq \Ext^{i}_R(M,R)$. %

\section{Preliminaries}
\label{sec:prelims}
We establish structural results on finite-dimensional smooth mod $p$ representations of $\GL_2(\cO_K)$ which will extensively be used in section \S~\ref{sec:subquot-non-ss}.

If $\sigma$ is a Serre weight, we write $\Proj_\Gamma\sigma$ (resp.~$\Inj_\Gamma\sigma$) for a projective cover (resp.~injective envelope) of $\sigma$ in the category of $\F[\Gamma]$-modules.
The objects $\Proj_{\wt{\Gamma}}\sigma$ and $\Inj_{\wt{\Gamma}}\sigma$ are defined similarly.
 We remark that $\Proj_\Gamma\sigma \cong \Inj_\Gamma\sigma$, as $\Gamma$ is a finite group, but $\Proj_{\wt{\Gamma}}\sigma \not\cong \Inj_{\wt{\Gamma}}\sigma$ by \cite[(2.6), Prop.~2.12]{HuWang2}.

 \subsection{Some \texorpdfstring{$\Gamma$}{Gamma}-representations}
\label{sec:EG-Gamma-rep}

 We collect a number of results on the combinatorics of Serre weights and injective envelopes.

Recall from \cite[\S~3]{BP} that given a Serre weight $\sigma$, there is an injective parametrization  $\JH(\Inj_\Gamma\sigma)\into \cI$ (where $\cI\defeq \cI(x_0,\dots,x_{f-1})$ is defined in \cite[\S~4]{BP}), which is bijective if $\sigma$ is $1$-generic.   For $\mu\in\cI$, we let $\mu(\sigma)$ be the Serre weight parametrized by $\mu$ as defined in \cite[(2.2)]{HuWang2} (recall that $\mu(\sigma)$ remains undefined if the formula in \emph{loc.~cit.}~is not well defined).
 We say $\tau_1,\tau_2\in \JH(\Inj_\Gamma\sigma)$ are \emph{compatible} (relative to $\sigma$) if the corresponding elements $\mu_1,\mu_2\in \cI$ are compatible  in the sense of \cite[Def.~4.10]{BP}.  We write $\mu_2\preceq \mu_1$ if $\cS(\mu_2)\subset \cS(\mu_1)$ and $\mu_2$ is compatible with $\mu_1$, where $\cS(-)\subset \{0,\dots,f-1\}$ is the subset defined in \cite[\S~4]{BP}. 
{It is easy to see that the set $\{\mu_2\in\cI: \mu_2\preceq \mu_1\}$ is in bijection with all subsets of $\cS(\mu_1)$,  via $\mu_2 \mapsto \cS(\mu_2)$, hence has cardinality $2^{|\cS(\mu_1)|}$.}

Recall from \cite[Cor.~3.12]{BP} that given $\tau\in \JH(\Inj_{\Gamma}\sigma)$ there exists a unique finite-dimensional $\Gamma$-representation $I(\sigma,\tau)$ such that $\soc_{\Gamma}I(\sigma,\tau)=\sigma$, $\cosoc_{\Gamma}I(\sigma,\tau)=\tau$ and $[I(\sigma,\tau):\sigma]=1$. 
By \cite[Cor.~4.11]{BP}, $\tau'\in \JH(I(\sigma,\tau))$ if and only if  $\tau' = \mu'(\sigma)$ with $\mu'\preceq \mu$, where $\tau = \mu(\sigma)$.

\begin{lem}\label{lem:I-sigma-tau-modular}
Assume that $\rhobar$ is $0$-generic. 
Let $\sigma,\tau\in W(\brho^{\ss})$.  Then $\tau\in\JH(\Inj_{\Gamma}\sigma)$. Moreover, $|\JH(I(\sigma,\tau))|=2^{|J_{\sigma}\mathbin\Delta J_{\tau}|}$ and 
\begin{align}
\JH(I(\sigma,\tau))&=\{\tau'\in W(\brho^{\ss}): J_{\sigma}\cap J_{\tau}\subset J_{\tau'}\subset J_{\sigma}\cup J_{\tau}\} \label{eq:JH-I-sigma-tau}\\
  &= \{\tau'\in W(\brho^{\ss}): J_{\sigma}\mathbin\Delta J_{\tau'}\subset J_{\sigma}\mathbin\Delta J_{\tau}\}.\label{eq:JH-I-sigma-tau2}
\end{align}
In particular, $\Ext^1_\Gamma(\tau,\sigma) \ne 0$ if and only if $|J_{\sigma}\mathbin\Delta J_{\tau}| = 1$.
\end{lem}

\begin{proof}
First, from \cite[Prop.~2.24]{HuWang} we deduce that $\tau\in \JH(\Inj_{\Gamma}\sigma)$ and that any irreducible constituent of $I(\sigma,\tau)$ is also an element of $W(\brho^{\rm ss})$. %

Let $\lambda_{\sigma}, \lambda_{\tau}\in \D^\ss$ correspond to $\sigma,\tau$, respectively. Then there exists a unique element  $\mu\in \mathcal{I}$ such that $\lambda_{\tau}=\mu\circ \lambda_{\sigma}$.  
By above we know that 
\begin{equation}\label{eq:JH-I-sigma-tau3}
  \JH(I(\sigma,\tau)) = \{ \mu'(\sigma) : \mu' \preceq \mu \text{ in }\cI \}.
\end{equation}
We claim that $\mu'(\sigma)$ is well defined for any $\mu'\preceq \mu$; this will imply that $\JH(I(\sigma,\tau))$  has cardinality $2^{|\cS(\mu)|}$. Given $\mu'\preceq \mu$ in $\cI$, one checks using the table in the proof of \cite[Lemma 2.6]{HuWang2} that $\mu'\circ\lambda_{\sigma}\in \D^{\rm ss}$, so that $\mu'(\sigma)\in W(\brho^{\rm ss})$ is the Serre weight  corresponding to $\mu'\circ\lambda_{\sigma}$ (which is well defined, as $\brho$ is $0$-generic), proving the claim.

From \cite[Lemma 2.6(i)]{HuWang2} we deduce that $\cS(\lambda_\tau) = \cS(\mu) \mathbin\Delta \cS(\lambda_\sigma)$, or equivalently $\cS(\mu) = J_\sigma \mathbin\Delta J_\tau$ (as $J_{\lambda} = \cS(\lambda)$ for all $\lambda \in \D^\ss$), which implies $|\JH(I(\sigma,\tau))|=2^{|J_{\sigma}\mathbin\Delta J_{\tau}|}$.
For the same reason, if $\tau'\in \JH(I(\sigma,\tau))$ corresponds to $\mu'\preceq\mu$ in $\cI$ (i.e. $\tau'=\mu'(\sigma)$), then $ \cS(\mu') = J_\sigma\mathbin\Delta J_{\tau'}$, so that~\eqref{eq:JH-I-sigma-tau2} (equivalently~\eqref{eq:JH-I-sigma-tau}) follows from~\eqref{eq:JH-I-sigma-tau3}.
The last assertion follows from $\Ext^1_\Gamma(\tau,\sigma) \ne 0$ if and only if $I(\sigma,\tau)$ has length 2.

\end{proof}

Recall from \cite[\S~\ref{bhhms4:sec:preliminaries-fin-len}]{BHHMS4} and \cite[Rk.~\ref{bhhms4:rmk:JH:conventions}]{BHHMS4} that given a character $\chi:I\ra\F^{\times}$ {with $\chi\ne\chi^s$},
there is an injective parametrization $\JH(\Ind_I^{\GL_2(\cO_K)}\chi) \into \mathcal{P}$ (where $\mathcal P \defeq \mathcal{P}(x_0,\dots,x_{f-1})$ is defined in \cite[\S~2]{BP}), which is bijective if $\chi$ is $1$-generic.
Moreover, $\mathcal{P}$ is in bijection with the subsets of $\{0,\dots,f-1\}$ via the map $\xi \mapsto \mathcal{S}(\xi)$ of \cite[eq.~\eqref{bhhms4:eq:S-xi1}]{BHHMS4}.
With this parametrization, the socle of $\Ind_I^{\GL_2(\cO_K)}\chi$ corresponds to the empty subset.
For $J \subset \{0,1,\dots,f-1\}$ let $\sigma_J$ denote the constituent of $\Ind_I^{\GL_2(\cO_K)} \chi$ parametrized by $\xi \in \mathcal P$ such that $\cS(\xi) = J$, if such a $\xi$ exists.
Finally, if $\sigma\in \JH(\Ind_I^{\GL_2(\cO_K)}\chi)$ is parametrized by $\xi \in \mathcal P$ we write $\cS(\sigma)\defeq\cS(\xi)$.

\begin{lem}\label{lem:I-sigma-tau-princ-series}
  Assume that $\chi$ is $1$-generic. 
  Let $J, J' \subset \{0,1,\dots,f-1\}$.  
  Then $\sigma_{J'}\in\JH(\Inj_{\Gamma}\sigma_J)$. 
  Moreover, $|\JH(I(\sigma_J,\sigma_{J'}))|=2^{|J\mathbin\Delta J'|}$ and 
  \begin{align*}
    \JH(I(\sigma_J,\sigma_{J'}))&=\{\sigma_{J''}: J\cap J'\subset J''\subset J\cup J'\} \\
    &= \{\sigma_{J''}: J\mathbin\Delta J''\subset J\mathbin\Delta J'\}.\notag
  \end{align*}
  In particular, $\Ext^1_\Gamma(\sigma_{J'},\sigma_J) \ne 0$ if and only if $|J\mathbin\Delta J'| = 1$.
\end{lem}

\begin{proof}
From \cite[Cor.~2.22]{HuWang} applied to the principal series tame type $\Ind_I^{\GL_2(\cO_K)}[\chi]$ we deduce that  $I(\sigma_J,\sigma_{J'})$ exists and that $\JH(I(\sigma_J,\sigma_{J'}))\subseteq \JH(\Ind_{I}^{\GL_2(\cO_K)}\chi)$.
As $\chi$ is $1$-generic, $|\JH(\Ind_I^{\GL_2(\cO_K)}\chi)|$ has cardinality $2^f$ by above, namely $\sigma_J$ is well defined for any subset $J\subseteq \{0,1,\dots,f-1\}$. 
The proof is then essentially identical to the proof of Lemma \ref{lem:I-sigma-tau-modular}, except that the bijection of $W(\brho^\ss)$ with $\{0,1,\dots,f-1\}$, $\sigma \mapsto J_\sigma$ is replaced by the bijection of $\JH(\Ind_I^{\GL_2(\cO_K)}\chi)$ with $\{0,1,\dots,f-1\}$, $\sigma_J \mapsto J$.

\end{proof}

\begin{lem}\label{lem:comp-JH}
Let $\sigma$ be a $1$-generic Serre weight and $\tau_1,\tau_2\in \JH(\Inj_{\Gamma}\sigma)$. Assume that $\tau_1,\tau_2$ are compatible (relative to $\sigma$). Then $\tau_2\in\JH(\Inj_{\Gamma}\tau_1)$ and $\JH(I(\tau_1,\tau_2))\subset \JH(\Inj_{\Gamma}\sigma)$.   
\end{lem}

\begin{proof}
Let $\mu_1,\mu_2\in \cI$  correspond to $\tau_1,\tau_2$ respectively. Since $\mu_1,\mu_2$ are compatible, one checks that there exists a unique element in $\cI$, denoted by $\mu_1\cap \mu_2$ (resp.~$\mu_1\cup \mu_2$), which is compatible with $\mu_1$ and $\mu_2$ such that $\cS(\mu_1\cap \mu_2)=\cS(\mu_1)\cap \cS(\mu_2)$ (resp.~$\cS(\mu_1\cup \mu_2)=\cS(\mu_1)\cup \cS(\mu_2)$). (See also \cite[\S~12]{BP} for the explicit construction of $\mu_1\cap \mu_2$.)

Let $\tau_{0}, \tau_{3}\in \JH(\Inj_\Gamma\sigma)$ correspond to $\mu_0\defeq\mu_1\cap \mu_2$, $\mu_3\defeq\mu_1\cup\mu_2\in\cI$ respectively.   
We first assume $\tau_0=\sigma$, equivalently $\cS(\mu_1)\cap \cS(\mu_2)=\emptyset$.  We have $\tau_1,\tau_2\in \JH(I(\sigma,\tau_3))$ by  \cite[Cor.~4.11]{BP}, and the genericity assumption on $\sigma$ implies that $I(\sigma,\tau_3)$ has length $2^{|\cS(\mu_3)|}$. We deduce from \cite[Lemma~2.20(iii)]{HuWang} that $\JH(I(\sigma,\tau_3))=\JH(I(\tau_1,\tau_2))$ ($\tau_2=\tau_1^c$ with the notation used there).

To treat the general case, we note that $I(\tau_0,\tau_3)$ exists and $\tau_1,\tau_2\in \JH(I(\tau_0,\tau_3))$ by \cite[Cor.~4.11]{BP},  so we may view $\tau_1,\tau_2$ as Jordan--H\"older factors of $\Inj_\Gamma\tau_0$.  
Let \[\lambda_1,\lambda_2,\lambda_3\in\cI(y_0,\dots,y_{f-1})\] be the element corresponding to $\tau_1,\tau_2,\tau_3\in\JH(\Inj_\Gamma\tau_0)$. 
 Using \cite[Lemmas~2.1, 2.7]{HuWang2} we get $\lambda_i\circ\mu_0=\mu_i$ for $i=1,2,3$.
By \cite[Lemma~2.6(i)]{HuWang2} we have $\cS(\mu_i)=\cS(\lambda_i)\mathbin\Delta\cS(\mu_0)$ or equivalently $\cS(\lambda_i)=\cS(\mu_i)\mathbin\Delta\cS(\mu_0)=\cS(\mu_i)\setminus \cS(\mu_0)$, 
so $\cS(\lambda_1)\cap \cS(\lambda_2)=\emptyset$ and $\cS(\lambda_1)\cup \cS(\lambda_2)=\cS(\lambda_3)$. 
Moreover, since $\mu_1,\mu_2,\mu_3$ are compatible,  $\lambda_1,\lambda_2,\lambda_3$ are also compatible by the table in the proof of \cite[Lemma~2.6]{HuWang2} (writing $\lambda_i = \mu_i \circ \mu_0^{-1}$, where $\mu_0^{-1}\in\cI$ is the unique element defined by demanding $\mu_0^{-1}\circ \mu_0=(x_0,\dots,x_{f-1})$), so that $\lambda_3=\lambda_1\cup \lambda_2$. 
Hence, by the previous paragraph we get $\JH(I(\tau_1,\tau_2))=\JH(I(\tau_0,\tau_3))$,  in particular $\JH(I(\tau_1,\tau_2))\subset \JH(\Inj_\Gamma\sigma)$ as $I(\tau_0,\tau_3)$ is a quotient of $I(\sigma,\tau_3)$.
\end{proof}

\subsection{More \texorpdfstring{$\Gamma$}{Gamma}-representations}
\label{sec:more-gamma-repr}

Recall from \cite[Def.~2.9]{HuWang2} that given $j\in\{0,\dots,f-1\}$ and $*\in \{+,-\}$ we define an $f$-tuple $\mu_j^*\in \bigoplus_{i=0}^{f-1}(\Z\pm x_i)$ as follows: if $f>1$ then $(\mu_j^*)_{j-1}(x_{j-1})\defeq p-2-x_{j-1}$, $(\mu_j^*)_{j}(x_{j})\defeq x_{j}\ast1$ and $(\mu_j^*)_{i}(x_i)\defeq x_i$ for $i\notin\{j-1,j\}$, while if $f = 1$ then $\mu_0^*(x_{0})\defeq p-2-(\ast 1)-x_0$.
If $\sigma$ is a $0$-generic Serre weight corresponding to a tuple $(s_0,\dots,s_{f-1})\in\{0,\dots,p-1\}^f$ we write $\mu^*_j(\sigma)$ for the Serre weight $\mu_j^*\big((s_0,\dots,s_{f-1})\big)\otimes\det^{e(\mu_j^*)(s_0,\dots,s_{f-1})}$, where $e(\mu_j^*)\in \Z\oplus \bigoplus_{i=0}^{f-1}\Z x_i$ is defined in \cite[\S~3]{BP}.
(Note that $\mu_j^-(\sigma)$ is undefined if $f \ge 2$ and $s_j = 0$ and $\mu_j^+(\sigma)$ is undefined if $f = 1$ and $s_j = p-2$.)

The following lemma is well known, but we state it for lack of convenient reference.

\begin{lem}\label{lem:mu-i+}
  Suppose that $\sigma = (r_0,\dots,r_{f-1}) \otimes \eta$ is any Serre weight such that $\mu_i^-(\sigma)$ is defined.
  If $f = 1$ we moreover suppose that $0 < r_0 < p-1$.
  Then the (unique up to isomorphism) nonsplit $\GL_2(\cO_K)/Z_1$-extension $0 \to \mu_i^-(\sigma) \to V \to \sigma \to 0$ is a quotient of $\Ind_I^{\GL_2(\cO_K)} \chi_{\sigma}$ {(hence is a $\Gamma$-representation)}, equivalently $\chi_{\sigma} \into V|_I$.
\end{lem}

\begin{proof}
  This follows from \cite[Thm.\ 2.4(iii) and Cor.\ 5.6(ii)]{BP}. 
\end{proof}

\begin{lem}\label{lem:soc-cosoc-sigma}
  Let $\sigma$ be a {1}-generic Serre weight.
  Let $Q$ be a quotient of $\Proj_\Gamma \sigma$ such that
  \begin{enumerate}
  \item $\soc_\Gamma(Q) \cong \sigma^{\oplus r}$ for some $r \ge 1$;
  \item $\rad_\Gamma(Q)/\soc_\Gamma(Q)$ is nonzero and does not admit $\sigma$ as a subquotient.
  \end{enumerate}
  Then $\rad_\Gamma(Q)/\soc_\Gamma(Q)$ is semisimple and there exists a subset $\cJ \subset \{0,1,\dots,f-1\}$ such that
  \begin{equation*}
    \rad_\Gamma(Q)/\soc_\Gamma(Q) \cong \bigoplus_{i \in \cJ} (\mu_i^+(\sigma) \oplus \mu_i^-(\sigma)).
  \end{equation*}
\end{lem}

\begin{proof}
{By the same argument as in the proof of \cite[Cor.\ 2.32]{HuWang2} (using \cite[Cor.\ 2.3]{HuWang2} for $\Gamma$-representations instead of \cite[Cor.\ 2.26]{HuWang2} for $\wt\Gamma$-representations), we prove that \[\rad_\Gamma(Q)/\soc_\Gamma(Q) \into \bigoplus_{i=0}^{f-1} (\mu_i^+(\sigma) \oplus \mu_i^-(\sigma))\] (in particular, it is semisimple). }
  It thus suffices to show that $\mu_i^+(\sigma) \in \JH(Q)$ if and only if $\mu_i^-(\sigma) \in \JH(Q)$.
  Note that $Q_\sigma$ surjects onto $Q$, where $Q_\sigma$ is the largest quotient of $\Proj_\Gamma \sigma/\rad_\Gamma^3(\Proj_\Gamma \sigma)$ whose socle is $\sigma$-isotypic.

  We now determine $Q_\sigma$ more explicitly.
  Let $A'_{\sigma,i}$ ($0 \le i \le f-1$) denote the $\Gamma$-representation of \cite[Def.\ 2.5]{HuWang}, which has $\soc_\Gamma(A'_{\sigma,i}) \cong \cosoc_\Gamma(A'_{\sigma,i}) \cong \sigma$ and $\rad_\Gamma(A'_{\sigma,i})/\soc_\Gamma(A'_{\sigma,i}) \cong \mu_i^+(\sigma) \oplus \mu_i^-(\sigma)$.
  Let $A'_\sigma$ denote the fiber product of all $A'_{\sigma,i}$ over their common cosocle $\sigma$. 
  (Up to twist this is dual to the notation $A'_\sigma$ in \cite{HuWang}.)
  Note that the natural injection $\rad_\Gamma(A'_\sigma) \into \bigoplus_i \rad_\Gamma(A'_{\sigma,i})$ is an isomorphism.
  (It surjects onto every factor, as $\rad_\Gamma(\cdot)$ preserves surjections, hence surjects onto the cosocle of the direct sum, which is multiplicity free.)
  Hence the cosocle of $A'_\sigma$ is still $\sigma$.
  Also we obtain a surjection $\psi : Q_\sigma \onto A'_\sigma$, and its kernel is $\sigma$-isotypic, because $\psi$ induces an isomorphism after applying the functor $\rad_\Gamma(\cdot)/\rad^2_\Gamma(\cdot)$, e.g.\ by \cite[Cor.\ 5.6(i)]{BP}. 
  As $\Ext^1_\Gamma(\sigma,\sigma) = 0$ we have a surjection $\soc_\Gamma(\psi) : \soc_\Gamma(Q_\sigma) \onto \soc_\Gamma(A'_\sigma) \cong \sigma^{\oplus f}$.
  On the other hand, $\soc_\Gamma(Q_\sigma) \cong \sigma^{\oplus f}$ by the dual version of \cite[Prop.\ 2.11]{HuWang} (alternatively, see \cite[Thm.\ 4.3]{AJL}),
  hence $Q_\sigma \cong A'_\sigma$.

  Write $0 \to L \to A'_\sigma \to Q \to 0$, with $L$ being the corresponding kernel.
  If $\mu_i^*(\sigma) \in \JH(L)$, then $L$ has to contain the unique subrepresentation of $\rad_\Gamma(A'_{\sigma,i}) \subset A'_\sigma$ with cosocle $\mu_i^*(\sigma)$.
  In particular, the natural map $\rad_\Gamma(A'_{\sigma,i}) \into A'_\sigma \to Q$ has to vanish on the socle, and hence is zero (by condition (i)).
  This proves that $\mu_i^{-*}(\sigma) \in \JH(L)$, as desired.
\end{proof}

Recall again from \cite[Def.~2.9]{HuWang2} that given $j\in\{0,\dots,f-1\}$ and $*\in \{+,-\}$ we define an $f$-tuple $\delta_j^*\in \bigoplus_{i=0}^{f-1}(\Z\pm x_i)$ by $(\delta_j^*)_{j}(x_{j})\defeq x_{j}\ast2$ and $(\delta_j^*)_{i}(x_i)\defeq x_i$ for $i\neq j$.
If $\sigma$ is a Serre weight corresponding to a tuple $(s_0,\dots,s_{f-1})\in\{0,\dots,p-1\}^f$ we write $\delta^*_j(\sigma)$ for the Serre weight $\delta_j^*\big((s_0,\dots,s_{f-1})\big)\otimes\det^{e(\delta_j^*)(s_0,\dots,s_{f-1})}$ (which is defined only if $s_j \ast 2 \in\{0,\dots,p-1\}$).
It follows from the definition that $\chi_{\delta_j^*(\sigma)}=\chi_\sigma\alpha_j^{*1}$.

\begin{lem}\label{lem:mu-pm}
Assume that $\brho$ is $1$-generic and let $\sigma \in W(\brho^\ss)$.
  For any $0 \le j \le f-1$,  there exists $* \in \{\pm\}$ such that
  \begin{equation}\label{eq:mu-pm-delta-pm}
\{\mu_j^+(\sigma), \mu_j^-(\sigma), \delta_j^+(\sigma), \delta_j^-(\sigma)\} \cap W(\brho^\ss) = \{\mu_j^*(\sigma)\}.
  \end{equation}
  Moreover, $J_{\mu_j^*(\sigma)} = J_\sigma \mathbin\Delta \{j\}$.
\end{lem}

\begin{proof}
Let $\sigma^c \in W(\brho^\ss)$ be determined by $J_{\sigma^c}=J_{\sigma}^c$.
Then $\JH(I(\sigma,\sigma^c)) = W(\brho^\ss)$ by Lemma~\ref{lem:I-sigma-tau-modular}.
{Recall from \S~\ref{sec:EG-Gamma-rep} that $\JH(\Inj_{\Gamma}\sigma)$ is parametrized by the set $\mathcal{I}$.}
Since $\delta_j^{\pm} \notin \cI$ we  deduce by \cite[Lemmas 2.1, 2.7]{HuWang2} that   $\delta_j^{\pm}(\sigma)$ does not occur in $\Inj_{\Gamma}\sigma$  for any $0\leq j\leq f-1$, hence $\delta_j^{\pm}(\sigma)\notin W(\brho^\ss)$.

Viewing $\sigma^c$ as a constituent in $\Inj_{\Gamma}\sigma$, it is parametrized by an element  $\lambda\in \cI$. %
Since $|\JH(I(\sigma,\sigma^c))|=2^f$, \cite[Cor.~4.11]{BP} implies that $\cS(\lambda)$ (defined above \cite[Lemma~\ref{bhhms4:lem:Dss-in-D0}]{BHHMS4}) equals $\{0,\dots,f-1\}$. For each $0\leq j\leq f-1$, there is a unique $*\in\{\pm\}$ such that $\mu_j^*$ is compatible (in the sense of \cite[Def.~4.10]{BP}) with $\lambda$. By \cite[Cor.~4.11]{BP} again and \cite[Lemmas 2.1, 2.7]{HuWang2}, we deduce that exactly one of $\mu_j^{\pm}(\sigma)$ occurs in  $\JH(I(\sigma,\sigma^c))$.  The final claim is a direct check.
\end{proof}

\subsection{Some \texorpdfstring{$\wt\Gamma$}{\textbackslash tilde Gamma}-representations}
\label{sec:some-wtgamma-repr}

We start by recalling some results from \cite{HuWang2}.
Let $\tau$ be a Serre weight and $\chi\defeq \tau^{I_1}$. 
For $n \ge 1$ let  $W_{\chi,n}\defeq (\Proj_{I/Z_1}\chi)/\m^n$, {where $\Proj_{I/Z_1}\chi$ is the linear dual of the injective envelope $\Inj_{I/Z_1}(\chi^{-1})$ (a projective cover of $\chi$ in the dual category). It is} a finite-dimensional representation of $I/Z_1$ over $\F$. 
We let $\overline{W}_{\chi,3}$ be the smallest quotient of $W_{\chi,3}$ such that $[W_{\chi,3}:\chi]=[\overline{W}_{\chi,3}:\chi]$. 
It is shown in \cite[Lemma~3.2]{HuWang2} that $\overline{W}_{\chi,3}$ fits into a short exact sequence
\begin{equation}
\label{eq:Wbar3}
0\ra \bigoplus_{\chi'} E_{\chi,\chi'}\ra \overline{W}_{\chi,3}\ra \chi\ra 0,
\end{equation}
where the direct sum is taken over the characters $\chi'$ such that $\Ext^1_{I/Z_1}(\chi',\chi)\neq 0$. 
Then $\overline{W}_{\chi,3}$, and hence also $\Ind_I^{\GL_2(\cO_K)}\overline{W}_{\chi,3}$, is annihilated by $\m_{K_1}^2$ \cite[Cor.\ 3.3]{HuWang2}. 

Recall from \cite[Thm.~2.23]{HuWang2} that given a 2-generic Serre weight $\sigma$ and $\tau\in \JH(\Inj_{\wt{\Gamma}}\sigma)$, there exists a unique finite-dimensional $\wt{\Gamma}$-module $I(\sigma,\tau)$ such that $\soc_{\wt{\Gamma}}I(\sigma,\tau)=\sigma$, $\cosoc_{\wt{\Gamma}}I(\sigma,\tau)=\tau$ and $[I(\sigma,\tau):\sigma]=1$.
(Note that this agrees with the definition of $I(\sigma,\tau)$ in \S~\ref{sec:EG-Gamma-rep} if $\tau \in \JH(\Inj_\Gamma\sigma)$.)
We also recall that $I(\sigma,\tau)$ is multiplicity free by \cite[Cor.~2.25]{HuWang2}.

Assume now that $\tau$ is 2-generic (so $\chi$ is 2-generic).
For $0 \le i \le f-1$ and a sign $* \in \{\pm\}$ let $W_i^* = W_i^*(\chi)$ denote the unique uniserial $I/Z_1$-representation of the form $\chi \-- \chi \alpha_i^{*1} \-- \chi$.
(It is a quotient of the $I/Z_1$-representation $\o W_{\chi,3}$ in \cite[\S~3.1]{HuWang2}, see also \S~\ref{sec:subspace-pi-k1-2} below.)
Let $Q_i^* = Q_i^*(\tau)$ denote the largest quotient of $\Ind_I^{\GL_2(\cO_K)} W_i^*$ with $\tau$-isotypic socle.
Then $Q_i^*$ is a $\wt\Gamma$-representation by \cite[Cor.\ 3.3]{HuWang2}.

\begin{lem}\label{lem:reps-Qi*}
  Suppose that $i \in \{0,\dots,f-1\}$.
  \begin{enumerate}
  \item The $\wt\Gamma$-representation $Q_i^-$ is uniserial of the form $\tau \-- \mu_i^-(\tau) \-- \tau$.
  \item The $\wt\Gamma$-representation $Q_i^+$ has the form 
    \begin{equation*}
      \xymatrix@R-1.5pc{
        & \mu_i^-(\tau) \ar@{-}[rd]\\
        \tau \ar@{-}[rd]\ar@{-}[ru] && \tau \\
        & \mu_i^+(\tau) \ar@{-}[rd]\ar@{-}[ru]\\
        && \delta_i^+(\tau)}
    \end{equation*}
  \end{enumerate}
 In particular, $\soc_{\wt\Gamma}(Q_i^*)=\tau$ for each $* \in \{\pm\}$.
\end{lem}

We remark that Lemma~\ref{lem:reps-Qi*} does not determine $\ker(Q_i^+ \onto \delta_i^+(\tau))$ uniquely up to isomorphism (this kernel is a suitable amalgam of the uniserial representations $\tau \-- \mu_i^-(\tau) \-- \tau$, $\tau \-- \mu_i^+(\tau) \-- \tau$, and this amalgam depends on a parameter in $\F^\times$), but this will not matter for us.

\begin{proof}
  (i)  
  Let $Y_i^-$ %
denote a uniserial $\wt\Gamma$-representation of the form $\tau \-- \mu_i^-(\tau) \-- \tau$, which exists by taking a suitable quotient of the representation $\Theta_\tau$ in \cite[Prop.\ 3.12]{HuWang2} (see also \cite[Cor.\ 3.16]{HuWang2}). 
  By \cite[Lemma 2.10]{HuWang2} it is easy to see that $Y_i^-$ is unique up to isomorphism (alternatively it follows once this lemma is proved).

 We now show that part (i) holds, \emph{assuming that} $W_i^- \into Y_i^-|_I$.
    Under this assumption, we show that the corresponding map $\Ind_I^{\GL_2(\cO_K)} W_i^- \onto Y_i^-$ is surjective.
    First note that $\JH(\cosoc_{\wt\Gamma} (\Ind_I^{\GL_2(\cO_K)} W_i^-)) \subset \{ \tau, \delta_i^-(\tau)\}$, as $\Ind_I^{\GL_2(\cO_K)} \chi\alpha_i^{-1}$ (resp.\ $\Ind_I^{\GL_2(\cO_K)} \chi$) has cosocle $\delta_i^-(\tau)$ (resp.\ $\tau$).
    Hence if the map $\Ind_I^{\GL_2(\cO_K)} W_i^- \onto Y_i^-$ is not surjective, then it has image $\tau$, which implies that it is trivial on $\Ind_I^{\GL_2(\cO_K)} \chi \subset \Ind_I^{\GL_2(\cO_K)} W_i^-$, contradicting the injectivity of $W_i^- \into Y_i^-|_I$.
    The surjection $\Ind_I^{\GL_2(\cO_K)} W_i^- \onto Y_i^-$ shows that $Q_i^- = Y_i^-$ by definition of $Q_i^-$, as $[\Ind_I^{\GL_2(\cO_K)} W_i^-:\tau]=2$, which completes part (i).

 It remains to check that $W_i^- \into Y_i^-|_I$. 
  Let $\chi'\defeq \chi\alpha_i^{-1}$
  and let $E_{\chi,\chi'}$ be the $I/Z_1$-representation which is the unique nonsplit extension of $\chi'$ by $\chi$ (see \S~\ref{sec:notation}). By \cite[Lemma~3.42(ii)]{BHHMS2} (applied with $\sigma=\tau$ and $\underline{Y}^{-\underline{i}}v=Y_i^{-1}v$) 
  there is 
  an injection $E_{\chi,\chi'} \into \tau|_I \into Y_i^-|_I$.
  As $\dim_\F(\tau) < q$, we know that $\tau|_I$ is multiplicity free by \cite[Lemma 2.7]{BP}.
  Let $u \in \tau^{I_1}$ (resp.\ $v\in \tau$) be an $H$-eigenvector with eigencharacter $\chi$ (resp.\ $\chi'=\chi\alpha_i^{-1}$), so $E_{\chi,\chi'} = \F u \oplus \F v$.
  On the other hand, let $w\in Y_i^-$ be an $H$-eigenvector with eigencharacter $\chi$, such that its image in $Y_i^-/\soc_{\wt\Gamma}(Y_i^-)$ is $I_1$-invariant.
  This is possible by Lemma~\ref{lem:mu-i+}. %

  We will prove that $\F u \oplus \F v\oplus \F w$ is $I$-stable (equivalently $I_1$-stable) and isomorphic to $W_i^-$. 
  Note that $(g-1)w \in \soc_{\wt\Gamma}(Y_i^-) = \tau$ for all $g \in I_1$ (by the choice of $w$), and that $w$ itself is not fixed by $I_1$, since otherwise there would be a surjection $\Ind_I^{\GL_2(\cO_K)}\chi\onto Y_i^-$, which is impossible as $\tau$ has
multiplicity $1$ in $\Ind_I^{\GL_2(\cO_K)}\chi$. %
As $\F u \oplus \F v$ is $I_1$-stable it is enough to prove that $(g-1)w\in \F u \oplus \F v$ for
$g\in I_1$.
  It then suffices to show that (note that $Z_1$ acts trivially on $Y_i^-)$:
  \begin{enumerate}[label=(\alph*)]
  \item $(g-1)w = 0$ for all $g \in \smatr{1}{\cO_K}01$;
  \item $(g-1)w \in \tau^{H = \chi} = \F u$ for all $g \in \smatr{1+p\cO_K}001$;
  \item $(g-1)w \in \tau^{H = \chi'} = \F v$ for all $g \in \smatr{1}0{p\cO_K}1$.
  \end{enumerate}

  To prove (a), let $Y_j \defeq  \sum_{a\in\F_q^\times}a^{-p^j}\smatr{1}{[a]}0{1}\in\F\bbra{\smatr1{\cO_K}01}$ for $0 \le j \le f-1$, so that $\F\bbra{\smatr{1}{\cO_K}01} = \F\bbra{Y_0,\dots,Y_{f-1}}$.
It is direct to check that $Y_jw$ is an $H$-eigenvector with eigencharacter $\chi\alpha_j$. However, we see from \cite[Lemma 2.7]{BP} that $\chi \alpha_j \notin \JH(\tau|_I)$ for all $0 \le j \le f-1$.
  Thus $Y_j w = 0$ for all $j$, so (a) holds.

  Part (b) is obvious.

  To prove (c), let $X_j \defeq  \sum_{a\in\F_q^\times}a^{-p^j}\smatr{1}0{[a]}{1}\in\F\bbra{\smatr{1}0{\cO_K}1}$ for $0 \le j \le f-1$. %
Write $\tau=(r_0,\dots,r_{f-1})$ up to twist.  By another application of  \cite[Lemma 2.7]{BP} we see that $\chi \alpha_j^{-(r_j+1)} \notin \JH(Y_i^-|_I)$ for all $j \ne i-1$, so using the $\GL_2(\cO_K)$-action on $Y_i^-$ we conclude that $X_j^{r_j+1} w = 0$ and hence $X_j^p w = 0$ for all $j \ne i-1$. 
  On the other hand, $X_j^p X_{j'}^p w = 0$ for all $j$, $j'$, as $Y_i^-$ is a $\wt\Gamma$-representation.
  As $\F\bbra{\smatr{1}0{p\cO_K}1} = \F\bbra{X_0^p,\dots,X_{f-1}^p}$, we deduce that $(g-1)w \in \F X_{i-1}^p w$, on which $H$ acts by $\chi \alpha_{i-1}^{-p} = \chi'$.
  
  (ii) 
  Using (i) we determine the submodule structure of $\Ind_I^{\GL_2(\cO_K)} W_i^+$ completely.
  This is done in Step 1 to Step 3 below.
  Write $\cS \defeq  \{0,1,\dots,f-1\}$ in what follows.
  For $J \subset \cS$ let $\sigma_J^0$ (resp.\ $\sigma_J^1$, resp.\ $\sigma_J^2$) denote the constituent parametrized by $J$
  in the bottom $\Ind_I^{\GL_2(\cO_K)} \chi$ (resp.\ $\Ind_I^{\GL_2(\cO_K)} \chi\alpha_i$, resp.\ the top $\Ind_I^{\GL_2(\cO_K)} \chi$), 
  (see \S~\ref{sec:EG-Gamma-rep} for this parametrization).
  In particular, we write $\sigma_J \defeq  \sigma_J^0 \cong \sigma_J^2$ and note that $\sigma_\cS \cong \tau$.
  Note that the constituents $\sigma_J$ occur with multiplicity 2, and that the $\sigma_J^1$ occur with multiplicity 1, cf.\ \cite[Lemma~\ref{bhhms4:lem:multfree}]{BHHMS4}. %
 
 For \ $s \in \{0,1\}$ \ write \ $V_J^s$ \ for \ the \ unique \ subrepresentation \ of \ $\Ind_I^{\GL_2(\cO_K)} E_{\chi,\chi\alpha_i} \subset \Ind_I^{\GL_2(\cO_K)} W_i^+$ with cosocle $\sigma_J^s$ {(not to be confused with the Serre weight $\sigma_J^{[s]}$ of \S~\ref{sec:notation}!)}, or equivalently for the image of \emph{any} nonzero map $\Proj_{\wt\Gamma}\sigma_{J}^s\ra \Ind_I^{\GL_2(\cO_K)} E_{\chi,\chi\alpha_i}$.
  Write $V_J^2$ for the image of \emph{some} map $\iota:\Proj_{\wt\Gamma}\sigma_{J}\ra\Ind_I^{\GL_2(\cO_K)} W_i^+$ such that the composite $\Proj_{\wt\Gamma}\sigma_{J}\xrightarrow{\iota}\Ind_I^{\GL_2(\cO_K)} W_i^+\onto \Ind_I^{\GL_2(\cO_K)}\chi$ is nonzero.
  We claim that  the $V_J^2$ are independent of the choice of $\iota$ (equivalently, $V_J^0 \subset V_J^2$). Indeed, as we recall at the beginning of \S~\ref{sec:some-wtgamma-repr}, $W_i^+$ is a quotient of $\overline{W}_{\chi,3}$. Using \cite[Cor.\ 3.3]{HuWang2} we see that $\Ind_I^{\GL_2(\cO_K)}\overline{W}_{\chi,3}$ is a $\wt\Gamma$-representation, so we can lift $\iota$ to $\phi:\Proj_{\wt\Gamma}\sigma_J\ra \Ind_I^{\GL_2(\cO_K)}\overline{W}_{\chi,3}$ such that the composite with $\Ind_I^{\GL_2(\cO_K)}\overline{W}_{\chi,3}\ra \Ind_I^{\GL_2(\cO_K)}\chi$ is nonzero.
By \cite[Prop.\ 3.10(i)]{HuWang2} (and its proof) we get $[\coker(\phi):\sigma_J]=0$, hence by \cite[Prop.\ 3.10(ii)]{HuWang2} the image of $\phi$ is independent of any choices, and consequently the image $V_J^2$ of $\iota$ is well defined. 
  Thus, to determine the submodule structure, it suffices to determine all minimal (proper) containments between the submodules of the form $V_J^0$, $V_J^1$, $V_J^2$.
  Here we say that a containment of two such modules is minimal if no other $V_{J'}^s$ lies strictly in between.

  \textbf{Step 1.} 
  By \cite[Thm.\ 2.4]{BP} and \cite[Lemma 3.7]{HuWang2} the minimal containments among the $V_J^0$ and $V_J^1$ are given by
  \begin{align}
    V_J^0 \subsetneq &V^0_{J\sqcup\{k\}},\quad   V_J^1 \subsetneq V^1_{J\sqcup\{k\}} && \text{for any $k \notin J$}; \label{eq:VJ0,VJ1} \\
    &\ \ V_{J \sqcup \{i\}}^0 \subsetneq V^1_{J} && \text{if $i \notin J$}. \label{eq:VJ0-VJ1}
  \end{align}
  Likewise, \cite[Lemma 3.8]{HuWang2} shows that the minimal containments between submodules of the form $V_J^1$ and $V_{J'}^2$ are given by
  \begin{equation}
    V_{J}^1 \subsetneq V^2_{J \sqcup \{i\}} \qquad\qquad \text{if $i \notin J$}.\label{eq:VJ1-VJ2}
  \end{equation}
  Likewise, by \cite[Thm.\ 2.4]{BP}, the minimal containments among the $V_J^2$ are given by
  \begin{equation}
    V_J^2 \subsetneq V^2_{J\sqcup\{k\}} \qquad\qquad \text{for any $k \notin J$}.\label{eq:VJ2}
  \end{equation}

  \textbf{Step 2.}
  We show that %
  the minimal containments between submodules of the form $V_J^0$ and $V_{J'}^2$ are given by
  \begin{equation}
    V_{J \sqcup \{i\}}^0 \subsetneq V^2_{J} \qquad\qquad \text{if $i \notin J$}.\label{eq:VJ0-VJ2}
  \end{equation}

  By dualizing (i) and replacing $\chi$ by $\chi^{-1}$ we deduce that the largest subrepresentation of $\Ind_I^{\GL_2(\cO_K)} W_i^+$ with cosocle $\sigma_\emptyset$ (and socle $\sigma_\emptyset$) is uniserial of the form 
  \begin{equation}
    \sigma_\emptyset^0 \-- \sigma_{\{i\}}^0 \-- \sigma_\emptyset^2.\label{eq:sigma-empty-002}
  \end{equation}
  (Note that the middle constituent cannot be $\sigma_{\{i\}}^2$ by~\eqref{eq:VJ2}.)

  Consider the statement
  \begin{equation*}
    A(J_1,J_2): \qquad V_{J_1}^0 \subsetneq V_{J_2}^2 \text{\ is a minimal containment}. %
  \end{equation*}
  Note by \eqref{eq:sigma-empty-002} that 
  \begin{equation}\label{eq:Ak-empty}
    \text{$A(\{k\},\emptyset)$ holds if and only if $k = i$.}
  \end{equation}
  Also note that
  \begin{equation}\label{eq:J1-J2-diff}
    A(J_1,J_2) \quad \Rightarrow\quad |J_1 \mathbin\Delta J_2| = 1 \quad \Rightarrow\quad (J_1 \subset J_2) \text{\ or\ } (J_2 \subset J_1),
  \end{equation}
  because if $A(J_1,J_2)$ holds, then $0\neq \Ext^1_{\wt\Gamma}(\sigma_{J_2},\sigma_{J_1})=\Ext^1_{\Gamma}(\sigma_{J_2},\sigma_{J_1})$, where the equality follows from \cite[Cor.~5.6(ii)]{BP}, so that $|J_1 \mathbin\Delta J_2| = 1$ by Lemma~\ref{lem:I-sigma-tau-princ-series}.  

  We now show that $A(J_1,J_2) \Rightarrow A(J_1 \sqcup \{k\},J_2 \sqcup \{k\})$ if $J_1  \supset J_2$ and $k \notin J_1$.
  By $A(J_1,J_2)$ and~\eqref{eq:VJ2} we deduce that $V_{J_1}^0 \subset V_{J_2}^2 \subset V_{J_2 \sqcup \{k\}}^2$. %
  In particular, $V_{J_2 \sqcup \{k\}}^2$ admits a (unique) quotient $Q$ {whose socle is the cosocle of $V_{J_1}^0$ (which is isomorphic to $\sigma_{J_1}$).}

 We prove that $Q$ is a $\Gamma$-representation isomorphic to $I(\sigma_{J_1},\sigma_{J_2 \sqcup \{k\}})$, by first showing that (a) $\JH(Q) \subset \JH(\Ind_I^{\GL_2(\cO_K)} \chi)$ and (b) $Q$ is multiplicity free.
    For (a), suppose by contradiction that $\sigma_J^1 \in \JH(Q)$ for some $J$. Then $Q$ admits a subrepresentation with socle $\sigma_{J_1}^0$ and cosocle $\sigma_J^1$, so $V_{J_1}^0 \subset V_J^1 \subset V_{J_2 \sqcup \{k\}}^2$.
    By Step 1 we deduce that $J_1 \subset J \sqcup \{i\} \subset J_2 \sqcup \{k\}$ (in particular, $i \notin J$).
    But $J_1 \subset J_2 \sqcup \{k\}$ implies equality by~\eqref{eq:J1-J2-diff}, contradicting $k\notin J_1$.
    For (b), suppose by contradiction that $V_{J_1}^0 \subset V_J^0 \subset V_J^2 \subset V_{J_2 \sqcup \{k\}}^2$ for some $J$.
    By Step 1 this gives $J_1 \subset J \subset J_2 \sqcup \{k\}$, leading to the same contradiction as before.
    Using (a) and (b) we deduce that $Q$ is a $\Gamma$-representation by \cite[Cor.~5.7]{BP}, so that $Q \cong I(\sigma_{J_1},\sigma_{J_2 \sqcup \{k\}})$, as claimed.

 As $Q \cong I(\sigma_{J_1},\sigma_{J_2 \sqcup \{k\}})$ as $\Gamma$-representation we deduce by Lemma~\ref{lem:I-sigma-tau-princ-series} and~\eqref{eq:J1-J2-diff} that $Q$ has length 4 and surjects onto the nonsplit extension $\sigma_{J_1 \sqcup \{k\}} \-- \sigma_{J_2 \sqcup \{k\}}$, i.e.\ there is a minimal containment $V_{J_1 \sqcup \{k\}}^s \subsetneq V_{J_2 \sqcup \{k\}}^2$ for some $s \in \{0,2\}$.
    As $J_1 \supset J_2$ we deduce by~\eqref{eq:VJ2} that $s = 0$.
    This establishes $A(J_1 \sqcup \{k\},J_2 \sqcup \{k\})$.

 Conversely, if $A(J_1 \sqcup \{k\},J_2 \sqcup \{k\})$, then $V_{J_1}^0 \subset V_{J_1 \sqcup \{k\}}^0 \subset V_{J_2 \sqcup \{k\}}^2$ by~\eqref{eq:VJ0,VJ1}.
    Using the same quotient $Q$ as above, we again have $Q \cong I(\sigma_{J_1},\sigma_{J_2 \sqcup \{k\}})$ as $\Gamma$-representation, which contains the nonsplit extension $\sigma_{J_1} \-- \sigma_{J_2}$, i.e.\ there is a minimal containment $V_{J_1}^0 \subsetneq V_{J_2}^s$ for some $s \in \{0,2\}$.
    By~\eqref{eq:VJ0,VJ1} we deduce that $s = 2$, so $A(J_1,J_2)$ holds.
 In particular, from~\eqref{eq:Ak-empty} we deduce that $A(J\sqcup \{k\},J)$ holds (for $k \notin J$) if and only if $k = i$.

  In the preceding paragraph we dealt with all cases when $J_1 \supsetneq J_2$.
  If $J_1 \subsetneq J_2$, then we have $V_{J_1}^0 \subsetneq V_{J_2}^0 \subsetneq V_{J_2}^2$, so that $V_{J_1}^0 \subsetneq V_{J_2}^2$ is not a minimal containment.
  We have thus confirmed the list of minimal containments between submodules of the form $V_J^0$ and $V_{J'}^2$ in~\eqref{eq:VJ0-VJ2}.

  \textbf{Step 3.}
  Recall that $Q_i^+$ is the largest quotient of $\Ind_I^{\GL_2(\cO_K)} W_i^+$ with socle $\tau \cong \sigma_\cS$.
  As $V_\cS^0 \subset V_\cS^2$, we have $\soc_{\wt\Gamma}(Q_i^+) = \tau$ and the submodules of $Q_i^+$ having irreducible cosocle are the images of the submodules $V_J^s$ of $\Ind_I^{\GL_2(\cO_K)} W_i^+$ that contain $V_\cS^0$.
  From Steps 1 and 2 we obtain precisely the following such submodules and containments:
  
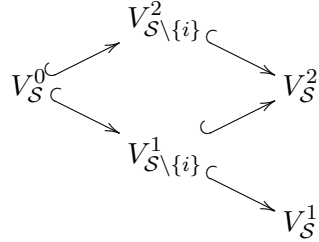
\begin{equation*}
    \xymatrix@R-1.5pc{
      & V_{\cS \setminus \{i\}}^2 \ar@{^{(}->}[rd]\\
      V_\cS^0 \ar@{^{(}->}[rd]\ar@{^{(}->}[ru] && V_\cS^2 \\
      & V_{\cS \setminus \{i\}}^1 \ar@{^{(}->}[rd]\ar@{^{(}->}[ru]\\
      && V_\cS^1}
  \end{equation*}
  This determines the submodule structure of $Q_i^+$ by Lemma~\ref{lem:submodules} (taking $M = \Ind_I^{\GL_2(\cO_K)} W_i^+$, $\o M = Q_i^+$, and all possible $\sigma$) below.
  It remains to observe that $\sigma_{\cS \setminus \{i\}}^2 \cong \mu_i^-(\tau)$, $\sigma_{\cS \setminus \{i\}}^1 \cong \mu_i^+(\tau)$, $\sigma_{\cS}^1 \cong \delta_i^+(\tau)$.
\end{proof}

\begin{lem}\label{lem:submodules}
  Suppose that $M$ is a finite length module over an artinian ring $A$, and that $\pi : M \onto \o M$ is a quotient morphism.
  Suppose that $\sigma$ and $\tau$ are simple $A$-modules and that $M_\sigma$ (resp.\ $M_\tau$) is a submodule of $M$ having cosocle $\sigma$ (resp.\ $\tau$).
  If the set of submodules of $M$ having cosocle $\sigma$ is totally ordered and $\pi(M_\sigma) \ne 0$, then
  \begin{equation*}
    M_\sigma \subset M_\tau \iff \pi(M_\sigma) \subset \pi(M_\tau).
  \end{equation*}
\end{lem}

\begin{proof}
  Let $N \defeq  \ker(\pi)$.
  For the nontrivial direction, we need to show that $M_\sigma \subset M_\tau+N$ implies $M_\sigma \subset M_\tau$. {Let $\Proj_A \sigma$ be the projective cover of $\sigma$ in the category of $A$-modules (which exists as $A$ is artinian).} 
  Pick $f : \Proj_A \sigma \to M$ that has image $M_\sigma$, and consider the commutative diagram
  \begin{equation*}
    \xymatrix{
      \Hom_A(\Proj_A \sigma, M_\tau \oplus N) \ar@{->>}[r] & \Hom_A(\Proj_A \sigma, M_\tau + N)\ar@{^{(}->}[d] \\
      \Hom_A(\Proj_A \sigma, M_\tau) \oplus \Hom_A(\Proj_A \sigma, N) \ar^{\cong}[u]\ar^-+[r] & \Hom_A(\Proj_A \sigma, M)
    }
  \end{equation*}
  As $M_\sigma \subset M_\tau+N$ and $\Proj_A \sigma$ is projective there exist $f_1 : \Proj_A \sigma \to M_\tau$ and $f_2 : \Proj_A \sigma \to N$ such that $f = f_1+f_2$.
  By the condition on the submodules of $M$, we know that $\im(f_1) \subset \im(f_2)$ or $\im(f_2) \subset \im(f_1)$.
  In the first case, $\im(f) \subset \im(f_2) \subset N$, contradiction.
  Hence $M_\sigma = \im(f) \subset \im(f_1) \subset M_\tau$, as desired.
\end{proof}

Let $W_i = W_i(\chi)$ 
denote the fiber product of $W_i^+$ and $W_i^-$ over their common cosocle $\chi$.
Let $Q_i = Q_i(\tau)$ denote the fiber product of $Q_i^+$ and $Q_i^-$ over their common quotient $I(\mu_i^-(\tau),\tau)$ (cf.\ Lemma~\ref{lem:reps-Qi*}).
We draw a diagram for $W_i$ and $Q_i$, but keep in mind that the submodule structure is more complicated since the socle has multiplicities in each case:
\begin{equation*}
  \xymatrix@R-1.5pc{
    &\chi\ar@{-}[r] & \chi\alpha_i^{-1}\ar@{-}[rd] &&&& \tau \ar@{-}[r] & \mu_i^-(\tau) \ar@{-}[rd]\\
    W_i: &&&\chi &&Q_i:& && \tau \\
    &\chi\ar@{-}[r] & \chi\alpha_i\ar@{-}[ru] &&&& \tau \ar@{-}[r] & \mu_i^+(\tau) \ar@{-}[rd]\ar@{-}[ru]\\
    &&&&& &&& \delta_i^+(\tau)}
\end{equation*}

\begin{lem}\label{lem:rep-Qi}
  The representation $\Ind_I^{\GL_2(\cO_K)} W_i$ has a unique quotient $Q$ with socle $\tau^{\oplus 2}$ and such that $[Q:\tau] = 3$, and this quotient is isomorphic to $Q_i$.
  Moreover, \[\JH(Q_i) = \{\tau, \mu_i^-(\tau), \mu_i^+(\tau), \delta_i^+(\tau)\}\] and $Q_i/\soc_{\wt\Gamma}(Q_i)$ is multiplicity free.
\end{lem}

\begin{proof}
  By exactness of induction, $\Ind_I^{\GL_2(\cO_K)} W_i$ is the fiber product of $\Ind_I^{\GL_2(\cO_K)} W_i^{\pm}$ over $\Ind_I^{\GL_2(\cO_K)} \chi$.
  We have a commutative diagram with exact rows:
  \begin{equation*}
    \xymatrix{
      0 \ar[r] & \Ind_I^{\GL_2(\cO_K)} W_i \ar[r]\ar[d] & \Ind_I^{\GL_2(\cO_K)} W_i^{+} \times \Ind_I^{\GL_2(\cO_K)} W_i^{-} \ar[r]\ar@{->>}[d] & \Ind_I^{\GL_2(\cO_K)} \chi\ar@{->>}^f[d] \ar[r] & 0 \\
      0 \ar[r] & Q_i \ar[r] & Q_i^+ \times Q_i^- \ar^{\alpha}[r] & I(\mu_i^-(\tau),\tau) \ar[r] & 0
      }
  \end{equation*}
  (For the right square, note that the natural map $\Ind_I^{\GL_2(\cO_K)} W_i^* \onto Q_i^* \onto I(\mu_i^-(\tau),\tau)$ factors through the $K_1$-coinvariants $\Ind_I^{\GL_2(\cO_K)} ((W_i^*)_{K_1}) = \Ind_I^{\GL_2(\cO_K)} (W_i^*/\chi)$,
  and hence through $\Ind_I^{\GL_2(\cO_K)} \chi$ because $\Ind_I^{\GL_2(\cO_K)} (W_i^*/\chi)$ is multiplicity free and $\mu_i^-(\tau) \in \JH(\Ind_I^{\GL_2(\cO_K)} \chi)$.)
  By the snake lemma, since no constituent of $\ker(f)$ occurs in $Q_i^\pm$ (by Lemma~\ref{lem:reps-Qi*}), the left vertical map is surjective.
Note that the map $\alpha$ sends $\soc_{\wt\Gamma}(Q_i^+ \times Q_i^-)$ to $0$, so $\soc_{\wt\Gamma}(Q_i) = \soc_{\wt\Gamma}(Q_i^+) \times \soc_{\wt\Gamma}(Q_i^-) = \tau^{\oplus 2}$.
As $[Q_i:\tau] = 3$, we deduce the existence of $Q$.
  Uniqueness of $Q$ is clear, since $[Q_i:\tau] = [\Ind_I^{\GL_2(\cO_K)} W_i:\tau] = 3$.
  The last statement follows from Lemma~\ref{lem:reps-Qi*}.
\end{proof}

Note that each $Q_i$ surjects onto $\tau$.
For a nonempty subset $\cJ \subset \{0,1,\dots,f-1\}$ let $Q_\cJ = Q_\cJ(\tau)$ denote the fiber product of all $Q_i$ ($i \in \cJ$) over $\tau$.
Let $\chi \defeq  \chi_\tau$ and let $W_\cJ = W_\cJ(\chi)$  
denote the fiber product of all $W_i$ (equivalently of all $W_i^\pm$) for $i \in \cJ$ over their common cosocle $\chi$.
\begin{lem}\label{lem:rep-WJ}\ 
\begin{enumerate}
\item The radical filtration of $W_{\cJ}$ is given by $\chi^{\oplus 2|\cJ|} \-- \bigoplus_{i\in\cJ}(\chi\alpha_i\oplus\chi\alpha_i^{-1}) \-- \chi$. Moreover, $\soc_{I}(W_{\cJ})\cong \rad_{I}^2(W_{\cJ})\cong\chi^{\oplus 2|\cJ|}$. 
\item The $K_1$-coinvariants of $W_{\cJ}$ fit in a short exact sequence $0\ra \bigoplus_{i\in\cJ}\chi\alpha_i\ra (W_{\cJ})_{K_1}\ra \chi\ra0$ with cosocle $\chi$.
\end{enumerate}
\end{lem}

\begin{proof}
(i) By construction of $W_\cJ$ as a fiber product we have an inclusion $\iota:\bigoplus_{i\in\cJ,*} \rad_I(W_i^{*}) \into W_\cJ$.
Its image is contained in $\rad_I(W_\cJ)$ because $\rad_I(W_i^{*}) \subset W_\cJ$ is the unique subrepresentation with cosocle $\chi \alpha_i^{*1}$ and $\rad_I(W_\cJ) \onto \rad_I(W_i^{*})$.
For length reasons, $\iota$ has to be an isomorphism onto $\rad_I(W_\cJ)$, which shows that $\cosoc_I(W_\cJ) \cong \chi$.
We also deduce the claims about $\rad_I(W_{\cJ})/\rad^2_I(W_{\cJ})$ and $\rad_I^2(W_{\cJ})$, as $\rad_I(W_i^{*})\cong (\chi \-- \chi\alpha_i^{*1})$.
The last assertion easily follows as $\rad^2_I(W_{\cJ})\subset \soc_I(W_{\cJ})\subset \soc_I(\bigoplus_{i\in\cJ}W_i)\cong \chi^{\oplus 2|\cJ|}$ (note that $\rad_I^3(W_{\cJ})=0$).

(ii) {By (i)} it is clear that $W_{\cJ}$ has a unique quotient, say $\mathcal{E}$, which fits in a short exact sequence as in the statement. By \cite[Lemma 2.4(ii)]{yongquan-algebra} $\mathcal{E}$ is annihilated by $\m_{K_1}$, so that $(W_{\cJ})_{K_1}\onto \mathcal{E}$. We prove that this is an isomorphism. Since $(W_{\cJ})_{K_1}$ has cosocle $\chi$ by (i), we have a surjection $\Proj_{I/K_1} \chi \onto (W_\cJ)_{K_1}$, which kills $\soc_I(\Proj_{I/K_1} \chi) = \chi$ because $\dim_\F(W_\cJ)_{K_1} \le 4|\cJ|+1 < p^f = \dim_\F(\Proj_{I/K_1} \chi)$.
Since $\Proj_{I/K_1} \chi/\soc_I(\Proj_{I/K_1} \chi)$ is multiplicity free~\cite[Lemma 6.1.3]{BHHMS1}, it follows that $(W_\cJ)_{K_1}$ is multiplicity free, hence $(W_\cJ)_{K_1}$ is a quotient of $W_{\cJ}/\rad^2_I(W_{\cJ})$ by (i). To conclude it suffices to prove that $\chi\alpha_i^{-1}$ does not occur in $(W_{\cJ})_{K_1}$.  Otherwise, $(W_{\cJ})_{K_1}$ would surject onto $E_{\chi\alpha_i^{-1},\chi}$ which is not annihilated by $\m_{K_1}$ by \cite[Lemma 2.4(ii)]{yongquan-algebra} again, contradiction.
\end{proof}

\begin{lem}\label{lem:rep-QJ}
  The representation $\Ind_I^{\GL_2(\cO_K)} W_{\cJ}$ has a unique quotient $Q$ with socle $\tau^{\oplus 2\vert\cJ\vert}$ and such that $[Q:\tau] = 2\vert\cJ\vert+1$, and this quotient is isomorphic to $Q_{\cJ}$.
  Moreover, $\JH(Q_\cJ) = \{\tau, \mu_i^\pm(\tau), \delta_i^+(\tau) : i \in \cJ \}$ and  $Q_\cJ/\soc_{\wt\Gamma}(Q_\cJ)$ is multiplicity free.
\end{lem}

\begin{proof}
  We have a commutative diagram with exact rows:
  \begin{equation*}
    \xymatrix{
      0 \ar[r] & \Ind_I^{\GL_2(\cO_K)} W_\cJ \ar[r]\ar^g[d] & \prod_{i \in \cJ} \Ind_I^{\GL_2(\cO_K)} W_i \ar[r]\ar@{->>}[d] & (\Ind_I^{\GL_2(\cO_K)} \chi)^{\oplus (|J|-1)}\ar@{->>}^f[d] \ar[r] & 0 \\
      0 \ar[r] & Q_\cJ \ar[r] & \prod_{i \in \cJ} Q_i \ar^\alpha[r] & \tau^{\oplus (|J|-1)} \ar[r] & 0
      }
  \end{equation*}
  We claim that the left vertical map $g$ is surjective.
  As $\coker(g)$ is a quotient of $\ker(f)$ and $\JH(\ker(f)) \cap \JH(\prod_{i \in \cJ} Q_i) = \{\mu_i^-(\tau) : i \in \cJ\}$ it follows that all the constituents of $\coker(g)$ are of the form $\mu_i^-(\tau)$ for some $i\in \cJ$, hence it suffices to show that $Q_\cJ$ cannot surject onto any $\mu_i^-(\tau)$, $i \in \cJ$.
  This is true, as $\mu_i^-(\tau)$ occurs with multiplicity one in $\prod_{i \in \cJ} Q_i$ (by Lemma~\ref{lem:rep-Qi}) and $Q_\cJ \onto Q_i \onto Q_i^- \onto I(\mu_i^-(\tau),\tau)$.
  By construction, the components of the map $\alpha$ are obtained as composition $Q_i \onto Q_i^- \onto \tau$ (or are zero), so the map sends $\prod_{i \in \cJ} \soc_{\wt\Gamma}(Q_i)$ to $0$ (by Lemma~\ref{lem:rep-Qi} for $Q_i^-$).
  Hence $\soc_{\wt\Gamma}(Q_\cJ) = \prod_{i \in \cJ} \soc_{\wt\Gamma}(Q_i) = \tau^{\oplus 2|\cJ|}$ and together with $[\prod_{i \in \cJ} Q_i : \tau]=3\vert\cJ\vert$ we deduce $[Q_{\cJ}:\tau]=3\vert\cJ\vert-(\vert\cJ\vert-1)=2\vert\cJ\vert+1$, hence we obtain the existence of $Q$ (by taking $Q=Q_{\cJ}$).
  Uniqueness and the last statement again follow easily.  
\end{proof}

Recall that $\cJ$ is a fixed subset of $\cS = \{0,\dots,f-1\}$.
Let $\Theta_\cJ=\Theta_\cJ(\tau) \subset Q_\cJ$ denote the largest subrepresentation such that $\cosoc_{\wt\Gamma}(\Theta_\cJ) \cong \tau$. 
(This exists and $[\Theta_\cJ:\tau] = [Q_\cJ:\tau] = 2|\cJ|+1$: 
by \cite[Prop.\ 3.10]{HuWang2}, noting that $W_\cS=\ovl{W}_{\chi,3}$, the representation $\Ind_I^{\GL_2(\cO_K)} W_\cS$ has a largest subrepresentation with cosocle $\tau$; the same is then true for any quotient representation, in particular for $\Ind_I^{\GL_2(\cO_K)} W_\cJ \onto Q_{\cJ}$.)
We note that $\JH(\Theta_\cJ) = \{\tau, \mu_i^{\pm}(\tau) : i \in \cJ \}$, with $\Theta_\cJ/\soc_{\wt\Gamma}(\Theta_\cJ)$ multiplicity free.
For $i \in \cJ$ let $\Psi_i=\Psi_i(\tau)\subset Q_\cJ$ be the unique subrepresentation such that $\cosoc_{\wt\Gamma}(\Psi_i) \cong \delta_i^+(\tau)$.
Then $\Psi_i \cong I(\tau,\delta_i^+(\tau))$, which is uniserial of shape $\tau \-- \mu_i^+(\tau) \-- \delta_i^+(\tau)$, %
as $$I(\tau,\delta_i^+(\tau)) \into \ker(Q_i^+ \onto I(\mu_i^-(\tau),\tau)) \into \ker(Q_i\onto \tau) \into Q_\cJ$$ (the first inclusion coming from Lemma \ref{lem:reps-Qi*}(ii)).
In particular, $\rad_{\wt\Gamma}(\Psi_i) \subset \Theta_\cJ$ for all $i \in \cJ$.

\begin{lem}\label{lem:Q-vs-Theta}
  The representation $Q_\cJ$ is the colimit of the diagram $(\Theta_\cJ \hookleftarrow \rad_{\wt\Gamma}(\Psi_i) \into \Psi_i)_{i \in \cJ}$ (with $2|\cJ|+1$ objects and $2|\cJ|$ morphisms).
\end{lem}

\begin{proof}
  We claim that $\cosoc_{\wt\Gamma}(Q_\cJ) \cong \tau \oplus \bigoplus_{i \in \cJ} \delta_i^+(\tau)$. 
  Suppose first that $Q_\cJ \onto \sigma$ for some irreducible $\sigma$.
  Then
  \begin{equation*}
    0 \ne \Hom_{\GL_2(\cO_K)}(\Ind_I^{\GL_2(\cO_K)} W_\cJ,\sigma) = \Hom_I(W_\cJ,\sigma) = \Hom_I((W_\cJ)_{K_1},\sigma).
  \end{equation*}
But $(W_\cJ)_{K_1}$ is an extension of $\chi$ by $\bigoplus_{i\in \cJ} \chi\alpha_i^{-1}$ by the last assertion in \cite[Lemma 2.4(ii)]{yongquan-algebra}.
  Hence $\sigma \in \{\tau, \delta_i^+(\tau) : i \in \cJ\}$.
  Conversely, it is enough to note that $Q_\cJ \onto Q_i \onto Q_i^+ \onto \delta_i^+(\tau) \oplus \tau$ for all $i \in \cJ$ by Lemma \ref{lem:reps-Qi*}(ii).

  Hence $Q_\cJ = \Theta_\cJ + \sum_{i \in \cJ} \Psi_i$.
  Write $\cJ = \{0 \le i_1 < \dots < i_n \le f-1\}$.
  Let $R_k \defeq  \Theta_\cJ + \sum_{j=1}^k \Psi_{i_j}$, with the convention $R_0=\Theta_\cJ$.  
  We will prove by induction that $R_{k} \cong R_{k-1} \oplus_{\rad_{\wt\Gamma}(\Psi_{i_k})} \Psi_{i_k}$ for $1 \le k \le n$, which will complete the proof.
  It suffices to show that $R_{k-1} \cap \Psi_{i_k} = \rad_{\wt\Gamma}(\Psi_{i_k})$.
  This is clear, as $\mu^+_{i_k}(\tau) \in \JH(\Theta_\cJ)$, which gives the inclusion $\supseteq$, and as $\delta^+_{i_k}(\tau) \notin \JH(R_{k-1})$.
  (Recall that these constituents occur with multiplicity one in $Q_\cJ$.)
\end{proof}

\section{Abstract setting}
\label{sec:abstract-setting}

Let $\rhobar:\Gal(\o K/K)\ra\GL_2(\F)$ be a continuous $0$-generic representation as in \S~\ref{sec:notation} and let $\pi$ denote a smooth representation of $\GL_2(K)$ over $\F$.
In this section we introduce and study certain assumptions on $\pi$ (relative to $\brho$) that play a key role in our work.

\subsection{Assumptions}
\label{subsec:assumptions}

\emph{From now until the end of this paper}, we assume that $\pi$ satisfies assumptions (i) (with $r = 1$) and (ii) in \cite[\S~\ref{bhhms2:grstr}]{BHHMS2} and assumption \ref{bhhms4:it:assum-iv} (with $r = 1$) in \cite[\S~\ref{bhhms4:sec:theorem}]{BHHMS4}, i.e.~ 
\begin{enumerate}
\item\label{it:assum-i} we have $\pi^{K_1}\cong D_0(\rhobar)$ as $\GL_2(\cO_K)$-representations (in particular, $\pi$ is admissible) and $\pi$ has central character $\det(\rhobar)\omega^{-1}$;
\item\label{it:assum-ii} for any $\lambda\in\P$ we have $[\pi[\fm^3]:\chi_\lambda]=[\pi[\fm]:\chi_\lambda]$;
\setcounter{enumi}{3}
\item\label{it:assum-iv} for any smooth character $\chi:I\ra \F^{\times}$ and any $i \ge 0$, $\Ext^i_{I/Z_1}(\chi,\pi)\neq0 $ only if
  $[\pi[\m]:\chi]\neq0$, in which case
  \[\dim_{\F}\Ext^i_{I/Z_1}(\chi,\pi)=\binom{2f}{i}.\]
\end{enumerate}

For later reference we also recall assumption (iii) of \cite[\S~\ref{bhhms2:sec:length-of-pi}]{BHHMS2}, \emph{though we will not assume it until \S~\ref{sec:subquot-ss}}:
\begin{enumerate}\setcounter{enumi}{2}
\item\label{it:assum-iii} there is a $\GL_2(K)$-equivariant isomorphism of $\Lambda$-modules
\[
\EE^{2f}_\Lambda(\pi^\vee) \cong \pi^\vee\otimes (\det(\rhobar)\omega^{-1}),
\]
where $\EE^{2f}_\Lambda(\pi^\vee)$ is endowed with the $\GL_2(K)$-action defined in \cite[Prop.~3.2]{Ko}.%
\end{enumerate}
Finally, we introduce a further assumption which will be used only in \S~\ref{sec:subquot-non-ss} (namely to verify equation~\eqref{eq:1st-term} in the introduction).
\begin{enumerate}{\setcounter{enumi}{4}}
\item 
\label{it:assum-v}
  We have
  \[ \dim_\F\Tor_1^{\gr(\Lambda)}(\gr(\Lambda)/\o\m^3,\gr_\m(\pi^{\vee})) = \dim_\F \Tor_1^{\Lambda}(\Lambda/\m^3,\pi^{\vee}), \]
  where $\o\m = (y_j,z_j : 0 \le j \le f-1)$ denotes the unique maximal graded ideal of $\gr(\Lambda)$ (see (\ref{eq:grlambda})).
\end{enumerate}

We first note the following consequence: %

\begin{lem}\label{lem:Ext1-Serre}
{Suppose assumptions \ref{it:assum-i} and \ref{it:assum-iv} hold.}
Let $\chi:I\ra \F^{\times}$ be a character such that $\chi\notin \JH(\pi^{I_1})$ and $Q$ be a quotient of $\Ind_I^{\GL_2(\cO_K)}\chi$. Then $\Ext^i_{\GL_2(\cO_K)/Z_1}(Q,\pi)=0$ for $i\in\{0,1\}$. In particular, this result holds when $Q=\tau$ is a Serre weight such that  $\chi_{\tau}\notin \JH(\pi^{I_1})$ by taking $\chi=\chi_\tau$.
\end{lem} 
\begin{proof}
Using \cite[Prop.~4.2]{breuil-buzzati}, the assumption on $\chi$ implies that 
 $\JH(\Ind_I^{\GL_2(\cO_K)}\chi)\cap W(\brho)=\emptyset$. 
Thus for any \emph{subquotient} $Q$ of $\Ind_I^{\GL_2(\cO_K)}\chi$ we have $\Hom_{\GL_2(\cO_K)}(Q,\pi)=0$, as \[\Hom_{\GL_2(\cO_K)}(\sigma,\pi)\neq 0\] if and only if $\sigma\in W(\brho)$ by assumption \ref{it:assum-i}.
 
Consider the short exact sequence $0\ra V\ra \Ind_I^{\GL_2(\cO_K)}\chi\ra Q\ra0$, where $V$ is the corresponding kernel. It induces a short exact sequence
\[\Hom_{\GL_2(\cO_K)}(V,\pi)\ra \Ext^1_{\GL_2(\cO_K)/Z_1}(Q,\pi)\ra \Ext^1_{\GL_2(\cO_K)/Z_1}(\Ind_I^{\GL_2(\cO_K)}\chi,\pi)=0,\]
where the first term vanishes by the last paragraph and the last term vanishes by assumption \ref{it:assum-iv} using Shapiro's lemma.  The result follows.
\end{proof}
\begin{rem}  
Even if it is not needed for this paper, it is natural to ask if $\Ext^1_{\GL_2(\cO_K)/Z_1}(\tau,\pi)=0$ for any Serre weight $\tau\notin W(\brho)$. It is possible to prove this for a globally defined representation $\pi(\brho)$ as in \cite[\S~\ref{bhhms4:sec:verify-assumpt-iv}]{BHHMS4}, in a similar way to \cite[Prop.~\ref{bhhms4:prop:dim-Ext}]{BHHMS4}, but we don't know how to deduce this property using only assumptions \ref{it:assum-i}--\ref{it:assum-iv}.
\end{rem} 

\subsection{Consequences of the assumptions}
\label{sec:subspace-pi-k1-2}

Recall from \S~\ref{sec:preliminaries} that $\fm_{K_1}$ denotes the maximal ideal of {the Iwasawa algebra} $\F\bbra{K_1/Z_1}$ and $\wt{\Gamma}= \F\bbra{\GL_2(\cO_K)/Z_1}/\fm_{K_1}^2$ (which is a finite-dimensional $\F$-algebra). 
Let $\pi$ be an admissible smooth $\GL_2(K)$-representation satisfying assumptions \ref{it:assum-i}, \ref{it:assum-ii} and \ref{it:assum-iv} above with $2$-generic underlying $\rhobar$.
In this subsection we explicitly determine the finite-dimensional $\wt{\Gamma}$-module $\pi[\fm_{K_1}^2]$.

Consider the $\wt{\Gamma}$-representation $\wt{D}_0(\brho)$ from  \S~\ref{sec:preliminaries}.
 As $\brho$ is 2-generic (which is precisely strongly generic in the sense of \cite[Def.~4.4]{HuWang2}) one can deduce from  \cite[Prop.~4.1, Thm.~4.6]{HuWang2} that $\wt{D}_0(\brho)$ is multiplicity free 
and has a direct  sum decomposition 
\[\wt{D}_0(\brho)= \bigoplus_{\sigma\in W(\brho)}\wt{D}_{0,\sigma}(\brho),\]
where $\wt{D}_{0,\sigma}(\brho)$ is the largest subrepresentation of $\Inj_{\wt{\Gamma}}\sigma$ containing $\sigma$ with multiplicity one and no other Serre weights of $W(\brho)$ (see also \cite[Thm.\ 8.4.2]{BHHMS1}). In particular, $\soc_{\wt{\Gamma}}(\wt{D}_{0,\sigma}(\brho))=\sigma$.

\begin{lem}\label{lem:ex-wtD}  
Let $\tau'$ be a Serre weight. 
\begin{enumerate}
\item If $\tau'\notin W(\brho)$, then $\Ext^1_{\wt{\Gamma}}(\tau',\wt{D}_0(\brho))=0$.
\item If  $\tau'\in\JH(D_0(\brho))\setminus W(\brho)$, then $\Ext^1_{\GL_2(\cO_K)/Z_1}(\tau',D_0(\brho))\cong\Ext^1_{\wt{\Gamma}}(\tau',D_0(\brho))=0$.
\end{enumerate}
\end{lem}
\begin{proof}
(i) This follows from the maximality of $\wt{D}_{0}(\brho)$ recalled above.

(ii) The first isomorphism is a general fact, because both $\tau'$ and $D_0(\brho)$ are annihilated by $\m_{K_1}$ (so any extension between them is automatically annihilated by $\m_{K_1}^2$). The second one  follows from (i) and the fact that $\Hom_{\wt{\Gamma}}(\tau',\wt{D}_0(\brho)/D_0(\brho))=0$ (as $\wt{D}_0(\brho)$ is multiplicity free).
\end{proof}

 Let $\tau$ be a Serre weight and $\chi\defeq \tau^{I_1}$ and recall the representation  $\overline{W}_{\chi,3}$ from \S~\ref{sec:some-wtgamma-repr}. {
If $\tau$ (hence $\chi$) is $2$-generic then}
by \cite[Prop.\ 3.10(i)]{HuWang2} for any 
Jordan--H\"older factor $\tau'$ of $\Ind_I^{\GL_2(\cO_K)}\overline{W}_{\chi,3}$ there exists a $\GL_2(\cO_K)$-equivariant morphism 
\begin{equation}\label{eq:phitau}
\phi_{\tau'}: \Proj_{\wt{\Gamma}}\tau'\ra \Ind_I^{\GL_2(\cO_K)}\overline{W}_{\chi,3}\end{equation}
such that $[\coker(\phi_{\tau'}):\tau']=0$. 
Note that by \cite[Prop.\ 3.10(ii)]{HuWang2}, the image $\im(\phi_{\tau'})$ is unique (even though $\phi_{\tau'}$ need not be unique up to scalar). %

For the following, we emphasize that $\phi_{\tau'}$ depends on $\tau$ (we always take $\chi \defeq  \tau^{I_1}$).

\begin{lem}\label{lem:coker}
Assume that $\tau\in W(\brho)$, {so $\tau$ is 2-generic}. Then
$\JH(\coker(\phi_{\tau}))\cap W(\brho)=\emptyset$, and \[\Ext^1_{\GL_2(\cO_K)/Z_1}(\tau',\tau)=0\] for any $\tau'\in \JH(\coker(\phi_{\tau}))$.
\end{lem}
\begin{proof}
  The first assertion is proved in  \cite[Cor.\ 4.14]{HuWang2} {(taking $\chi = \tau^{I_1}$ there)}.  The second is essentially a consequence of \cite[Cor.\ 3.11]{HuWang2}. To see this, let $\tau'$ be a %
  Serre weight such that $\Ext^1_{\GL_2(\cO_K)/Z_1}(\tau',\tau)\neq 0$ (this is equivalent to $\Ext^1_{\Gamma}(\tau',\tau)\neq 0$ by \cite[Lemma~\ref{bhhms4:lem:ext1}]{BHHMS4}, {noting that $\tau'$ is automatically 0-generic by \cite[Cor.\ 5.6(ii)]{BP}}).
We need to prove that $\tau'\notin \JH(\coker(\phi_{\tau}))$. By \cite[Prop.\ 3.12(ii)]{HuWang2} (where $\mathscr{E}(\tau)$ in \emph{loc.~cit.} is the set of Serre weights $\tau''$ such that $\Ext^1_{\Gamma}(\tau'',\tau)\neq0$), we know that $\tau'$ is a Jordan--H\"older factor of $\Ind_I^{\GL_2(\cO_K)}\overline{W}_{\chi,3}$, so we have a morphism $\phi_{\tau'}$ as in \eqref{eq:phitau}.  Since $[\coker(\phi_{\tau'}):\tau']=0$, it suffices to prove that $\im(\phi_{\tau'})\subset \im(\phi_{\tau})$.  But this follows from  \cite[Cor.~3.11(a), (c)]{HuWang2}. Indeed, if $\tau'\in \JH(\Ind_I^{\GL_2(\cO_K)}\chi)$, then we conclude by \cite[Cor.~3.11(a)]{HuWang2}, as $J(\tau)=\{0,\dots,f-1\}$ in the notation of \emph{loc.~cit.}; if $\tau'\in \JH(\Ind_I^{\GL_2(\cO_K)}\chi')$ for some $\chi'\in \JH(\overline{W}_{\chi,3})$ with $\chi'\neq \chi$, then we conclude by \cite[Cor.~3.11(c)]{HuWang2}.
\end{proof}

\begin{cor}\label{cor:coker-ext1}
  If $\tau\in W(\brho)$, then
  \begin{equation*}
    \Ext^1_{\wt{\Gamma}}(\coker(\phi_{\tau}),\wt{D}_0(\brho))=\Ext^1_{\wt{\Gamma}}(\coker(\phi_{\tau}),\tau)=0.
  \end{equation*}
\end{cor}

\begin{proof}
Using $\Ext^1_{\wt{\Gamma}}(\tau',\tau)=\Ext^1_{\GL_2(\cO_K)/Z_1}(\tau',\tau)$ for any Serre weight $\tau'$, the first term is $0$ by d\'evissage from Lemma~\ref{lem:ex-wtD}(i) and the first assertion in Lemma \ref{lem:coker}, and the second term is $0$ by d\'evissage from the second assertion in Lemma \ref{lem:coker}.
\end{proof}

\begin{lem}\label{lem:coker-j}
Assume that $\tau\in W(\brho)$.
Then $\coker(\phi_{\tau})$ has a direct sum decomposition 
\begin{equation}\label{eq:coker-decomp}
\coker(\phi_{\tau})\cong \bigoplus_{j=0}^{f-1}\coker(\phi_{\tau})_j,\end{equation}
where $\coker(\phi_{\tau})_j$ is a quotient of $\Ind_I^{\GL_2(\cO_K)}\chi\alpha_j$ for $0\leq j\leq f-1$. Moreover, 
\begin{enumerate}
\item if $\chi\alpha_j\in\JH(\pi^{I_1})$, then $\Ext^1_{\GL_2(\cO_K)/Z_1}\big(\coker(\phi_\tau)_j,D_0(\brho)\big)=0$;
\item if $\chi\alpha_j\notin \JH(\pi^{I_1})$, then $\Ext^1_{\GL_2(\cO_K)/Z_1}\big(\coker(\phi_\tau)_j,\pi\big)=0$. 
\end{enumerate}
\end{lem}
\begin{rem}
Although it will not be used in this paper, we have the following explicit description of $\coker(\phi_\tau)_j$: it is the unique quotient of $\Ind_I^{\GL_2(\cO_K)}\chi\alpha_j$ consisting of the Jordan--H\"older factors parametrized by the subsets of $\{0,\dots,f-1\}$ that contain $j$.  
\end{rem}

\begin{proof}
By construction, $\im(\phi_\tau)$ contains the image of any morphism $\Proj_{\wt{\Gamma}}\tau\ra \Ind_I^{\GL_2(\cO_K)}\overline{W}_{\chi,3}$, and in particular contains the subrepresentation $\Ind_I^{\GL_2(\cO_K)}\chi^{\oplus 2f}\subset \Ind_I^{\GL_2(\cO_K)}\overline{W}_{\chi,3}$ (recall that $\chi^{\oplus 2f}\subset \overline{W}_{\chi,3}$ by \eqref{eq:Wbar3} and that $\cosoc_\Gamma\big(\Ind_I^{\GL_2(\cO_K)}\chi\big)=\tau$).
Thus, the quotient map $\Ind_I^{\GL_2(\cO_K)}\overline{W}_{\chi,3}\onto \coker(\phi_\tau)$ factors through $\Ind_I^{\GL_2(\cO_K)}W_{\chi,2}\onto \coker(\phi_{\tau})$.  Recall that $W_{\chi,2}$ fits into a short exact sequence
\[0\ra \bigoplus_{j=0}^{f-1}(\chi\alpha_j\oplus \chi\alpha_j^{-1})\ra W_{\chi,2}\ra\chi\ra0.\]

We have a commutative diagram with exact rows
\begin{equation*}
  \xymatrix{0 \ar[r] & \ker(q) \ar[r]\ar[d] & \Proj_{\wt{\Gamma}}\tau \ar^-q[r]\ar^{\phi_\tau}[d] & \Ind_I^{\GL_2(\cO_K)}\chi \ar[r]\ar@{=}[d] & 0 \\
    0 \ar[r] & \bigoplus_{j=0}^{f-1}\Ind_I^{\GL_2(\cO_K)}\chi\alpha_j^{\pm 1}\ar[r] & \Ind_I^{\GL_2(\cO_K)}W_{\chi,2}\ar[r] & \Ind_I^{\GL_2(\cO_K)}\chi \ar[r] & 0
  }
\end{equation*}
so that we have a surjection $\gamma:\bigoplus_{j=0}^{f-1}\Ind_I^{\GL_2(\cO_K)}\chi\alpha_j^{\pm 1} \onto \coker(\phi_\tau)$ by the snake lemma.
As $\bigoplus_{j=0}^{f-1}\Ind_I^{\GL_2(\cO_K)}\chi\alpha_j^{\pm 1}$ is multiplicity free (for instance by \cite[Lemma~\ref{bhhms4:lem:multfree}]{BHHMS4}) we deduce an isomorphism $\coker(\phi_\tau)\cong \bigoplus_{j=0}^{f-1}\coker(\phi_\tau)^{\pm}_j$, where $\coker(\phi_\tau)^{\pm}_j \defeq \gamma(\Ind_I^{\GL_2(\cO_K)}\chi\alpha^{\pm 1}_j)$ (in particular it is a quotient of $\Ind_I^{\GL_2(\cO_K)}\chi\alpha^{\pm 1}_j$).
If $\coker(\phi_\tau)^{-}_j\neq 0$ then $\coker(\phi_\tau)$ and a fortiori $\Ind_I^{\GL_2(\cO_K)}W_{\chi,2}$ would surject onto $\delta_j^-(\tau)$ (the cosocle of $\Ind_I^{\GL_2(\cO_K)}\chi\alpha_j^{-1}$).
But this is not true by Frobenius reciprocity, as one checks that $\Hom_{I}(W_{\chi,2}, \delta_j^-(\tau))=0$ by \cite[Lemma 3.42(ii)]{BHHMS2}.
We thus get the decomposition \eqref{eq:coker-decomp} by taking $\coker(\phi_\tau)_j\defeq \coker(\phi_\tau)^+_j$. 

(i) By Lemma \ref{lem:ex-wtD}(ii) and the first statement in Lemma \ref{lem:coker}, it suffices to show that $\JH(\coker(\phi_\tau)_j)\subset \JH(D_0(\brho))$ when $\chi\alpha_j\in \JH(\pi^{I_1})$. In fact, we prove the following stronger statement: if $\chi'\in \JH(\pi^{I_1})$ then $\JH(\Ind_I^{\GL_2(\cO_K)}\chi')\subset \JH(D_0(\brho))$.  By \cite[eq.~\eqref{bhhms4:eq:JH-D0-rho}]{BHHMS4} we have $\JH(\Proj_{\Gamma}\sigma)\subset \JH(D_0(\brho))$ for any $\sigma\in W(\brho)$. Now, since  $\chi'\in \JH(\pi^{I_1})$, we have  \[\JH(\Ind_I^{\GL_2(\cO_K)}\chi')\cap W(\brho)\neq \emptyset\] by \cite[Prop.~4.2]{breuil-buzzati}. Thus it suffices to prove that 
\begin{equation}\label{eq:JH-inclusion}\JH(\Ind_I^{\GL_2(\cO_K)}\chi')\subset \JH(\Proj_{\Gamma}\sigma') 
\end{equation}
for any $\sigma'\in \JH(\Ind_I^{\GL_2(\cO_K)}\chi')$. 
We prove \eqref{eq:JH-inclusion} for any character $\chi':I\ra \F^{\times}$ satisfying $\chi'\neq \chi'^s$. Let $\Proj_{W(\F)[\Gamma]}\sigma'$ be the projective cover of $\sigma'$ in the category of $W(\F)[\Gamma]$-modules. Let $[\chi']:I\ra W(\F)^{\times}$ be the Teichm\"uller lift of $\chi'$. Since $\sigma'\in \JH(\Ind_I^{\GL_2(\cO_K)}\chi')$, there is a non-zero morphism $\gamma:\Proj_{W(\F)[\Gamma]}\sigma'\ra \Ind_I^{\GL_2(\cO_K)}[\chi']$. Inverting $p$, the latter representation is irreducible over $W(\F)[1/p]$ as $[\chi']\neq [\chi'^s]$, so $\gamma$ is surjective after inverting $p$. We conclude by the Brauer--Nesbitt theorem {(cf.\ \cite[Chap.~15, Thm.~32]{serre-book}).} 

(ii) It is a direct consequence of Lemma \ref{lem:Ext1-Serre}.
\end{proof}

\begin{lem}\label{lem:4.18}
Let $\tau\in W(\brho)$ and $Q$ be a quotient of $\Proj_{\wt{\Gamma}}\tau$ such that $\rad_{\wt\Gamma}(Q)\subset \wt{D}_0(\brho)$ (hence $\rad_{\wt\Gamma}(Q)$ is multiplicity free). Then $Q$ is a quotient of $\im(\phi_{\tau})$. %
\end{lem}
\begin{proof}
{In this proof, if $M$ is a finite-dimensional \  $\wt{\Gamma}$-module, we \ write \  $\rad(M)$, \  $\soc(M)$ and \  $\cosoc(M)$ for $\rad_{\wt\Gamma}(M)$, $\soc_{\wt\Gamma}(M)$ and $\cosoc_{\wt\Gamma}(M)$ respectively.}
We may assume that $Q \ne 0$. %
Since $\rad(Q)$ is multiplicity free by assumption, we have $[\rad(Q):\tau]\leq 1$ and $[Q:\tau]\leq 2$.
Since $\rad(Q)\subset \widetilde{D}_0(\brho)=\bigoplus_{\sigma\in W(\brho)}\wt{D}_{0,\sigma}(\brho)$ (which is multiplicity free), we have a decomposition \[\rad(Q)=\bigoplus_{\sigma\in W(\brho)}V_{\sigma}\]  for some subrepresentations $V_\sigma \subset \wt{D}_{0,\sigma}(\brho)$. If $V_{\sigma}\neq0$, let $Q_\sigma$ be the quotient of $Q$ by its largest subrepresentation in which $\sigma$ does not occur, %
so $\soc(Q_{\sigma})\cong\sigma$ (even if $\sigma = \tau$, as $\cosoc(Q_\sigma)=\tau$) and $0\ra V_{\sigma}\ra Q_{\sigma}\ra\tau\ra0$. 
Assume first $\sigma\neq \tau$ and $V_{\sigma}\neq0$. By \cite[Lemma 4.10]{HuWang2}, the natural morphism
\begin{equation}\label{eq:isom-ext1}
  \Ext^1_{\wt{\Gamma}}(\tau,\sigma) \to \Ext^1_{\wt{\Gamma}}(\tau,V_\sigma) 
\end{equation}
is an isomorphism. Since $Q_{\sigma}$ has cosocle $\tau$, we deduce that $V_{\sigma}=\sigma$. 
Assume next that $\sigma=\tau$ and also $V_{\tau}\neq 0$. Then \cite[Cor.~4.9]{HuWang2} implies that the natural inclusion $D_{0,\tau}(\brho)\hookrightarrow \wt{D}_{0,\tau}(\brho)$
induces an isomorphism \begin{equation}\label{eq:Ext1-tau-isom}\Ext^1_{\wt{\Gamma}}(\tau,D_{0,\tau}(\brho))\simto \Ext^1_{\wt{\Gamma}}(\tau,\wt{D}_{0,\tau}(\brho)).\end{equation}
Letting  $A\defeq V_{\tau}\cap D_{0,\tau}(\brho)\neq 0$, we obtain a commutative diagram  with exact rows
\[\xymatrix{0\ar[r]& \Ext^1_{\wt{\Gamma}}(\tau,A)\ar[r]\ar@{^{(}->}[d]&\Ext^1_{\wt{\Gamma}}(\tau,V_{\tau})\ar[r]\ar@{^{(}->}[d]&\Ext^1_{\wt{\Gamma}}(\tau,V_{\tau}/A)\ar@{^{(}->}[d]\\
0\ar[r]&\Ext^1_{\wt{\Gamma}}(\tau,D_{0,\tau}(\brho))\ar^{\sim}[r]&\Ext^1_{\wt{\Gamma}}(\tau,\wt{D}_{0,\tau}(\brho))\ar[r]&\Ext^1_{\wt{\Gamma}}(\tau,\wt{D}_{0,\tau}(\brho)/D_{0,\tau}(\brho))}\]
where all the vertical arrows are easily seen to be injective (as $\wt{D}_{0,\tau}(\brho)$ is multiplicity free). A diagram chase together with \eqref{eq:Ext1-tau-isom} shows that the class of $Q_{\tau}$ in $\Ext^1_{\wt{\Gamma}}(\tau,V_{\tau})$ lies in the image of $\Ext^1_{\wt{\Gamma}}(\tau,A)$. Since $Q_{\tau}$ has cosocle $\tau$, we have $A=V_{\tau}$, namely $V_{\tau}\subset D_{0,\tau}(\brho)$. Altogether we get $\rad(Q)\subset D_0(\brho)$. 

Now we prove the lemma. If $\tau$ does not occur in $\rad(Q)$, then $[Q:\tau]=1$ and $V_{\tau}=0$. Moreover,  the discussion in the last paragraph implies that $\rad(Q)=\bigoplus_{\sigma\in \JH(\soc(Q))}\sigma$ (provided $\rad(Q)\neq 0$). Thus  $Q$ is a $\Gamma$-representation by \cite[Lemma~\ref{bhhms4:lem:ext1}]{BHHMS4} (and the first sentence in its proof), 
and \cite[Cor.~3.14]{HuWang2} (applied with $m=0$) implies that  $Q$ is a certain quotient of $\Theta_{\tau}$, where  $\Theta_{\tau}$ in \emph{loc.~cit.} is a quotient of $\im(\phi_{\tau})$ constructed in \cite[Prop.\ 3.12]{HuWang2}. As a consequence, $Q$ is a quotient of $\im(\phi_{\tau})$.   

If $\tau$ occurs in $\rad(Q)$, then $\tau$ must occur in $\soc(Q)$ as  $\rad(Q)\subset D_0(\brho)$ and  $\tau \in \JH(\soc(D_0(\brho)))$.    In this case $Q$  satisfies the following conditions:  
\begin{enumerate}
\item[(1)] $[Q:\tau]=2$, $\tau\hookrightarrow \soc(Q)$, and $\cosoc(Q)\cong\tau$;
\item[(2)] $\rad(Q)$ is a subrepresentation of $D_0(\brho)$.
\end{enumerate}
It is proved in the last paragraph of the proof of \cite[Prop.~4.18]{HuWang2} that such a representation is a quotient of $\Theta_{\tau}$, hence of $\im(\phi_{\tau})$. The argument goes as follows. Firstly, by \cite[Lemma~4.10]{HuWang2} condition (2) implies that $\JH(\soc(Q))$ is contained in $\{\tau\}\cup \mathscr{E}(\tau)$, where $\mathscr{E}(\tau)$ denotes the set of Serre weights $\tau'$ such that $\Ext^1_{\Gamma}(\tau',\tau)\neq0$  (equivalently, $\Ext^1_{\Gamma}(\tau,\tau')\neq0$). Secondly,  using (1) and the fact $\Ext^1_{\widetilde{\Gamma}}(\tau,\tau)=0$, one shows that the socle of $C\defeq Q/\tau$ is contained in $\mathscr{E}(\tau)$ and, using again \cite[Lemma~4.10]{HuWang2} and $\cosoc_{\widetilde{\Gamma}}(Q)=\tau$, that $C$   fits in a short exact sequence
\[0\ra S\ra C\ra\tau\ra 0\]
for some subrepresentation $S$ of $\bigoplus_{
\tau'\in\mathscr{E}(\tau)}\tau'$. Then we conclude by \cite[Cor.~3.14]{HuWang2}.
\end{proof}

\begin{lem}\label{lem:HW4.20}
Assume that $\tau\in  W(\brho)$. Then $ \Hom_{\GL_2(\cO_K)}(\Proj_{\wt{\Gamma}}\tau,\pi)$ has dimension $1$ over $\F$.
\end{lem}

\begin{proof}
\textbf{Step 1.} We prove that $\Hom_{\GL_2(\cO_K)}(\tau,\pi/\pi[\m_{K_1}])=0$. Suppose by contradiction that $\Hom_{\GL_2(\cO_K)}(\tau,\pi/\pi[\m_{K_1}])\neq0$. The pullback of $\tau$ gives   a subrepresentation $V\subset \pi|_{\GL_2(\cO_K)}$ which (using assumption \ref{it:assum-i}) fits into a nonsplit extension
\begin{equation}
\label{eq:non-split:Chr}
0\ra D_0(\brho)\ra V\ra \tau\ra0.
\end{equation}
Note that $V$ is a $\wt{\Gamma}$-representation but not a $\Gamma$-representation. By the projectivity of $\Proj_{\wt{\Gamma}}\tau$, there exists  a  $\wt{\Gamma}$-equivariant morphism $q:\Proj_{\wt{\Gamma}}\tau\ra V$ whose composition with $V\twoheadrightarrow \tau$ is the natural surjection $\Proj_{\wt{\Gamma}}\tau\twoheadrightarrow \tau$.   Let $V_\tau$ denote the image of $q$, which has cosocle $\tau$. Clearly $V_{\tau}$ satisfies the conditions (on $Q$) in Lemma \ref{lem:4.18}, so there exists a surjection $\im(\phi_{\tau})\onto V_\tau$ and we denote by $\beta$ the composition $\im(\phi_{\tau})\onto V_\tau \into V$.

We introduce a 3-step filtration on $M\defeq\Ind_I^{\GL_2(\cO_K)}\overline{W}_{\chi,3}$ as follows. Let $S\subset \{0,\dots,f-1\}$ be the set of indices $j$ such that $\chi\alpha_j\in \JH(\pi^{I_1})$.  Put  $M_2\defeq \im(\phi_{\tau})\subset M$ and 
\[M_1\defeq \ker\big(M\twoheadrightarrow \bigoplus_{j\notin S}\coker(\phi_{\tau})_j\big), \]
where we used \eqref{eq:coker-decomp}.
Then $0\subset M_2\subset M_1\subset M$ with 
\begin{equation}\label{eq:filt-M}
M_1/M_2\cong \bigoplus_{j\in S}\coker(\phi_{\tau})_j,\quad M/M_1\cong \bigoplus_{j\notin S}\coker(\phi_{\tau})_j.\end{equation} 

By Lemma \ref{lem:coker}, Lemma \ref{lem:coker-j}(i) and \eqref{eq:non-split:Chr}, \eqref{eq:filt-M} we have $\Ext^1_{\GL_2(\cO_K)/Z_1}(M_1/M_2,V)=0$, so the natural morphism 
$\Hom_{\GL_2(\cO_K)}(M_1,V)\ra \Hom_{\GL_2(\cO_K)}(M_2,V)$ is surjective. Thus we can lift $\beta$ to $\beta':M_1\ra V$, which we view as a morphism $\beta':M_1\ra \pi$ (as $V\subset \pi$). Next, since $\Ext^1_{\GL_2(\cO_K)/Z_1}(M/M_1,\pi)=0$ by Lemma \ref{lem:coker-j}(ii) and \eqref{eq:filt-M}, we can further lift $\beta'$ to a morphism $\beta'':M\ra \pi$. By Frobenius reciprocity, we obtain an $I$-equivariant morphism $\overline{W}_{\chi,3}\ra\pi|_I$ whose image is then contained in $\pi[\m^3]$ and has cosocle $\chi$.
 As $\tau\in W(\rhobar)$ we have $\chi\into \pi^{I_1}=\pi[\m]$. 
By assumption \ref{it:assum-ii} we conclude that  $\chi$ does not appear in $\pi[\m^3]/\pi[\m]$, hence $\overline{W}_{\chi,3}\ra\pi|_I$ factors through $\overline{W}_{\chi,3}\onto \chi\into \pi|_{I}$.
 Correspondingly,   $\beta''$ itself factors through $M\onto \Ind_I^{\GL_2(\cO_K)}\chi\ra \pi$, so $\im(\beta'')$ is contained in $\pi[\m_{K_1}]$. But this is not true by construction of $\beta''$, contradiction. 

\textbf{Step 2.} Suppose by contradiction that $\dim_{\F}\Hom_{\GL_2(\cO_K)}(\Proj_{\wt{\Gamma}}\tau,\pi)\geq 2$. By \cite[Prop.~4.18]{HuWang2}, which requires $\brho$ to be $2$-generic and condition (a) at the beginning of \cite[\S~4.3]{HuWang2} to hold, we also have $\dim_{\F}\Hom_{\GL_2(\cO_K)}(\Theta_{\tau},\pi)\geq 2$. 
 (Recall from \cite[\S~3.3]{HuWang2} that $\Theta_{\tau}$ is the smallest quotient of $\Proj_{\wt{\Gamma}}\tau/\rad^3_{\wt{\Gamma}}\big(\Proj_{\wt{\Gamma}}\tau\big)$ such that
$\big[\Proj_{\wt{\Gamma}}\tau/\rad^3_{\wt{\Gamma}}\big(\Proj_{\wt{\Gamma}}\tau\big):\tau\big]=[\Theta_{\tau}:\tau]$ 
 and that $\Theta_{\tau}$ fits into a short exact sequence
\[
0\ra \bigoplus_{\tau'}E_{\tau,\tau'}\ra \Theta_{\tau}\ra \tau\ra 0,
\]
where the direct sum is taken over the Serre weights $\tau'$  such that $\Ext^1_{\Gamma}(\tau',\tau)\neq0$; see \cite[Cor.~3.16]{HuWang2}.)
Thus,  there exists a $\GL_2(\cO_K)$-equivariant morphism $\gamma:\Theta_{\tau}\ra \pi$ which does not factor through the cosocle of $\Theta_\tau$. Since $\pi[\m_{K_1}]\cong D_0(\brho)$ is multiplicity free and $\tau\cong \cosoc_{\wt{\Gamma}}(\Theta_\tau)$ occurs in $\soc_{\GL_2(\cO_K)}(\pi)$, we deduce that $\im(\gamma)$ is not contained in $\pi[\m_{K_1}]$. However, $\rad_{\wt{\Gamma}}(\Theta_\tau)$ is annihilated by $\m_{K_1}$ (by \cite[Cor.~3.16]{HuWang2}), so the image $U$ of $\rad_{\wt{\Gamma}}(\Theta_\tau)$ is contained in $\pi[\m_{K_1}]$. The inclusions $U\subset \pi[\m_{K_1}]\subset \pi$ induce natural maps
\begin{equation}
\label{eq:ext:vanish}
\Ext^1_{\GL_2(\cO_K)/Z_1}(\tau,U)\ra \Ext^1_{\GL_2(\cO_K)/Z_1}(\tau,\pi[\m_{K_1}])\ra \Ext^1_{\GL_2(\cO_K)/Z_1}(\tau,\pi).
\end{equation}
The first map is injective, because  by assumption \ref{it:assum-i} either $\Hom_{\GL_2(\cO_K)}(\tau,\pi[\m_{K_1}]/U)=0$ (if $\tau\in\JH(U)$) or  the map $\Hom_{\GL_2(\cO_K)}(\tau,\pi[\m_{K_1}])\ra \Hom_{\GL_2(\cO_K)}(\tau,\pi[\m_{K_1}]/U)$ is surjective (if $\tau\notin \JH(U)$). The second map is also injective by Step 1. However, viewing $\im(\gamma)$ as a (non-zero) element in $\Ext^1_{\GL_2(\cO_K)/Z_1}(\tau,U)$, which is sent to $0$ via the natural map $\Ext^1_{\GL_2(\cO_K)/Z_1}(\tau,U)\ra\Ext^1_{\GL_2(\cO_K)/Z_1}(\tau,\im(\gamma))$, we conclude that $\im(\gamma)$ is sent to zero via \eqref{eq:ext:vanish} since the latter factors through $\Ext^1_{\GL_2(\cO_K)/Z_1}(\tau,U)\ra \Ext^1_{\GL_2(\cO_K)/Z_1}(\tau,\im(\gamma))$.
This gives the desired contradiction.
\end{proof}

\begin{prop}\label{prop:ii'}
  Suppose that $\pi$ satisfies assumptions \ref{it:assum-i}, \ref{it:assum-ii} and \ref{it:assum-iv} {with a 2-generic underlying $\brho$}. Then 
\begin{equation}\label{eq:m2=wtD}
\pi[\fm_{K_1}^2]\cong\wt{D}_0(\brho).
\end{equation}
\end{prop}
\begin{proof} 
It follows from Lemma \ref{lem:HW4.20} that $[\pi[\m_{K_1}^2]:\sigma]=1$ for any $\sigma\in W(\brho)$ (recall that $[\pi[\m_{K_1}^2]:\sigma]$ is precisely the dimension of $\Hom_{\GL_2(\cO_K)}(\Proj_{\widetilde{\Gamma}}\sigma,\pi)$). From the construction of $\wt{D}_{0}(\brho)$ we deduce an inclusion $\pi[\fm_{K_1}^2]\subset \wt{D}_0(\brho)$. Suppose the inclusion is strict, and choose a Serre weight $\tau\into \wt{D}_0(\brho)/\pi[\fm_{K_1}^2]$. Let $V_{\tau}\subset \wt{D}_0(\brho)$ be a subrepresentation with cosocle $\tau$ and such that the composition $V_{\tau}\into \wt{D}_0(\brho)\onto \wt{D}_0(\brho)/\pi[\fm_{K_1}^2]$ coincides with the chosen inclusion $\tau\into \wt{D}_0(\brho)/\pi[\fm_{K_1}^2]$.
As $D_0(\brho)\subset \pi[\fm_{K_1}^2]$ by assumption \ref{it:assum-i}, we have  $\tau\in \JH(\wt{D}_0(\brho))\setminus \JH(D_0(\brho))$, so in particular $\chi_{\tau}\notin \JH(\pi^{I_1})$. Applying $\Hom_{\GL_2(\cO_K)/Z_1}(-,\pi)$ to $0\ra \rad_{\wt{\Gamma}}(V_{\tau})\ra V_{\tau}\ra \tau\ra0$ and using Lemma \ref{lem:Ext1-Serre}, we obtain an isomorphism
\[\Hom_{\GL_2(\cO_K)}(V_{\tau},\pi)\simto \Hom_{\GL_2(\cO_K)}(\rad_{\wt{\Gamma}}(V_{\tau}),\pi).\]
Thus, the natural inclusion $\rad_{\wt{\Gamma}}(V_{\tau})\subset\pi$ lifts to an embedding $V_{\tau}\into \pi$, whose image is contained in $\pi[\m_{K_1}^2]$ as $V_{\tau}$ is annihilated by $\m_{K_1}^2$. This gives a contradiction as $\tau\notin\JH(\pi[\fm_{K_1}^2])$ ($\wt{D}_0(\brho)$ being multiplicity free). 
\end{proof}

\section{On the Hilbert series of \texorpdfstring{$\pi$}{pi}}
\label{sec:hilbert-polynomial}

Let $\pi$ be a smooth mod $p$ representation of $\GL_2(K)$ over $\F$ satisfying assumptions \ref{it:assum-i}, \ref{it:assum-ii} and \ref{it:assum-iv} of \S~\ref{sec:abstract-setting}. %
In this section we compute the Hilbert series of $\gr_{\m}(\pi^{\vee})$.

If $M=\bigoplus_{n\leq 0} M_n$ is a graded $\F$-vector space {with $\dim_{\F}M_n<+\infty$ for all $n$}, we define the \emph{Hilbert series} 
\[h_M(t)\defeq \sum_{n\geq 0} \dim_{\F}(M_{-n})t^n\in \Z\bbra{t}.\]
{In particular, if $\Pi$ is any admissible smooth representation of $\GL_2(K)$ the Hilbert series $h_\Pi(t) \defeq h_{\gr_\m(\Pi^\vee)}(t) \in \Z\bbra t$ is defined.}
\begin{thm1}\label{thm:Hilbert}\
{Assume that $\brho$ is $9$-generic and that $\pi$ satisfies assumptions \ref{it:assum-i}, \ref{it:assum-ii} and \ref{it:assum-iv} of \S~\ref{sec:abstract-setting}.}
\begin{enumerate}
\item If $\brho$ is irreducible, then $\displaystyle h_{{\pi}}(t)=\frac{(3+t)^f}{(1-t)^f}-1$. %
\item If $\brho$ is split reducible, then $\displaystyle h_{{\pi}}(t)=\frac{(3+t)^f}{(1-t)^f}+1$.
\item If $\brho$ is nonsplit reducible and $d_{\brho}\defeq |J_{\brho}|$ (so $d_{\brho} < f$, see \eqref{eq:P} for $J_{\brho}$), then $\displaystyle h_{{\pi}}(t)=2^{f-d_{\brho}}\cdot   \frac{(1+t)^{f-d_{\brho}}(3+t)^{d_{\brho}}}{(1-t)^f}$.
\end{enumerate}
\end{thm1}

\begin{rem1}
  {Note that the denominator of $h_\pi(t)$ equals $(1-t)^f$ expresses the fact that the Gelfand--Kirillov dimension of $\pi$ equals $f$. 
    (By \cite[Lemma 5.1.3]{BHHMS1}, the Gelfand--Kirillov dimension of $\pi$ equals the dimension of $\gr_\m(\pi^\vee)$ as an $\o R$-module, hence equals the dimension of $(\gr_\m(\pi^\vee))_{\o\m}$ as an $\o R_{\o\m}$-module by \cite[Ex.\ 1.5.25]{BH93}, hence equals the exponent of $(1-t)$ in the denominator of $h_\pi(t)$ by \cite[Thms.\ 13.2, 13.4]{Ma}, cf.\ the discussion on \cite[p.\ 97]{Ma}.)}
\end{rem1}

\begin{rem1}\label{rem:Hilbert-pi}
 Note that if we put $t = 0$ we recover the dimension formula for $D_0(\brho)^{I_1} = \pi^{I_1}$ in \cite[Thm.\ 1.1]{BP}.
\end{rem1}

\begin{proof}
{We first recall that, under our assumptions, by~\cite[Thm.~\ref{bhhms4:thm:CMC}]{BHHMS4} we have 
  \[\gr_{\m}(\pi^{\vee}) \cong N \defeq \bigoplus_{\lambda\in\P}\chi_{\lambda}^{-1}\otimes R/\fa(\lambda),
  \] 
  where $\mathfrak{a}(\lambda)$ is the ideal of $R$ associated to $\lambda\in\P$ in \eqref{eq:id:al}.
  It remains to determine $h_N(t)$.}

{We note the following elementary but useful formulas.
First, if $M, M'$ are two graded $\F$-vector spaces, then 
\begin{equation}\label{eq:Hilbert-tensor}h_{M\otimes M'}(t)=h_M(t)h_{M'}(t).\end{equation}
Second, we have for any integer $n\geq 0$:
\begin{align}
\label{eq:binom1}\frac{1}{2}\Big[(2+x)^n-(2-x)^n\Big]&=\sum_{0\leq i\leq n,\ i\ \mathrm{odd}}\binom{n}{i} 2^{n-i}x^i,\\
\label{eq:binom2}\frac{1}{2}\Big[(2+x)^n+(2-x)^n\Big]&=\sum_{0\leq i\leq n,\ i\ \mathrm{even}}\binom{n}{i}2^{n-i}x^i.
\end{align}
}

By definition of $N$, we have $h_{N}(t)=\sum_{\lambda\in\P} h_{R/\mathfrak{a}(\lambda)}(t)$. Recalling $\mathfrak{a}(\lambda)=(t_j; 0\leq j\leq f-1)$ with $t_j\in\{y_j,z_j,y_jz_j\}$ and noting that \[h_{\F[y_i]}(t)=h_{\F[z_i]}(t)=1/(1-t),\ \ h_{\F[y_i,z_i]/(y_iz_i)}(t)=(1+t)/(1-t),\] we obtain by \eqref{eq:Hilbert-tensor} that 
\[h_{R/\mathfrak{a}(\lambda)}(t)=\frac{(1+t)^{|\mathcal{A}(\lambda)|}}{(1-t)^f},\]
where $\mathcal{A}(\lambda)\defeq \{j: t_j=y_jz_j\}$. Hence, we are reduced to counting the cardinality of $\lambda\in \P$ such that $|\mathcal{A}(\lambda)|=s$ for a given $0\leq s\leq f$. 

(i) Given $\lambda\in \P$, we define an element $\overline{\lambda}\in \D$ as follows:
\[\overline{\lambda}_0(x_0)\defeq
\begin{cases}
  x_0-1 & \text{if $\lambda_0(x_0)\in\{x_0-1,x_0+1\}$},\\
  p-2-x_0 & \text{if $\lambda_0(x_0)\in\{p-2-x_0,p-x_0\}$},\\
  \lambda_0(x_0) & \text{otherwise}
\end{cases}
\] 
and if $j\neq 0$,
\[\overline{\lambda}_j(x_j)\defeq 
\begin{cases}
  x_j & \text{if $\lambda_j(x_j)\in\{x_j,x_j+2\}$},\\
  p-3-x_j & \text{if $\lambda_j(x_j)\in\{p-1-x_j,p-3-x_j\}$},\\
  \lambda_j(x_j) & \text{otherwise}.
\end{cases}
\]
It is easy to see that $\o{\lambda}\in \D$. By \cite[Def.~3.57]{BHHMS2},  we have $t_j=y_jz_j$  if and only if $\lambda_j(x_j)\in \{x_j+1,p-2-x_j\}$ if $j\neq 0$ (resp. $\lambda_0(x_0)\in \{x_0,p-1-x_0\}$), thus $\mathcal{A}(\lambda)=\mathcal{A}(\o{\lambda})$.  On the other hand, given  $\overline{\lambda}\in \D$, there exist exactly  $2^{|\{0,\dots,f-1\}\setminus \mathcal{A}(\overline{\lambda})|}$ elements $\lambda\in \P$ giving rise to $\overline{\lambda}$. As a consequence we have  
$h_N(t)=Q_N(t)/(1-t)^f$ with
 \[Q_N(t)\defeq \sum_{0\leq s\leq f,\ s\ \mathrm{odd}} 2^{f-s}\cdot 2\binom{f}{s}(1+t)^s=(3+t)^f-(1-t)^f, 
 \]
{where the first equality follows from  Lemma \ref{lem:Hilbert-count} below and the second from  \eqref{eq:binom1} (with $x=1+t$).}  
The result follows. 

(ii) The proof is similar to (i) using Lemma \ref{lem:Hilbert-count}(ii) below and \eqref{eq:binom2}.

(iii) 
Let $\overline{\P}\subset \P$ be the subset introduced in the proof of \cite[Prop.~3.61]{BHHMS2}, namely $\lambda\in\overline{\P}$ if and only if 
\begin{equation}\label{eq:overline-P}
\lambda_j(x_j)\in\{x_j,x_j+1,p-1-x_j,p-2-x_j,p-3-x_j\}\end{equation}
and $\lambda_j(x_j)=p-1-x_j$ implies $j\notin J_{\brho}$ (recall from~\eqref{eq:P} that $\lambda_j(x_j)=p-3-x_j$ implies $j\in J_{\brho}$). In the proof of \cite[Prop.~3.61]{BHHMS2} a map $\P\to\overline{\P}$, $\lambda\mapsto \overline{\lambda}$ is defined, which satisfies $\mathcal{A}(\lambda)=\mathcal{A}(\overline{\lambda})$ and for any $\o\lambda\in\overline{\P}$, there exist exactly $2^{|\{0,\dots,f-1\}\setminus \mathcal{A}(\overline{\lambda})|}$ elements $\lambda$ in $\P$ giving rise to $\overline{\lambda}$.
Using Lemma \ref{lem:Hilbert-count}(iii) below (with $|\mathcal{A}(\lambda)|=f-d_{\brho}+s$), we then obtain $h_N(t)=Q'_N(t)/(1-t)^f$, where 
\[Q'_N(t)\defeq \sum_{0\leq s\leq d_{\brho}}2^{d_{\brho}-s}\cdot 2^{f-d_{\brho}}\binom{d_{\brho}}{s}(1+t)^{(f-d_{\brho})+s}=2^{f-d_{\brho}}(1+t)^{f-d_{\brho}}  (3+t)^{d_{\brho}}\] 
(recall $d_{\brho}= |J_{\brho}|$), proving the result.
\end{proof}

\begin{rem1}
{We note that our proof determines $h_N(t)$ in each case without any genericity conditions on $\brho$.}%
\end{rem1}

\begin{lem1}\label{lem:Hilbert-count}\ 
\begin{enumerate}
\item If $\brho$ is irreducible, then $|\mathcal{A}(\lambda)|$ is odd for all $\lambda\in\D$. For any subset $J\subset\{0,\dots,f-1\}$ with $|J|$ odd, there exist exactly $2$ elements $\lambda\in\D$ such that $\mathcal{A}(\lambda)=J$. As a consequence, for any $0\leq s\leq f$ which is odd, the set of  $\lambda\in \D$ with $|\mathcal{A}(\lambda)|=s$ has cardinality $2\binom{f}{s}$. 
 \item If $\brho$ is split reducible, then $|\mathcal{A}(\lambda)|$ is even for all $\lambda\in\D$. For any subset $J\subset\{0,\dots,f-1\}$ with $|J|$ even, there exist exactly $2$ elements $\lambda\in\D$ such that $\mathcal{A}(\lambda)=J$. As a consequence, for any $0\leq s\leq f$ which is even, the set of  $\lambda\in \D$ with $|\mathcal{A}(\lambda)|=s$ has cardinality $2\binom{f}{s}$.
\item If $\brho$ is nonsplit reducible, then $J_{\brho}^c\subset \mathcal{A}(\lambda)$ for any $\lambda\in\overline{\P}$ (where $\o\P$ is defined in the proof of Theorem \ref{thm:Hilbert}(iii)),  and for any $J\subset J_{\brho}$ the set of $\lambda\in\overline{\P}$ with $\mathcal{A}(\lambda)=J\sqcup J_{\brho}^c$ has cardinality $2^{f-d_{\brho}}$. In particular, we always have $f-d_{\brho}\leq |\mathcal{A}(\lambda)|\leq f$, and for any $0\leq s\leq d_{\brho}$ the set of $\lambda\in\overline{\P}$ with $|\mathcal{A}(\lambda)|=f-d_{\brho}+s$ has cardinality $2^{f-d_{\brho}}\binom{d_{\brho}}{s}$.
\end{enumerate}
\end{lem1}
\begin{proof}
(i) By the definition of  $\mathcal{A}(\lambda)$, we have 
\begin{equation}\label{eq:A-irred}\mathcal{A}(\lambda)=  \{j: \lambda_j(x_j)\in\{x_j+1,p-2-x_j\}\ \mathrm{if}\ j\neq 0,\ \mathrm{or}\ \lambda_0(x_0)\in \{x_0,p-1-x_0\}\ \mathrm{if}\ j=0\}.\end{equation}
By definition of $\D$ (see \cite[\S~11]{BP}) and of $\mathcal{A}(\lambda)$, we check that $\lambda_{j}(x_{j})$ is determined by $\lambda_{j-1}(x_{j-1})$ and the value of $\mathbf{1}_{\mathcal{A}(\lambda)}(j)$ for any $j$.  
  For example, if $\lambda_0(x_0)=x_0$  and $f \ge 2$, then $\lambda_1(x_1)=x_1$ (resp.~$\lambda_1(x_1)=p-2-x_1$) if $1\notin \mathcal{A}(\lambda)$ (resp.~$1\in \mathcal{A}(\lambda)$). 
This implies that $\lambda\in \D$ is determined by $\lambda_0(x_0)$ and $\mathcal{A}(\lambda)$. Moreover, one checks that:
\begin{itemize}
\item $|\mathcal{A}(\lambda)\cap \{1,\dots,f-1\}|$ is even if $\lambda_0(x_0)\in \{x_0,p-1-x_0\}$ (by showing that $|\{j\neq 0:\lambda_j(x_j)=p-2-x_j\}|=|\{j\neq0:\lambda_j(x_j)=x_j+1\}|$), and is odd if $\lambda_0(x_0)\in \{x_0-1,p-2-x_0\}$ (by showing that $|\{j\neq 0:\lambda_j(x_j)=p-2-x_j\}|=|\{j\neq0:\lambda_j(x_j)=x_j+1\}|\pm 1$). As a consequence, $|\mathcal{A}(\lambda)|$ is always odd by \eqref{eq:A-irred}. 
 \item Conversely, if $\lambda_0(x_0)\in \{x_0,p-1-x_0\}$ (resp.~$\lambda_0(x_0)\in\{x_0-1,p-2-x_0\}$) and $J\subset \{1,\dots,f-1\}$ is even (resp.~odd), then there exists a unique $\lambda\in \D$ with given value at $j=0$ and such that $J=\mathcal{A}(\lambda)\cap \{1,\dots,f-1\}$.  
\end{itemize}
Thus, for any $J\subset \{0,\dots,f-1\}$ with $|J|$ odd, there exist exactly two $\lambda\in \D$ with $\mathcal{A}(\lambda)=J$. The result follows from this.   

(ii) The proof is similar to (and simpler than) (i). 
In this case, one has $\mathcal{A}(\lambda)=\{j:\lambda_j(x_j)\in\{x_j+1,p-2-x_j\}\}$ and it follows directly from the definition of $\D$ that the subsets $|\{j:\lambda_j(x_j)=x_j+1\}|$ and $|\{j:\lambda_j(x_j)=p-2-x_j\}|$ of $\Z/f\Z$ are interlaced, i.e.\ between any two distinct elements of one subset there exists an element of the other, and hence of the same cardinality.

(iii) %
By the proof of \cite[Lemma~3.59]{BHHMS2}, there is a bijection between $\overline{\P}$ and $\D^{\rm ss}$ as follows: $\lambda\in \overline{\P}$ corresponds to $\mu\in \D^{\ss}$ defined by
\[\mu_j(x_j)\defeq
\begin{cases}
  p-3-x_j & \text{if $\lambda_j(x_j)=p-1-x_j$},\\
  \lambda_j(x_j) & \text{otherwise}.
\end{cases}
\] 
One checks that $\mathcal{A}(\lambda)=(\mathcal{A}(\mu)\cap J_{\brho})\sqcup J_{\brho}^c$,  so in particular $J_{\brho}^c\subset \mathcal{A}(\lambda)$.
Thus for a given $J\subset J_{\brho}$, 
\[|\{\lambda\in\overline{\P}:\mathcal{A}(\lambda)=J\sqcup J_{\brho}^c\}|=|\{\mu\in\D^{\ss}: \mathcal{A}(\mu)\cap J_{\brho}=J\}|,\]
where $\mathcal{A}(\mu)$ is formed with respect to $\brho^{\ss}$. 
If $|J|$ is even, then for any $J'\subset J_{\brho}^c$ with $|J'|$ being even, there exist exactly $2$ elements $\mu\in\D^{\ss}$ such that $\mathcal{A}(\mu)=J\sqcup J'$ by (ii), so  the cardinality of $\mu\in\D^{\ss}$ satisfying $\mathcal{A}(\mu)\cap J_{\brho}=J$ is  $\sum_{0\leq i\leq f-d_{\brho},\ i\ \mathrm{even}}2\binom{f-d_{\brho}}{i}=2^{f-d_{\brho}}$.   Similarly, if $|J|$ is odd, then the cardinality of $\mu\in\D^{\ss}$ satisfying $\mathcal{A}(\mu)\cap J_{\brho}=J$ is  $\sum_{0\leq i\leq f-d_{\brho},\ i\ \mathrm{odd}}2\binom{f-d_{\brho}}{i}=2^{f-d_{\brho}}$.  This proves the second statement and the last one easily follows.
\end{proof}

In the rest of this section, we assume that $\brho$ is split reducible. %
For $\lambda\in \P$, recall the set $J_{\lambda}\subset\{0,\dots,f-1\}$  defined in \eqref{eq:J-lambda}.
For $i\in\{0,\dots,f\}$ put
\[N_{(i)}\defeq\bigoplus_{\lambda\in \P, |J_\lambda| = i}\chi_{\lambda}^{-1}\otimes R/\fa(\lambda).\] 
The following result computes the Hilbert series of $N_{(i)}$.%

\begin{prop1}\label{prop:Hilbert-split}
Assume $\brho$ is split reducible. Then for any $0\leq i\leq f$, 
\[h_{N_{(i)}}(t)=\frac{2\sum\limits_{0\leq s\leq i}\binom{f}{2s}\binom{f-2s}{i-s}(1+t)^{2s}}{(1-t)^f}. \] 
\end{prop1}
\begin{rem1}
  Together with \cite[Cor.~\ref{bhhms4:cor:finite-length}]{BHHMS4}(ii), %
  Proposition \ref{prop:Hilbert-split} gives the Hilbert series {$h_{\pi'}(t)$} %
  for any subquotient $\pi'$ of $\pi$  if $\brho$ is split reducible {and $\max\{9,2f+1\}$-generic}.
\end{rem1}

\begin{proof}
Since $\brho$ is split reducible,  $|\mathcal{A}(\lambda)|=2|\{j:\lambda_j(x_j)=x_j+1\}|$  by the proof of Lemma \ref{lem:Hilbert-count}(ii). Since $\{j:\lambda_j(x_j)=x_j+1\}\subset J_{\lambda}$, we deduce  $|\mathcal{A}(\lambda)|/2\leq |J_{\lambda}|$. Fix $0\leq i\leq f$. As in the proof of Theorem \ref{thm:Hilbert}, we have 
\[h_{N_{(i)}}(t)=\frac{\sum_{0\leq s\leq i}|\P_{i,s}|(1+t)^{2s}}{(1-t)^f}\]
where 
$\P_{i,s}\defeq\{\lambda\in \P: |J_{\lambda}|=i, |\mathcal{A}(\lambda)|=2s\}.$ Thus, it suffices to show \[|\P_{i,s}|=2\sum_{0\leq s\leq i}\binom{f}{2s}\binom{f-2s}{i-s}.\]

Let $\lambda\in \P_{i,s}$ (with $0\leq s\leq i\leq f$) and write  $\mathcal{A}(\lambda)=\{0 \le j_1 < j_1' < \dots < j_{s} < j_s' < f\}$. Assume first $\lambda_{j_1}(x_{j_1})=x_{j_1}+1$; we call it case $+$. Then one checks that $\lambda$ is uniquely determined by $(\mathcal{A}(\lambda),J_{\lambda}\setminus \mathcal{A}(\lambda))$ as follows:
\begin{itemize}
\item%
$\lambda_{j_k}=x_{j_k}+1$ and $\lambda_{j_k'}(x_{j_k'})=p-2-x_{j_k'}$ for $1\leq k\leq s$  by the proof of Lemma \ref{lem:Hilbert-count}(ii);
\item%
if $j_k<j<j_k'$ for some $k$, then $\lambda_j(x_j)\in \{x_j,x_j+2\}$, and $\lambda_j(x_j)=x_j+2$ if and only if $j\in J_{\lambda}\setminus \mathcal{A}(\lambda)$;
\item%
if $j_k'<j<j_{k+1}$ for some $k$ (in $\Z/f\Z$), then $\lambda_j(x_j)\in \{p-1-x_j,p-3-x_j\}$, and $\lambda_j(x_j)=p-3-x_j$ if and only if $j\in J_{\lambda}\setminus \mathcal{A}(\lambda)$.
\end{itemize}
Conversely, an element $\lambda\in\P$ satisfying the above conditions belongs to $\P_{i,s}$ with $i=|\{j:\lambda_j(x_j)=\{x_j+1,x_j+2,p-3-x_j\}\}|$. Similar statements hold if $\lambda_{j_1}(x_{j_1})=p-2-x_{j_1}$; we call it case $-$.

The above discussion implies that sending $\lambda$ to $\big(\mathcal{A}(\lambda), \mathrm{case}\ \pm, J_{\lambda}\setminus  \mathcal{A}(\lambda)
\big)$ gives a bijection between $\P_{i,s}$ and the set of triples $(J,\pm,J')$ satisfying
\[J\subset \{0,\dots,f-1\},\quad |J|=2s,\quad J'\subset J^c,\quad |J'|=i-s.\]
Thus $|\P_{i,s}|=2 \sum_{0\leq s\leq i}\binom{f}{2s}\binom{f-2s}{i-s}$ as desired.
\end{proof}

\section{On the structure of subquotients of \texorpdfstring{$\pi$}{pi} in the semisimple case}
\label{sec:subquot-ss}

We determine the $\m_{K_1}^2$-torsion of any subquotient of $\pi$, where $\pi$ is any smooth mod $p$ representation of $\GL_2(K)$ satisfying assumptions \ref{it:assum-i}--\ref{it:assum-iv} of \S~\ref{sec:abstract-setting} and the underlying Galois representation $\rhobar:\Gal(\o K/K)\ra\GL_2(\F)$ is semisimple and {sufficiently} generic.
By \cite[Cor.\ 3.90]{BHHMS2} and Proposition~\ref{prop:ii'} we may and will assume that $\brho$ is split reducible.

\begin{prop1}\label{prop:K1-square-invariants}
Assume that $\brho$ is split reducible and $\max\{9,2f+1\}$-generic.
\begin{enumerate}
\item 
Let $\pi'$ be a subquotient of $\pi$. Then there exists a (unique) subset $\Sigma'\subset\{0,\dots,f\}$ such that \[\pi'[\m_{K_1}^2]\cong\bigoplus_{i\in\Sigma'}\wt{D}_0(\brho)_i,\]
where $\wt{D}_{0}(\brho)_i\defeq \bigoplus_{\sigma\in W(\brho),~|J_{\sigma}|=i}\wt{D}_{0,\sigma}(\brho)$ for $0\leq i\leq f$.
\item Let $\pi_1\subset\pi_2$ be subrepresentations of $\pi$.  Then the induced sequence of $\wt{\Gamma}$-modules
\[0\ra \pi_1[\m_{K_1}^2]\ra \pi_2[\m_{K_1}^2]\ra (\pi_2/\pi_1)[\m_{K_1}^2]\ra0\]
is split exact. 
\end{enumerate}
\end{prop1}

\begin{rem1}\label{rem:K1-invt-exact}
As a consequence of Proposition \ref{prop:K1-square-invariants}(ii), if $\pi_1\subset\pi_2$ are subrepresentations of $\pi$, then the induced sequence of $\Gamma$-representations 
\[0\ra \pi_1^{K_1}\ra \pi_2^{K_1}\ra(\pi_2/\pi_1)^{K_1}\ra0\]
is split exact. This strengthens \cite[Lemma~\ref{bhhms4:lem:soc-exact}]{BHHMS4}. 
\end{rem1}

\begin{proof}
We first prove (i) for any subrepresentation $\pi' = \pi_1$.
Let $\Sigma' = \Sigma_1$ be the unique subset such that $\soc_{\GL_2(\cO_K)}(\pi_1)=\bigoplus_{\ell(\sigma)\in\Sigma_1}\sigma$.
 (See \cite[Prop.\ 3.2.2]{BHHMS4} for the existence of $\Sigma_1$.) 
First, since $\pi_1[\fm_{K_1}^2]\subseteq \pi[\fm_{K_1}^2]$, we deduce from Proposition~\ref{prop:ii'} %
that 
\[%
\pi_1[\m_{K_1}^2]\subseteq \bigoplus_{i\in\Sigma_1}\wt{D}_0(\brho)_i. \]
Denote by $Q$ the quotient $\big(\bigoplus_{i\in\Sigma_1}\wt{D}_0(\brho)_i\big)/\pi_1[\m_{K_1}^2]$;  we want to prove $Q=0$. By \cite[Thm.\ 4.6]{HuWang2}, $\wt{D}_0(\brho)$ is multiplicity free, so $\JH(Q)\cap W(\brho)=\emptyset$. Consider the natural morphisms
\[Q\hookrightarrow \pi[\fm_{K_1}^2]/\pi_1[\fm_{K_1}^2]\hookrightarrow (\pi/\pi_1)[\fm_{K_1}^2]\hookrightarrow \pi/\pi_1\]
which induce an embedding $\soc_{\GL_2(\cO_K)}(Q)\hookrightarrow \soc_{\GL_2(\cO_K)}(\pi/\pi_1)$. But $\JH(\soc_{\GL_2(\cO_K)}(\pi/\pi_1))\subset W(\brho)$ by \cite[Lemma~\ref{bhhms4:lem:soc-exact}]{BHHMS4}, so we must have $\soc_{\GL_2(\cO_K)}(Q)=0$, equivalently $Q=0$.

(ii) As in the proof of \cite[Cor.~\ref{bhhms4:cor:split-In}]{BHHMS4}, 
it suffices to treat the special case $\pi_2=\pi$.
We again define $\Sigma_1$ by the equality $\soc_{\GL_2(\cO_K)}(\pi_1)=\bigoplus_{\ell(\sigma)\in\Sigma_1}\sigma$, so $\pi_1[\fm_{K_1}^2]=\bigoplus_{i\in \Sigma_1}\widetilde{D}_0(\brho)_i$ by the preceding paragraph.
Thus there is an inclusion $\bigoplus_{i\notin \Sigma_1}\wt{D}_0(\brho)_i\cong \pi[\m_{K_1}^2]/\pi_1[\m_{K_1}^2] \subseteq (\pi/\pi_1)[\m_{K_1}^2]$. Suppose that this is not an equality. Then $(\pi/\pi_1)[\m_{K_1}^2]$ contains a subrepresentation $V$ which fits into a \emph{nonsplit} $\wt{\Gamma}$-extension
\begin{equation}\label{eq:exact-seq-V}
0\ra \bigoplus_{i\notin \Sigma_1}\wt{D}_0(\brho)_i\ra V\ra \tau\ra0
\end{equation}
for some Serre weight $\tau$. 
(The extension is nonsplit by \cite[Lemma~\ref{bhhms4:lem:soc-exact}]{BHHMS4}.)
We have $\tau\in W(\brho)$ by  Lemma \ref{lem:ex-wtD}(i) and we let again $\chi \defeq  \tau^{I_1}$. 

By the projectivity of $\Proj_{\wt{\Gamma}}\tau$, there exists  a  $\wt{\Gamma}$-equivariant morphism $\beta:\Proj_{\wt{\Gamma}}\tau\ra V$ whose composition with $V\twoheadrightarrow \tau$ is the natural projection $\Proj_{\wt{\Gamma}}\tau\twoheadrightarrow \tau$.   Let $V_\beta$ denote the image of $\beta$, which has cosocle $\tau$. By \eqref{eq:exact-seq-V}, $V_{\beta}$ satisfies the conditions in Lemma \ref{lem:4.18}, so it is a quotient of $\im(\phi_{\tau})$, namely $\beta$ factors through $\im(\phi_{\tau})\ra V$.

By %
Corollary~\ref{cor:coker-ext1} we have $\Ext^1_{\wt{\Gamma}}(\coker(\phi_{\tau}),V)=0$ by d\'evissage using \eqref{eq:exact-seq-V}.   Hence, using the short exact sequence $0\ra \im(\phi_{\tau})\ra \Ind_I^{\GL_2(\cO_K)}\overline{W}_{\chi,3}\ra \coker(\phi_{\tau})\ra0$, we can lift the map $\im(\phi_{\tau})\rightarrow V$ of the previous paragraph to 
\[\beta':\Ind_I^{\GL_2(\cO_K)}\overline{W}_{\chi,3}\ra V\ (\hookrightarrow\pi/\pi_1).\]
The splitting statement in \cite[Cor.~\ref{bhhms4:cor:split-In}]{BHHMS4} with $n = 3$ implies that the natural sequence
\[0\ra \Hom_I(\overline{W}_{\chi,3},\pi_1)\ra \Hom_I(\overline{W}_{\chi,3},\pi)\ra \Hom_I(\overline{W}_{\chi,3},\pi/\pi_1)\ra0\]
is  exact, so combined with Frobenius reciprocity we obtain a  morphism 
\[\beta'':\Ind_I^{\GL_2(\cO_K)}\overline{W}_{\chi,3}\ra \pi\] whose composition with $\pi \onto \pi/\pi_1$ gives $\beta'$. By \cite[Prop.\ 6.4.6]{BHHMS1}, any $I$-equivariant morphism  $\overline{W}_{\chi,3}\ra \pi$  factors through $\overline{W}_{\chi,3}\twoheadrightarrow \chi$, hence $\beta''$ factors as $\Ind_{I}^{\GL_2(\cO_K)}\overline{W}_{\chi,3}\twoheadrightarrow \Ind_I^{\GL_2(\cO_K)}\chi\ra \pi$. In particular, the image of $\beta''$ is contained in $\pi^{K_1}$ and has cosocle $\tau$. Since $\tau$ occurs in  $\soc_{\GL_2(\cO_K)}(\pi)$ and not elsewhere in $\pi^{K_1}$ (as $\pi^{K_1}$ is multiplicity free), the image of $\beta''$ (hence also the image of $\beta$) is just $\tau$.   This gives a contradiction, proving (ii), as $V$ is a nonsplit extension by assumption and $V_\beta$ has cosocle $\tau$.  

Finally, (i) is a direct consequence of the first paragraph of the proof and of (ii). 
\end{proof}

We can now prove Theorem~\ref{thm:intro-yitong-subquot} in the semisimple case.

\begin{cor1}\label{cor:yitong-subquot-ss}
  Assume that $\brho$ is semisimple and $\max\{9,2f+1\}$-generic.
  Then for any subquotient $\pi'$ of $\pi$ we have
  \begin{equation*}
    \dim_{\F\ppar{X}} D_{\xi}^{\vee}(\pi')=|\JH(\pi'^{K_1}) \cap W(\rhobar)|.
  \end{equation*}
\end{cor1}

\begin{proof}
  We have $\dim_{\F\ppar{X}} D_{\xi}^{\vee}(\pi')=|\JH(\soc_{\GL_2(\cO_K)} \pi')|$ by \cite[Cor.\ \ref{bhhms4:cor:finite-length}(i)]{BHHMS4} (if $\brho$ is split reducible) and \cite[Prop.\ 3.87(ii)]{BHHMS2} (if $\brho$ is irreducible, noting that $\pi' = \pi$ by \cite[Cor.\ 3.90]{BHHMS2} in that case).
  It suffices to show that $\JH(\pi'^{K_1}) \cap W(\brho) = \JH(\soc_{\GL_2(\cO_K)} \pi')$.
  If $\brho$ is irreducible this is clear, as $\pi'^{K_1} = \pi^{K_1} = D_0(\brho)$ by assumption~\ref{it:assum-i}.
  If $\brho$ is split reducible, then $\pi'^{K_1} = \bigoplus_{i\in\Sigma'}{D}_0(\brho)_i$ by Proposition~\ref{prop:K1-square-invariants}, keeping the notation there, and the result follows.
\end{proof}

\section{On the structure of subquotients of \texorpdfstring{$\pi$}{pi} in the non-semisimple case}
\label{sec:subquot-non-ss}

We prove many results on the structure of subquotients of $\pi$ as $I$- and $\GL_2(\cO_K)$-representations.

From now on $\pi$ denotes an admissible smooth representation of $\GL_2(K)$ over $\F$ satisfying assumptions \refeq{it:assum-i}--\refeq{it:assum-v} of \S~\ref{sec:abstract-setting}, with underlying Galois representation $\rhobar$ which is nonsplit reducible and {0}-generic.

The main results of this section include the description of the $I_1$- and $K_1$-invariants as well as of the $\GL_2(\cO_K)$-socle of any subquotient of $\pi$.
These results all depend on determining the $I_1$-socle filtration of any subquotient $\pi'$ of $\pi$ (equivalently, the associated graded module of $\pi'^\vee$ for the $\m$-adic filtration), which is the subject of subsection~\ref{sec:grad-module-subq}.

We again suppose that $\pi_1 \subset \pi$ is a subrepresentation of $\pi$ and let $\pi_2\defeq \pi/\pi_1$. 
Let $i_0 \defeq  i_0(\pi_1) \in \{-1,\dots,f\}$, cf.\ \cite[Thm.~\ref{bhhms4:thm:conj2}]{BHHMS4}.
To simplify notation, for $\lambda \in \P$ we let $d_\lambda \defeq  \max\{i_0+1-|J_\lambda|,0\}$.

\subsection{The graded module of subquotient representations of \texorpdfstring{$\pi$}{pi}}
\label{sec:grad-module-subq}

We describe $\gr_\m(\pi'^\vee)$, where $\pi'$ is any subquotient of $\pi$ (Corollary \ref{cor:gr-pi'}).
We start with quotients $\pi_2 = \pi/\pi_1$ of $\pi$:

\begin{thm}\label{thm:gr-pi2}
  {Assume that $\brho$ is $\max\{9,2f+3\}$-generic.}
  We have an isomorphism of graded $\gr(\Lambda)$-modules with compatible $H$-actions,
  \begin{equation}\label{eq:gr-pi2}
    \gr_\m(\pi_2^\vee) \cong \bigoplus_{\lambda \in \P} \chi_\lambda^{-1} \otimes \frac{\mathfrak{a}_1^{i_0}(\lambda)}{\mathfrak{a}(\lambda)}(-d_\lambda).
  \end{equation}
\end{thm}

 We recall the ideal $\fa_1^{i_0}(\lambda)$ of $\o R$ from~\cite[\eqref{bhhms4:eq:fa-1}]{BHHMS4}.
  Let $J_1 \defeq  \{ j \in J_{\rhobar}^c : \lambda_j(x_j) = p-1-x_j \}$ and $J_2 \defeq  \{ j \in J_{\rhobar}^c : \lambda_j(x_j) = x_j \}$.
  If $i_0 \ge |J_\lambda|$, then $\fa_1^{i_0}(\lambda)$ is the ideal generated by $\fa(\lambda)$ (cf.\ \eqref{eq:id:al}) and all $\prod_{j \in J_1'} y_j \prod_{j \in J_2'} z_j$, where $J_1' \subset J_1$, $J_2' \subset J_2$, $|J'_1|+|J'_2| = i_0+1-|J_\lambda|$.
  Otherwise, $\fa_1^{i_0}(\lambda) = \o R$.
  In particular, $\mathfrak a_1^{i_0}(\lambda) = \mathfrak a(\lambda)$ if $|J_1|+|J_2| < i_0+1-|J_\lambda|$.

The grading shifts in~\eqref{eq:gr-pi2} are such that all nonzero direct summands contribute in degree 0, but vanish in degree 1.
Note also from the definitions that $\mathfrak{a}_1^{i_0}(\lambda)/\mathfrak{a}(\lambda) = 0$ if $|\{ j \in J_{\rhobar}^c : \lambda_j(x_j) \in \{ p-1-x_j, x_j \} \}| < d_\lambda$.
(The converse is also true, by comparing equations~\cite[\eqref{bhhms4:eq:mult-a1} and \eqref{bhhms4:eq:mult-a}]{BHHMS4}, or alternatively see the proof of \cite[Cor.~\ref{bhhms4:cor:subquot-F}]{BHHMS4}.)

\begin{rem}\label{rem:gr-pi2-CM}
  Theorem~\ref{thm:gr-pi2} implies that $\gr_\m(\pi_2^\vee)$ is Cohen--Macaulay or zero as $\gr(\Lambda)$-module.
  (By \cite[Prop.~\ref{bhhms4:prop:nonsplit-I1}]{BHHMS4} and \cite[Cor.~\ref{bhhms4:cor:gr-pi1}]{BHHMS4}, each nonzero $\mathfrak{a}_1^{i_0}(\lambda)/\mathfrak{a}(\lambda)$ is Cohen--Macaulay, as the Cohen--Macaulay property is closed under direct summands, shifts in grading, and direct sums.)
\end{rem}

\begin{rem}\label{rem:killed-by-J}
  {Theorem~\ref{thm:gr-pi2} shows that $\gr_\m(\pi_2^\vee)$ is killed by the ideal $J$, i.e.\ is an $\o R$-module.
  This is a priori not obvious.
  A similar comment applies to Corollary~\ref{cor:gr-pi'}.}
\end{rem}

\begin{rem}\label{rem:i0-f-1-or-f}
  When $i_0 = f$, Theorem~\ref{thm:gr-pi2} is trivially true, because $\pi_2 = 0$ and $\mathfrak{a}_1^{f}(\lambda) = \mathfrak{a}(\lambda)$ for all $\lambda \in \P$.
  (By \cite[Lemma~\ref{bhhms4:lem:J-sets-add-up}]{BHHMS4} we have $f+1-|J_\lambda| = |J_1|+|J_2|+|J_{\lambda^*}|+1 > |J_1|+|J_2|$, where $\lambda\mapsto \lambda^*$ is the involution of $\mathscr{P}$ defined in \cite[Def.~3.62]{BHHMS2}.)
  When $i_0 = f-1$, Theorem~\ref{thm:gr-pi2}  can be proved as follows, assuming that $\brho$ is only $\max\{9,2f+1\}$-generic.
  By \cite[Thm.~\ref{bhhms4:thm:fin-length-nonsplit}(ii)]{BHHMS4}, $\pi_2$ is irreducible in this case, so $\pi_2$ is the principal series $\Ind_{B(K)}^{\GL_2(K)} (\chi_1 \otimes \chi_2\omega^{-1})$ by \cite[Prop.\ 10.8]{HuWang2}, where $\o\rho \cong \smatr{\chi_1} * 0 {\chi_2}$ and hence $\chi_1|_{I_K} = \omega_f^{\sum_{j=0}^{f-1} (r_j+1) p^j}$ and $\chi_2|_{I_K} = 1$.
  (Here $B(K)$ denotes the Borel subgroup of upper-triangular matrices of $\GL_2(K)$.)
  We apply the combinatorial Proposition~\ref{prop:gr-m-semisimple-nonsplit} below (or argue directly) to deduce
  \begin{equation}\label{eq:i0-f-1}
    \bigoplus_{\lambda \in \P} \chi_\lambda^{-1} \otimes \frac{\mathfrak{a}_1^{i_0}(\lambda)}{\mathfrak{a}(\lambda)}(-d_\lambda)
    \cong \chi_{\lambda'}^{-1} \!\otimes \!\frac{\o R}{(z_j : 0 \le j \le f-1)} \ \oplus \ \chi_{\lambda''}^{-1}\! \otimes \!\frac{\o R}{(y_j : 0 \le j \le f-1)},
  \end{equation}
  where $\lambda',\lambda'' \in \P^\ss$ are given by $\lambda'_j(x_j) = p-3-x_j$, $\lambda''_j(x_j) = x_j+2$ for all $j$.
  We calculate $\chi_{\lambda'} = (\chi_2|_{I_K})\omega^{-1} \otimes \chi_1|_{I_K}$ and $\chi_{\lambda''} = \chi_1|_{I_K} \otimes (\chi_2|_{I_K})\omega^{-1}$.
  We conclude by \cite[Prop.\ 3.76(ii)]{BHHMS2}.
\end{rem}

\begin{lem}\label{lem:lift-graded-homo}
  Suppose that $M$ is a graded $\gr(\Lambda)$-module. %
  Let $N \defeq  (\mathfrak{a}_1^{i_0}(\lambda)/\mathfrak{a}(\lambda))(-d_\lambda)$ for some $\lambda \in \P$.
  Then the natural map 
  \begin{equation*}
    \HOM_{\gr(\Lambda)}(N,M)_0 \to \HOM_{\gr(\Lambda)}(N,M/\gr_{\le-3} M)_0
  \end{equation*}
  is an isomorphism.
\end{lem}

Recall that $\HOM(N,M)_0$ denotes the graded morphisms $N \to M$ (of degree 0).

\begin{proof}
  \textbf{Step 1.}
  Suppose for the moment that $S$ is a graded ring and $N$ any finitely presented graded $S$-module.
  For any subset $D \subset \Z$ we say that \emph{$N$ has relations in degrees $D$} as $S$-module if there exists a graded exact sequence of the form $\bigoplus_{i=1}^n S(-d_i) \to S^{\oplus m} \to N \to 0$ with $d_i \in D$ for all $i$ (in particular, $N$ is generated by its degree 0 part).

  We claim that if $N$ has relations in degrees $D$ as $S/I$-module, where $I$ is an ideal of $S$ that is generated by finitely many homogeneous elements $s_i$ whose degrees are contained in $D$, then the same is true as $S$-module.
  
  To see this, by assumption we can find a surjective graded homomorphism $(S/I)^{\oplus m} \to N\to 0$, whose kernel is generated by finitely many homogeneous elements $x_j$ whose degrees are contained in $D$ for all $j$.
  Lift each $x_j$ to a homogeneous element $\wt x_j$ of $S^{\oplus m}$ of the same degree.
  By composition we have a surjective morphism $S^{\oplus m} \to N \to 0$ of graded $S$-modules.
  Its kernel is generated by all $s_i e_k$ (where $e_k$ denotes the standard $\F$-basis of $S^{\oplus m}$) and all $\wt x_j$, as desired.

  \textbf{Step 2.} 
  We show that $N = (\mathfrak{a}_1^{i_0}(\lambda)/\mathfrak{a}(\lambda))(-d_\lambda)$ has relations in degrees $\{-1,-2\}$ as $\gr(\Lambda)$-module.
  Since $N$ is a graded $\o R/\mathfrak{a}(\lambda)$-module and $\o R/\mathfrak{a}(\lambda)$ is obtained from $\gr(\Lambda)$ by quotienting by an ideal generated by homogeneous elements of degrees $-1$ and $-2$, by Step 1 it suffices to show that $N$ has relations in degrees $\{-1,-2\}$ as $\o R/\mathfrak{a}(\lambda)$-module.

  Note that $j \in J_1 \sqcup J_2$ implies that $t_j = y_jz_j$ and let $d \defeq  d_\lambda$ for short.
  By interchanging $y_j$ and $z_j$ for some $j$ and permuting $\{0,1,\dots,f-1\}$, we may assume that 
  \[N = (y_{i_1}\cdots y_{i_d} : 0 \le i_1 < \cdots < i_d < e)(-d)\]
  as graded $\o R/\mathfrak{a}(\lambda)$-module, for some $0 \le e \le f$.
  This module is generated by the elements $X_I \defeq  \prod_{i \in I} y_i$ (of degree 0) for subsets $I \subset E \defeq  \{0,1,\dots,e-1\}$ with $|I| = d$.
  We claim that the relations are generated by
  \begin{equation}\label{eq:relations}
    \begin{aligned}
      z_i X_I &= 0 && \text{for all $i \in I$}; \\
      y_i X_{I'\setminus \{i\}} &= y_j X_{I'\setminus \{j\}} && \text{for all $i \ne j$ in $I' \subset E$, $|I'| = d+1$}.
    \end{aligned}
  \end{equation}
  Note that $\o R/\mathfrak{a}(\lambda)$ has as $\F$-basis all monomials of the form $\prod_j w_j^{\ge 0}$ with $w_j \in \{y_j,z_j\}\setminus \{t_j\}$.
  If $\sum_{I} f_I X_I = 0$ with $f_I \in \o R/\mathfrak{a}(\lambda)$, then using relations~\eqref{eq:relations}, without loss of generality, $f_I$ does not contain any $z_i$ ($i \in I$) and $y_i$ ($i \in E \setminus I$, $i < \min I$).
  The map $f_I \mapsto f_I X_I$ is injective for such $f_I$, and moreover for every monomial term in $f_I X_I$, $I$ is the set of $d$ largest elements $i$ of $E$ such that $y_i$ divides it.
  This shows that $f_I = 0$ for all $I$, proving that we have found all relations, and indeed the relations are in degree $-1$.

  \textbf{Step 3.}
  By Step 2 we have an exact sequence $\bigoplus_{i=1}^n \gr(\Lambda)(-d_i) \to \gr(\Lambda)^{\oplus m} \to N \to 0$ with $d_i \in \{-1,-2\}$ for all $i$.
  We get a commutative diagram
  \begin{equation*}
    \xymatrix@R-1pc{
      0 \ar[r] & \HOM_{\gr(\Lambda)}(N,M)_0 \ar[r]\ar[d] & \HOM_{\gr(\Lambda)}(\gr(\Lambda),M)^{\oplus m}_0 \ar[r]\ar[d] & \bigoplus_{i=1}^n \HOM_{\gr(\Lambda)}(\gr(\Lambda)(-d_i),M)_0 \ar[d] \\
      0 \ar[r] & \HOM_{\gr(\Lambda)}(N,\o M)_0 \ar[r] & \HOM_{\gr(\Lambda)}(\gr(\Lambda),\o M)^{\oplus m}_0 \ar[r] & \bigoplus_{i=1}^n \HOM_{\gr(\Lambda)}(\gr(\Lambda)(-d_i),\o M)_0 
      }
    \end{equation*}
  where $\o M \defeq  M/\gr_{\le-3} M$.
  As $\HOM_{\gr(\Lambda)}(\gr(\Lambda)(-i),M)_0 = M_{i}$ for all $i \in \Z$, the middle and right vertical arrows are isomorphisms, hence so is the left one.
\end{proof}

Fix $n \ge 1$, which we will specify later{, and assume that $\brho$ is $(2n-1)$-generic}.
Recall $\tau \defeq  \tau^{(n)} \subset \pi$ from \cite[\S~\ref{bhhms4:sec:representation-tau}]{BHHMS4}, so $\tau = \bigoplus_{\lambda \in \P} \tau_\lambda$ with $\tau_\lambda\defeq \tau_\lambda^{(n)}$ and $\soc_I(\tau_\lambda) \cong \chi_\lambda$.
Let $\o\tau \defeq  \tau[\m^{n}]=\pi[\m^{n}]$ (last statement of \cite[Lemma~\ref{bhhms4:lem:isom-modcI})]{BHHMS4} and $\o\tau_\lambda \defeq  \tau_\lambda[\m^{n}]$ for $\lambda \in \P$, so $\o\tau = \bigoplus_{\lambda \in \P} \o\tau_\lambda$.
Let $\Theta$ denote the image of $\o\tau$ in $\pi_2$.
As $\tau$ is multiplicity free by \cite[Cor.~\ref{bhhms4:cor:tau-multfree}(i)]{BHHMS4} (for $r=1$),  %
we have $\o\tau \cap \pi_1 = \bigoplus_{\lambda \in \P} (\o\tau_\lambda \cap \pi_1)$ and $\Theta = \bigoplus_{\lambda \in \P} \Theta_\lambda$, where $\Theta_\lambda$ is the image of $\o\tau_\lambda$ in $\pi_2$.
For the same reason,
\begin{equation}\label{eq:filtrations}
    F_{-i} \Theta^\vee = \m^i \o\tau^\vee \cap \Theta^\vee = \Big(\bigoplus_{\lambda \in \P} \m^i \o\tau_\lambda^\vee\Big) \cap \Theta^\vee = \bigoplus_{\lambda \in \P} (\m^i \o\tau_\lambda^\vee \cap \Theta^\vee) = \bigoplus_{\lambda \in \P} F_{\lambda,-i} \Theta_\lambda^\vee
  \end{equation}
  for all $i \in \Z_{\ge 0}$, where $F$ (resp.\ $F_\lambda$) denotes the filtration on $\Theta^\vee$ (resp.\ $\Theta_\lambda^\vee$) induced from the $\m$-adic filtration on $\o\tau^\vee$ (resp.\ $\o\tau_\lambda^\vee$).
  In particular, $\gr_F(\Theta^\vee) = \bigoplus_{\lambda \in \P} \gr_{F_\lambda}(\Theta_\lambda^\vee)$ and $F_{-n} \Theta^\vee = 0$.

Suppose that $n \ge 1$ and that $\brho$ is {$(2n-1)$}-generic.
The following lemma determines the submodule structure of $\tau_\lambda^{(n)}[\m^n]$ (and hence of $\pi[\m^n] = \tau^{(n)}[\m^n]$ if $r=1$ by \cite[Cor.~\ref{bhhms4:cor:tau-multfree}(ii)]{BHHMS4}), since $\bigoplus_{\lambda\in\mathscr{P}}\tau_{\lambda}^{(n)}$ is multiplicity free by \cite[Cor.~\ref{bhhms4:cor:tau-multfree}(i)]{BHHMS4} (as $\brho$ is {$(2n-1)$}-generic).

\begin{lem}\label{lem:submodule-structure}  
  Suppose that $\brho$ is  {$(n-1)$-generic}  %
  and keep the above notation.
  Suppose that $\lambda \in \P$.
  For any $\chi,\chi' \in \JH(\tau^{(n)}_\lambda[\m^n])$ with $\Ext^1_{I/Z_1}(\chi,\chi') \ne 0$, upon perhaps interchanging $\chi$ and $\chi'$, there exist $\ell_j \in \Z$ for $0 \le j \le f-1$, $\ve \in \{\pm1\}$, and $0 \le j_0 \le f-1$ such that 
  \begin{enumerate}
  \item $\chi = \chi_\lambda \prod_j \alpha_j^{\ell_j}$ and $\chi' = \chi \alpha_{j_0}^{\ve}$;
  \item $\ell_j \ge 0$ if $t_j = y_j$, $\ell_j \le 0$ if $t_j = z_j$, and $\sum_j |\ell_j| < n$;
  \item $|\ell_{j_0}+\ve| < |\ell_{j_0}|$;
  \item $E_{\chi',\chi}$ (the unique nonsplit extension of $\chi$ by $\chi'$, see \S~\ref{sec:preliminaries}) is a subquotient of $\tau^{(n)}_\lambda[\m^n]$.
  \end{enumerate}
  In particular, either $E_{\chi,\chi'}$ or $E_{\chi',\chi}$ occurs as subquotient of $\tau^{(n)}_\lambda[\m^n]$.
\end{lem}

\begin{proof}
  By construction of $\tau^{(n)}_\lambda$ (cf.\ the proofs of \cite[Lemma~\ref{bhhms4:lem:tau-embed}]{BHHMS4} and \cite[Prop.\ 9.19]{HuWang2}), as $\chi\in \tau^{(n)}_\lambda[\m^n]$ we can write $\chi = \chi_\lambda \prod_j \alpha_j^{\ell_j}$ for some $\ell_j \in \Z$ satisfying condition (ii). As $\Ext^1_{I/Z_1}(\chi,\chi') \ne 0$ we have $\chi' = \chi \alpha_{j_0}^{\ve}$ for some $\ve \in \{\pm1\}$ and some $0 \le j_0 \le f-1$. By the genericity condition we deduce that condition (ii) holds for $(\ell_0,\dots,\ell_{j_0}+\ve,\dots,\ell_{f-1})$.
 (As $\brho$ is {$(n-1)$}-generic, $|\ell_j| < n \le \frac{p-1}2$, and $\ve p^{j_0}+\sum_{j=0}^{f-1} \ell_jp^j \equiv \sum_{j=0}^{f-1} \ell'_jp^j \pmod{p^f-1}$ for integers $|\ell'_j|  <  \frac{p-1}2$ implies $(\ell_0,\dots,\ell_{j_0}+\ve,\dots,\ell_{f-1})=(\ell'_0,\dots,\ell'_{f-1})$.) 
  Interchanging $\chi$ and $\chi'$, if necessary, we may assume that (iii) holds.
  Then (iv) holds, as the nonsplit extension of $\alpha_{j_0}^{\ell_{j_0}}$ by $\alpha_{j_0}^{\ell_{j_0}+\ve}$ occurs in the $j_0$-th tensor factor defining $\tau^{(n)}_\lambda$.
\end{proof}

\begin{proof}[Proof of Theorem~\ref{thm:gr-pi2}]
  By Remark \ref{rem:i0-f-1-or-f} we can assume throughout the proof that $i_0\leq f-2$.
  Let $N_2'$ denote the right-hand side of the theorem, i.e.\ $N_2' \defeq  \bigoplus_{\lambda \in \P} N_{2,\lambda}'$ with
  \[N_{2,\lambda}' \defeq  \chi_\lambda^{-1} \otimes \frac{\mathfrak{a}_1^{i_0}(\lambda)}{\mathfrak{a}(\lambda)}(-d_\lambda).\]

  \textbf{Step 1.} We show that $\gr_\m(\pi_2^\vee)/ \o\m^3 \onto N_2'/\o\m^3$ as graded $\gr(\Lambda)$-modules with compatible $H$-actions.

  Consider the commutative diagram
  \begin{equation*}
    \xymatrix{
      0 \ar[r] & \bigoplus_{\lambda \in \P} \chi_\lambda^{-1} \otimes \frac{\mathfrak{a}_1^{i_0}(\lambda)}{\mathfrak{a}(\lambda)} \ar[r]\ar^{\cong}[d] & \bigoplus_{\lambda \in \P} \chi_\lambda^{-1} \otimes \frac{\o R}{\mathfrak{a}(\lambda)} \ar[r]\ar^{\cong}[d] & \bigoplus_{\lambda \in \P} \chi_\lambda^{-1} \otimes \frac{\o R}{\mathfrak{a}_1^{i_0}(\lambda)} \ar[r]\ar^{\cong}[d] & 0 \\
      0 \ar[r] & \gr_{F'}(\pi_2^\vee) \ar[r]\ar[d] & \gr(\pi^\vee) \ar[r]\ar@{->>}[d] & \gr(\pi_1^\vee) \ar[r]\ar@{->>}[d] & 0 \\
      0 \ar[r] & \gr_F(\Theta^\vee) \ar[r]\ar@{=}[d] & \gr(\o\tau^\vee) \ar[r]\ar@{=}[d] & \gr((\o\tau\cap \pi_1)^\vee) \ar[r]\ar@{=}[d] & 0 \\
      0 \ar[r] & \bigoplus_{\lambda \in \P} \gr_{F_\lambda}(\Theta_\lambda^\vee) \ar[r] & \bigoplus_{\lambda \in \P} \gr(\o\tau_\lambda^\vee) \ar[r] & \bigoplus_{\lambda \in \P} \gr((\o\tau_\lambda \cap \pi_1)^\vee) \ar[r] & 0
    }
  \end{equation*}
  of graded $\gr(\Lambda)$-modules with compatible $H$-actions, where $F'$ denotes the filtration on $\pi_2^\vee$ induced by the $\m$-adic filtration on $\pi^\vee$ and we recall that $F$ denotes the filtration on $\Theta^\vee$ induced by the $\m$-adic filtration on $\o\tau^\vee$.
  The top vertical maps are isomorphisms by \cite[Cor.~\ref{bhhms4:cor:gr-pi1}]{BHHMS4} (see also the proof of \cite[Prop.~\ref{bhhms4:prop:nonsplit-I1}]{BHHMS4}).
  From $\o\tau = \pi[\m^{n}]$ we get $(\o\tau\cap \pi_1)[\m^{n}] = \pi_1[\m^{n}]$.
  Hence the middle and right vertical maps are isomorphisms in degrees $> -n$, %
  and so the same is true of the left vertical map.
  The middle vertical composition $\bigoplus_{\lambda \in \P} \chi_\lambda^{-1} \otimes (\mathfrak{a}_1^{i_0}(\lambda)/\mathfrak{a}(\lambda)) \to \bigoplus_{\lambda \in \P} \gr(\o\tau_\lambda^\vee)$ is an isomorphism in degree 0, respecting the direct sum decomposition (by $H$-equivariance).
  As its domain is generated by its degree 0 part as $\gr(\Lambda)$-module, it follows that the middle vertical composition respects the direct sum decomposition, and
  hence the same is true for the left and right vertical maps.
  We deduce that for each $\lambda \in \P$ the morphism 
  \begin{equation}\label{eq:N_2-lambda'}
    N_{2,\lambda}'(d_\lambda) = \chi_\lambda^{-1} \otimes \frac{\mathfrak{a}_1^{i_0}(\lambda)}{\mathfrak{a}(\lambda)} \to \gr_{F_\lambda}(\Theta_\lambda^\vee)
  \end{equation}
  of graded $\gr(\Lambda)$-modules is an isomorphism in degrees $> -n$.

  We now show that
  \begin{equation}\label{eq:compare-fil}
    \text{$F_{\lambda,-d_\lambda-i}(\Theta_\lambda^\vee) = \m^i \Theta_\lambda^\vee$ \ \ for any $\lambda \in \P$, $i \ge 0$.}
  \end{equation}
  To see this, note that if $d_\lambda = 0$, then $\mathfrak{a}_1^{i_0}(\lambda) = \o R$, hence the $\lambda$-part of the above commutative diagram shows that the natural map $\gr_{F_\lambda}(\Theta_\lambda^\vee) \into \gr(\o\tau_\lambda^\vee)$ is an isomorphism, which implies that the natural map $\Theta_\lambda^\vee \into \o\tau_\lambda^\vee$ of filtered $\Lambda$-modules is an isomorphism, as desired.
  Suppose now that $d_\lambda > 0$.
  We first obtain from the previous diagram the following diagram: %
  \begin{equation*}
    \xymatrix{
      0 \ar[r] & \chi_\lambda^{-1} \otimes \frac{\mathfrak{a}_1^{i_0}(\lambda) + \m^n}{\mathfrak{a}(\lambda)+\m^n} \ar[r]\ar^{\cong}[d] & \chi_\lambda^{-1} \otimes \frac{\o R}{\mathfrak{a}(\lambda)+\m^n} \ar[r]\ar^{\cong}[d] & \chi_\lambda^{-1} \otimes \frac{\o R}{\mathfrak{a}_1^{i_0}(\lambda)+\m^n} \ar[r]\ar^{\cong}[d] & 0 \\
      0 \ar[r] & \gr_{F_\lambda}(\Theta_\lambda^\vee) \ar[r] & \gr(\o\tau_\lambda^\vee) \ar[r] & \gr((\o\tau_\lambda \cap \pi_1)^\vee) \ar[r] & 0
    }
  \end{equation*}
  By the definition of $\mathfrak{a}(\lambda)$ the middle isomorphism shows that
  \begin{equation*}
    \JH(\o\tau_\lambda^\vee) = \bigg\{ \chi_\lambda^{-1} \prod_j \alpha_j^{\ell_j} : \ell \in L \bigg\},
  \end{equation*}
  where
  \begin{equation*}
    L \defeq  \bigg\{ \ell = (\ell_j)_{j=0}^{f-1} : \text{$\ell_j \le 0$ if $t_j = y_j$, $\ell_j \ge 0$ if $t_j = z_j$},\ \sum_j |\ell_j| < n \bigg\},
  \end{equation*}
  as well as 
  \begin{equation*}
    \JH(\m^i \o\tau_\lambda^\vee) = \bigg\{ \chi_\lambda^{-1} \prod_j \alpha_j^{\ell_j} : \ell \in L,\ i \le \sum_j |\ell_j| \bigg\}.
  \end{equation*}
  By the definition of $\mathfrak{a}_1^{i_0}(\lambda)$ the left isomorphism shows that
  \begin{equation*}
    \JH(\Theta_\lambda^\vee) = \bigg\{ \chi_\lambda^{-1} \prod_j \alpha_j^{\ell_j} : \ell \in L,\; |\{j \in J_1 : \ell_j > 0\}|+|\{j \in J_2 : \ell_j < 0\}| \ge d_\lambda \bigg\},
  \end{equation*}
  where we recall that $J_1 = \{ j \in J_{\rhobar}^c : \lambda_j(x_j) = p-1-x_j \}$, $J_2 = \{ j \in J_{\rhobar}^c : \lambda_j(x_j) = x_j \}$.
  In particular, $\Theta_\lambda^\vee \subset \m^{d_\lambda} \o\tau_\lambda^\vee$ and hence $\m^i \Theta_\lambda^\vee \subset \Theta_\lambda^\vee \cap \m^{d_\lambda+i} \o\tau_\lambda^\vee$ for all $i \ge 0$.
  Conversely, to show $\Theta_\lambda^\vee \cap \m^{d_\lambda+i} \o\tau_\lambda^\vee \subset \m^i \Theta_\lambda^\vee$, by multiplicity freeness it suffices to show that $\JH(\Theta_\lambda^\vee) \cap \JH(\m^{d_\lambda+i} \o\tau_\lambda^\vee) \subset \JH(\m^i \Theta_\lambda^\vee)$.
  Take $\chi \defeq  \chi_\lambda^{-1} \prod_j \alpha_j^{\ell_j}$ with $\ell \in L$, $|\{j \in J_1 : \ell_j > 0\}|+|\{j \in J_2 : \ell_j < 0\}| \ge d_\lambda$, and $d_\lambda+i \le \sum_j |\ell_j|$.
  With the help of Lemma \ref{lem:submodule-structure} it is easy to show that there exist characters $\chi_{i'} \in \JH(\Theta_\lambda^\vee)$ ($0 \le i' \le i$) with $\chi_0 \defeq  \chi$ and such that the unique nonsplit extension $E_{\chi_{i'-1},\chi_{i'}}$ (of $\chi_{i'}$ by $\chi_{i'-1}$) occurs as subquotient of $\Theta_\lambda^\vee$ for all $0 < i' \le i$. 
  (If $i > 0$ then we find $\chi_1$ as follows: if there is $j \notin J_1 \sqcup J_2$ such that $\ell_j \ne 0$, choose such a $j$;
  otherwise, choose $j$ such that $|\ell_j| > 0$ and is as small as possible.
  Then $\chi_1 \defeq  \chi_0 \alpha_j^{-\sgn(\ell_j)}$ is still an element of $\JH(\Theta_\lambda^\vee)$, and we have decreased $\sum_j |\ell_j|$ by 1.
  Proceed inductively to find all $\chi_{i'}$.)
  We deduce that $\chi$ occurs in $\rad^i \Theta_\lambda^\vee = \m^i \Theta_\lambda^\vee$, proving~\eqref{eq:compare-fil}.

  We now let $n \defeq  i_0+4$.
  As $d_\lambda \le i_0+1 < n-2$ we obtain from~\eqref{eq:N_2-lambda'} and~\eqref{eq:compare-fil} an isomorphism of graded $\gr(\Lambda)$-modules
  \begin{equation*}
    N_{2,\lambda}'/\o\m^3 = N_{2,\lambda}'/\gr_{\le -3} N_{2,\lambda}' \cong \gr_{F_\lambda}(\Theta_\lambda^\vee(-d_\lambda))/\gr_{F_\lambda,\le -3}(\Theta_\lambda^\vee(-d_\lambda)) = \gr(\Theta_\lambda^\vee)/\o\m^3.
  \end{equation*}
  Hence
  \begin{equation}
\label{eq:grad2:surj}
    \gr(\pi_2^\vee)/\o\m^3 \onto \gr(\Theta^\vee)/\o\m^3 \cong \bigoplus_{\lambda \in \P} \gr(\Theta_\lambda^\vee)/\o\m^3 \cong N_2'/\o\m^3,
  \end{equation}
  as desired.

  \textbf{Step 2.} We show that $\gr_\m(\pi_2^\vee)/ \o\m^3 \cong N_2'/\o\m^3$ as graded $\gr(\Lambda)$-modules with compatible $H$-actions.

  From the cohomology long exact sequence we get
  \begin{equation}\label{eq:coho-seq}
    0 \to \coker\big(\Tor_1^\Lambda(\Lambda/\m^3,\pi^\vee) \to \Tor_1^\Lambda(\Lambda/\m^3,\pi_1^\vee)\big) \to \pi_2^\vee/\m^3 \to \pi^\vee/\m^3 \to \pi_1^\vee/\m^3 \to 0.
  \end{equation}
  We let
  \begin{equation*}
    C \defeq  \coker\big(\Tor_1^\Lambda(\Lambda/\m^3,\pi^\vee) \to \Tor_1^\Lambda(\Lambda/\m^3,\pi_1^\vee)\big)
  \end{equation*}
  and give it the induced filtration as quotient of $\Tor_1^\Lambda(\Lambda/\m^3,\pi_1^\vee)$.

  First we show that $\gr(C)$ is a subquotient of
  \begin{equation}\label{eq:coker-Tor1-gr}
    C' \defeq  \coker\big(\Tor_1^{\gr(\Lambda)}(\gr(\Lambda)/\o\m^3,\gr_{\m}(\pi^{\vee})) \to \Tor_1^{\gr(\Lambda)}(\gr(\Lambda)/\o\m^3,\gr_{\m}(\pi_1^{\vee}))\big).
  \end{equation}
  Notice that $\gr(C)$ is a quotient of
  \begin{equation*}
    \coker\big(\gr(\Tor_1^{\Lambda}(\Lambda/\m^3,\pi^{\vee})) \to \gr(\Tor_1^{\Lambda}(\Lambda/\m^3,\pi_1^{\vee}))\big),
  \end{equation*}
  because if we have a filtered exact sequence $X \to Y \to C \to 0$ with $C$ carrying the induced filtration, then $\coker(\gr(X) \to \gr(Y))$ surjects onto $\gr(C)$ by \cite[Thm.\ I.4.2.4(1)]{LiOy}.
  Then, as in the proof of \cite[Prop.~\ref{bhhms4:prop:Tor-inj}]{BHHMS4}, we consider the morphism of spectral sequences that converges to this morphism:
  \begin{equation*}
    \begin{gathered}
      \xymatrix{E_i^r\ar@{=>}[r]\ar[d]&\Tor_i^{\Lambda}(\Lambda/\m^3,\pi^{\vee})\ar[d]\\
        E'^r_i\ar@{=>}[r]&\Tor_i^{\Lambda}(\Lambda/\m^3,\pi_1^{\vee}).}
    \end{gathered}
  \end{equation*}
  (Referring to that proof, we have $E_i^0 = \gr(\Lambda/\m^3 \otimes_\Lambda M_i) \cong \gr(\Lambda/\m^3) \otimes_{\gr(\Lambda)} \gr(M_i)$ by \cite[Lemma I.6.14]{LiOy}, so $E_i^1 = \Tor_i^{\gr(\Lambda)}(\gr(\Lambda)/\o\m^3,\gr_{\m}(\pi^{\vee}))$.)
  Assumption~\ref{it:assum-v} says that $\dim_\F E_1^\infty = \dim_\F E_1^1$.
  It easily follows that $\coker(E_1^{r+1} \to E'^{r+1}_1)$ is a subquotient of $\coker(E_1^r \to E'^r_1)$ for any $r \ge 1$ (recall that $E_1'^{r+1}$ is a subquotient of $E_1'^r$, while $E^{r+1}_1=E^{r}_1$ by the preceding sentence).
  This implies the claim by taking $r$ sufficiently large.

  From the sequence~\eqref{eq:coho-seq} we see that
  \begin{equation}\label{eq:dim-C}
    \begin{aligned}
      \dim_\F(C) &= \dim_\F(\pi_2^\vee/\m^3)-\dim_\F(\pi^\vee/\m^3)+\dim_\F(\pi_1^\vee/\m^3) \\
      & = \dim_\F(\gr_{\m}(\pi_2^\vee)/\o\m^3)-\dim_\F(\gr(\pi^\vee)/\o\m^3)+\dim_\F(\gr(\pi_1^\vee)/\o\m^3).
    \end{aligned}
  \end{equation}
  By Step 1 we know that
  \begin{equation}\label{eq:dim-gr-pi2}
    \begin{aligned}
      \dim_\F(\gr_{\m}(\pi_2^\vee)/\o\m^3) &\ge \dim_\F(N_2'/\o\m^3) = \sum_{\lambda \in \P} \dim_\F(N_{2,\lambda}'/\o\m^3)\\
      &= \sum_{\lambda \in \P} \dim_\F(N_{2,\lambda}/\o\m^3) = \dim_\F(\gr_{F'}(\pi_2^\vee)/\o\m^3),
    \end{aligned}
  \end{equation}
  where $N_{2,\lambda} \defeq  \chi_\lambda^{-1} \otimes (\mathfrak{a}_1^{i_0}(\lambda)/\mathfrak{a}(\lambda)) = N_{2,\lambda}'(d_\lambda)$ and we used \cite[Cor.~\ref{bhhms4:cor:gr-pi1}]{BHHMS4} for the last equality.
  Combining equations~\eqref{eq:dim-C}, \eqref{eq:dim-gr-pi2} together with the fact that $\gr(C)$ is a subquotient of $C'$ (cf.\ \eqref{eq:coker-Tor1-gr}) we obtain
  \begin{equation}\label{eq:ineq}
    \dim_\F(C') \ge \dim_\F(C) \ge \dim_\F(\gr_{F'}(\pi_2^\vee)/\o\m^3)-\dim_\F(\gr(\pi^\vee)/\o\m^3)+\dim_\F(\gr(\pi_1^\vee)/\o\m^3).
  \end{equation}
  The exact sequence
  \begin{equation*}
    0 \to C' \to \gr_{F'}(\pi_2^\vee)/\o\m^3 \to \gr(\pi^\vee)/\o\m^3 \to \gr(\pi_1^\vee)/\o\m^3 \to 0
  \end{equation*}
  shows that equality holds in~\eqref{eq:ineq}, and hence in~\eqref{eq:dim-gr-pi2}.
  As equality holds in~\eqref{eq:dim-gr-pi2}, the surjection $\gr_\m(\pi_2^\vee)/ \o\m^3 \onto N_2'/\o\m^3$ in \eqref{eq:grad2:surj} has to be an isomorphism.

  \textbf{Step 3.} By Lemma~\ref{lem:lift-graded-homo} we get a graded morphism (of degree $0$) $N_2' \to \gr_\m(\pi_2^\vee)$, which has to be surjective by the graded Nakayama lemma.
  Recall that $N_2'$ is Cohen--Macaulay by Remark~\ref{rem:gr-pi2-CM}.
  By Step 4 of the proof of \cite[Prop.~\ref{bhhms4:prop:nonsplit-I1}]{BHHMS4} we have $\mathcal{Z}(N_2') = \mathcal{Z}(N_2^{i_0}) = \mathcal{Z}(\gr_\m(\pi_2^\vee))$.
  Using the same argument as in the last paragraph of the proof of \cite[Prop.~\ref{bhhms4:prop:nonsplit-I1}]{BHHMS4} (i.e.\ $N_2'$ is Cohen--Macaulay and the two modules have the same cycle), we deduce that the morphism $N_2' \onto \gr_\m(\pi_2^\vee)$ is an isomorphism.
\end{proof}

\begin{cor}\label{cor:gr-pi'}%
{Assume that $\brho$ is $(4f+1)$-generic.}
  Suppose $\pi' = \pi_1'/\pi_1$ is any nonzero subquotient, where $\pi_1 \subsetneq \pi_1' \subset \pi$.
  Then we have an isomorphism of graded $\gr(\Lambda)$-modules with compatible $H$-actions,
  \begin{equation}\label{eq:gr-pi'}
    \gr_\m(\pi'^\vee) \cong \bigoplus_{\lambda\in\mathscr{P}}\chi_{\lambda}^{-1}\otimes \frac{\fa_1^{i_0}(\lambda)}{\fa_1^{i_0'}(\lambda)}(-d_\lambda),
  \end{equation}
  where $-1 \le i_0 \defeq  i_0(\pi_1) < i_0'\defeq i_0(\pi_1')   \le f$ and $d_\lambda \defeq  \max\{i_0+1-|J_\lambda|,0\}$.
\end{cor}

\begin{proof}
 {We first assume that $f \ge 2$.
  Then $4f+1 \ge \max\{9,2f+3\}$, so}
we may assume that $i_0' \le f-1$ by Theorem~\ref{thm:gr-pi2}.
  Let $N'$ denote
 the right-hand side of~\eqref{eq:gr-pi'}.
  Let $\pi_2 \defeq  \pi/\pi_1$ and $\pi'_2 \defeq  \pi/\pi'_1$, so
  \begin{equation*}
    0 \to \pi' \to \pi_2 \to \pi_2' \to 0.
  \end{equation*}
  Let $d_\lambda' \defeq  \max\{i_0'+1-|J_\lambda|,0\}$.

  \textbf{Step 1.} 
  We show that $N'/\o\m^n \cong \gr(\pi'^\vee)/\o\m^n$, where $n \defeq  \max\{i_0'-i_0,2\}+1 (\le {f+1})$.%

  Let $n' \defeq  n+i_0'+1 {(\le {2f+1})}$
  and let $\o\tau \defeq  \tau^{(n')}[\m^{n'}]$.
  {Note that by assumption $\brho$ is $(2n'-1)$-generic.}
  Define $\Theta = \bigoplus_{\lambda \in \P} \Theta_\lambda$ (resp.\ $\Theta' = \bigoplus_{\lambda \in \P} \Theta'_\lambda$), as the image of $\o\tau$ in $\pi_2$ (resp.\ $\pi_2'$).
  Then $\Theta_\lambda'^\vee \subset \Theta^\vee_\lambda$ for all $\lambda \in \P$.
  By~\eqref{eq:compare-fil} applied to $\Theta^\vee_\lambda $ and $\Theta'^\vee_\lambda$ we have
  \begin{equation}\label{eq:compare-fil2}
    \m^i \Theta_\lambda^\vee \cap \Theta'^\vee_\lambda = \m^{i+d_\lambda}\o\tau_\lambda^\vee \cap \Theta'^\vee_\lambda = \m^{i+d_\lambda-d_\lambda'}\Theta'^\vee_\lambda
  \end{equation}
  for all $i \in \Z$.

  From~\eqref{eq:compare-fil} and~\eqref{eq:N_2-lambda'} we have
  \begin{equation*}
    \gr(\Theta_\lambda^\vee)(d_\lambda) \cong \gr_{F_\lambda}(\Theta_\lambda^\vee) \cong \left(\chi_{\lambda}^{-1}\otimes \frac{\fa_1^{i_0}(\lambda)}{\fa(\lambda)}\right)_{\ge -n'+1},
  \end{equation*}
  using the notation $(\cdot)_{\ge -n'+1}$ as in \cite[Lemma~\ref{bhhms4:lem:graded-Tor}]{BHHMS4}, hence
  \begin{equation}\label{eq:gr-Theta-lambda}
    \gr(\Theta_\lambda^\vee) \cong \left(\chi_{\lambda}^{-1}\otimes \frac{\fa_1^{i_0}(\lambda)}{\fa(\lambda)}(-d_\lambda)\right)_{\ge -n'+d_\lambda+1}
  \end{equation}
  and likewise for $\gr(\Theta_\lambda'^\vee)$.
  As $i_0+1 \ge d_\lambda$, by Theorem \ref{thm:gr-pi2} the natural surjection $\gr(\pi_2^\vee) \onto \gr(\Theta^\vee)$ is an isomorphism in degrees $\ge -n'+i_0+2$, and likewise for $\gr(\pi_2'^\vee)$  in degrees $\ge -n'+i_0'+2$.
  As $-n+1 = -n'+i_0'+2  > -n'+i_0+2$, we obtain that the natural surjections $\pi_2^\vee \onto \Theta^\vee$ and $\pi_2'^\vee \onto \Theta'^\vee$ induce isomorphisms
  \begin{equation}\label{eq:pi2-vs-Theta}
    \pi_2^\vee/\m^i\pi_2^\vee \congto \Theta^\vee/\m^i \Theta^\vee, \qquad \pi_2'^\vee/\m^i\pi_2'^\vee \congto \Theta'^\vee/\m^i \Theta'^\vee
  \end{equation}
  for all $0 \le i \le n$.

  Suppose $0 \le i \le n$.
  From the exact sequence $\pi_2'^\vee/\m^i\pi_2'^\vee \to \pi_2^\vee/\m^i\pi_2^\vee \to \pi'^\vee/\m^i\pi'^\vee \to 0$ and~\eqref{eq:pi2-vs-Theta}, we obtain that
  \begin{equation*}
    \pi'^\vee/\m^i\pi'^\vee \cong \bigoplus_{\lambda\in\mathscr{P}} \Theta^\vee_\lambda/(\Theta'^\vee_\lambda + \m^i\Theta^\vee_\lambda).
  \end{equation*}
  By the line above together with~\eqref{eq:pi2-vs-Theta} we see that the kernel of $\pi_2^\vee/\m^i\pi_2^\vee \onto \pi'^\vee/\m^i\pi'^\vee$ is identified with
  \begin{equation*}
   \bigoplus_{\lambda\in\mathscr{P}} (\Theta'^\vee_\lambda + \m^i\Theta^\vee_\lambda)/\m^i \Theta_\lambda^\vee \cong \bigoplus_{\lambda\in\mathscr{P}}\Theta'^\vee_\lambda/\m^{i+d_\lambda-d_\lambda'} \Theta'^\vee_\lambda,
  \end{equation*}
  where the isomorphism follows from~\eqref{eq:compare-fil2},
  and hence we have an exact sequence
  \begin{equation*}
    0 \to \bigoplus_{\lambda\in\mathscr{P}} \Theta'^\vee_\lambda/\m^{i+d_\lambda-d_\lambda'} \Theta'^\vee_\lambda \to \bigoplus_{\lambda\in\mathscr{P}} \Theta^\vee_\lambda/\m^i \Theta^\vee_\lambda \to \pi'^\vee/\m^i\pi'^\vee \to 0.
  \end{equation*}
  Therefore the filtration on the left term induced by the $\m$-adic filtration on the middle term is the $\m$-adic filtration up to a shift by $d_\lambda'-d_\lambda$.
  Taking graded pieces for $i = n$, we obtain
  \begin{equation*}
    0 \to \bigoplus_{\lambda\in\mathscr{P}} (\gr(\Theta_\lambda'^\vee)/\o\m^{n+d_\lambda-d_\lambda'})(d_\lambda'-d_\lambda) \to \bigoplus_{\lambda\in\mathscr{P}} \gr(\Theta_\lambda^\vee)/\o\m^n \to \gr(\pi'^\vee)/\o\m^n \to 0
  \end{equation*}
  By~\eqref{eq:gr-Theta-lambda} and its analogue for $\gr(\Theta_\lambda'^\vee)$ we obtain
  \begin{equation*}
    0 \to \left(\bigoplus_{\lambda\in\mathscr{P}}\chi_{\lambda}^{-1}\otimes \frac{\fa_1^{i'_0}(\lambda)}{\fa(\lambda)}(-d_\lambda)\right)_{\ge -n+1} \to \left(\bigoplus_{\lambda\in\mathscr{P}}\chi_{\lambda}^{-1}\otimes \frac{\fa_1^{i_0}(\lambda)}{\fa(\lambda)}(-d_\lambda)\right)_{\ge -n+1} \to \gr(\pi'^\vee)/\o\m^n \to 0.
  \end{equation*}
  Therefore, $\gr(\pi'^\vee)/\o\m^n \cong (N')_{\ge -n+1} \cong N'/\o\m^n$ (the second isomorphism holds since $N'$ is generated by its elements of degree 0, which follows from the definition of the ideal $\fa_1^{i_0}(\lambda)$), as we wanted to show.
  
  \textbf{Step 2.} 
  We lift the isomorphism $ N'/\o\m^n \congto \gr(\pi'^\vee)/\o\m^n =(\gr(\pi'^\vee))_{\ge -n+1}$ to a homomorphism $N' \to \gr(\pi'^\vee)$.
  Consider the short exact sequence of graded $\gr(\Lambda)$-modules, 
\begin{equation*} 
    0 \to \frac{\fa_1^{i_0'}(\lambda)}{\fa(\lambda)}(-d_\lambda) \to \frac{\fa_1^{i_0}(\lambda)}{\fa(\lambda)}(-d_\lambda) \to \frac{\fa_1^{i_0}(\lambda)}{\fa_1^{i_0'}(\lambda)}(-d_\lambda) \to 0.
  \end{equation*}
  Going back to the proof of Lemma~\ref{lem:lift-graded-homo} (noting that $n \ge 3$), we know that the middle term has relations generated in degrees $-1$, $-2$, and the left term has generators in degree $d_\lambda-d_\lambda'$.
  As $0 \le d_\lambda'-d_\lambda \le i_0'-i_0$, the right term has relations in degree $-1$, $-2$, \dots, $-\max\{i_0'-i_0,2\}=-n+1$.
  (Note that when $d_\lambda'-d_\lambda=0$ then $d_\lambda' = d_\lambda = 0$ by definition since $i_0 < i_0'$, and hence $\fa_1^{i_0}(\lambda) = \fa_1^{i_0'}(\lambda) = \o R$.)
  By Step 3 of the proof of Lemma~\ref{lem:lift-graded-homo}, we deduce the desired lifting $N' \to \gr(\pi'^\vee)$.

  \textbf{Step 3.}
  Using Steps 1 and 2, we conclude that the homomorphism $N' \to \gr(\pi'^\vee)$ is an isomorphism exactly as in Step 3 of the proof of Theorem~\ref{thm:gr-pi2}.
  (Note that $N'$ is Cohen--Macaulay by \cite[Cor.~\ref{bhhms4:cor:subquot}]{BHHMS4}.
  Also note that $\cZ(\gr(\pi'^\vee)) = \cZ(\gr(\pi_2^\vee)) - \cZ(\gr(\pi_2'^\vee)) = \cZ(N_2^{i_0})-\cZ(N_2^{i'_0}) = \cZ(N')$.)

  Finally we assume that $f=1$, and we just assume that $\brho$ is 0-generic.
  We prove that $\pi$ has length $2$ and fits into a short exact sequence $0\ra \pi_0\ra \pi\ra \pi_1\ra0$, where $\pi_0,\pi_1$ are irreducible principal series as in \cite[Cor.~3.92]{BHHMS2}, namely they are dual to each other in the sense that
  \begin{equation}\label{eq:pi-dual}
    \EE^{2}_{\Lambda}(\pi_i^{\vee})\cong\pi_{1-i}^{\vee}\otimes(\det(\brho)\omega^{-1}).
  \end{equation}
  Indeed,  by assumption \ref{it:assum-i} we know that $W(\brho)=\{\sigma_0\}$ is a singleton and it is easy to see that $\pi_0\defeq\langle \GL_2(K)\cdot \sigma_0\rangle$ is an irreducible principal series (as in the proof of \cite[Cor.~3.92]{BHHMS2}) and that $\pi_0=\soc_{\GL_2(K)}(\pi)$.
  Using assumption \ref{it:assum-iii} and \eqref{eq:pi-dual}, we deduce that $\pi$ has a quotient isomorphic to $\pi_1$.
  We need to prove that $V=0$, where $V\defeq \ker(\pi\onto \pi_1)/\pi_0$.
  By \cite[Thm.~3.67]{BHHMS2} (with $r=1$) there is a surjection of $\gr(\Lambda)$-modules with compatible $H$-action
  \begin{equation}\label{eq:ontof=1}
    N\defeq\bigoplus_{\lambda\in\mathscr{P}}\chi_{\lambda}^{-1} \otimes R/\fa(\lambda)\onto \gr_{\m}(\pi^\vee)
  \end{equation}
  and $m(\gr_{\m}(\pi^{\vee}))\leq 4$  by \cite[Cor.~3.71]{BHHMS2}, where $m(-)$ denotes the total multiplicity of $\overline{R}$-modules.
  On the other hand, \cite[Prop.~3.76(ii)]{BHHMS2} implies that $\gr_{\m}(\pi_i^{\vee})$ is an $\overline{R}$-module and  $m(\gr_{\m}(\pi_i^{\vee}))=2$ for $i=0,1$, so we deduce  $m(V)=0$  by the additivity of $m(-)$, equivalently $\dim_{\F}V<+\infty$.
  However, this forces   $\Ext^1_{\GL_2(K)}(V,\pi_0)=0$ by \cite[Lemma~4.3.9, Prop.~4.3.32(1)]{emerton-ordII}, hence $V=0$ (as $\pi_0=\soc_{\GL_2(K)}(\pi)$).
  In all, we deduce that $\pi$ has length $2$ and that \eqref{eq:ontof=1} is an isomorphism (as the graded module $N$ in \eqref{eq:ontof=1} is Cohen--Macaulay) which determines $\gr_{\m}(\pi^{\vee})$.
  Moreover, using \cite[Prop.~3.76(ii)]{BHHMS2} again we check that $\gr_{\m}(\pi'^{\vee})$ is as in \eqref{eq:gr-pi'} for any proper subquotient $\pi'$ of $\pi$.
\end{proof}

\begin{cor}\label{cor:pi2-mn-mult-free}%
 Assume $\brho$ is $\max\{9,2f+3,2n+2i_0+1\}$-generic for some $n\geq 1$.
  If $\pi' = \pi_1'/\pi_1$ is any subquotient, where $\pi_1 \subsetneq \pi_1' \subset \pi$, then $\pi'[\m^n]$ is multiplicity free as $I$-representation.
\end{cor}

\begin{proof}
 Since $\pi'$ injects into $\pi_2 = \pi/\pi_1$ and hence $\pi'[\m^n] \subset \pi_2[\m^n]$ we are reduced to the case where $\pi'=\pi_2$.
  We need to show that $\gr(\pi_2^\vee)/\o\m^n$ is multiplicity free.
  By Theorem~\ref{thm:gr-pi2}, and as ${\mathfrak{a}_1^{i_0}(\lambda)}/{\mathfrak{a}(\lambda)}$ is generated by elements of degree $-d_\lambda$ for each $\lambda \in \P$, it is equivalent to showing that
  \begin{equation*}
    \bigoplus_{\lambda \in \P} \chi_\lambda^{-1} \otimes \left(\frac{\mathfrak{a}_1^{i_0}(\lambda)}{\mathfrak{a}(\lambda)}(-d_\lambda)\right)_{> -n}
  \end{equation*}
  is multiplicity free.
  Hence it is sufficient that
  \begin{equation*}
    \bigoplus_{\lambda \in \P} \chi_\lambda^{-1} \otimes \left(\frac{\o R}{\mathfrak{a}(\lambda)}\right)_{> -(n+d_\lambda)}
  \end{equation*}
  is multiplicity free.
  This follows from \cite[Lemma~\ref{bhhms4:lem:N/I-multifree}]{BHHMS4}, as $n+d_\lambda \le n+i_0+1$ and as $N/\mathcal{I}^{(n+i_0+1)}N$ surjects onto $N/\o\m^{n+i_0+1} N {= N_{>-(n+i_0+1)}}$ {(as $h_j$ kills $N$)}.
\end{proof}

\begin{rem}\label{rem2:i0-f-1-or-f}
  Just like in Remark~\ref{rem:i0-f-1-or-f} there is an easier proof when $i_0(\pi_1) \in \{f-1,f\}$ {with $\rhobar$ being $\max\{9,2n-1,2f+1\}$-generic}.
  In the first case, the multiplicity freeness {follows from \eqref{eq:i0-f-1}, by applying \cite[Lemma~\ref{bhhms4:lem:compare-I1-invts}(ii)]{BHHMS4} with $m = 2n-2$ and $\lambda = \lambda'' \in \P^\ss$ (so $t_j = y_j$ for all $j$)}; in the second case it is trivial.
\end{rem}

We conclude this section with a further result regarding the structure of $\pi$.
{Let $\pi_s$ be an admissible smooth representation satisfying assumptions \ref{it:assum-i}--\ref{it:assum-v} with respect to $\rhobar^{\ss}$.}
The most optimistic expectation,  {at least for globally defined $\pi = \pi(\brho)$ and $\pi_s = \pi(\brho^\ss)$ as in \S~\ref{sec:global:setting}} (cf.\ the comments after \cite[Thm.\ 19.10]{BP}), is that $\pi^\ss \cong \pi_s$ and moreover that $\pi$ and $\pi_s$ both have length $f+1$.
{(In fact, the first expectation implies the second by Remark \ref{rem:length-f+1} below.)}
The following proposition provides new evidence for this expectation.
For any $\lambda' \in \P^\ss$ we let $\fa^\ss(\lambda')$ denote the ideal of $R$ generated by all $t_j^\ss = t_j^\ss(\lambda') \in \{y_j,z_j,y_jz_j\}$, which are defined as in (\ref{eq:id:al}) but for the Galois representation $\brho^\ss$. Recall $\fa^\ss(\lambda') \supset \ker(R \to \o R)$, so we often think of it as ideal of $\o R$. %

\begin{prop}\label{prop:gr-m-semisimple-nonsplit}
  For any $-1 \le i_0 \le f-1$ we have an isomorphism
  \begin{equation*}
    \bigoplus_{\lambda\in\P}\chi_{\lambda}^{-1}\otimes \frac{\fa_1^{i_0}(\lambda)}{\fa_1^{i_0+1}(\lambda)}(-d_\lambda) \cong \bigoplus_{\lambda'\in\P^\ss,\; \ell(\lambda') = i_0+1}\chi_{\lambda'}^{-1}\otimes \frac{\o R}{\fa^\ss(\lambda')}
  \end{equation*}
  of graded $\gr(\Lambda)$-modules with compatible $H$-actions.
\end{prop}

\begin{rem}\label{rem:gr-m-semisimple-nonsplit}
  For any $-1 \le i_0 \le f-1$ let $\pi' \defeq  \pi_1'/\pi_1$, where $i_0(\pi_1) = i_0$ and $i_0(\pi'_1) = i_0+1$ (if such $\pi_1$, $\pi_1'$ exist) and let $\pi'_s$   denote the subquotient of $\pi_s$ corresponding to the subset {$\P' \defeq \{\lambda'\in\P^\ss,\; \ell(\lambda') = i_0+1\}$} in \cite[Cor.~\ref{bhhms4:cor:finite-length}(ii)]{BHHMS4} (if it exists).
  If $\pi'$ and $\pi'_s$ exist (for example, if $\pi$ and $\pi_s$ have length $f+1$), then by Corollary~\ref{cor:gr-pi'}  and \cite[Cor.~\ref{bhhms4:cor:finite-length}(iii)]{BHHMS4} {(provided that $\brho$ is $\max\{9,4f+1\}$-generic)}, Proposition~\ref{prop:gr-m-semisimple-nonsplit} asserts that
  \begin{equation}\label{eq:gr-isom}
    \gr_\m(\pi'^\vee) \cong \gr_\m(\pi_s'^\vee)
  \end{equation}
  as graded $\gr(\Lambda)$-modules with compatible $H$-actions.
  If $i_0+1 \in \{0,f\}$, we even know that $\pi' \cong \pi_s'$ are isomorphic principal series (compare \cite[Prop.\ 10.8]{HuWang2} with \cite[Cor.\ 3.92]{BHHMS2}). %
  More interestingly, if $f = 2$ and $i_0 = 0$ we know that $\pi'$ and $\pi'_s$ exist (and are supersingular) by \cite[Thm.\ 1.7]{HuWang2}, \cite[Cor.\ 3.92]{BHHMS2}, and hence~\eqref{eq:gr-isom} holds (provided $\brho$ is $\max\{9,4f+1\}$-generic).
\end{rem}

\begin{proof}
  Fix $\lambda\in\P$.
  As usual, let $J_1 \defeq  \{ j \in J_{\rhobar}^c : \lambda_j(x_j) = p-1-x_j \}$ and $J_2 \defeq  \{ j \in J_{\rhobar}^c : \lambda_j(x_j) = x_j \}$, and let $J \defeq  J_1 \sqcup J_2$.
  We show that
  \begin{equation}\label{eq:lambda-term}
    \chi_{\lambda}^{-1}\otimes \frac{\fa_1^{i_0}(\lambda)}{\fa_1^{i_0+1}(\lambda)}(-d_\lambda) \cong \bigoplus_{J' \subset J,\; |J_\lambda|+|J'|=i_0+1}\chi_{\lambda'(J')}^{-1}\otimes \frac{\o R}{\fa^\ss(\lambda'(J'))},
  \end{equation}
  where $\lambda'(J') \in \P^\ss$ is defined by
  \begin{equation*}
    \lambda'(J')_j(x_j) \defeq 
    \begin{cases}
      p-3-x_j & \text{if $j \in J_1' \defeq  J' \cap J_1$,} \\
      x_j+2 & \text{if $j \in J_2' \defeq  J' \cap J_2$,} \\
      \lambda_j(x_j) & \text{otherwise.}
    \end{cases}
  \end{equation*}
  It is easy to check that this implies the proposition, by taking a direct sum over all $\lambda \in \P$.

  If $i_0+1-|J_\lambda| < 0$, then~\eqref{eq:lambda-term} trivially holds: by definition, $\fa_1^{i_0}(\lambda) = \fa_1^{i_0+1}(\lambda) = \o R$, so both sides are zero.

  Suppose that $i_0+1-|J_\lambda| \ge 0$, so $d_\lambda = i_0+1-|J_\lambda|$.
  Note that $\chi_{\lambda'(J')}^{-1} = \chi_\lambda^{-1} \prod_{j \in J_1'}\alpha_j \prod_{j \in J_2'}\alpha_j^{-1}$, and $\prod_{j \in J_1'}\alpha_j \prod_{j \in J_2'}\alpha_j^{-1}$ gives the action of $H$ on the degree $d_\lambda$ polynomial 
  \[p_{J'} \defeq  \prod_{j \in J_1'}y_j \prod_{j \in J_2'}z_j \in R \cong \F[y_j,z_j : 0 \le j \le f-1].\]
  Therefore, by twisting both sides by $\chi_\lambda(d_\lambda)$, it suffices to show that
  \begin{equation*}
    \frac{\fa_1^{i_0}(\lambda)}{\fa_1^{i_0+1}(\lambda)} \cong \bigoplus_{J' \subset J,\; |J'|=d_\lambda} p_{J'}\otimes \frac{\o R}{\fa^\ss(\lambda'(J'))}.
  \end{equation*}
  We have
  \begin{equation*}
    \frac{\fa_1^{i_0}(\lambda)}{\fa_1^{i_0+1}(\lambda)} \cong \frac{I(J_1,J_2,d_\lambda)+\fa(\lambda)}{I(J_1,J_2,d_\lambda+1)+\fa(\lambda)} \cong \frac{I(J_1,J_2,d_\lambda)}{I(J_1,J_2,d_\lambda+1)+I(J_1,J_2,d_\lambda) \cap \fa(\lambda)}
  \end{equation*}
  (recall the ideals $I(J_1,J_2,d_\lambda)$ from \cite[Def.~\ref{bhhms4:def:IJideals}]{BHHMS4}).
  Note that $I(J_1,J_2,d_\lambda) = (p_{J'} : J' \subset J,\; |J'|=d_\lambda)$ and $\fa(\lambda) = (t_j : 0 \le j \le f-1)$ with $t_j = y_jz_j$ for all $j \in J$ (see equation \eqref{eq:id:al}).
  By \cite[Prop.\ 1.2.1]{herzog-hibi} the ideal $I(J_1,J_2,d_\lambda) \cap \fa(\lambda)$ is generated by the monomials $p_{J'} z_j$ ($j \in J_1'$), $p_{J'} y_j$ ($j \in J_2'$), and $p_{J'} t_j$ ($j \notin J'$), where $J' \subset J$ runs through all subsets with $|J'|=d_\lambda$.
  Hence the ideal $I(J_1,J_2,d_\lambda+1)+I(J_1,J_2,d_\lambda) \cap \fa(\lambda)$ is generated by the monomials $p_{J'} z_j$ ($j \in J_1' \sqcup (J_2\setminus J_2')$), $p_{J'} y_j$ ($j \in J_2' \sqcup (J_1\setminus J_1')$), and $p_{J'} t_j$ ($j \notin J$), where again $J' \subset J$ runs through all subsets with $|J'|=d_\lambda$.
  Since, from equation \eqref{eq:id:al} and the definition of $\lambda'(J')$, $\fa^\ss(\lambda'(J')) = \fa(\lambda) + (z_j : j \in J_1' \sqcup (J_2\setminus J_2')) + (y_j : j \in J_2' \sqcup (J_1\setminus J_1'))$ we deduce that
  \begin{equation*}
    I(J_1,J_2,d_\lambda+1)+I(J_1,J_2,d_\lambda) \cap \fa(\lambda) = \sum_{J' \subset J,\; |J'|=d_\lambda} p_{J'} \cdot \fa^\ss(\lambda'(J')).
  \end{equation*}
  In particular, 
\[
  \frac{I(J_1,J_2,d_\lambda)}{I(J_1,J_2,d_\lambda+1)+I(J_1,J_2,d_\lambda) \cap \fa(\lambda)}  \cong  \frac{\sum_{J' \subset J,\;
|J'|=d_\lambda} p_{J'}R}{\sum_{J' \subset J,\; |J'|=d_\lambda} p_{J'} \fa^\ss(\lambda'(J'))},
\]  
so for each index $J'$, multiplication induces a homomorphism
  \begin{equation*}
    p_{J'}\otimes \frac{R}{\fa^\ss(\lambda'(J'))} \to \frac{I(J_1,J_2,d_\lambda)}{I(J_1,J_2,d_\lambda+1)+I(J_1,J_2,d_\lambda) \cap \fa(\lambda)}
  \end{equation*}
  and passing to the direct sum induces a surjective homomorphism
  \begin{equation*}
    \bigoplus_{J' \subset J,\; |J'|=d_\lambda} p_{J'}\otimes \frac{R}{\fa^\ss(\lambda'(J'))} \onto \frac{I(J_1,J_2,d_\lambda)}{I(J_1,J_2,d_\lambda+1)+I(J_1,J_2,d_\lambda) \cap \fa(\lambda)},
  \end{equation*}
  which we need to show is an isomorphism.
  Suppose that we have an element $f = (p_{J'} \otimes [f_{J'}])_{J'}$ in the kernel, or equivalently that $\sum_{J'} p_{J'}f_{J'} = \sum_{J'} p_{J'}g_{J'}$ in $R$ for some $g_{J'} \in \fa^\ss(\lambda'(J'))$.
  By replacing $f_{J'}$ by $f_{J'}-g_{J'}$ we may assume that $g_{J'} = 0$ for all $J'$, i.e.\ that $\sum_{J'} p_{J'}f_{J'} = 0$.
  Fix $J'$ and let $\fb(J') \defeq  (z_j : j \in J_2\setminus J_2') + (y_j : j \in J_1\setminus J_1') \subset \fa^\ss(\lambda'(J'))$.
  From $p_{J'}f_{J'} = - \sum_{J'' \ne J'} p_{J''}f_{J''}$ we deduce that $p_{J'}f_{J'} \in \fb(J')$.
  As multiplication by $p_{J'}$ is injective on $R/\fb(J')$ (a polynomial ring), it follows that $f_{J'} \in \fb(J') \subset \fa^\ss(\lambda'(J'))$, so $f = 0$, as desired.
\end{proof}

\subsection{\texorpdfstring{$I_1$}{I\_1}-invariants and \texorpdfstring{$\GL_2(\cO_K)$}{GL\_2(O\_K)}-socle of subquotient representations of \texorpdfstring{$\pi$}{pi}}
\label{sec:further-properties}

We describe the $I_1$-invariants of subquotients of $\pi$. We deduce that $\pi$ is multiplicity free and give a description of the $\GL_2(\cO_K)$-socles of subquotients of $\pi$.

We start with the $I_1$-invariants of quotients $\pi_2 = \pi/\pi_1$ of $\pi$:

\begin{prop}\label{prop:pi2-I1-invt}
  {Assume that $\brho$ is $\max\{9,2f+3\}$-generic.} %
  Let $i_0 = i_0(\pi_1)$ with $-1 \le i_0 \le f$ be as in \cite[Thm.~\ref{bhhms4:thm:conj2}]{BHHMS4}.
  Then $\pi_2^{I_1}$ is multiplicity free and
  \begin{equation*}
    \JH(\pi_2^{I_1}) = \{ \chi_\lambda : \text{$\lambda \in \P, |J_\lambda| > i_0$ or $\lambda \in \P^\ss \setminus \P, |J_\lambda| = i_0+1$} \}.
  \end{equation*}
\end{prop}

\begin{proof}
As $H$-representations $\pi_2^{I_1}$ is dual to the degree $0$ part of $\gr_{\m}(\pi_2^{\vee})$, namely, $\F\otimes_{\gr(\Lambda)}\gr_{\m}(\pi_2^{\vee})$. Comparing Theorem~\ref{thm:gr-pi2}  and \cite[Cor.~\ref{bhhms4:cor:gr-pi1}]{BHHMS4}, we see that there is an isomorphism $\F\otimes_{\gr(\Lambda)}\gr_{\m}(\pi_2^{\vee})\cong \F\otimes_{\gr(\Lambda)}\gr_F(\pi_2^{\vee})$ compatible with $H$-actions (but not gradings), where $F$ denotes the filtration on $\pi_2^{\vee}$ induced by the $\m$-adic filtration on $\pi^{\vee}$. The result then follows from \cite[Cor.~\ref{bhhms4:cor:subquot-F}]{BHHMS4}.
\end{proof}

We can generalize to subquotients: %

\begin{cor}\label{cor:subquot-I1-invt}
  {Assume that $\brho$ is $\max\{9,2f+3\}$-generic.}
  Suppose $\pi' = \pi_1'/\pi_1$ is any nonzero subquotient, where $\pi_1 \subsetneq \pi_1' \subset \pi$.
  Let $i_0 \defeq  i_0(\pi_1)$, $i_0' \defeq  i_0(\pi_1')$, so $-1 \le i_0 < i_0' \le f$.
  Then $\pi'^{I_1}$ is multiplicity free and
  \begin{equation*}
    \JH(\pi'^{I_1}) = \{ \chi_\lambda : \text{$\lambda \in \P,\; i_0 < |J_\lambda| \le i'_0$ or $\lambda \in \P^\ss \setminus \P,\; |J_\lambda| = i_0+1$} \}.
  \end{equation*}
\end{cor}

\begin{proof}
  Since we have injections \[(\pi_1')^{I_1}/\pi_1^{I_1}\into \pi'^{I_1} \into (\pi/\pi_1)^{I_1},\] we deduce from \cite[Cor.~\ref{bhhms4:cor:conj1}]{BHHMS4} and Proposition~\ref{prop:pi2-I1-invt} that
  \begin{multline*}
    \{ \chi_\lambda : \text{$\lambda \in \P, i_0 < |J_\lambda| \le i'_0$} \} \subset \JH(\pi'^{I_1}) \\
    \subset \{ \chi_\lambda : \text{$\lambda \in \P, i_0 < |J_\lambda|$ or $\lambda \in \P^\ss \setminus \P, |J_\lambda| = i_0+1$} \}.
  \end{multline*}
  In particular, $\pi'^{I_1}$ is multiplicity free. On the other hand, note that $\pi'^{I_1}$ is the kernel of the natural map $(\pi/\pi_1)^{I_1}\ra (\pi/\pi_1')^{I_1}$. For any $\lambda\in \P^{\ss}\setminus \P$ with $|J_\lambda|=i_0+1$, $\chi_\lambda\in \JH((\pi/\pi_1)^{I_1})$ and maps to zero in $(\pi/\pi'_1)^{I_1}$ by Proposition \ref{prop:pi2-I1-invt}, so $\chi_{\lambda}\in \JH(\pi'^{I_1})$ for such $\lambda$.

 To finish the proof it remains to show that $\chi_\lambda$ for $\lambda \in \P, |J_\lambda| > i_0'$ does not occur in $\pi'^{I_1}$.
  By \cite[Cor.~\ref{bhhms4:cor:conj1}]{BHHMS4}, $\chi_\lambda$ does not occur in $\pi_1'^{I_1}$ and $\pi_1^{I_1}$, and by Proposition~\ref{prop:pi2-I1-invt} it occurs in $(\pi/\pi_1)^{I_1}$ with multiplicity 1. 
  Hence the $\chi_\lambda$-eigenspace of $(\pi/\pi_1)^{I_1}$ is the image of the $\chi_\lambda$-eigenspace of $\pi^{I_1}$ and so maps to zero in $H^1(I_1/Z_1,\pi_1)$.
  The following diagram then shows that $\chi_\lambda$ cannot occur in $\pi'^{I_1} =(\pi'_1/\pi_1)^{I_1}$, since it does not occur in $\pi_1'^{I_1}$:
  \begin{equation*}
    \begin{gathered}[b]
      \xymatrix{
        0 \ar[r] & \pi_1^{I_1} \ar[r]\ar@{=}[d] & \pi_1'^{I_1} \ar[r]\ar@{^{(}->}[d]& (\pi_1'/\pi_1)^{I_1} \ar[r]\ar@{^{(}->}[d] & H^1(I_1/Z_1,\pi_1) \ar@{=}[d]\\
        0 \ar[r] & \pi_1^{I_1} \ar[r] & \pi^{I_1} \ar[r]& (\pi/\pi_1)^{I_1} \ar[r] & H^1(I_1/Z_1,\pi_1). } \\[-\dp\strutbox]
    \end{gathered}
    \qedhere %
  \end{equation*}
\end{proof}

\begin{rem}\label{rem:length-f+1}
Let $\pi_s$ be an admissible smooth representation satisfying assumptions \ref{it:assum-i}--\ref{it:assum-v} with respect to $\rhobar^{\ss}$.
  We remark that if $\pi^\ss \cong \pi_s$ (as one might hope is true in analogy with $\GL_2(\Q_p)$) or even just $\pi^\ss \cong \pi_s^\ss$, then $\pi$ and $\pi_s$ have length exactly $f+1$  provided $\rhobar$ is $\max\{9,2f+3\}$-generic. %
  (Sketch proof: we calculate $(\pi^\ss)^{I_1}$ using Corollary~\ref{cor:subquot-I1-invt} and $(\pi_s^\ss)^{I_1}$ using \cite[Cor.~\ref{bhhms4:cor:split-In}]{BHHMS4}.
  If $\Sigma \subset \{0,1,\dots,f\}$ denotes the subset of elements of the form $i_0(\pi_1)+1$ for some subrepresentation $\pi_1 \subsetneq \pi$, then we deduce that $\P^\ss = \P \cup \{ \lambda \in \P^\ss : |J_\lambda| \in \Sigma \}$.
  As $\brho$ is non-semisimple, $J_{\brho} \ne \{0,1,\dots,f-1\}$ and one easily shows that for any $1 \le k \le f$ there exists $\lambda \in \P^\ss \setminus \P$ with $|J_\lambda| = k$.
  We deduce that $\Sigma = \{0,1,\dots,f\}$, i.e.\ $\pi$ has length $f+1$.)
\end{rem}

\begin{cor}\label{cor:subquot-K-soc}
  {Assume that $\brho$ is $\max\{9,2f+3\}$-generic.} 
  Suppose $\pi' = \pi_1'/\pi_1$ is any nonzero subquotient, where $\pi_1 \subsetneq \pi_1' \subset \pi$.
  Let $i_0 \defeq  i_0(\pi_1)$, $i_0' \defeq  i_0(\pi_1')$, so $-1 \le i_0 < i_0' \le f$. 
  Then 
  \begin{equation}\label{eq:subquot-K-soc}
    \soc_{\GL_2(\cO_K)}(\pi') \cong \bigoplus_{\sigma \in W(\brho), i_0 < \ell(\sigma) \le i'_0} \sigma \oplus \bigoplus_{\sigma \in W(\brho^\ss) \setminus W(\brho), \ell(\sigma) = i_0+1} \sigma.
  \end{equation}
  In particular, $\soc_{\GL_2(\cO_K)}(\pi')$ is multiplicity free.
\end{cor}

For the proof, recall from \cite[\S~\ref{bhhms4:sec:k_1-invar-subr}%
]{BHHMS4} that $D_0(\brho)$ admits an increasing filtration $0=D_0(\brho)_{\leq -1}\subsetneq D_0(\brho)_{\leq 0}\subsetneq \cdots \subsetneq D_0(\brho)_{\leq i}\subsetneq \cdots \subsetneq D_0(\brho)_{\leq f}=D_0(\brho),$
where $D_0(\brho)_{\leq i}$ is the largest $\Gamma$-subrepresentation of $D_0(\brho)$ not containing any $\tau\in W(\brho^\ss)$ with $\ell(\tau)>i$ as subquotient.
We set $D_{0}(\brho)_i\defeq D_0(\brho)_{\leq i}/D_0(\brho)_{\leq i-1}$, $D_{0,\sigma}(\brho)_{\leq i}\defeq D_{0,\sigma}(\brho)\cap D_{0}(\brho)_{\leq i}$ and  
$D_{0,\tau}(\brho)_{i} \defeq  D_0(\brho)_i \cap D_{0,\tau}(\brho^\ss)$ for $\sigma\in W(\brho)$ and $\tau\in W(\brho^\ss)$ (see \emph{loc.~cit}.~and note that $D_0(\brho)_i\into D_0(\brho^\ss)_i$ using \cite[Prop.~13.1]{BP} so that the latter intersection is taken in $D_0(\brho^\ss)$, see also \cite[Prop.~5.2]{Hu-SMF}). 

\begin{proof}
  By Corollary~\ref{cor:subquot-I1-invt}, $\soc_{\GL_2(\cO_K)}(\pi')$ is multiplicity free.
  We prove the inclusion ``$\supseteq$''.
  If $\sigma \in W(\brho), i_0 < \ell(\sigma) \le i'_0$, then $\sigma \subset \pi_1'|_{\GL_2(\cO_K)}$ but $\sigma \not\subset \pi_1|_{\GL_2(\cO_K)}$ by \cite[Cor.~\ref{bhhms4:cor:pi1-K-soc}]{BHHMS4}, so $\sigma \subset \pi'|_{\GL_2(\cO_K)}$.
  If $\sigma \in W(\brho^\ss), \ell(\sigma) = i_0+1$, then $\sigma \subset D_0(\brho)_{i_0+1}$ by~\cite[eq.~\eqref{bhhms4:eq:soc-i}]{BHHMS4}, which by \cite[Thm.~\ref{bhhms4:thm:conj2}]{BHHMS4} injects into
  \begin{equation*}
    \pi_1'^{K_1}/\pi_1^{K_1} = D_0(\brho)_{\le i_0'}/D_0(\brho)_{\le i_0} \into \pi'^{K_1}.
  \end{equation*}
  (In particular, the right-hand side of~\eqref{eq:subquot-K-soc} injects into $\soc_{\GL_2(\cO_K)}(\pi_1'^{K_1}/\pi_1^{K_1})$.)

  Now we prove the inclusion ``$\subseteq$''.
  Suppose that $\tau$ is a Serre weight such that $\tau \subset \pi'|_{\GL_2(\cO_K)}$.
  By Corollary~\ref{cor:subquot-I1-invt}, we know $\tau^{I_1} = \chi_\lambda$, where either $\lambda \in \P$, $i_0 < |J_{\lambda}| \le i_0'$ or $\lambda\in \P^\ss\setminus \P$, $|J_{\lambda}| = i_0+1$.

  If $\lambda \in \P$, $i_0 < |J_{\lambda}| \le i_0'$, then $\tau^{I_1}$ lifts to $\pi_1'^{I_1}$ by \cite[Cor.~\ref{bhhms4:cor:conj1}]{BHHMS4} and hence %
  $\tau$ is the image of a morphism $\Ind_I^{\GL_2(\cO_K)} \chi_\lambda \to \pi_1' \onto \pi'$.
  In particular, $\tau \into \pi_1'^{K_1}/\pi_1^{K_1} \cong D_0(\brho)_{\le i'_0}/D_0(\brho)_{\le i_0}$, so by d\'evissage and~\cite[eq.~\eqref{bhhms4:eq:soc-i}]{BHHMS4}, $\tau \in W(\brho^\ss)$ with $i_0 < \ell(\tau) \le i_0'$.
  Suppose that $\tau \notin W(\brho)$ (or we are done).
  By \cite[Lemma~\ref{bhhms4:lem:I1-invt-nonss}]{BHHMS4} we deduce that the image of the above morphism $\Ind_I^{\GL_2(\cO_K)} \chi_\lambda \to \pi_1' \subset \pi$ is isomorphic to $I(\sigma,\tau)$, where $\sigma \in W(\brho)$ is determined (via equation \eqref{eq:J-lambda}) by $J_\sigma = J_{\brho} \cap J_\lambda$.
  As the image of the composition is $\tau$, it follows that $\rad_\Gamma(I(\sigma,\tau)) \subset \pi_1^{K_1} = D_0(\brho)_{\le i_0}$, so from \cite[Lemma~\ref{bhhms4:lem:JHinW}]{BHHMS4} (taking $\tau'$ such that $|J_{\tau'}| = |J_\tau|-1$) and \cite[eq.~\eqref{bhhms4:eq:soc-i}]{BHHMS4} we deduce that $\ell(\tau) = i_0+1$, as desired.

  Suppose that $\lambda\in \P^\ss\setminus \P$, $|J_{\lambda}| = i_0+1$. %
  From the proof of Corollary~\ref{cor:subquot-I1-invt} we know that the 1-dimensional subspace $(\pi')^{I=\chi_\lambda}$ is the image of the morphism $W(\chi_\mu,\chi_\lambda) \into \pi_1' \onto \pi'$, where $J_1$, $J_2$ and $\mu \in \P$ are defined as in equations \cite[eq.~\eqref{bhhms4:eq:J1:J2}, eq.~\eqref{bhhms4:eq:mu_j}]{BHHMS4}.
  By Frobenius reciprocity we have a corresponding morphism $\Ind_I^{\GL_2(\cO_K)} W(\chi_\mu,\chi_\lambda) \to \pi_1'$ and we denote by $V$ its image.
  Note that $V/(V \cap \pi_1) \cong \tau$.
  By \cite[Lemma~\ref{bhhms4:lem:W-embeds}(i)]{BHHMS4} we have $\soc_{\GL_2(\cO_K)}(V) \cong \sigma$, where $\sigma \in W(\brho)$ is determined by $J_\sigma = J_{\brho} \cap J_\mu$, and by \cite[Lemma~\ref{bhhms4:lem:I1-invt-nonss}]{BHHMS4} and \cite[eq.~\eqref{bhhms4:eq:mu_j}]{BHHMS4} it has parameter
  \begin{align*}
    \mathcal{S}(\sigma) &= \{ j : \mu_j(x_j)\in\{x_j,\un{x_j+1},p-2-x_j,p-3-x_j\} \} \\
    &= \{ j : \lambda_j(x_j) \in \{ x_j,\un{x_j+1},\dotuline{x_j+2},p-2-x_j,\un{p-3-x_j} \}\}
  \end{align*}
  inside $\Ind_I^{\GL_2(\cO_K)} \chi_\mu$ ({see \S~\ref{sec:EG-Gamma-rep}} for $\mathcal{S}(\sigma)$).
  Here we use the convention {(from the proof of  \cite[Lemma~\ref{bhhms4:lem:I1-invt-nonss}]{BHHMS4})} that an underlined entry is only allowed when $j \in J_{\rhobar}$, and similarly that an entry with a dotted underline is only allowed when $j \notin J_{\rhobar}$.
  Let $\tau' \in W(\brho^\ss)$ be determined by $J_{\tau'} = J_\lambda$, and by~\cite[Lemma~\ref{bhhms4:lem:I1-invt-component}]{BHHMS4} it has parameter
  \begin{equation*}
    \mathcal{S}(\tau') = \{ j : \lambda_j(x_j)\in\{x_j,x_j+1,p-2-x_j,p-3-x_j\} \} \\
  \end{equation*}
  inside $\Ind_I^{\GL_2(\cO_K)} \chi_\lambda$.  As
  \begin{equation*}
    (\mathcal{S}(\sigma) \sqcup J_1) \setminus J_2 = \{ j : \lambda_j(x_j) \in \{ x_j,\un{x_j+1},p-2-x_j,p-3-x_j \}\} \subset \mathcal{S}(\tau'),
  \end{equation*}
  and $J_2 \subset \mathcal{S}(\sigma) \setminus \mathcal{S}(\tau')$ (and since $\mathcal{S}(\sigma) \cap J_1=\mathcal{S}(\tau') \cap J_2=\emptyset$) we deduce from \cite[Prop.~\ref{bhhms4:prop:HW}(ii), Prop.~\ref{bhhms4:prop:cond-K1inv}]{BHHMS4} that $\tau' \in \JH(V^{K_1})$.
  Since $\ell(\tau') = i_0+1$, it follows from \cite[Thm.~\ref{bhhms4:thm:conj2}]{BHHMS4}
 that $\tau' \notin \JH(\pi_1^{K_1})$, so $\tau' \cong \tau$.
  Thus $\tau\in W(\brho^\ss)$ and $\ell(\tau)=i_0+1$, i.e.~$\tau$ appears in the right-hand side of \eqref{eq:subquot-K-soc}, as desired.
\end{proof}

\begin{rem}\label{rem:subquot-K-soc}
  The proof shows that $\soc_{\GL_2(\cO_K)}(\pi') = \soc_{\GL_2(\cO_K)}(\pi_1'^{K_1}/\pi_1^{K_1})$, but by~\eqref{eq:subquot-K-soc} it is bigger than $\soc_{\GL_2(\cO_K)}(\pi_1')/\soc_{\GL_2(\cO_K)}(\pi_1)$, in general.
\end{rem}

Recall from subsection~\ref{sec:preliminaries} the $\wt{\Gamma}$-representation $\wt{D}_0(\brho)$ and from Proposition~\ref{prop:ii'} that \[\pi[\fm_{K_1}^2]\cong \wt{D}_0(\brho).\]

We now define the increasing filtration $(\wt{D}_0(\brho)_{\le i})_{-1 \le i \le f}$ on $\wt{D}_0(\brho)$ by letting $\wt{D}_0(\brho)_{\le i}$ be the largest $\wt\Gamma$-subrepresentation of $\wt{D}_0(\brho)$ that does not contain any $\tau \in W(\brho^\ss)$ with $\ell(\tau) > i$ as subquotient.
Equivalently, it is the largest subrepresentation $V$ of $\Inj_{\wt\Gamma} (\soc_{\GL_2(\cO_K)} \pi)$ such that 
\begin{enumerate}
\item $[V : \sigma] = 1$ for all $\sigma \in W(\brho)$, $\ell(\sigma) \le i$;
\item $[V : \tau] = 0$ for all $\tau \in W(\brho^\ss)$, $\ell(\tau) > i$.
\end{enumerate}
Then $\wt{D}_0(\brho)_{\le i} = \bigoplus_{\sigma \in W(\brho)} \wt{D}_{0,\sigma}(\brho)_{\le i}$ for a unique subrepresentation $\wt{D}_{0,\sigma}(\brho)_{\le i} \subset \wt{D}_{0}(\brho)_{\le i}$.

We remark (though will not need) that $\wt D_0(\brho)_{\le i} \cap D_0(\brho) = D_0(\brho)_{\le i}$ and that
all properties of the filtration $D_0(\brho)_{\le i}$ before \cite[Lemma~\ref{bhhms4:lem:inter}]{BHHMS4} generalize to the filtration $\wt D_0(\brho)_{\le i}$.

\begin{cor}\label{cor:pi1-mK1-squared}
{Assume that $\brho$ is $\max\{9,2f+3\}$-generic.} %
  Let $i_0 = i_0(\pi_1)$ with $-1 \le i_0 \le f$ be as in \cite[Thm.~\ref{bhhms4:thm:conj2}]{BHHMS4}.
  Then 
  \begin{equation*}
    \pi_1[\m_{K_1}^2] \cong \wt{D}_0(\brho)_{\le i_0}.
  \end{equation*}
\end{cor}

\begin{proof}
  (Cf.\ the proof of Proposition~\ref{prop:K1-square-invariants}(i).)
  By \cite[Thm.~\ref{bhhms4:thm:conj2}]{BHHMS4} we have $\pi_1[\m_{K_1}] = D_0(\brho)_{\le i_0}$, and as $\pi[\m_{K_1}^2]$ is multiplicity free we deduce from the injection $\pi_1[\m_{K_1}^2]/\pi_1[\m_{K_1}]\into \pi[\m_{K_1}^2]/\pi[\m_{K_1}]$ that no element of $W(\brho^\ss)$ occurs in $\pi_1[\m_{K_1}^2]/\pi_1[\m_{K_1}]$.
  As a consequence,
  \begin{equation*}
    \pi_1[\m_{K_1}^2] \subset \wt{D}_0(\brho)_{\le i_0}.
  \end{equation*}
  Let $Q \defeq  \wt{D}_0(\brho)_{\le i_0}/\pi_1[\m_{K_1}^2]$, which injects into $\pi[\m_{K_1}^2]/\pi_1[\m_{K_1}^2]$ and hence into $\pi_2[\m_{K_1}^2]$.
  If $Q \ne 0$, pick an irreducible subrepresentation $\sigma \subset Q \subset \pi_2|_{\GL_2(\cO_K)}$.
  Then $\sigma \in W(\brho^\ss)$ and $\ell(\sigma) > i_0$ by Corollary~\ref{cor:subquot-K-soc} (with $i_0' = f$), contradicting that $\sigma$ contributes to $\wt{D}_0(\brho)_{\le i_0}$.
  Hence $Q = 0$, as we wanted to show.
\end{proof}

\subsection{\texorpdfstring{$K_1$-}{K\_1-}-invariants of subquotient representations of \texorpdfstring{$\pi$}{pi}}
\label{sec:k1-quotient}

We describe the $K_1$-invariants of subquotients of $\pi$ (Corollary \ref{cor:K1-subquot}). The proofs in this section are subtle (and sometimes technical), in particular use the results of the preceding two sections and certain $\wt{\Gamma}$-representations that are not multiplicity free (Lemma \ref{lem:reps-Qi*}).

We start with the $K_1$-invariants of quotients $\pi_2 = \pi/\pi_1$ of $\pi$:

\begin{thm}\label{thm:pi2-K1} %
  {Assume that $\brho$ is $\max\{9,2f+3\}$-generic.} %
  Let $i_0 = i_0(\pi_1)$ with $-1 \le i_0 \le f$ be as in \cite[Thm.~\ref{bhhms4:thm:conj2}]{BHHMS4}.
  We have
  \begin{equation*}
    \pi_2^{K_1} \cong D_0(\brho^\ss)_{i_0+1} \oplus_{D_0(\brho)_{i_0+1}} (D_0(\brho)/D_0(\brho)_{\le i_0}).
  \end{equation*}
\end{thm}

To prepare for the proof, we first need some lemmas.

Recall from \S~\ref{sec:notation} that $D_0(\brho)=\bigoplus_{\sigma\in W(\rhobar)}D_{0,\sigma}(\brho)$, and from \cite[\S~13]{BP} that $D_{0,\sigma}(\brho)$ is maximal (for the inclusion) with respect to the two properties $\soc_{\GL_2(\cO_K)}(D_{0,\sigma}(\brho))=\sigma$ and $\JH(D_{0,\sigma}(\brho)/\sigma)\cap W(\brho)=\emptyset$.
In particular, $D_{0,\sigma}(\brho^{\ss})\subset D_{0,\sigma}(\brho)$.

We now define and study the important subrepresentation $\cW=\bigoplus_{\sigma\in W(\brho)}\mathcal{W}_{\sigma} \subset D_0(\brho)$ as well as its image $\cW_2=\bigoplus_{\sigma \in W(\brho)}\mathcal{W}_{2,\wt{\sigma}} \subset D_0(\brho)/D_{0}(\brho)_{\leq i_0}$.

\begin{lem}\label{lem:calW}
{Assume that $\rhobar$ is $0$-generic.} %
There exists a unique subrepresentation $\mathcal{W}\subset D_0(\brho)$ such that $\JH(\mathcal{W})=W(\brho^{\ss})$. Moreover, $\mathcal{W}$ has a direct sum decomposition $\mathcal{W}=\bigoplus_{\sigma\in W(\brho)}\mathcal{W}_{\sigma}$, where $\soc_{\GL_2(\cO_K)}(\mathcal{W}_{\sigma})=\sigma$ and $\JH(\mathcal{W}_{\sigma})=\{\tau\in W(\brho^\ss):J_{\brho}\cap J_{\tau}=J_{\sigma}\}$ {(in particular, $\mathcal{W}_{\sigma}\subset D_{0,\sigma}(\brho)$ for all $\sigma\in W(\brho)$)}. The cosocle $\wt\sigma$ of  $\mathcal{W}_{\sigma}$ is irreducible and we have $J_{\wt\sigma} = J_{\sigma}\sqcup J_{\brho}^c$. 
\end{lem}
\begin{proof}
This is a direct consequence of \cite[Lemma~\ref{bhhms4:lem:JHinW}]{BHHMS4}.
\end{proof}

For $\sigma\in W(\rhobar)$ let $\wt{\sigma}\in W(\brho^\ss)$ be the element such that $J_{\wt\sigma} = J_{\sigma}\sqcup J_{\brho}^c$, thus $\mathcal{W}_{\sigma} \cong I(\sigma,\wt{\sigma})$. Clearly, $\sigma\mapsto \wt{\sigma}$ gives a bijection between $W(\brho)$ and $\{\tau\in W(\brho^\ss):J_{\tau}\supseteq J_{\brho}^c\}$. For convenience, we write $\mathcal{W}_{\wt\sigma}\defeq \cal{W}_{\sigma}$. 

Let $\mathcal{W}_2\subset \pi_2$ (resp.\ $\mathcal{W}_{2,\wt{\sigma}}$) be the image of $\mathcal{W}$ (resp.\ $\mathcal{W}_{\wt{\sigma}}$) in $D_0(\brho)/D_{0}(\brho)_{\leq i_0} \cong \pi^{K_1}/\pi_1^{K_1} \subset \pi_2^{K_1}$.
In particular, $\JH(\cW_2) = \{ \tau\in W(\brho^\ss): \ell(\tau) \ge i_0+1 \}$.
As $\cW$ is multiplicity free, we have $\mathcal{W}_2=\bigoplus_{\sigma \in W(\brho)}\mathcal{W}_{2,\wt{\sigma}}$ (note that, contrary to $\mathcal{W}_{\wt{\sigma}}$, $\sigma \notin\JH(\mathcal{W}_{2,\wt{\sigma}})$ in general).

\begin{lem}\label{lem:W2}\ 
{Assume that $\brho$ is {$0$}-generic} %
and that $\sigma \in W(\brho)$.
  \begin{enumerate}
  \item 
  \label{it:lem:W2:1}
  We have 
    \begin{equation*}\JH(\mathcal{W}_{2,\wt{\sigma}})=\{\tau\in W(\brho^\ss): J_{\brho}\cap J_{\tau} =J_{\sigma},\; \ell(\tau)\ge i_0+1\}.\end{equation*}
    Moreover, if $\ell(\sigma)\geq i_0+1$, then $\mathcal{W}_{2,\wt{\sigma}}=\mathcal{W}_{\wt\sigma}\cong I(\sigma,\wt{\sigma})$.
  \item 
  \label{it:lem:W2:2}
    Suppose that $\tau,\tau' \in \JH(\cW_2)$ are such that $\Ext^1_{\Gamma}(\tau',\tau)\neq0$. 
    Then the nonsplit extension $\tau\-- \tau'$ occurs in $\mathcal{W}_{2}$ as a subquotient if and only if $J_{\tau'}=J_{\tau}\sqcup\{j\}$ for some $j\notin J_{\brho}$.
  \item 
   \label{it:lem:W2:3}
   We have
    \begin{equation*}\soc_\Gamma(\mathcal{W}_{2,\wt{\sigma}})=\bigoplus \big\{\tau\in W(\brho^\ss): J_{\brho}\cap J_{\tau} =J_{\sigma},\; \ell(\tau)=\max\{i_0+1,\ell(\sigma)\}\big\}.\end{equation*}
  \end{enumerate}
\end{lem}

\begin{proof}
  (i) Note that $\JH(\mathcal{W}_{2,\wt{\sigma}}) \subset \JH(\cW_2) \cap \JH(\mathcal{W}_{\wt{\sigma}})$, and equality has to hold since both sides form a partition of $\JH(\cW_2)$ as $\sigma$ varies.
  The first formula follows, and it implies the second part, as $\mathcal{W}_{2,\wt{\sigma}}$ is a quotient of $\mathcal{W}_{\wt{\sigma}} \cong I(\sigma,\wt{\sigma})$.

  (ii) 
  ``$\Leftarrow$'': %
  The nonsplit extension $\tau \-- \tau'$ is isomorphic to $I(\tau,\tau')$.
  Assuming $J_{\tau'}=J_{\tau}\sqcup\{j\}$ for some $j\notin J_{\brho}$, let $\sigma \in W(\brho)$ be determined by $J_\sigma = J_{\brho} \cap J_\tau = J_{\brho} \cap J_{\tau'}$.
  By \cite[Lemma~\ref{bhhms4:lem:JHinW}]{BHHMS4} we know that $\tau$ occurs in $I(\sigma,\tau')$ and that $\tau'$ occurs in $I(\sigma,\wt\sigma)$.
  From $I(\tau,\tau') \twoheadleftarrow I(\sigma,\tau') \into I(\sigma,\wt\sigma) \onto \mathcal{W}_{2,\wt\sigma}$ and the multiplicity freeness of $\cW_{\wt\sigma} \cong I(\sigma,\wt\sigma)$, we deduce that $\tau \-- \tau'$ occurs as subquotient of $\mathcal{W}_{2,\wt\sigma}$.

  ``$\Rightarrow$'': if $\tau \-- \tau'$ occurs as subquotient, then $\tau,\tau' \in \JH(\cW_{2,\wt\sigma})$ for some $\sigma \in W(\brho)$.
  Thus $J_\tau \mathbin\Delta J_{\tau'} \subset J_{\brho}^c$ by (i) and moreover $|J_\tau \mathbin\Delta J_{\tau'}| = 1$ by Lemma~\ref{lem:I-sigma-tau-modular}, i.e.\ $J_{\tau'}=J_{\tau}\sqcup\{j\}$ or $J_{\tau}=J_{\tau'}\sqcup\{j\}$ for some $j\notin J_{\brho}$.
  If we had $J_{\tau}=J_{\tau'}\sqcup\{j\}$, then $\tau' \-- \tau$ would occur as subquotient by ``$\Leftarrow$''.
  This contradicts the fact that, by multiplicity freeness, at most one of $\tau \-- \tau'$, $\tau' \-- \tau$ can occur.
  Hence $J_{\tau'}=J_{\tau}\sqcup\{j\}$.

  (iii) This follows from (i) and (ii).
\end{proof}

\begin{lem}\label{lem:Ext1-W2}\
Assume that $\brho$ is $1$-generic. Let $\tau'\in W(\brho^\ss)$. 
\begin{enumerate}
\item
\label{it:lem:Ext1-W2:1}
If $\sigma\in W(\brho)$, %
then the natural morphism 
\begin{equation}\label{eq:Ext1-W2}
  \Ext^1_{\Gamma}(\tau',\soc_\Gamma(\mathcal{W}_{2,\wt{\sigma}}))\ra \Ext^1_{\Gamma}(\tau',\mathcal{W}_{2,\wt{\sigma}})
\end{equation}
is surjective. %
\item
\label{it:lem:Ext1-W2:2}
If $0 \to \mathcal{W}_{2,\wt{\sigma}} \to V \to \tau' \to 0$ is a nonsplit extension of $\Gamma$-representations and $V'\subset V$, $V'\nsubseteq \mathcal{W}_{2,\wt{\sigma}}$ is any subrepresentation with cosocle $\tau'$, then $\rad_\Gamma(V')$ is semisimple and $[V':\tau'] = 1$.
\end{enumerate}
\end{lem}

\begin{proof}
Consider a nonsplit $\Gamma$-extension $0\ra \mathcal{W}_{2,\wt\sigma}\ra V\ra \tau'\ra0$, and let $V' \subset V$, $V' \not\subset \mathcal{W}_{2,\wt\sigma}$ be such that $\cosoc_\Gamma(V') \cong \tau'$.
Write $\soc_\Gamma(V') = \bigoplus_{i=1}^n \tau_i$ and let $m \defeq  \max\{i_0+1,\ell(\sigma)\}$.
(Note that, if $\ell(\sigma) \ge i_0+1$, then $n = 1$ and $\tau_1 = \sigma$ by the last statement of Lemma \ref{lem:W2}\ref{it:lem:W2:1}.)
As $V$ is nonsplit, $\soc_\Gamma(V) = \soc_\Gamma(\mathcal{W}_{2,\wt\sigma})$, so by Lemma \ref{lem:W2}\ref{it:lem:W2:3} we deduce that $\tau_i\in W(\brho^\ss)$, $\ell(\tau_i) = m$, and the $\tau_i$ are pairwise distinct.

We claim that $V'$ is multiplicity free, or equivalently that $[V':\tau'] = 1$.
If not, then $V'$ has a quotient $\o V'$ with $\soc_\Gamma(\o V') = \tau'$ and $[\o V':\tau'] = 2$, and we get a contradiction by Lemmas~\ref{lem:soc-cosoc-sigma} and \ref{lem:mu-pm} applied with $Q=\overline
V'$ and $\sigma=\tau'$, as $\JH(V) \subset W(\brho^\ss)$.

By Lemma~\ref{lem:W2}\ref{it:lem:W2:1} and \ref{it:lem:W2:3} we know that 
\begin{equation}
\label{eq:inc:radV'}
\JH(\rad_\Gamma(V')) \subset \JH(\mathcal{W}_{2,\wt\sigma}) \subset \{ \tau \in W(\brho^\ss) : \ell(\tau) \ge m \} 
\end{equation}
and also that $\tau \in \JH(\soc_\Gamma(\mathcal{W}_{2,\wt\sigma}))$ implies $\ell(\tau) = m$.
As $\Ext^1_\Gamma(\tau',\rad_\Gamma(V'))\neq 0$ we obtain by d\'evissage that $\Ext^1_\Gamma(\tau',\tau)\neq 0$ for some constituent $\tau$ of $\rad_\Gamma(V')$ and hence, by the last assertion of Lemma \ref{lem:I-sigma-tau-modular} and~\eqref{eq:inc:radV'}, we deduce that $\ell(\tau') \ge \ell(\tau)-1\geq m-1$.

We claim that $\ell(\tau') \neq m$.
As $V'$ is multiplicity free by above, $V'$ admits a unique quotient $\o V'$ such that $\soc_\Gamma(\o V') = \tau_1$ (recall that $\tau_1\subseteq \soc_\Gamma(V')$), so $\o V'\cong I(\tau_1,\tau')$ by \cite[Cor.~3.12]{BP}.
Assume by contradiction that $\ell(\tau')=\ell(\tau_1)=m$, and note that $\tau' \not\cong \tau_1$ by multiplicity freeness of $V'$.
By Lemma \ref{lem:I-sigma-tau-modular} applied to $\o V'\cong I(\tau_1,\tau')$, we deduce that $\o V'$ has a Jordan--H\"older constituent $\tau''\neq \tau'$ (e.g.~that corresponding to $J_{\tau_1}\cap J_{\tau'}\subsetneq J_{\tau_1}$) satisfying $|J_{\tau''}|<|J_{\tau_1}|=m$.
This contradicts \eqref{eq:inc:radV'}, proving the claim.

Arguing as in the previous paragraph (replacing $\tau_1$ by $\tau_i\subseteq \soc_\Gamma(V')$), we have a surjection $V'\onto I(\tau_i,\tau')$ and hence $\rad_\Gamma(V')\onto \rad_\Gamma(I(\tau_i,\tau'))$ for each $1 \le i \le n$.
As $\rad_\Gamma(V') \subset \mathcal{W}_{2,\wt\sigma}$ we conclude that $\JH(\rad_\Gamma(I(\tau_i,\tau'))) \subset \JH(\mathcal{W}_{2,\wt\sigma})$.
By Lemma~\ref{lem:I-sigma-tau-modular} applied to $I(\tau_i,\tau')$ and Lemma~\ref{lem:W2}(i) we deduce that
\begin{equation}\label{eq:J}
  \{ J : J \mathbin\Delta J_{\tau_i} \subset J_{\tau'} \mathbin\Delta J_{\tau_i} \text{\ for some $i$ and $J \ne J_{\tau'}$} \} \subset \{ J : J_\sigma \subset J \subset J_{\wt\sigma},\; |J| \ge m \}.
\end{equation}
Fix $1 \le i \le n$. 
If $J_{\tau_i} \cap J_{\tau'}\neq J_{\tau'}$, then by~\eqref{eq:J} applied to $J_{\tau_i} \cap J_{\tau'}$ we deduce that $|J_{\tau_i} \cap J_{\tau'}|\geq m=|J_{\tau_i}|$.
Hence $J_{\tau_i} \cap J_{\tau'}$ equals $J_{\tau'}$ or $J_{\tau_i}$, i.e.\ 
$J_{\tau'} \subset J_{\tau_i}$ or $J_{\tau_i} \subset J_{\tau'}$ for any $i$.

If $\ell(\tau') \in \{m-1,m+1\}$, then by above $|J_{\tau'} \mathbin\Delta J_{\tau_i}| = 1$ for all $i$, and it follows from Lemma~\ref{lem:I-sigma-tau-modular} that $I(\tau_i,\tau')$ has length 2. %
The natural map $V'\ra \bigoplus_i I(\tau_i,\tau')$ is injective, as it is injective on socles, so $\rad_\Gamma(V') \into \bigoplus_i \rad_\Gamma(I(\tau_i,\tau')) = \bigoplus_i \tau_i$.
We conclude that $\rad_\Gamma(V')$ is semisimple.
Hence the class $[V]$ of $V$, which is by construction the image of $[V']$ under $\Ext^1_{\Gamma}(\tau',\rad_\Gamma(V'))\to \Ext^1_{\Gamma}(\tau',\mathcal{W}_{2,\wt{\sigma}})$, is in fact the image of $[V']$ under the composition
\begin{equation*}
  \Ext^1_{\Gamma}(\tau',\rad_\Gamma(V'))\to \Ext^1_{\Gamma}(\tau',\soc_\Gamma(\mathcal{W}_{2,\wt{\sigma}})) \to \Ext^1_{\Gamma}(\tau',\mathcal{W}_{2,\wt{\sigma}}),
\end{equation*}
i.e.~is in the image of \eqref{eq:Ext1-W2}.

We suppose finally till the end of that proof that $\ell(\tau') > m+1$, and we will derive a contradiction (so that this case does not happen).
Note that the assumption implies $(J_\sigma \subset{}) J_{\tau_i} \subset J_{\tau'}$ for all $i$.
If there exists $j \in J_{\tau'} \setminus J_{\wt\sigma}$, then $J = J_{\tau_1} \sqcup \{j\}$ belongs to the left-hand side of~(\ref{eq:J}) (using $\ell(\tau') \ne m+1$), but (obviously) not to its right-hand side, a contradiction.
Hence $J_{\tau'} \subset J_{\wt\sigma}$, and thus $\tau' \in \JH(\mathcal{W}_{2,\wt\sigma})$ (using Lemma~\ref{lem:W2}(i)).
Let $V''$ denote the unique subrepresentation of $\mathcal{W}_{2,\wt\sigma}$ with cosocle $\tau'$.

We now show that $V' \cong V''$.
As both $V'$ and $V''$ are multiplicity free, it is enough to show that $\soc_\Gamma(V') = \soc_\Gamma(V'')$ by (the dual of) \cite[Prop.\ 3.6, Cor.\ 3.11]{BP}.
(The references imply that $\Proj_{\Gamma}\tau'$ admits a maximal multiplicity-free quotient $R$. 
As $R$ is multiplicity free, the quotients of $R$ are determined by their socles.)
If $\ell(\sigma) \ge i_0+1$, this is obvious, as $\soc_\Gamma(\mathcal{W}_{2,\wt\sigma}) \cong \sigma$ is irreducible.
If $\ell(\sigma) \le i_0$, then  by Lemma~\ref{lem:W2}(ii) and (iii) %
we have $\soc_\Gamma(V'') \cong \bigoplus_J \tau_J$, where the direct sum runs over all the $J\subset\{0,\dots,f-1\}$ such that $J_\sigma \subset J \subset J_{\tau'}$ and $|J| = i_0+1$.
(Here, $\tau_J$ is the element of $W(\brho^\ss)$ such that $J_{\tau_J} = J$, see \S~\ref{sec:notation}.)
Thus $\soc_\Gamma(V') \subset \soc_\Gamma(V'')$, since $J_{\tau_i} \subset J_{\tau'}$ for all $1 \le i \le n$.
We claim that if $\tau_J \in \JH(\soc_\Gamma(V'))$ (for some $J\subset\{0,\dots,f-1\}$ such that $J_\sigma \subset J \subset J_{\tau'}$ and $|J| = i_0+1$), then $\tau_{(J \sqcup \{j\}) \setminus \{j'\}} \in \JH(\soc_\Gamma(V'))$ for any $j \in J_{\tau'} \setminus J$ and any $j' \in J\setminus J_\sigma$.
To see this: from $\JH(I(\tau_J,\tau')) \subset \JH(V')$ we get $\tau_{J \sqcup \{j\}} \in \JH(\rad_\Gamma(V'))$ and from Lemma~\ref{lem:W2}(ii) we deduce that $\tau_{(J \sqcup \{j\}) \setminus \{j'\}} \in \JH(\soc_\Gamma(V'))$, as desired.
As $|J_\sigma| < |J| = i_0+1 < |J_{\tau'}|-1$, by iteration of the claim above we conclude that any $J\subset\{0,\dots,f-1\}$ such that $J_\sigma \subset J \subset J_{\tau'}$ and $|J| = i_0+1$ satisfies $\tau_J \in \JH(\soc_\Gamma(V'))$. %
Hence $\soc_\Gamma(V'') \subset \soc_\Gamma(V')$, so indeed $V' \cong V''$.

We next claim that $\rad_\Gamma(V') = \rad_\Gamma(V'')$ is indecomposable.
We already know that $\rad_\Gamma(V')$, $\rad_\Gamma(V'')$ are isomorphic subrepresentations of $\mathcal{W}_{2,\wt\sigma}$, and $\mathcal{W}_{2,\wt\sigma}$ is multiplicity free, so $\rad_\Gamma(V') = \rad_\Gamma(V'')$.
The indecomposability is obvious if $\ell(\sigma) \ge i_0+1$, as $\soc_\Gamma(V')=\soc_\Gamma(V'')$ is then irreducible, so suppose $\ell(\sigma) \le i_0$.
Following the argument of the previous paragraph, we know that the uniserial representations of the form $\tau_J \-- \tau_{J \sqcup \{j\}}$ and $\tau_{(J \sqcup \{j\}) \setminus \{j'\}} \-- \tau_{J \sqcup \{j\}}$ occur as subquotients of $\cW_{2,\wt\sigma}$ by Lemma~\ref{lem:W2}(ii) and hence of $\rad_\Gamma(V')$.
This shows by the same iteration as in the preceding paragraph that all constituents of $\soc_\Gamma(V')$ lie in the same indecomposable component of $\rad_\Gamma(V')$.
Therefore $\rad_\Gamma(V')$ is indecomposable.

By the preceding two paragraphs, we can pick an isomorphism $f : V' \congto V''$.
By indecomposability of $\rad_\Gamma(V')$, we may rescale $f$ so that $f|_{\rad_\Gamma(V')}$ is the identity on $\rad_\Gamma(V') = \rad_\Gamma(V'')$.
This means that $V'$ and $V''$ define the same class in $\Ext^1_\Gamma(\tau',\rad_\Gamma(V'))$, up to scalar, so some linear combination splits, implying $\tau' \in \JH(\soc_\Gamma(V))$.
Since $\soc_\Gamma(\mathcal{W}_{2,\wt\sigma}) = \soc_\Gamma(V)$ by definition of $V$, this contradicts that $\ell(\tau') > m+1$ (if $\tau \in \JH(\soc_\Gamma(\mathcal{W}_{2,\wt\sigma}))$, then $\ell(\tau) = m$).
\end{proof}

\begin{prop}\label{prop:pi2K1}
{Assume that $\brho$ is $\max\{9,2f+3\}$-generic.} 
Let $i_0=i_0(\pi_1)$ with $-1\leq i_0\leq f$ be as in \cite[Thm.~\ref{bhhms4:thm:conj2}]{BHHMS4} and $\pi_2\defeq \pi/\pi_1$. 
Then $D_0(\brho^\ss)_{i_0+1}$ injects into $\pi_2|_{\GL_2(\cO_K)}$.
\end{prop}

\begin{proof}
Since $D_0(\brho^\ss)_{i_0+1}=\bigoplus_{\tau\in W(\brho^\ss),\ell(\tau)=i_0+1}D_{0,\tau}(\brho^\ss)$ and is multiplicity free, it suffices to prove that $D_{0,\tau}(\brho^\ss)$ injects into $\pi_2|_{\GL_2(\cO_K)}$ for any $\tau\in W(\brho^\ss)$ with $\ell(\tau)=i_0+1$. Let $\lambda\in \D^\ss$ be the element corresponding to $\tau$.
As in Step 2 of the proof of \cite[Thm.~\ref{bhhms4:thm:conj2}]{BHHMS4} we define %
\[J_1\defeq \{j\in J_{\brho}^c:\lambda_j(x_j)=p-3-x_j\}, \quad \wt{J}_1\defeq \{j:\lambda_j(x_j)\in\{x_j+1,p-2-x_j\}\},
\] 
the element $\mu\in\P$ by $\mu_j(x_j)=p-1-x_j$ if $j\in J_1$ and $\mu_j(x_j)=\lambda_j(x_j)$ otherwise, and the character $\chi''$ by
\[\chi''\defeq \chi_{\mu}\prod_{j\in J_1\sqcup \wt{J}_1}\alpha_j^{-1}.\]
We then have $W(\chi_{\mu},\chi'')\hookrightarrow \pi|_{I}$, hence a $\GL_2(\cO_K)$-equivariant morphism as in Step 4 of the proof of \cite[Thm.~\ref{bhhms4:thm:conj2}]{BHHMS4}:
\[\widetilde{\kappa}:\Ind_I^{\GL_2(\cO_K)}W(\chi_{\mu}^s,\chi''^s)\ra \pi|_{\GL_2(\cO_K)}.\]
Let $\sigma_1\in W(\brho)$ and $\tau_1=\delta(\tau)\in W(\brho^\ss)$ be as in Step 4 of the proof of \cite[Thm.~\ref{bhhms4:thm:conj2}]{BHHMS4}, so that in particular $\im(\wt{\kappa})$ has socle $\sigma_1$ and $I(\sigma_1,\tau_1)$ embeds into $\im(\wt\kappa)^{K_1}$. By \cite[Lemma~\ref{bhhms4:lem:JHinW}]{BHHMS4}, for any $\tau'\in \JH(I(\sigma_1,\tau_1))$ with $\tau'\neq \tau_1$, we have $\tau'\in W(\brho^\ss)$ with $\ell(\tau')<\ell(\tau_1)=i_0+1 $, hence $\rad_\Gamma(I(\sigma_1,\tau_1))\subset \pi_1$ and $I(\sigma_1,\tau_1)\not\subset \pi_1$ by \cite[Thm.~\ref{bhhms4:thm:conj2}]{BHHMS4}
 and \cite[eq.~\eqref{bhhms4:eq:JH-D0-leqi}]{BHHMS4}. 

Consider the composite morphism \[\Ind_I^{\GL_2(\cO_K)}W(\chi_{\mu}^s,\chi''^s)\ra \pi|_{\GL_2(\cO_K)}\onto \pi_2|_{\GL_2(\cO_K)};\]
we claim that it factors through 
\begin{equation}\label{eq:Ind-pi2}\Ind_I^{\GL_2(\cO_K)}W(\chi_{\lambda}^s,\chi''^s)\ra \pi_2|_{\GL_2(\cO_K)}.\end{equation}
It suffices to prove that the image of $W(\chi_{\mu}^s,\chi''^s)\hookrightarrow \pi|_I\twoheadrightarrow
\pi_2|_I$ has socle $\chi_{\lambda}^s$ or equivalently that the image of $W(\chi_{\mu},\chi'')\hookrightarrow
\pi|_I\twoheadrightarrow \pi_2|_I$ has socle $\chi_{\lambda}$. 
This follows from Proposition~\ref{prop:pi2-I1-invt} and the following two facts:
\begin{enumerate}
\item[(a)] under the morphism $W(\chi_{\mu},\chi'')\hookrightarrow \pi|_I$, $\rad_I(W(\chi_{\mu},\chi_{\lambda}))$ is sent into $\pi_1$, as any constituent is of the form $\chi_\nu$ with $\nu \in \P^\ss$, $\ell(\nu) < i_0+1$ (this follows from \cite[Lemma~\ref{bhhms4:lem:Wchi}]{BHHMS4}, the recipe~\cite[eq.~\eqref{bhhms4:eq:mu_j}]{BHHMS4}, and since $\ell(\lambda) = i_0+1$);
\item[(b)] by the discussion after \cite[eq.~\eqref{bhhms4:eq:V/I}]{BHHMS4} we have in particular that
\[ \big(\JH(W(\chi_{\lambda},\chi''))\setminus \{\chi_{\lambda}\}\}\big) \cap \JH(D_0(\brho^\ss)^{I_1})=\emptyset.\] 
 \end{enumerate}
We note by \cite[Prop.\ 4.2]{breuil-buzzati} that fact (b) is equivalent to
\begin{equation}\label{eq:inter-W-rhob-ss}
  \JH\bigg(\Ind_I^{\GL_2(\cO_K)}W(\chi_{\lambda}^s,\chi''^s)/\Ind_I^{\GL_2(\cO_K)} \chi_\lambda^s\bigg)\cap W(\brho^\ss)=\emptyset.
\end{equation}

Let $V$ be the image of \eqref{eq:Ind-pi2} and $V_{\lambda}\subset V$ be the image of $\Ind_I^{\GL_2(\cO_K)}\chi_{\lambda}^s$ in $\pi_2$, so that $V/V_\lambda$ is a quotient of $\Ind_I^{\GL_2(\cO_K)}W(\chi_{\lambda}^s,\chi''^s)/\Ind_I^{\GL_2(\cO_K)} \chi_\lambda^s$ and hence $\JH(V/V_\lambda) \cap W(\brho^\ss)=\emptyset$ by~\eqref{eq:inter-W-rhob-ss}. 
From the exact sequence $0\ra \soc_{\GL_2(\cO_K)}(V_\lambda)\ra\soc_{\GL_2(\cO_K)}(V)\ra\soc_{\GL_2(\cO_K)}(V/V_\lambda)$, as
\[\JH(\soc_{\GL_2(\cO_K)}(V)) \subset \JH(\soc_{\GL_2(\cO_K)}(\pi_2)) \subset W(\brho^\ss)\]
(by Corollary~\ref{cor:subquot-K-soc}) and $\JH(\soc_{\GL_2(\cO_K)}(V/V_\lambda)) \cap W(\brho^\ss)=\emptyset$, we deduce that the natural map $\soc_{\GL_2(\cO_K)}(V)\ra\soc_{\GL_2(\cO_K)}(V/V_\lambda)$ is zero, i.e.\ $\soc_{\GL_2(\cO_K)}(V_\lambda) = \soc_{\GL_2(\cO_K)}(V)$.

We claim that $\soc_{\GL_2(\cO_K)}(V_\lambda) = \soc_{\GL_2(\cO_K)}(V) = \tau_1$.
By the first paragraph, the representation $\tau_1$ injects into $\soc_{\GL_2(\cO_K)}(V)$, hence into  $\soc_{\GL_2(\cO_K)}(V_\lambda)$.
Conversely, if $\tau_2 \subset \soc_{\GL_2(\cO_K)}(V_\lambda)$, then we obtain surjections $\Ind_I^{\GL_2(\cO_K)} \chi_\lambda^s \onto V_\lambda \onto I(\tau_2,\tau^{[s]})$
and the final representation surjects onto $I(\delta(\tau),\tau^{[s]})$ by \cite[Lemma~12.8(ii)]{BP} and \cite[Lemma~15.2]{BP} (with $\mathcal S^- = \mathcal S^+ = \emptyset$ here).
As $\delta(\tau) = \tau_1$ occurs in $\soc_{\GL_2(\cO_K)}(V_\lambda)$, by multiplicity freeness of $V_\lambda$ we deduce that $\tau_2 = \tau_1$.

As in Step 4 of the proof of \cite[Thm.~\ref{bhhms4:thm:conj2}]{BHHMS4},  \cite[Lemma~\ref{bhhms4:lem:BP-gene}]{BHHMS4} and the preceding paragraph imply that $V$ contains $D_{0,\delta(\tau)}(\brho^{\ss})$, hence $D_{0,\delta(\tau)}(\brho^{\ss})$ injects into $\pi_2^{K_1}$. As $\ell(\delta(\tau))=\ell(\tau)$ and $\delta(\cdot)$ is periodic, we deduce that $D_{0,\tau}(\brho^{\ss})$ also injects into $\pi_2^{K_1}$, as desired.
\end{proof}

We define the $\Gamma$-representation $D_{i_0} \defeq  D_0(\brho^\ss)_{i_0+1} \oplus_{D_0(\brho)_{i_0+1}} (D_0(\brho)/D_0(\brho)_{\le i_0})$.

\begin{lem}\label{lem:Ext-Di0}\ 
{Assume that $\brho$ is $\max\{9,2f+3\}$-generic.} 
\begin{enumerate}
\item 
\label{it:lem:Ext-Di0:1}
$D_{i_0}$ is multiplicity free.
\item $D_{i_0}$ injects into $\pi_2^{K_1}$.
\item 
\label{it:lem:Ext-Di0:3}
We have $D_{i_0}^{I_1} \cong \pi_2^{I_1}$ and $\soc_{\GL_2(\cO_K)}(D_{i_0}) \cong \soc_{\GL_2(\cO_K)}(\pi_2)$.
Both representations are multiplicity free.
In particular, $\JH(\soc_{\GL_2(\cO_K)}(D_{i_0}))\subset W(\brho^\ss)$.
\item 
\label{it:lem:Ext-Di0:4}
We have $\mathcal{W}_2 \subset D_{i_0}$ and $\JH(D_{i_0}/\mathcal{W}_2) \cap W(\brho^\ss) = \emptyset$.
\item Let $\tau'$ be a Serre weight. If  $\Ext^1_{\Gamma}(\tau',D_{i_0})\neq 0$, then either $\tau'\in W(\brho^\ss)$ %
or \[\tau'\in \bigcup_{j\geq i_0+2}\JH(D_0(\brho^\ss)_j)\backslash \JH(D_0(\brho)_j).\]
\end{enumerate}
\end{lem}
\begin{proof}
(i) By construction, $D_{i_0}$ is a successive extension of the form
\begin{equation}\label{eq:filt-Di0}D_{0}(\brho^{\rm ss})_{i_0+1} \-- D_0(\brho)_{i_0+2}\-- \cdots\-- D_0(\brho)_{f}.\end{equation}
More precisely, it inherits from \cite[eq.~\eqref{bhhms4:eq:fil-D0}]{BHHMS4} a filtration with graded pieces $D_{0}(\brho^{\rm ss})_{i_0+1}$, $D_0(\brho)_{i_0+2}$, \dots, $D_0(\brho)_{f}$.
Since $D_{0}(\brho)_j$ injects into $D_0(\brho^{\ss})_j$ for all $j$ and $D_0(\brho^{\ss})$ is multiplicity free, $D_{i_0}$ is also multiplicity free.  

(ii) Clearly, $\pi^{K_1}/\pi_1^{K_1} \cong D_0(\brho)/D_0(\brho)_{\leq i_0}$ injects into $\pi_2^{K_1}$.
Recall that
\[\Sigma\defeq \soc_{\Gamma} (D_0(\brho^{\ss})_{i_0+1}) =\soc_{\Gamma} (D_0(\brho)_{i_0+1})\]
and that $\JH(D_0(\brho^{\ss})_{i_0+1}/\Sigma) \cap W(\brho^\ss) = \emptyset$.
Hence the restriction maps
\begin{equation*}
  \Hom_{\GL_2(\cO_K)}(D_0(\brho^{\ss})_{i_0+1},\pi_2) \to \Hom_{\GL_2(\cO_K)}(D_0(\brho)_{i_0+1},\pi_2) \to \Hom_{\GL_2(\cO_K)}(\Sigma,\pi_2).
\end{equation*}
are injective using Corollary \ref{cor:subquot-K-soc} (applied to $\pi_2$), and therefore bijective by Proposition \ref{prop:pi2K1} (for each $\tau \in W(\brho^\ss)$ with $\ell(\tau)=i_0+1$ we have an injection $D_{0,\tau}(\brho^\ss) \into \pi_2$).
Thus any injection $f : D_0(\brho)/D_0(\brho)_{\leq i_0} \into \pi_2^{K_1}$ can be extended to a map $\wt f: D_{i_0} \to \pi_2$.
From the short exact sequence
\[0 \to D_0(\brho)/D_0(\brho)_{\le i_0} \to D_{i_0} \to D_0(\brho^{\ss})_{i_0+1}/D_0(\brho)_{i_0+1} \to 0,\]
together with \[\soc_\Gamma(D_{i_0})\subset W(\brho^\ss)\] (which follows from \eqref{eq:filt-Di0}) and $W(\brho^\ss) \cap \JH( D_0(\brho^\ss)_{i_0+1}/D_0(\brho)_{i_0+1} )=\emptyset$ (which follows from \cite[eq.~\eqref{bhhms4:eq:soc-i}]{BHHMS4}), we deduce that $\soc_{\Gamma}(D_{i_0})=\soc_{\Gamma}(D_0(\brho)/D_0(\brho)_{\le i_0})$.
We conclude that $\wt f$ is also injective and (ii) follows.

(iii) 
We claim that the inclusion $D_{i_0}^{I_1} \subset \pi_2^{I_1}$, deduced from (ii), is in fact an equality.
Indeed, \cite[Cor.~\ref{bhhms4:cor:conj1}]{BHHMS4} and~\cite[Lemma~\ref{bhhms4:lem:I1-invt-component}]{BHHMS4} show respectively that 
\[
\{ \chi_\lambda : \text{$\lambda \in \P$, $|J_\lambda| \geq  i_0+1$}\}\subseteq \JH\Big(\big(D_0(\brho)/D_0(\brho)_{\le i_0}\big)^{I_1}\Big)
\] and 
\[
\{\chi_\lambda : \text{$\lambda \in \P^{\ss}$, $|J_\lambda| = i_0+1$}\}\subseteq \JH\Big(D_0(\brho^\ss)_{i_0+1}^{I_1}\Big).
\]
By the definition of $D_{i_0}$ we obtain inclusions
\begin{equation*}
  \{ \chi_\lambda : \text{$\lambda \in \P, |J_\lambda| > i_0$ or $\lambda \in \P^\ss \setminus \P, |J_\lambda| = i_0+1$} \} \subset \JH(D_{i_0}^{I_1}) \subset \JH(\pi_2^{I_1}),
\end{equation*}
so that equality holds by Proposition~\ref{prop:pi2-I1-invt}.

The equality of $\soc_{\Gamma}(D_{i_0})$ and $\soc_{\GL_2(\cO_K)}(\pi_2)$ follows from the chain of inclusions
\[
\soc_{\Gamma}(D_0(\brho)/D_0(\brho)_{\le i_0})\subseteq \soc_{\Gamma}(D_{i_0})\subseteq \soc_{\Gamma}(\pi_2^{K_1})=\soc_{\Gamma}(D_0(\brho)/D_0(\brho)_{\le i_0}),
\]
where the first inclusion follows from the construction of $D_{i_0}$, the second from part (ii) and the last equality follows from Remark~\ref{rem:subquot-K-soc}.

(iv) By definition, $\mathcal{W}_2 \subset D_0(\brho)/D_0(\brho)_{\le i_0} \subset D_{i_0}$.
We have an exact sequence
\begin{equation*}
  0 \to (D_0(\brho)/D_0(\brho)_{\le i_0})/\cW_2 \to D_{i_0}/\mathcal{W}_2 \to D_0(\brho^{\ss})_{i_0+1}/D_0(\brho)_{i_0+1} \to 0
\end{equation*}
and a surjection $D_0(\brho)/\cW \onto (D_0(\brho)/D_0(\brho)_{\le i_0})/\cW_2$.
As no constituents of $D_0(\brho^{\ss})_{i_0+1}/D_0(\brho)_{i_0+1}$ and $D_0(\brho)/\cW$ lie in $W(\brho^\ss)$ (see the proof of (ii) in the former case), we deduce the final claim.

(v) If $\Ext^1_{\Gamma}(\tau',D_{i_0}) \ne 0$, then $\Ext^1_{\Gamma}(\tau',D) \ne 0$ for one of the graded pieces appearing in \eqref{eq:filt-Di0}.
  If $\Ext^1_{\Gamma}(\tau',D_0(\brho)_j)\neq0$ for some $j\geq i_0+2$, then using the exact sequence $0\ra D_0(\brho)_j\ra D_0(\brho^\ss)_{j}\ra R_j\ra 0$, where $R_j$ is the corresponding quotient, we see that either $\Ext^1_{\Gamma}(\tau',D_0(\brho^\ss)_j)\neq0$ or $\Hom_{\Gamma}(\tau',R_j)\neq0$.
  In the latter case, $\tau' \in \JH(D_0(\brho^\ss)_j)\backslash \JH(D_0(\brho)_j)$ by multiplicity freeness.

Therefore, $\Ext^1_{\Gamma}(\tau',D_{i_0}) \ne 0$ implies $\tau' \in \bigcup_{j\geq i_0+2}\JH(D_0(\brho^\ss)_j)\backslash \JH(D_0(\brho)_j)$ or for some $j \ge i_0+1$ we have $\Ext^1_{\Gamma}(\tau',D_0(\brho^\ss)_j)\neq0$.
  The latter implies $\tau'\in W(\brho^\ss)$ by \cite[Lemmas 2.25, 2.26]{HuWang} (and recalling $D_0(\brho^{\ss})_j=\bigoplus_{\tau\in W(\brho^\ss),\ell(\tau)=j}D_{0,\tau}(\brho^\ss)$).
\end{proof}

The following result strengthens Lemma \ref{lem:Ext-Di0}(v).
 
\begin{cor}\label{cor:Ext-Di0}
{Assume that $\brho$ is $\max\{9,2f+3\}$-generic.}
If $\Ext^1_{\Gamma}(\tau',D_{i_0})\neq0$ for some Serre weight $\tau'$, then $\tau'\in W(\brho^\ss)$.
\end{cor}
\begin{proof}
By Lemma \ref{lem:Ext-Di0}(v), it suffices to show $\Ext^1_{\Gamma}(\tau',D_{i_0})=0$ if $\tau'\in \JH(D_{0}(\brho^\ss)_j)\backslash \JH(D_0(\brho)_j)$ for some $j\geq i_0+2$. 
Fix such a Serre weight $\tau'$, so $\tau' \notin \JH(D_{i_0})$ (and $\tau' \notin \JH(D_0(\rhobar))$). By contradiction let $0\ra D_{i_0}\ra V\ra \tau'\ra0$ be a nonsplit extension of $\Gamma$-representations and $V'\subset V$, $V'\nsubseteq D_{i_0}$ any subrepresentation with cosocle $\tau'$.  %
Say $\tau' \in \JH(D_{0,\tau''}(\brho^\ss))$, for $\tau'' \in W(\brho^\ss)$, $\ell(\tau'') = j > i_0+1$.
Note that $\tau' \not\cong \tau''$, as $\tau' \notin W(\brho^\ss)$ by~\cite[eq.~\eqref{bhhms4:eq:soc-i}]{BHHMS4}.
Pick any $\kappa \in \JH(\soc_{\Gamma}(V'))$, %
so $\kappa \in W(\brho^\ss)$ by Lemma \ref{lem:Ext-Di0}(iii) and $V' \onto I(\kappa,\tau')$, so $\tau'' \in \JH(I(\kappa,\tau')) \subset \JH(V')$ by \cite[Lemma 12.8(ii)]{BP}.
As $\tau' \not\cong \tau''$, $\tau''$ even occurs in $\rad_{\Gamma}(V') \subset D_{i_0}$, so $\tau''$ occurs in $\cW_2 \subset D_{i_0}$ by Lemma~\ref{lem:Ext-Di0}(iv).
Define $\sigma \in W(\brho)$ by $J_\sigma = J_{\brho} \cap J_{\tau''}$, so $\tau'' \in \JH(\cW_{2,\wt\sigma})$ (where $\cW_2$ and $\cW_{2,\wt\sigma}$ are defined just before Lemma~\ref{lem:W2}).
We claim that $\tau' \in \JH(\Inj_\Gamma \sigma)$, which gives a contradiction by~\cite[eq.~\eqref{bhhms4:eq:JH-D0-i}]{BHHMS4} since $\JH(\Inj_\Gamma \sigma) \subset \JH(D_0(\brho))$ (cf.\ \cite[eq.~\eqref{bhhms4:eq:JH-D0-rho}]{BHHMS4}). 

If $\ell(\sigma) \ge i_0+1$, then by the last assertion of Lemma~\ref{lem:W2}(i) the socle of the unique subrepresentation of $\cW_2$ with cosocle $\tau''$ is $\sigma$, so $\sigma \into V'$ and hence $V' \onto I(\sigma,\tau')$, which implies $\tau' \in \JH(\Inj_\Gamma \sigma)$, as claimed.

If $\ell(\sigma)\leq i_0$,  the socle of the unique subrepresentation of $\cW_2$ with cosocle $\tau''$ is the direct sum of all $\tau_J \in W(\brho^\ss)$ with $\ell(\tau_J) = i_0+1$, $J_\sigma \subset J \subset J_{\tau''} (\subset J_\sigma \sqcup J_{\brho}^c)$, by Lemma \ref{lem:W2}, and each such $\tau_J$ must inject into $V'$.
(Again, $\tau_J$ is the element of $W(\brho^\ss)$ such that $J_{\tau_J} = J$.)
We deduce as before that $\tau' \in \JH(\Inj_\Gamma \tau_J)$ for each such $\tau_J$, equivalently  $\tau_J\in\JH(\Inj_\Gamma\tau')$ by \cite[Lemma~3.2]{BP}.  
To prove the claim it suffices to prove the following two statements:
\begin{enumerate}
\item[(a)] for any two subsets $J,J'\subset \{0,\dots,f-1\}$, if $\tau_J,\tau_{J'}\in \JH(\Inj_\Gamma\tau')$, then $\tau_{J\cap J'}\in \JH(\Inj_\Gamma\tau')$;
\item[(b)]  $\bigcap_J J=J_\sigma$, where $J$ runs over all elements in $X\defeq\{J:  J_\sigma\subset J\subset J_{\tau''}, |J|=i_0+1\}$. 
\end{enumerate}
For (a), we first observe that by \cite[Lemma~12.6]{BP} $\tau_J,\tau_{J'}$ are automatically compatible in the sense that their corresponding elements in $\cI$ are compatible (relative to $\tau'$). Thus Lemma \ref{lem:comp-JH} implies that $\JH(I(\tau_{J},\tau_{J'}))\subset \JH(\Inj_\Gamma\tau')$, in particular $\tau_{J\cap J'}\in \JH(\Inj_\Gamma\tau')$ using Lemma \ref{lem:I-sigma-tau-princ-series}. For (b), we note  that $\ell(\sigma)\leq i_0$ and $\ell(\tau'')\geq i_0+2$, so $J_\sigma\subsetneq J\subsetneq J_{\tau''}$ for any $J\in X$. Fix $J\in X$ and $j'\in J_{\tau''}\setminus J$. Then $J_j\defeq (J\sqcup \{j'\})\setminus\{j\}\in X$  for any  $j\in J\setminus J_\sigma$. It is direct to check that $J_\sigma=J\cap(\bigcap_j J_j)$, from which (b) follows.    
\end{proof}

For $\sigma \in W(\brho)$ and $0 \le i \le f$ let us define for convenience {the $\Gamma$-representations}
\begin{gather*}
  D_{0,\sigma}(\brho)_{(i)} \defeq  \frac{D_{0,\sigma}(\brho)_{\le i}}{D_{0,\sigma}(\brho)_{\le i-1}} \cong \bigoplus_{\tau \in W(\brho^\ss), \ell(\tau)=i, J_\sigma = J_{\brho} \cap J_\tau} D_{0,\tau}(\brho)_i\\
  \noalign{\noindent (using \cite[eq.~\eqref{bhhms4:eq:D0-sigma-i-decomp}]{BHHMS4}) and}
  D_{0,\sigma}(\brho^\ss)_{(i)} \defeq  \bigoplus_{\tau \in W(\brho^\ss), \ell(\tau)=i, J_\sigma = J_{\brho} \cap J_\tau} D_{0,\tau}(\brho^\ss).
\end{gather*}
(Note that $D_{0,\sigma}(\brho^\ss)_{(i)}$ depends on $\brho$, not just on $\brho^\ss$!)
Hence
\begin{equation}\label{eq:D0-rho-i}
  D_0(\brho)_i = \bigoplus_{\sigma \in W(\brho)} D_{0,\sigma}(\brho)_{(i)} \qquad\text{and}\qquad D_0(\brho^\ss)_i = \bigoplus_{\sigma \in W(\brho)} D_{0,\sigma}(\brho^\ss)_{(i)}.
\end{equation}
(In the second case, note that $W(\brho^\ss) = \coprod_{\sigma \in W(\brho)} \{\tau \in W(\brho^\ss) : J_{\brho} \cap J_\tau = J_\sigma\}$.)
Note that the injection $D_0(\brho)_i\into D_0(\brho^\ss)_i$ (cf.~above \cite[eq.~\eqref{bhhms4:eq:soc-i}]{BHHMS4}) respects the direct sum decompositions \eqref{eq:D0-rho-i}, 
as\ $D_{0,\tau}(\brho)_{i} = D_0(\brho)_i \cap D_{0,\tau}(\brho^\ss)$ (cf.~above \cite[eq.~\eqref{bhhms4:eq:D0-sigma-i-decomp}]{BHHMS4}).

\begin{lem}\label{lem:D=A+B}\ 
{Assume that $\brho$ is $\max\{9,2f+3\}$-generic.}
\begin{enumerate}
\item 
\label{it:D=A+B:1}
There is a direct sum decomposition $D_{i_0}=\bigoplus_{\sigma \in W(\brho)} D_{i_0,\sigma}$, where
  \[D_{i_0,\sigma} \defeq  D_{0,\sigma}(\brho^\ss)_{(i_0+1)} \oplus_{D_{0,\sigma}(\brho)_{(i_0+1)}} \big(D_{0,\sigma}(\brho)/D_{0,\sigma}(\brho)_{\le i_0}\big) \qquad\text{for $\sigma \in W(\brho)$}.\]
Moreover,
$D_{i_0,\sigma} = D_{0,\sigma}(\brho)$ if $\ell(\sigma) \ge i_0+1$.
\item 
Fix $\sigma \in W(\brho)$.
We have a natural injection $\mathcal{W}_{2,\wt\sigma} \into D_{i_0,\sigma}$ and $\JH(D_{i_0,\sigma}/\mathcal{W}_{2,\wt\sigma}) \cap W(\brho^\ss) = \emptyset$.
Moreover, 
\begin{equation}\label{eq:soc-D-i0-sigma}
  \soc_{\Gamma}(D_{i_0,\sigma}) \cong \bigoplus_{\tau \in W(\brho^\ss), \ell(\tau)=\max\{i_0+1,\ell(\sigma)\}, J_\sigma = J_{\brho} \cap J_\tau} \tau.
\end{equation}
\item  
Fix $\sigma \in W(\brho)$ and let $\tau'\in W(\brho^\ss)$. Suppose that $0 \to D_{i_0,\sigma} \to V \to \tau' \to 0$ is a nonsplit extension of $\Gamma$-representations and $V'\subset V$, $V'\nsubseteq D_{i_0,\sigma}$ is any subrepresentation with cosocle $\tau'$. %
If $[V':\tau'] = 1$, then $\rad_{\Gamma}(V')$ is semisimple and contained in $\mathcal{W}_{2,\wt\sigma} \subset D_{i_0,\sigma}$.
If $[V':\tau'] = 2$, then $\tau' \cong \sigma$ and $\ell(\sigma) \ge i_0+1$.

\end{enumerate}
\end{lem}

\begin{proof}
(i) Recall that $D_0(\brho) = \bigoplus_{\sigma \in W(\brho)} D_{0,\sigma}(\brho)$, compatibly with filtrations.
As $D_0(\brho)_{i_0+1}\into D_0(\brho^\ss)_{i_0+1}$, and the injection respects the direct sum decompositions \eqref{eq:D0-rho-i}, the direct sum decomposition claimed in (i) follows.

We now prove the last claim of (i).
If $\ell(\sigma) \ge i_0+1$, then $D_{0,\sigma}(\brho)_{\le i_0} = 0$, and $D_{0,\sigma}(\brho)_{(i_0+1)} = D_{0,\sigma}(\brho^\ss)_{(i_0+1)} = 0$ if $\ell(\sigma) > i_0+1$.
If $\ell(\sigma) = i_0+1$, then it follows from the definitions of $D_{0,\sigma}(\brho)_{\le i_0+1}$ and of $D_{0}(\brho)_{\le i_0+1}$ (cf.\ \cite[\S~\ref{bhhms4:sec:k_1-invar-subr}]{BHHMS4}), and from the inclusion $D_{0,\sigma}(\brho^\ss)\subset D_{0,\sigma}(\brho)$ (cf.~the introduction to \cite[\S~\ref{bhhms4:sec:finite-length-nonss}]{BHHMS4}) that $D_{0,\sigma}(\brho^\ss) \subset D_{0,\sigma}(\brho)_{\le i_0+1}$, so $D_{0,\sigma}(\brho)_{(i_0+1)} = D_{0,\sigma}(\brho^\ss)_{(i_0+1)} (= D_{0,\sigma}(\brho^\ss))$.

(ii) 
Consider the diagram {of $\Gamma$-representations}
\begin{equation*}
  \xymatrix{\cW \ar@{^{(}->}[r]\ar@{->>}[d] & D_0(\brho)\ar@{->>}[d] \\
  \cW_2 \ar@{^{(}->}[r] & D_0(\brho)/D_0(\brho)_{\le i_0} \ar@{^{(}->}[r] & D_{i_0}}
\end{equation*}
Each term in this diagram has a natural direct sum decomposition indexed by $\sigma \in W(\brho)$.
The top horizontal arrow preserves the decompositions because $D_0(\brho)$ is multiplicity free and as $\JH(\cW_\sigma) \subset \JH(D_{0,\sigma}(\brho))$ for all $\sigma \in W(\brho)$ (Lemma \ref{lem:calW}), which implies that the bottom horizontal map preserves the decompositions.
The property $\JH(D_{i_0,\sigma}/\mathcal{W}_{2,\wt\sigma}) \cap W(\brho^\ss) = \emptyset$ then follows from Lemma~\ref{lem:Ext-Di0}\ref{it:lem:Ext-Di0:4}.
By Lemma~\ref{lem:Ext-Di0}\ref{it:lem:Ext-Di0:3} and \ref{it:lem:Ext-Di0:4} we know that $\soc_{\Gamma}(\cW_2) = \soc_{\Gamma}(D_{i_0})$, and hence $\soc_{\Gamma}(\cW_{2,\wt\sigma}) = \soc_{\Gamma}(D_{i_0,\sigma})$ for each $\sigma$.
Formula~\eqref{eq:soc-D-i0-sigma} follows from Lemma~\ref{lem:W2}\ref{it:lem:W2:3}.

(iii)  
As $V$ is a nonsplit extension, we have $\soc_{\Gamma}(V) = \soc_{\Gamma}(D_{i_0,\sigma})$ hence $\JH(\soc_{\Gamma}(V)) \subset \JH(\soc_{\Gamma}(D_{i_0})) \subset W(\brho^\ss)$ by Lemma \ref{lem:Ext-Di0}\ref{it:lem:Ext-Di0:3}.
Since $D_{i_0,\sigma}$ is multiplicity free by Lemma \ref{lem:Ext-Di0}\ref{it:lem:Ext-Di0:1}, for each $\tau\in \JH(\soc_{\Gamma}(V'))$ there is a unique largest quotient $V'_{\tau}$ of $V'$ with socle $\tau$ (and cosocle $\tau'$).

We first show that
\begin{align}
\label{eq:implications}
  [\soc_{\Gamma}(V'):\tau'] = 0 \ \Leftrightarrow\ [V':\tau'] = 1 \ &\Rightarrow\  \rad_{\Gamma}(V') \subset \mathcal{W}_{2,\wt\sigma}\\
  &\Rightarrow\  \text{$\rad_{\Gamma}(V')$ is semisimple}.\nonumber
\end{align}
If $[\soc_{\Gamma}(V'):\tau'] = 0$ then $[V':\tau]=1$ for each $\tau \in \JH(\soc_{\Gamma}(V'))$ as $D_{i_0,\sigma}$ is multiplicity free, hence $V'_\tau \cong I(\tau,\tau')$ by \cite[Cor.~3.12]{BP}.
As $V'$ injects into $\bigoplus_{\tau\in\JH(\soc_{\Gamma}(V'))} V'_{\tau}$, we deduce that $\rad_{\Gamma}(V')$ injects into $\bigoplus_{\tau\in\JH(\soc_{\Gamma}(V'))} \rad_{\Gamma}(V'_{\tau})$, so $[\rad_{\Gamma}(V'):\tau'] = 0$ as $[\rad_{\Gamma}(V'_\tau):\tau'] = 0$ for all $\tau$, i.e.\ $[V':\tau'] = 1$.
The converse of the first implication in \eqref{eq:implications} is obvious.
Still assuming $[\soc_{\Gamma}(V'):\tau'] = 0$, by Lemma~\ref{lem:I-sigma-tau-modular} we obtain $\JH(V_{\tau}')\subset W(\brho^\ss)$ for all $\tau \in \JH(\soc_{\Gamma}(V'))$, thus $\JH(V')\subset W(\brho^\ss)$.
This implies that $\rad_{\Gamma}(V') \subset \mathcal{W}_{2,\wt\sigma}$ by Lemma~\ref{lem:Ext-Di0}(iv), so $\rad_{\Gamma}(V')$ is semisimple by Lemma \ref{lem:Ext1-W2}\ref{it:lem:Ext1-W2:2} (applied to the pushout of $0 \to \rad_{\Gamma}(V') \to V' \to \tau' \to 0$ along the injection $\rad_{\Gamma}(V') \subset \mathcal{W}_{2,\wt\sigma}$, which is still nonsplit, as it is contained in $V$ which is nonsplit by assumption). %

To conclude, it suffices to show that $\tau' \not\cong \sigma$ or $\ell(\sigma) \le i_0$ imply $[\soc_{\Gamma}(V'):\tau'] = 0$, hence $[V':\tau'] = 1$.
If $\tau' \not\cong \sigma$ and $\ell(\sigma) \ge i_0+1$, then $D_{i_0,\sigma} = D_{0,\sigma}(\brho)$ by part (i), so $\soc_{\Gamma}(V') = \sigma$ and we are done.
If $\ell(\sigma) \le i_0$, assume by contradiction that $\tau' \subset \soc_{\Gamma}(V')$.
Then $\tau' \subset \soc_{\Gamma}(V') \subset D_{i_0,\sigma}$, so $\tau' \subset \soc_{\Gamma}(\cW_{2,\wt\sigma})$ by part (ii), so $\tau' \notin W(\brho)$ and $\ell(\tau')=i_0+1$ by Lemma~\ref{lem:W2}\ref{it:lem:W2:3}. 
By \cite[Cor.\ 2.32]{HuWang2}, $V'_{\tau'}$ has $3$ socle layers, of the shape
\[\tau' \-- \soc_1(V'_{\tau'}) \-- \tau',\]
and $\soc_1(V'_{\tau'})$ contains at least one element of $W(\brho^\ss)$ by Lemmas~\ref{lem:soc-cosoc-sigma} and~\ref{lem:mu-pm}, say $\tau''$. 
Since $\tau',\tau''\in W(\brho^\ss)$ and $\tau'\--\tau''$ is a subquotient of $D_{i_0,\sigma}$ (hence of $\mathcal{W}_{2,\wt\sigma}$ by part (ii) and Lemma \ref{lem:Ext-Di0}\ref{it:lem:Ext-Di0:4}), we have $J_{\tau''}=J_{\tau'}\sqcup\{j''\}$ for some $j''\notin J_{\brho}$ by Lemma~\ref{lem:W2}(ii).
On the other hand, as $\tau'\notin W(\brho)$ there exists $j'\in J_{\tau'}\setminus J_{\brho}$ (so $j' \ne j''$).
Let $\tau\in W(\brho^\ss)$ be the element corresponding to $J_{\tau''}\setminus\{j'\}$.
Then $\ell(\tau)=i_0+1$ and the extension $\tau\-- \tau''$ occurs in $V'$ as a subquotient by Lemma~\ref{lem:W2}(ii). 
By part (ii) and Lemma \ref{lem:W2}\ref{it:lem:W2:3} we have $\tau\subset \soc_{\Gamma}(\mathcal{W}_{2,\wt{\sigma}}) \subset \soc_{\Gamma}(D_{i_0,\sigma})$.
Thus $\tau$ occurs in $\soc_{\Gamma}(V')$ (using that $D_{i_0,\sigma}$ containing $\rad_{\Gamma}(V')$ is multiplicity free).
As $\tau \ne \tau'$ we have $[V':\tau] = 1$ so $\tau$ occurs in $\soc_{\Gamma}(V')$. %
As $\tau \ne \tau'$, $V'$ has quotient $V_{\tau}'$ which is isomorphic to $I(\tau,\tau')$ (arguing as in the case $[\soc_{\Gamma}(V'):\tau'] = 0$). 
By Lemma \ref{lem:I-sigma-tau-modular} and as $J_{\tau} = (J_{\tau'} \sqcup \{j''\})\setminus \{j'\}$, it has length $4$ and a constituent $\tau''' \in W(\brho^\ss)$, with $J_{\tau'''}=J_{\tau}\setminus \{j''\}$ in $\soc_1(I(\tau,\tau'))$. 
But this implies that the nonsplit extension $\tau\-- \tau'''$ occurs in $D_{i_0,\sigma}$ (hence in $\mathcal{W}_{2,\wt\sigma}$), which is impossible by Lemma~\ref{lem:W2}(ii).
Hence $[\soc_{\Gamma}(V'):\tau'] = 0$, as desired.
\end{proof}

\begin{proof}[Proof of Theorem~\ref{thm:pi2-K1}]
If $i_0 = f$ this is trivial (both sides are zero).
If $i_0 = f-1$, then as in Remark \ref{rem:i0-f-1-or-f} we have $\pi_2 \cong \Ind_{B(K)}^{\GL_2(K)} (\chi_1 \otimes \chi_2\omega^{-1})$, where $\o\rho \cong \smatr{\chi_1} * 0 {\chi_2}$.
Theorem~\ref{thm:pi2-K1} boils down to showing $\Ind_I^K (\chi_1 \otimes \chi_2\omega^{-1}) \cong D_0(\brho^\ss)_f$.
This is true by \cite[Rk.~14.9(i)]{BP}.  
From now on we will assume that $i_0 \le f-2$. Furthermore, if $i_0=-1$ the result is just assumption \ref{it:assum-i} from Section~\ref{subsec:assumptions}, so we will assume $i_0\geq 0$ and so $f\geq 2$. 

Recall that $D_{i_0} \subset \pi_2^{K_1}$ and that both representations have the same $I_1$-invariants and the same $\GL_2(\cO_K)$-socle by Lemma~\ref{lem:Ext-Di0}(ii), (iii).

For a contradiction, assume $D_{i_0}\subsetneq \pi_2^{K_1}$ with $Q$ being the quotient. Pick a Serre weight $\tau'$ which injects into $Q$. Then we obtain an extension of $\Gamma$-representations
\[0\ra D_{i_0}\ra V\ra \tau'\ra0,\]
which is nonsplit by the above discussion, and moreover $V \subset \pi_2^{K_1}$. In other words, $V$ defines a nonzero element in $\Ext^1_{\Gamma}(\tau',D_{i_0})$. Hence, by Corollary \ref{cor:Ext-Di0}, we have $\tau'\in W(\brho^\ss)$. %

For $\sigma \in W(\brho)$ let $V_\sigma$ denote the quotient of $V$ defined as the pushout of $V \hookleftarrow D_{i_0} \onto D_{i_0,\sigma}$, so that $V \into \bigoplus_{\sigma \in W(\brho)} V_\sigma$ (recall $D_{i_0} = \bigoplus_{\sigma \in W(\brho)} D_{i_0,\sigma}$ from Lemma~\ref{lem:D=A+B}(i)).
(We caution the reader that, despite the notation, $\sigma \notin \JH(V_\sigma)$ in general.)
If there exists $\sigma' \in W(\brho)$ such that $\tau' \in \JH(D_{i_0,\sigma'})$ (there can be at most one such $\sigma'$ by Lemma~\ref{lem:Ext-Di0}(i)) and $[V_{\sigma'}] = 0$ in $\Ext^1_\Gamma(\tau',D_{i_0,\sigma'})$, then 
choose any splitting $s : \tau' \into V_{\sigma'}$ and let $V'$ be the image of any morphism $\Proj_\Gamma \tau' \to V$ whose composition with $V \onto V_{\sigma'}$ is the map $\Proj_\Gamma \tau' \onto \tau' \xrightarrow{s} V_{\sigma'}$.
Otherwise, let $V' \subset V$, $V' \not\subset D_{i_0}$ denote any subrepresentation with cosocle $\tau'$.
In either case, $V' \subset V$, $V' \not\subset D_{i_0}$ and $\cosoc_{\Gamma}(V') \cong \tau'$.
Moreover, $\rad_{\Gamma}(V') \subset \rad_{\Gamma}(V) \subset D_{i_0}$ is multiplicity free and
\begin{equation}\label{eq:soc-V'}
  \JH(\soc_{\Gamma}(V')) \subset \JH(\soc_{\Gamma}(D_{i_0})) \subset W(\brho^\ss)
\end{equation}
by Lemma~\ref{lem:Ext-Di0}(i) and (iii).

For any $\sigma \in W(\brho)$ let $V'_\sigma$ denote the image of $V' \into V \onto V_\sigma$, so $\cosoc_{\Gamma}(V'_\sigma) \cong \tau'$ and $V'_\sigma \not\subset D_{i_0,\sigma}$.
We show that 
\begin{equation}\label{eq:claim-V'-sigma}
  \text{$[V_\sigma] = 0$ in $\Ext^1_\Gamma(\tau',D_{i_0,\sigma})$ if and only if $V'_\sigma \cong \tau'$.}
\end{equation}
If $[V_\sigma] = 0$ and $\tau' \in \JH(D_{i_0,\sigma})$, then $V'_{\sigma} \cong \tau'$ by construction in the preceding paragraph.
If $[V_\sigma] = 0$ and $\tau' \notin \JH(D_{i_0,\sigma})$, then $V'_{\sigma} \cong \tau'$ (the unique subrepresentation of $D_{i_0,\sigma} \oplus \tau'$ with cosocle $\tau'$).
Conversely, the ``if'' direction of~\eqref{eq:claim-V'-sigma} is true, as $V'_\sigma \not\subset D_{i_0,\sigma}$ provides the splitting.

For later reference we show that
\begin{equation}\label{eq:rad-V'}
  \rad_{\Gamma}(V') \cong \bigoplus_{\sigma \in W(\brho)} \rad_{\Gamma}(V'_\sigma).
\end{equation}
By construction, $V' \into \bigoplus_{\sigma \in W(\brho)} V'_\sigma$, so $\rad_{\Gamma}(V') \into \bigoplus_{\sigma \in W(\brho)} \rad_{\Gamma}(V'_\sigma)$.
As $V'$ surjects onto $V'_\sigma$, we deduce $\rad_{\Gamma}(V')$ surjects onto $\rad_{\Gamma}(V'_\sigma)$ for all $\sigma$.
Since the $\rad_{\Gamma}(V'_\sigma) \subset D_{i_0,\sigma}$ for $\sigma \in W(\brho)$ have disjoint constituents, it follows that the injection $\rad_{\Gamma}(V') \into \bigoplus_{\sigma \in W(\brho)} \rad_{\Gamma}(V'_\sigma)$ is an isomorphism.

We now distinguish cases.

\textbf{Step 1.} Assume $[V':\tau'] = 1$. %
We will show that $\chi_{\tau'}$ contributes \emph{twice} to $\pi_2[\m^2]$, but this contradicts Corollary~\ref{cor:pi2-mn-mult-free}. 

We claim that $\rad_\Gamma(V')$ is semisimple.
By~\eqref{eq:rad-V'} it suffices to show that $\rad_\Gamma(V'_\sigma)$ is semisimple for any $\sigma \in W(\brho)$.
If $[V_{\sigma}] = 0$, then $V'_\sigma \cong \tau'$ by~\eqref{eq:claim-V'-sigma} and we are done.
If $[V_\sigma] \ne 0$ we deduce that $\rad_\Gamma(V'_\sigma)$ is semisimple by Lemma~\ref{lem:D=A+B}(iii) (using the assumption $[V':\tau'] = 1$).
This establishes the claim.

In the following three paragraphs we show that for any $\Gamma$-subrepresentation $V'' \subset \pi_2^{K_1}$ such that $\cosoc_\Gamma(V'') \cong \tau'$ and $\rad_\Gamma(V'')$ is semisimple there exist $\cJ\subset \{0,\dots,f-1\}$ and a map $\wt f : Q_\cJ=Q_\cJ(\tau') \to \pi_2^{K_1}$ %
such that $\Theta_\cJ=\Theta_\cJ(\tau')$ surjects onto $V''$, where $Q_\cJ(\tau')$ (resp.\ $\Theta_\cJ(\tau')$) was defined just before Lemma~\ref{lem:rep-QJ} (resp.\ Lemma~\ref{lem:Q-vs-Theta}).

Note that $\tau \in \JH(\rad_\Gamma(V''))$ implies $\Ext_\Gamma^1(\tau',\tau) \ne 0$ and $\tau \in W(\brho^\ss)$ (by~(\ref{eq:soc-V'}), as $\rad_\Gamma(V'') \subset \soc_\Gamma(V'')$).
\ \ For \ \ $* \in \{\pm\}$ \ \ define \ \ $\cJ^* \defeq  \{ 0 \le i \le f-1 : \mu_i^*(\tau') \into V'' \}$, \ \ so \ \ that \ $\rad_\Gamma(V'') \cong \bigoplus_{i \in \cJ^+} \mu_i^+(\tau') \oplus \bigoplus_{i \in \cJ^-} \mu_i^-(\tau')$ by \cite[Lemma~\ref{bhhms4:lem:ext1}]{BHHMS4}.
We note that $\cJ^+ \cap \cJ^- = \emptyset$ by %
Lemma~\ref{lem:mu-pm}.

Suppose that $i \in \cJ^+$ (or more generally that $\mu_i^+(\tau') \into \pi_2|_{\GL_2(\cO_K)}$).
We claim that the nonsplit extension $\mu_i^+(\tau') \-- \delta_i^+(\tau')$ embeds into $D_{i_0}$.
By Corollary~\ref{cor:subquot-K-soc} (applied to $\pi_2$) we have two cases.
If $\mu_i^+(\tau') \in W(\brho)$ and $\ell(\mu_i^+(\tau')) \ge i_0+1$, then $\mu_i^+(\tau') \-- \delta_i^+(\tau')$ embeds into $D_{0,\mu_i^+(\tau')}(\brho)$ (by Lemma~\ref{lem:mu-pm} and the definition of $D_{0,\mu_i^+(\tau')}(\brho)$), so into $D_{i_0}$.
If $\mu_i^+(\tau') \in W(\brho^\ss)\setminus W(\brho)$ and $\ell(\mu_i^+(\tau')) = i_0+1$, then $\mu_i^+(\tau') \-- \delta_i^+(\tau')$ similarly embeds into $D_{0,\mu_i^+(\tau')}(\brho^\ss) \subset D_0(\brho^\ss)_{i_0+1} \subset D_{i_0}$.
In either case we deduce from Lemma~\ref{lem:mu-i+}  %
(note $\mu_i^+(\tau') = \mu_i^-(\delta_i^+(\tau'))$) that
\begin{equation}\label{eq:chi-delta-tau-pi2}
  \chi_{\delta_i^+(\tau')} \in \JH(\pi_2^{I_1}) \quad\forall\ i \in \cJ^+,
\end{equation}
and that
\begin{equation}\label{eq:chi-delta-tau-pi}
  \chi_{\delta_i^+(\tau')} \in \JH(\pi^{I_1}) \text{\quad if $i \in \cJ^+$ and $\mu_i^+(\tau') \in W(\brho)$}.
\end{equation}

Let $\cJ \defeq  \cJ^+ \sqcup \cJ^-$.
Now we show that there exists a map $\wt f : Q_\cJ=Q_\cJ(\tau') \to \pi_2^{K_1}$ such that $\Theta_\cJ=\Theta_\cJ(\tau')$ surjects onto $V''$. 
Clearly there exists a surjection $f : \Theta_{\cJ}/\soc_{\wt\Gamma}(\Theta_{\cJ}) \onto V''$, which is unique up to scalar (both sides are multiplicity free, with cosocle $\tau'$, and we know all their constituents).
We show that $f : \Theta_\cJ \onto \Theta_{\cJ}/\soc_{\wt\Gamma}(\Theta_{\cJ})\onto V'' \subset \pi_2^{K_1}$ can be extended to a map $Q_\cJ \to \pi_2^{K_1}$.
By Lemma~\ref{lem:Q-vs-Theta} (with $\Psi_i = \Psi_i(\tau')$) it suffices to show that the map $f|_{\rad_{\wt\Gamma}(\Psi_i)}$ extends to $\Psi_i$ for all $i \in \cJ$.
(The extension is automatically unique, as $\delta_i^+(\tau') \not\into \pi_2|_{\GL_2(\cO_K)}$ by Lemma~\ref{lem:mu-pm}.)
Note that $f|_{\soc_{\wt\Gamma}(\Psi_i)} = 0$, as $\soc_{\wt\Gamma}(\Psi_i) \subset \rad_{\wt\Gamma}(\Psi_i) \subset \Theta_\cJ$ and $f(\soc_{\wt\Gamma}(\Theta_\cJ)) = 0$.
If $i \in \cJ^-$, $f|_{\rad_{\wt\Gamma}(\Psi_i)} = 0$, as $\mu_i^+(\tau') \not\into \pi_2|_{\GL_2(\cO_K)}$ by Corollary~\ref{cor:subquot-K-soc} and Lemma~\ref{lem:mu-pm}, so the extension to $\Psi_i$ is trivial.
If $i \in \cJ^+$, $f|_{\rad_{\wt\Gamma}(\Psi_i)}$ factors through an injection of $\mu_i^+(\tau')$ into $\pi_2^{K_1}$ and by above this extends to an injection of $\mu_i^+(\tau') \-- \delta_i^+(\tau')$ into $\pi_2^{K_1}$.

By Lemma~\ref{lem:rep-QJ} the map $\wt f$ gives rise to a homomorphism $\Ind_I^{\GL_2(\cO_K)} W_\cJ \onto Q_\cJ \to \pi_2^{K_1}$, and by Frobenius reciprocity we get a map $\o f : W_\cJ \to \pi_2^{K_1}|_I$. {Since $\o{f}$ factors through $W_{\cJ}\onto (W_{\cJ})_{K_1}$ which is killed by $\m^2$ by Lemma \ref{lem:rep-WJ}(ii), we get $\o{f}(W_{\cJ})\subset \pi_2[\m^2]$. In particular, as $\cosoc_I(W_{\cJ})\cong\chi_{\tau'}$, we see that $\chi_{\tau'}$ contributes to $\pi_2[\m^2]$.} %

We now specialize to $V'' \defeq  V'$, in which case $\o f(W_\cJ) \not\subset D_{i_0}$. If $\o f(W_\cJ) \subset \pi_2[\m] = \pi_2^{I_1}$, then $\o f(W_\cJ) \subset D_{i_0}^{I_1} \subset D_{i_0}$ by the first statement of Lemma~\ref{lem:Ext-Di0}(iii), contradiction. Since $\o f(W_\cJ)$ is $K_1$-invariant, by Lemma \ref{lem:rep-WJ}(ii)  we deduce that $\rad_I(\o f(W_\cJ)) \subset \bigoplus_{i \in \cJ} \chi_{\tau'} \alpha_i$. If $\chi_{\tau'} \alpha_i \into \o f(W_\cJ) \subset \pi_2$, then Frobenius reciprocity gives us a nonzero map $\Ind_I^{\GL_2(\cO_K)} \chi_{\tau'}\alpha_i \to \im(\wt f) \subset \pi_2$, and hence $[\im(\wt f) : \delta_i^+(\tau')] \ne 0$.
By construction, $[\im(\wt f) : \delta_i^+(\tau')] = 0$ for all $i \in \cJ^-$, so $\rad_I(\o f(W_\cJ)) \subset \bigoplus_{i \in \cJ^+} \chi_{\tau'} \alpha_i$. (In fact, we have $\rad_I(\o f(W_\cJ)) = \bigoplus_{i \in \cJ^+} \chi_{\tau'} \alpha_i$, but we will not need this.)

Choose now $i \in \cJ^+$ such that $\chi_{\tau'} \alpha_i \subset \o f(W_\cJ)$.
Let $\lambda' \in \D^\ss$ correspond to $\tau' \in W(\brho^\ss)$, and for $i \in \cJ^+$ we let $\lambda^{(i)} \in \D^\ss$ correspond to $\mu_i^+(\tau') \in W(\brho^\ss)$ and $\lambda'^{(i)} \in \P^\ss$ correspond to $\delta_i^+(\tau')$.
(See \S~\ref{sec:notation} for $\D^\ss$ and $\P^\ss$.)
More precisely, note that 
\begin{gather*}
  (\lambda'_{i},\lambda^{(i)}_{i},\lambda'^{(i)}_{i}) \in \Big\{ (x_{i},x_{i}+1,x_{i}+2), (p-3-x_{i},p-2-x_{i},p-1-x_{i}) \Big\},\\
  \lambda'_{i-1} = p-2-\lambda^{(i)}_{i-1} = \lambda'^{(i)}_{i-1}, \qquad \text{and}\\
  \lambda'_j = \lambda^{(i)}_j = \lambda'^{(i)}_j \qquad \forall\ j \notin \{i-1,i\}.  
\end{gather*}
In particular, it follows that
\begin{equation}\label{eq:lengths-lambda}
  J_{\lambda^{(i)}} = J_{\lambda'^{(i)}} = J_{\lambda'}\mathbin\Delta \{i\}, \qquad \ell(\lambda^{(i)}) = \ell(\lambda'^{(i)}) = \ell(\lambda')\pm 1.
\end{equation}

Recall that the $I$-representations $\Theta = \bigoplus_{\mu \in \P} \Theta_\mu$ and $\o\tau = \bigoplus_{\mu \in \P} \o\tau_\mu$ were defined just before Lemma~\ref{lem:submodule-structure} (taking $n = i_0+4 {(\le f+2)}$ and {noting that by assumption $\brho$ is $(2n-1)$-generic}).
Recall that in Step~2 of the proof of Theorem \ref{thm:gr-pi2} we showed that the natural map \eqref{eq:grad2:surj} is an isomorphism. %
Equivalently, $\Theta[\m^3] = \pi_2[\m^3]$, and hence $\pi_2[\m^2] = \bigoplus_{\mu \in \P} \Theta_\mu[\m^2]$.
Moreover, using the exact sequence $0 \ra \o\tau_\mu \cap \pi_1 \ra \o\tau_\mu \ra \Theta_\mu\ra 0$ and the dual of \eqref{eq:compare-fil} (recalling that $F_{\mu,-i} \Theta_\mu^\vee = \m^i \o\tau_\mu^\vee \cap \Theta_\mu^\vee$, cf.\ (\ref{eq:filtrations})), we see that
\begin{equation*}
  \Theta_\mu[\m^i] = \frac{\o\tau_\mu[\m^{i+d_\mu}]}{(\o\tau_\mu \cap \pi_1)[\m^{i+d_\mu}]}
\end{equation*}
for all $i\geq 0$ (recall that $d_\mu = \max\{i_0+1-\ell(\mu),0\}$).
By ~(\ref{eq:chi-delta-tau-pi2}) we have $\chi_{\tau'} \alpha_i = \chi_{\lambda'^{(i)}} \in \JH(\pi_2^{I_1}) = \JH(\pi_2[\m]) = \JH(\Theta[\m])$, and moreover it occurs in $\o\tau_{\mu^{(i)}}$, where $\mu^{(i)} \in \P$ is obtained from $\lambda'^{(i)}$ as $\mu$ is obtained from $\lambda$ in~\cite[eq.~\eqref{bhhms4:eq:mu_j}]{BHHMS4}.
As the $I$-representation $\o\tau$ is multiplicity free (\cite[Cor.~\ref{bhhms4:cor:tau-multfree}(ii)]{BHHMS4}) we deduce that $\chi_{\tau'} \alpha_i$ occurs in $\Theta_{\mu^{(i)}}[\m]$.
Since the nonsplit extension $\chi_{\tau'}\alpha_i \-- \chi_{\tau'}$ is a quotient of $\o f(W_\cJ) \subset \pi_2[\m^2] = \Theta[\m^2]$, it follows again by multiplicity freeness of $\o\tau$ that $\chi_{\tau'}$ occurs in the direct summand $\Theta_{\mu^{(i)}}[\m^2]$ of $\Theta[\m^2]$ as well.
Dually, $\chi_{\tau'}^{-1}$ occurs in $\gr_{\m,-1}(\Theta_{\mu^{(i)}}^\vee)$.

Define $J_1^{(i)}$, $J_2^{(i)}$ exactly as in~\cite[eq.~\eqref{bhhms4:eq:J1:J2}]{BHHMS4} (with $\lambda'^{(i)}$ instead of $\lambda$) and let
\begin{align*}
  \wt J_1^{(i)} &\defeq  \{j \notin J_{\brho} : \mu^{(i)}_j = p-1-x_j \} = \{j \notin J_{\brho} : \lambda'^{(i)}_j \in \{ p-3-x_j,p-1-x_j \}\}, \\
  \wt J_2^{(i)} &\defeq  \{j \notin J_{\brho} : \mu^{(i)}_j = x_j \} = \{j \notin J_{\brho} : \lambda'^{(i)}_j \in \{ x_j,x_j+2 \}\},
\end{align*}
so $J_1^{(i)} \subset \wt J_1^{(i)}$, $J_2^{(i)} \subset \wt J_2^{(i)}$, $\wt J_1^{(i)} \cap \wt J_2^{(i)} = \emptyset$.
Note that $\ell(\mu^{(i)}) = \ell(\lambda'^{(i)})-|J_1^{(i)}|-|J_2^{(i)}| = \ell(\lambda^{(i)})-|J_1^{(i)}|-|J_2^{(i)}|$, where the first equality was noted just after~\cite[eq.~\eqref{bhhms4:eq:mu_j}]{BHHMS4}.
We claim that
\begin{equation}\label{eq:d-mu-i}
  d_{\mu^{(i)}} = |J_1^{(i)}|+|J_2^{(i)}|.
\end{equation}
If $\mu_i^+(\tau') \notin W(\brho)$, then $\ell(\lambda^{(i)}) = \ell(\mu_i^+(\tau')) = i_0+1$.
If $\mu_i^+(\tau') \in W(\brho)$ and $\ell(\lambda^{(i)}) \ge i_0+1$, then by above $\lambda'^{(i)} \in \P$ (as $\chi_{\delta_i^+(\tau')} \in \JH(\pi^{I_1})$ by~(\ref{eq:chi-delta-tau-pi})), so $\mu^{(i)} = \lambda'^{(i)}$ and $J_1^{(i)} = J_2^{(i)} = \emptyset$.
The claim follows.

As noted just after~\cite[eq.~\eqref{bhhms4:eq:mu_j}]{BHHMS4} we have $\chi_{\lambda'^{(i)}} = \chi_{\mu^{(i)}} \prod_{j \in J_1^{(i)}} \alpha_j^{-1} \prod_{j \in J_2^{(i)}} \alpha_j$, or equivalently (using $\chi_{\lambda'} \alpha_i = \chi_{\lambda'^{(i)}}$):
\begin{equation}\label{eq:chi-lambda'}
  \chi_{\lambda'}^{-1} = \chi_{\mu^{(i)}}^{-1} \alpha_i \prod_{j \in J_1^{(i)}} \alpha_j \prod_{j \in J_2^{(i)}} \alpha_j^{-1}.
\end{equation}
Recall from~\eqref{eq:N_2-lambda'} and \eqref{eq:compare-fil} that $\gr_{\m,-1}(\Theta_{\mu^{(i)}}^\vee) \cong \chi_{\mu^{(i)}}^{-1} \otimes \gr_{-1-d_{\mu^{(i)}}}(\mathfrak{a}_1^{i_0}(\mu^{(i)})/\mathfrak{a}(\mu^{(i)}))$ {(using $-1-d_{\mu^{(i)}} > -n$)}, and that
\begin{equation}\label{eq:a1-mod-a}
  \frac{\mathfrak{a}_1^{i_0}(\mu^{(i)})}{\mathfrak{a}(\mu^{(i)})} \cong \frac{I(\wt J_1^{(i)},\wt J_2^{(i)},d_{\mu^{(i)}})}{I(\wt J_1^{(i)},\wt J_2^{(i)},d_{\mu^{(i)}}) \cap \mathfrak{a}(\mu^{(i)})}
\end{equation}
by \ \cite[eq.~\eqref{bhhms4:eq:fa-1}]{BHHMS4}.
\ As \ $\chi_{\lambda'}^{-1}$ \ occurs \ in \ $\gr_{\m,-1}(\Theta_{\mu^{(i)}}^\vee)$, \ we \ deduce \ that \ there \ exists \ a \ monomial \ $m \in I(\wt J_1^{(i)},\wt J_2^{(i)},d_{\mu^{(i)}})$, $m \notin \mathfrak{a}(\mu^{(i)})$ of degree $d_{\mu^{(i)}}+1$ that has $H$-eigencharacter $\chi_{\lambda'}^{-1}\chi_{\mu^{(i)}}$.
By~(\ref{eq:chi-lambda'}) the monomial 
\begin{equation*}
m' \defeq 
\begin{cases}
  y_{i} \prod_{j \in J_1^{(i)}} y_j \prod_{j \in J_2^{(i)}} z_j & \text{if $i \notin J_2^{(i)}$},\\
  \prod_{j \in J_1^{(i)}} y_j \prod_{j \in J_2^{(i)} \setminus \{i\}} z_j & \text{if $i \in J_2^{(i)}$}
\end{cases}
\end{equation*}
has the same $H$-eigencharacter.
Since in addition $m$ and $m'$ have degree at most 2 in each variable (by \cite[Def.~\ref{bhhms4:def:IJideals}]{BHHMS4}
 in the first case) and are not multiples of $y_jz_j$ for any $j$ (as $m \notin \mathfrak{a}(\mu^{(i)})$ in the first case), we deduce using $p-1 > 4$ that $m = m'$.
Hence $m'$ has degree $d_{\mu^{(i)}}+1$. 
Using (\ref{eq:d-mu-i}) it follows that $i \notin J_2^{(i)}$ and $m = y_{i} \prod_{j \in J_1^{(i)}} y_j \prod_{j \in J_2^{(i)}} z_j$.
If $i \in J_{\brho}$, then $\mu^{(i)}_{i} = \lambda'^{(i)}_{i} \in \{x_{i}+2,p-1-x_{i}\}$ 
(by \cite[eq.~\eqref{bhhms4:eq:J1:J2}, eq.~\eqref{bhhms4:eq:mu_j}]{BHHMS4}) and hence $t_{i} = y_{i}$ (where $t_i = t_i(\mu^{(i)})$ is defined in \eqref{eq:id:al}), so $m$ is a multiple of $t_i \in \mathfrak{a}(\mu^{(i)})$, contradiction.
Hence $i \in J_{\brho}^c$, and we deduce moreover that $\lambda'^{(i)}_{i} \ne x_{i}+2$ from $i \notin J_2^{(i)}$ and \cite[eq.~\eqref{bhhms4:eq:J1:J2}]{BHHMS4}.

We are left with the case where $\lambda'^{(i)}_{i} = p-1-x_{i}$ and $i \in J_{\brho}^c$.
In this case %
$J_{\tau'} = J_{\mu_i^+(\tau')} \sqcup \{i\}$ %
by~(\ref{eq:lengths-lambda}) since $\lambda'_i=p-3-x_i$, hence $\tau' \notin W(\brho)$ (as $i \in J_{\brho}^c$).
Moreover, as $\mu_i^+(\tau') \into \pi_2|_{\GL_2(\cO_K)}$, we have $\ell(\tau') = \ell(\mu_i^+(\tau'))+1 > i_0+1$, so $\tau' \in \JH(\cW_2)$, so $\tau' \in \JH(\cW_{2,\wt\sigma})$, where $\sigma \in W(\brho)$ is determined by $J_\sigma = J_{\brho} \cap J_{\tau'}$ (Lemma~\ref{lem:W2}(i)).
We claim that the unique subrepresentation $W'$ of $\cW_{2,\wt\sigma} \subset \cW_2 \subset D_{i_0}$ having cosocle $\tau'$ has semisimple radical.
We have two cases by Corollary~\ref{cor:subquot-K-soc}.
If $\mu_i^+(\tau') \in W(\brho)$ and $\ell(\mu_i^+(\tau')) \ge i_0+1$, then $\sigma \cong \mu_i^+(\tau')$ and $W' \cong (\mu_i^+(\tau') \-- \tau')$ by Lemma \ref{lem:W2}\ref{it:lem:W2:1}, \ref{it:lem:W2:3}.
If $\mu_i^+(\tau') \in W(\brho^\ss)$ and $\ell(\mu_i^+(\tau')) = i_0+1$, then $\ell(\sigma) = |J_{\brho} \cap J_{\tau'}| \le |J_{\mu_i^+(\tau')}| = i_0+1$ (using $i \in J_{\brho}^c$) and $\ell(\tau') = i_0+2$, so $\rad_\Gamma(W')$ is semisimple by Lemma \ref{lem:W2}\ref{it:lem:W2:2}, \ref{it:lem:W2:3}, proving the claim.
By the third paragraph of Step~1 (applied with $V'' = V'$, resp.\ $W'$), we obtain two morphisms $Q_\cS \to \pi_2^{K_1}$ that are linearly independent, as $V' \not\subset D_{i_0}$ and $W' \subset D_{i_0}$.
{(Note that the subset $\cJ \subset \cS$ may differ in the two cases, but $Q_\cS$ surjects onto any $Q_\cJ$, and likewise $W_\cS$ surjects onto any $W_\cJ$.)}
Since we showed that the maps $\o f: W_\cS \to \pi_2$ corresponding to $\Ind_I^{\GL_2(\cO_K)} W_\cS \onto Q_\cS \to \pi_2^{K_1}$ have images contained in $\pi_2[\m^2]$, we deduce that $\chi_{\tau'}$ contributes \emph{twice} to $\pi_2[\m^2]$, as we wanted to show. %

\textbf{Step 2.} Assume $[V':\tau'] = 2$. %
We will show that $\chi_{\tau'}$ contributes \emph{twice} to $\pi_2[\m^3]$, but this contradicts Corollary~\ref{cor:pi2-mn-mult-free}.
 (Note that $2n+2i_0+1 \le 6+2(f-2)+1 = 2f+3$ when $n = 3$ and as we assumed $i_0 \le f-2$.)

From~\eqref{eq:rad-V'} and since $\rad_{\Gamma}(V') \subset \rad_{\Gamma}(V) \subset D_{i_0}$ is multiplicity free we deduce that $\tau' \in \JH(\rad_{\Gamma}(V'_\sigma))$ for a unique $\sigma \in W(\brho)$.
By \eqref{eq:claim-V'-sigma} and the last statement of Lemma~\ref{lem:D=A+B}(iii) (applied to $V'_\sigma \subset V_\sigma$) it follows that $\tau' \cong \sigma \in W(\brho)$ and $\ell(\tau') \ge i_0+1$, hence $[V'_{\tau'}:\tau'] = 2$.
Using $\rad_{\Gamma}(V'_{\tau'}) \subset D_{i_0,\tau'}$ and $\ell(\tau') \ge i_0+1$ it follows moreover from Lemma \ref{lem:D=A+B}\ref{it:D=A+B:1} that $\soc_{\Gamma}(V'_{\tau'}) = \tau'$, and hence $\tau' \into \rad_{\Gamma}(V')$ by~\eqref{eq:rad-V'}.
The natural surjection $\rad_{\Gamma}(V') \onto \rad_{\Gamma}(V'/\tau')$ induces an isomorphism $\rad_{\Gamma}(V')/\tau' \congto \rad_{\Gamma}(V'/\tau')$.
We apply again~\eqref{eq:rad-V'} to deduce that
\begin{equation}\label{eq:rad-V'/tau'}
  \rad_{\Gamma}(V')/\tau' \cong \bigoplus_{\sigma \ne \tau'} \rad_{\Gamma}(V'_\sigma) \oplus (\rad_{\Gamma}(V'_{\tau'})/\tau').
\end{equation}
As in the second paragraph of Step~1, $\rad_{\Gamma}(V'_\sigma)$ is semisimple for all $\sigma \ne \tau'$, and $\rad_{\Gamma}(V'_{\tau'})/\tau'$ is semisimple by Lemma~\ref{lem:soc-cosoc-sigma}.
In conclusion, $\rad_{\Gamma}(V'/\tau')$ is semisimple. %
As $\cosoc_{\Gamma}(V'/\tau') \cong\tau'$ we can write, as in the fourth paragraph of Step~1,
\begin{equation}\label{eq:mu-pm-J}
  \bigoplus_{\sigma \ne \tau'} \rad_{\Gamma}(V'_\sigma) \cong \bigoplus_{i \in \cJ^+} \mu_i^+(\tau') \oplus \bigoplus_{i \in \cJ^-} \mu_i^-(\tau')
\end{equation}
and (using moreover Lemma~\ref{lem:soc-cosoc-sigma}),
\begin{equation}\label{eq:mu-pm-J'}
  \rad_{\Gamma}(V'_{\tau'})/\tau' \cong \bigoplus_{i \in \cJ'} (\mu_i^+(\tau') \oplus \mu_i^-(\tau'))
\end{equation}
for some subsets $\cJ^+$, $\cJ^-$, $\cJ'$ of $\{0,1,\dots,f-1\}$.
For $i \in \cJ^*$, $* \in \{\pm\}$ we have $\mu_i^*(\tau') \into V' \into \pi_2|_{\GL_2(\cO_K)}$ by~\eqref{eq:rad-V'}.
Therefore $\cJ^+ \cap \cJ^- = \emptyset$ by Corollary~\ref{cor:subquot-K-soc} and Lemma~\ref{lem:mu-pm}, and note that $(\cJ^+ \sqcup \cJ^-) \cap \cJ' = \emptyset$ as $\rad_{\Gamma}(V') \subset D_{i_0}$ is multiplicity free.
Let $\cJ \defeq  \cJ' \sqcup \cJ^+ \sqcup \cJ^-$, and write $Q_\cJ=Q_\cJ(\tau')$, $\Theta_\cJ=\Theta_\cJ(\tau')$, $\Psi_i = \Psi_i(\tau')$ as in Step~1.

We show that $\dim_\F \Hom_{\GL_2(\cO_K)}(\Theta_{\cJ},\pi_2|_{\GL_2(\cO_K)}) \ge 2$. %
It suffices to show that \[\dim_\F \Hom_{\wt\Gamma}(\Theta_{\cJ},V') \geq 2.\] 
As we have a trivial map $\Theta_{\cJ} \onto \tau' \into V'$, it suffices to find a map whose image is not irreducible.
We follow the argument in \cite[Cor.\ 3.14]{HuWang2}.
By~\eqref{eq:rad-V'/tau'}, \eqref{eq:mu-pm-J}, and \eqref{eq:mu-pm-J'} there exists a surjection $f : \Theta_{\cJ}/\soc_{\wt\Gamma}(\Theta_{\cJ}) \onto V'/\tau'$ (just as in the sixth paragraph of Step~1).
The same proof as in \cite[Cor.\ 3.13]{HuWang2} shows that $\Ext^1_{\wt\Gamma}(\Theta_{\cJ},\tau') = 0$, so we can lift $f$ to $\wt f : \Theta_{\cJ} \to V'$ whose image is not contained in $\tau' \subset V'$, as desired.

We show that the restriction map
\begin{equation}\label{eq:Q-Theta}
  \Hom_{\GL_2(\cO_K)}(Q_{\cJ},\pi_2|_{\GL_2(\cO_K)}) \to \Hom_{\GL_2(\cO_K)}(\Theta_{\cJ},\pi_2|_{\GL_2(\cO_K)})
\end{equation}
is surjective (even an isomorphism).
By Lemma~\ref{lem:Q-vs-Theta} it suffices to show that the restriction map
\begin{equation}\label{eq:Theta-i}
  \Hom_{\GL_2(\cO_K)}(\Psi_i,\pi_2|_{\GL_2(\cO_K)}) \to \Hom_{\GL_2(\cO_K)}(\rad_{\wt\Gamma}(\Psi_i),\pi_2|_{\GL_2(\cO_K)})
\end{equation}
is an isomorphism for any $i \in \cJ$.
The map~\eqref{eq:Theta-i} is injective, as $\Psi_i \cong (\tau' \-- \mu_i^+(\tau') \-- \delta_i^+(\tau'))$ and $\delta_i^+(\tau') \not\into \pi_2|_{\GL_2(\cO_K)}$ by Lemma~\ref{lem:mu-pm}, so it suffices to show that the map~\eqref{eq:Theta-i} is surjective.
By  Lemma~\ref{lem:mu-i+}, $\rad_{\wt\Gamma}(\Psi_i) \cong (\tau' \-- \mu_i^+(\tau'))$ is a quotient of $\Ind_I^{\GL_2(\cO_K)} \chi_{\mu_i^+(\tau')}$  (note that $\mu_i^-(\mu_i^+(\tau')) = \tau'$ as $f\geq 2$), so by Lemma~\ref{lem:Ext-Di0}(iii) (and Frobenius reciprocity) we see that the right-hand side of~\eqref{eq:Theta-i} is at most 1-dimensional.
It thus suffices to show that $\Hom_{\GL_2(\cO_K)}(\Psi_i,\pi_2|_{\GL_2(\cO_K)}) \ne 0$ for all $i \in \cJ$.

Suppose $i \in \cJ$.
If $\mu_i^+(\tau') \into \pi_2|_{\GL_2(\cO_K)}$, then the nonsplit extension $\mu_i^+(\tau') \-- \delta_i^+(\tau')$ embeds into $D_{i_0} \subset \pi_2^{K_1}$ exactly as in the fifth paragraph of Step~1, so we are done.
Otherwise, $i \in \cJ^-$ and hence $\mu_i^-(\tau') \in W(\brho^\ss)$.
As $\tau' \in W(\brho)$ and $\ell(\tau') \ge i_0+1$ (see the second paragraph of Step~2), the uniserial representation $\Psi_i = (\tau' \-- \mu_i^+(\tau') \-- \delta_i^+(\tau'))$ injects into $\wt D_0(\brho)$ %
(by the definition of $\wt D_0(\brho)$ in \S~\ref{sec:preliminaries}, noting that $\mu_i^+(\tau'), \delta_i^+(\tau') \notin W(\brho)$ by Lemma~\ref{lem:mu-pm}) 
and even into $\wt D_0(\brho)/\wt D_0(\brho)_{\le i_0} \into \pi_2|_{\GL_2(\cO_K)}$, where this last injection comes from Corollary~\ref{cor:pi1-mK1-squared}.
We have proved that~\eqref{eq:Q-Theta} is an isomorphism.

As $\Ind_I^{\GL_2(\cO_K)} W_{\cJ}$ surjects onto $Q_{\cJ}$ by Lemma~\ref{lem:rep-QJ}, we deduce from the surjectivity of~\eqref{eq:Q-Theta} and from $\dim_\F \Hom_{\GL_2(\cO_K)}(\Theta_{\cJ},\pi_2|_{\GL_2(\cO_K)}) \ge 2$ that
\begin{equation*}
  \dim_\F \Hom_I(W_{\cJ},\pi_2|_I) = \dim_\F \Hom_{\GL_2(\cO_K)}(\Ind_I^{\GL_2(\cO_K)} W_{\cJ},\pi_2|_{\GL_2(\cO_K)}) \ge 2.
\end{equation*}
As $\cosoc_I(W_{\cJ}) = \chi_{\tau'}$ and $W_\cJ$ is killed by $\m^3$, it follows that $\chi_{\tau'}$ occurs at least twice in $\pi_2[\m^3]$, as we wanted to show.
\end{proof}

\begin{cor}\label{cor:K1-subquot}
  {Assume that $\brho$ is $\max\{9,2f+3\}$-generic.}
  Suppose $\pi' = \pi_1'/\pi_1$ is any nonzero subquotient, where $\pi_1 \subsetneq \pi_1' \subset \pi$.
  Let $i_0 \defeq  i_0(\pi_1)$, $i_0' \defeq  i_0(\pi_1')$, so $-1 \le i_0 < i_0' \le f$.
  Then 
   \begin{equation*}
     \pi'^{K_1} \cong D_0(\brho^\ss)_{i_0+1} \oplus_{D_0(\brho)_{i_0+1}} (D_0(\brho)_{\le i_0'}/D_0(\brho)_{\le i_0}).
   \end{equation*}
\end{cor}

\begin{proof}
  Note that $\pi'^{K_1}$ is the kernel of the natural map $(\pi/\pi_1)^{K_1} \to (\pi/\pi_1')^{K_1}$.
  Let us write again $D_{i_0} \defeq  D_0(\brho^\ss)_{i_0+1} \oplus_{D_0(\brho)_{i_0+1}} (D_0(\brho)/D_0(\brho)_{\le i_0})$.
  By Theorem~\ref{thm:pi2-K1} the above natural map is identified with a map
  \begin{equation*}
    \theta : D_{i_0} \to D_{i'_0}.
  \end{equation*}
  Let $\theta_0 \defeq  \theta|_{D_0(\brho)/D_0(\brho)_{\le i_0}}$.
  We have $\soc_{\GL_2(\cO_K)}(\pi/\pi_1') \subset \{ \sigma \in W(\brho^\ss): \ell(\sigma) \ge i_0'+1 \}$ by Corollary~\ref{cor:subquot-K-soc}, and $\JH(D_0(\brho^\ss)_{i_0+1}) \cap W(\brho^\ss)$ is disjoint from that set since $i_0 < i_0'$, so $D_0(\brho^\ss)_{i_0+1} \subset \ker(\theta)$ and hence $\ker(\theta) = D_0(\brho^\ss)_{i_0+1} \oplus_{D_0(\brho)_{i_0+1}} \ker(\theta_0)$.
  On the other hand, by comparison with $\pi^{K_1} = D_0(\brho)$ we see that $\theta_0$ is the natural surjection $D_0(\brho)/D_0(\brho)_{\le i_0} \onto D_0(\brho)/D_0(\brho)_{\le i'_0}$.
  The result follows.
\end{proof}

We can now extend \cite[Thm.~1.2]{YitongWang} to subquotients.

\begin{cor}\label{cor:yitong-subquot}
  Keep the notation and assumptions of Corollary~\ref{cor:K1-subquot}. Then we have
  \begin{equation*}
    \dim_{\F\ppar{X}} D_{\xi}^{\vee}(\pi')=|\JH(\pi'^{K_1}) \cap W(\rhobar^\ss)|=\sum_{i_0 < i \le i_0'} \binom fi.
  \end{equation*}
\end{cor}

\begin{proof}
  By exactness of $D_{\xi}^{\vee}$ and \cite[Cor.~\ref{bhhms4:cor:yitong}]{BHHMS4} we know that the two outside terms are equal.
  By Corollary~\ref{cor:K1-subquot} (using \eqref{eq:filt-Di0}, \cite[eq.~\eqref{bhhms4:eq:JH-D0-i}]{BHHMS4} and that $W(\rhobar^\ss)\subseteq \JH(D_0(\brho))$), we deduce that $\JH(\pi'^{K_1}) \cap W(\rhobar^\ss) = \{ \sigma \in W(\brho^\ss) : i_0 < \ell(\sigma) \le i_0' \}$, and the result follows.
\end{proof}

\section{Global arguments}
\label{sec:global-arg}

{In this section we prove that certain globally defined smooth mod $p$ representations of $\GL_2(K)$ satisfy assumption \ref{it:assum-v} of \S~\ref{sec:abstract-setting} (besides assumptions \ref{it:assum-i}--\ref{it:assum-iv} of \S~\ref{sec:abstract-setting}).}

\subsection{Global setting}
\label{sec:global:setting}
We define smooth mod $p$ representations of $\GL_2(K)$ that arise from the mod $p$ \'etale cohomology of suitable Shimura curves, and recall why they satisfy assumptions \ref{it:assum-i}--\ref{it:assum-iv} of \S~\ref{sec:abstract-setting}.
We then show in \S~\ref{subsec:verify:v} below that they furthermore satisfy assumption \ref{it:assum-v}.

Let $F$ be a totally real number field in which $p$ is unramified, and let $S_p$ denote the set of places of $F$ above $p$.
For each finite place $w$ of $F$ we denote by $F_w$ the completion of $F$ at $w$.
We fix a quaternion algebra $D$ over $F$, with center $F$ such that $D$ splits at all places in $S_p$ and at exactly one infinite place. 
We let $S_D$ denote the set of places of $F$ at which $D$ ramifies.

We fix a continuous representation $\rbar:\Gal(\ovl{F}/F)\ra\GL_2(\F)$ and define $S_{\rbar}$ to be the set of places where $\rbar$ ramifies. 
We write $\rbar_w$ for $\rbar|_{\Gal(\ovl{F}_w/F_w)}$. 
We assume that:
\begin{itemize}
\item $\rbar|_{\Gal(\ovl{F}/F(\sqrt[p]{1}))}$ is absolutely irreducible;
\item if $p=5$, then the image of $\rbar(\Gal(\ovl{F}/F(\sqrt[p]{1})))$ in $\PGL_2(\F)$ is not isomorphic to $A_5$;
\item for all $w\in S_p$, $\rbar_w$ is $0$-generic;
\item for all $w\in S_D$, $\rbar_w$ is non-scalar.
\end{itemize}
We now fix $v\in S_p$ {and let $\psi : G_F \to W(\F)\s$ be the Teichm\"uller lift of $\omega \det(\rbar)$}.
Following \cite[\S~6.5]{EGS} with the corrections of \cite[Rk.~8.1.3]{BHHMS1} (see also~\cite[\S~3.3, \S~3.4]{BD})
we have a compact open subgroup $U^v$ of $(D\otimes_F\mathbb{A}_F^{\infty,v})^\times$ and a smooth representation of $U^v(\mathbb{A}_F^{\infty,v})^\times$ on a finite-dimensional $\F$-vector space which we denote by $\ovl{M}^v$, and on which $(\mathbb{A}_F^{\infty,v})^\times$ acts by $\psi^{-1}$. %
(Following the notation of \cite[\S~6.5]{EGS}~this is the mod $p$-reduction of the inflation to $\prod_{w\in S\setminus\{v\}}K_w\prod_{w\notin S\cup\{v\}}(\cO_{D})^\times_w$ of the $\prod_{w\in S\setminus\{v\}}K_w$-representation $L$ over $W(\F)$ of \emph{loc.~cit}., where furthermore $(\mathbb{A}_F^{\infty,v})^\times$ acts via $\psi^{-1}$. %
Again, the representation $L$ should be corrected following \cite[Rk.~8.1.3(i)]{BHHMS1}, in particular $\psi$ in \cite[\S~6.5]{EGS} should be replaced everywhere by {its inverse}.) %
We set $\rhobar\defeq \rbar_v^\vee$  and following \cite[eq.~(3.3)]{BD} (which treats the case where $\rbar$ is split at all $w\in S_p$, but generalizes to the remaining cases by \cite[\S~6.5]{EGS}) we define the ``local factor''
\begin{align}
\label{eq:pi:rhobar}
\pi(\rhobar)&\defeq\Hom_{U^v}\!\Big(\!\,\ovl{M}^v, \Hom_{\Gal(\o F/F)}\!\big(\rbar, \varinjlim_{V}H^1_{{\rm \acute et}}({X}_{V} \times_F \overline F, \F)\big)\Big)[\fm'_{\rbar}],
\end{align}
where $X_V$ denotes the smooth projective Shimura curve over $F$ associated to $V$ constructed with the convention ``$\varepsilon=-1$'' (see \cite[\S~3.1]{BD} and \cite[\S~2]{BDJ}), 
the colimit runs over all compact open subgroups of $(D\otimes_F\mathbb{A}_F^{\infty})^\times$, and $\fm'_{\rbar}$ is the maximal ideal denoted $\fm'$ in \cite[\S~3.3]{BD} and in \cite[p.~50]{EGS} (though the context of \emph{loc.~cit}.~is slightly different since they use patching functors).
We assume from now on that 
\begin{itemize}
\item $\Hom_{\Gal(\o F/F)}\big(\rbar, \varinjlim_{V}H^1_{{\rm \acute et}}({X}_{V} \times_F \overline F, \F)\big)\neq 0$.
\end{itemize}
In particular $\pi(\rhobar)\neq 0$ by \cite[Thm.~3.7.1]{BD} under the condition that $\rbar$ is reducible at all $w\in S_p$, but the proof extends to the general case using the material of \cite[\S~6.5]{EGS}.

We define the ring $R_\infty$ as in \cite[\S~8.1]{BHHMS1}, with the set $S$ in \emph{loc.~cit}.~taken to be $S_D \cup S_{\o r}$ and the rings $R^{\psi_w}_{\rbar_w}$ {($w \in (S_D \cup S_{\o r}) \setminus S_p$)} and $R_{\rbar_w}^{(0,-1),\tau_w,\psi_w}$ {($w\in S_p\setminus\{v\}$)} of \emph{loc.~cit}.~replaced by the rings $R_w^{\min}$ of \cite[\S~6.5]{EGS}.
By \cite[Thm.~7.2.1]{EGS} and 
\cite[Lemma 3.4.1]{BD} the rings $R_w^{\min}$ are formally smooth over $W(\F)$ (of dimension $3$ or $3+3[F_w:\Qp]$ according to whether $w\in S_p$ or not), so that $R_\infty$ is formally smooth over $W(\F)$ of relative dimension $4|S_D\cup S_{\ovl{r}}|+2[F_v:\Qp]+q-1$ for some integer $q\geq [F:\Q]$.

We can now follow the construction of \cite[\S~6.4]{EGS}
(where the definition of $S(\sigma)_{\fm}$ of \emph{loc.~cit}.~should be corrected as explained in \cite[Rk.~8.1.3(iii)]{BHHMS1}).
We obtain a patching functor (in the sense of \cite[\S~6.1]{EGS}) $M_\infty$ defined on the category of continuous representations of $\GL_2(\cO_{F_v})$ on finite type $W(\F)$-modules with central character ${\psi}^{-1}$, and taking values in the category of $R_\infty$-modules of finite type, such that \[
M_\infty(\sigma_v)/\mathfrak{m}_\infty\cong \big(\Hom_{\GL_2(\cO_{F_v})}(\sigma_v,\pi(\rhobar))\big)^\vee
\]
and which moreover satisfies $\dim_{\F} (M_\infty(\sigma_v)/\mathfrak{m}_\infty)\leq 1$ for any Serre weight $\sigma_v$, by \cite[\S~6.5]{EGS}.
Furthermore the construction of \cite[Thm.~6.2]{DoLe} gives a finitely generated module $\mathbb{M}_\infty$ over $R_\infty\bbra{\GL_2(\cO_{F_v})}$ such that $\mathbb{M}_\infty/\fm_\infty\cong \pi(\rhobar)^\vee$ and \[
M_\infty(\sigma_v)=\Hom^{\textnormal{cont}}_{W(\F)\bbra{\GL_2(\cO_{F_v})}}(\mathbb{M}_\infty,\sigma_v^\vee)^\vee
\]
(where $(-)^\vee\defeq \Hom^{\textnormal{cont}}_{W(\F)}(-,E/W(\F))$).

\begin{prop}
\label{prop:exist:r=1}
If $\rhobar$ is $12$-generic then $\pi(\rhobar)$ satisfies assumptions \ref{it:assum-i}--\ref{it:assum-iv} with ``$r=1$''.
\end{prop}

\begin{proof}
The proofs of the results of \cite[\S~8.2, \S~8.3, Thms.~8.4.1, 8.4.2, 8.4.3]{BHHMS1}, \cite[\S~6]{YitongWangGKD} go through {verbatim} {for our $\pi(\brho)$}, replacing all occurrences of $r$ in \emph{loc.~cit}.~by $1$.
(Note that the hypothesis \cite[\S~8.1, item~(iii)(b)]{BHHMS1} and \cite[\S~1, item (ii)]{YitongWangGKD} are satisfied as all $R_w^{\min}$ above are formally smooth over $W(\F)$.)
In particular \cite[Thm.~1.9]{BHHMS1}, \cite[Thm.~6.3(ii)]{YitongWangGKD} hold, with $r=1$, for $\pi(\rhobar)$ so that $\pi(\rhobar)$ satisfies assumption \ref{it:assum-i} and \ref{it:assum-ii} (for the latter, using \cite[Prop.~6.4.6]{BHHMS1} which holds for a not necessarily semisimple $\rhobar$).
Similarly \cite[Thm.~1.10]{BHHMS1}, \cite[Thm.~6.3(i)]{YitongWangGKD} hold for $\pi(\rhobar)$ so that $\pi(\rhobar)$ satisfies assumption \ref{it:assum-iii} (via \cite[Thm.~8.2]{HuWang2}).
Finally the proofs of \cite[Lemma~\ref{bhhms4:lem:isom-Tor}, Prop.~\ref{bhhms4:prop:dim-Ext}]{BHHMS4} go through verbatim replacing $r$ and $\pi$ in \emph{loc.~cit}.~with $1$ and $\pi(\rhobar)$ respectively, so $\pi(\rhobar)$ satisfies assumption \ref{it:assum-iv}.
(Note that assumption \ref{it:assum-i} is also satisfied by the main result of \cite{LMS, HuWang, DanWild}.)
\end{proof}

\begin{rem}
  {Using \cite[\S~6.5]{EGS} it should be possible to modify the definition of $\pi(\rhobar)$ and generalize the above result to the definite case (i.e.~when $D$ is ramified at all infinite places).}
\end{rem}

\subsection{Verifying assumption \texorpdfstring{\ref{it:assum-v}}{(v)}}
\label{subsec:verify:v}

We keep the setup of \S~\ref{sec:global:setting}.
The goal of this section is to prove the following result:

\begin{prop}\label{prop:degenerates}
{Assume that $\brho$ is 9-generic.} 
  If $\pi(\rhobar)$ satisfies assumptions \refeq{it:assum-i}, \refeq{it:assum-ii} and \refeq{it:assum-iv} of \S~\ref{sec:abstract-setting}, then it also satisfies assumption~\ref{it:assum-v}.
\end{prop}

To simplify notation we let $\pi \defeq \pi(\rhobar)$  and assume that it satisfies assumptions \refeq{it:assum-i}, \refeq{it:assum-ii} and \refeq{it:assum-iv} in the remainder of this section.

\begin{rem}\label{rem:degenerates}
  In fact, we will even establish a canonical isomorphism 
  \begin{equation*}
    \Tor_1^{\gr(\Lambda)}(\gr(\Lambda)/\o\m^n,\gr_\m(\pi^{\vee}))\cong \gr(\Tor_1^{\Lambda}(\Lambda/\m^n,\pi^{\vee})) 
  \end{equation*}
  for $n = 3$.
  We remark that the $n = 2$ case can be proved
  by a similar, but significantly shorter, argument.
  The $n = 1$ case was established in \cite[Cor.~\ref{bhhms4:cor:Tor-pi}(i)]{BHHMS4} (taking $i = 1$ there).
\end{rem}

The proof of Proposition~\ref{prop:degenerates} requires a number of preliminary results.

\begin{lem}\label{lem:buzzati-prop4.3-reformulation}
  Assume that {$\brho$ is $0$-generic.} Then for any $\lambda \in \P$ we have
  \begin{equation*}
    \JH(\Ind_I^{\GL_2(\cO_K)} \chi_\lambda) \cap W(\brho) = \{ \tau \in W(\brho) : J'^\min \subset J_\tau \subset J'^\max \},
  \end{equation*}
  where
  \begin{align*}
    J'^\min &\defeq  \{ j \in J_{\brho} : \lambda_j \in \{x_j+2,p-3-x_j\}\}, \\
    J'^\max &\defeq  \{ j \in J_{\brho} : \lambda_j \notin \{x_j,p-1-x_j\}\}.
  \end{align*}
\end{lem}

\begin{proof}
  Note that we can replace $\lambda$ with $\lambda^{[s]}$ {(see \cite[eq.~\eqref{bhhms4:eq:[s]}]{BHHMS4})} without changing the validity of the lemma.
  By \cite[Prop.\ 4.3]{breuil-buzzati} we have
  \begin{equation*}
    \JH(\Ind_I^{\GL_2(\cO_K)} \chi_\lambda^s) \cap W(\brho) = \{ \sigma_J : J^{\min} \subset J \subset J^{\max} \},
  \end{equation*}
  where $\sigma_J \in W(\brho)$ denotes the Serre weight defined by $\nu_J \defeq  \mu_J \circ \lambda$, with $\mu_J \in \mathcal{P}$ determined by $\mu_{J,j} \in \{p-2-x_j,p-1-x_j\}$ if and only if $j \in J$.
  As $\sigma_J \in W(\brho)$, we deduce that $\mu_J \circ \lambda \in \D$ by \cite[Lemmas 2.1, 2.7]{HuWang2}.
  Also recall from \cite[Prop.\ 4.3]{breuil-buzzati} that
  \begin{align*}
    J^{\min} &= \delta(\{ j : \lambda_j \in \{p-1-x_j,x_j+2\} \text{\ or\ } (\lambda_j = x_j+1, j\notin J_{\brho} \}), \\
    J^{\max} &= \delta(\{ j : \lambda_j \notin \{p-3-x_j,x_j\} \text{\ and\ } (\lambda_j = p-2-x_j \Rightarrow j\in J_{\brho}) \}).
  \end{align*}
  Let $J(\lambda) \defeq  \{ j : \lambda_j \in \{p-3-x_j,p-2-x_j,p-1-x_j\}\}$.
  As $\nu_J = \mu_J \circ \lambda$, we have
  \begin{equation*}
    j \in \delta(J_{\nu_J}) \iff \nu_{J,j} \in \{p-3-x_j,p-2-x_j\} \iff j \in J \mathbin\Delta J(\lambda).
  \end{equation*}
  Equivalently,
  \begin{equation}\label{eq:set-K}
    J_{\nu_J} = \delta^{-1}(J) \mathbin\Delta K, \text{\ where\ } K \defeq  \delta^{-1}(J(\lambda)) = \{ j : \lambda_j \in \{p-3-x_j,p-1-x_j,x_j+1\}\}.
  \end{equation}
  From basic set theory, $J^{\min} \subset J \subset J^{\max}$ if and only if $J'^{\min} \subset \delta^{-1}(J) \mathbin\Delta K \subset J'^\max$,
  where
  \begin{align*}
    J'^\min &\defeq  (\delta^{-1}(J^{\min}) \setminus K) \sqcup (K \setminus \delta^{-1}(J^{\max})), \\
    J'^\max &\defeq  (\delta^{-1}(J^{\max}) \setminus K) \sqcup (K \setminus \delta^{-1}(J^{\min})).
  \end{align*}
  Finally, it follows from the definitions that
  \begin{align*}
    J'^\min &= \{ j : \lambda_j = x_j+2\} \sqcup \{ j : \lambda_j = p-3-x_j\}, \\
    J'^\max &= \{ j \in J_{\brho} : \lambda_j \in \{x_j+2,p-2-x_j\}\} \sqcup \{ j \in J_{\brho} : \lambda_j \in \{p-3-x_j,x_j+1\}\}.\qedhere
  \end{align*}
\end{proof}

\begin{lem}\label{lem:intersec}
  Suppose that $A_0 = \F[x_1,\dots,x_n]$ and $A = \F\bbra{x_1,\dots,x_n}$.
  If $I_0$, $J_0$ are ideals of $A_0$, then $I_0 A \cap J_0 A = (I_0 \cap J_0)A$ as ideals of $A$.
\end{lem}

\begin{proof}
  This is a special case of \cite[Thm.~7.4(ii)]{Ma}, as $A$ is flat over $A_0$.
\end{proof}

The following lemma follows exactly as in \cite[Lemma 3.6.2]{LLLM2} and \cite[Prop.\ 8.1.1]{EGS}.

\begin{lem}\label{lem:Ann}
  Suppose $V$ is a finite length smooth representation of $\GL_2(\cO_K)$ over $\F$ that is multiplicity free.
  Suppose that the scheme-theoretic supports of the $R_\infty$-modules $M_\infty(\sigma)$ (cf.\ 
\cite[\S~\ref{bhhms4:sec:verify-assumpt-iv}]{BHHMS4}) are reduced and do not share any irreducible components for $\sigma$ running through $\JH(V)$.
  Then 
  \begin{equation*}
    \Ann_{R_\infty} M_\infty(V) = \bigcap_{\sigma \in \JH(V)} \Ann_{R_\infty} M_\infty(\sigma).
  \end{equation*}
\end{lem}

\begin{lem}\label{lem:Tor_i-stanley-reisner}
  Suppose $R = \F[X_j, Y_j\ (1 \le j \le k)]$, $I = (X_jY_j\ (1 \le j \le k),\ Y_j Y_{j'}\ (1 \le j < j' \le k))$.
  We have 
  \begin{equation*}
    \dim_\F \Tor_i^R(\F,R/I) =
    \begin{cases}
      1 & \text{if $i = 0$},\\
      i \binom{k+1}{i+1} & \text{if $i > 0$}.
    \end{cases}
  \end{equation*}
\end{lem}

\begin{proof}
  Note that $R/I$ is the Stanley--Reisner ring $\F[\Delta]$ associated to the simplicial complex $\Delta$ whose minimal non-faces are $\{X_j, Y_j\}$ ($1 \le j \le k$) and $\{Y_j, Y_{j'}\}$ ($1 \le j < j' \le k$) \cite[\S~5]{BH93}.
  Also, $\dim_\F \Tor_i^R(\F,R/I)$ is the rank of the degree $i$ term in any minimal graded free resolution of $R/I$ as $R$-module.
  Let $\mathscr V \defeq  \{X_1,\dots,Y_k\}$ denote the set of vertices. %
  By \cite[Thm.\ 5.5.1, Thm.\ 5.3.2]{BH93} we have
  \begin{equation}\label{eq:hochster}
    \dim_\F \Tor_i^R(\F,R/I) = \sum_{\mathscr W \subset \mathscr V} \dim_\F \wt H_{|\mathscr W|-i-1}(|\Delta_{\mathscr W}|;\F),
  \end{equation}
  where $\Delta_{\mathscr W}$ denotes the subcomplex obtained by all faces of $\Delta$ whose vertices are contained in $\mathscr W$ with geometric realization $|\Delta_{\mathscr W}|$, and where $\wt H_j$ denotes the $j$-th reduced homology group (by convention, $H_{-1}(\emptyset;\F) = \F$).
  By definition of $\Delta$, if $\mathscr W$ contains at least two $X_j$, then $|\Delta_{\mathscr W}|$ is contractible (so the term indexed by $\mathscr W$ in~\eqref{eq:hochster} vanishes).
  Similarly, if $\mathscr W$ contains $X_j$ for precisely one $j$, but it does not contain $Y_j$, then $|\Delta_{\mathscr W}|$ is contractible.
  If $\mathscr W$ contains $X_j$ for precisely one $j$ and it also contains $Y_j$, then $|\Delta_{\mathscr W}|$ is homotopic to a disjoint union of 2 points.
  If $\mathscr W$ contains no $X_j$, then $|\Delta_{\mathscr W}|$ is a disjoint union of $|\mathscr W|$ points.
  If $|\Delta_{\mathscr W}|$ is homotopic to a disjoint union of $s \ge 2$ points, the term $\dim_\F \wt H_{|\mathscr W|-i-1}(|\Delta_{\mathscr W}|;\F)$ equals $s-1$ in degree $i = |\mathscr W|-1$ and 0 otherwise.

Now let us compute $\dim_{\F}\Tor_i^R(\F,R/I)$ via \eqref{eq:hochster}.  If $i > 0$, the only contribution then comes from the $\binom k{i+1}$ subsets $\mathscr W$ of $\{Y_1,\dots,Y_k\}$ of cardinality $i+1$ (each contributing $i$) and the $k\cdot \binom {k-1}{i-1}$ subsets $\mathscr W$ that contain precisely one $X_j$ and also $Y_j$ (each contributing $1$).
  The lemma easily follows.
\end{proof}

Recall  from \S~\ref{sec:some-wtgamma-repr} that for $n \ge 1$ we denote $W_{\chi,n} = (\Proj_{I/Z_1} \chi)/\m^n \cong \chi \otimes_\F \Lambda/\m^n$.
(For $n=2,3$ the structure of $W_{\chi,n}$ is completely explicit, see \cite[\S~3.1]{HuWang2}.)
For $\lambda \in \P$ we let
\begin{equation*}
k_\lambda \defeq  |\{0 \le j \le f-1 : t_j \ne y_jz_j \}|.
\end{equation*}
(Recall that the $t_j$, depending on $\lambda$, are defined in~(\ref{eq:id:al}).)
We recall that $\overline{R}_{\infty}=R_{\infty}\otimes_{\cO}\F$.

\begin{prop}\label{prop:patched-module-V_chi-3}
  {Assume that $\brho$ is $2$-generic}. 
  Then for any $\lambda \in \P$ we can find compatible isomorphisms
  \begin{equation*}
    \o R_\infty \cong \F\bbra{X_j, Y_j\ (1 \le j \le \ell), Z_m\ (\ell < m \le N)}
  \end{equation*}
  for some integer $N \ge 2f$ and
  \begin{equation*}
    M_\infty(\Ind_I^{\GL_2(\cO_K)} W_{\chi_\lambda,2}) \cong \o R_\infty/(X_jY_j\ (1 \le j \le \ell),\ Y_i Y_j\ (1 \le i < j \le k_\lambda), \ Z_m (\ell < m \le 2f)),
  \end{equation*}
  where $\ell \defeq  |J_{\brho}|$.
\end{prop}

\begin{proof}
  Let $\chi \defeq \chi_\lambda$ and $V_\chi \defeq  \Ind_I^{\GL_2(\cO_K)} W_{\chi,2}$.
  We have
  \begin{equation}\label{eq:M_infty-V_chi}
    M_\infty(V_\chi) / \m_\infty \cong \Hom_I((\Proj_{I/Z_1} \chi)/\m^2,\pi)^\vee \cong \Hom_I(\Proj_{I/Z_1} \chi,\pi[\m^2])^\vee,
  \end{equation}
  and this is one-dimensional by \cite[Thm.\ 1.3(ii)]{HuWang2}, so $M_\infty(V_\chi)$ is a cyclic $\o R_\infty$-module.

  We first show that %
  \begin{equation}\label{eq:V-chi-cap-W-rho}
    \JH(V_{\chi}) \cap W(\brho) = \Big\{ \sigma \in W(\brho) : \big\lvert (J_\sigma \setminus J'') \mathbin\Delta J'\big\rvert \le 1 \Big\},
  \end{equation}
  where $J' \defeq  \{ j \in J_{\brho} : \lambda_j \in \{x_j+2,p-3-x_j\}\}$ and $J'' \defeq  \{ j \in J_{\brho} : \lambda_j \in \{x_j+1,p-2-x_j\}\}$.
  Note that Lemma~\ref{lem:buzzati-prop4.3-reformulation} applied to $\chi$ gives $J'^\min = J'$ and $J'^\max = J' \sqcup J''$.
  On the other hand, by \cite[Lemma~\ref{bhhms4:lem:compare-I1-invts}(ii)]{BHHMS4} (with $m=1$), $\chi' \defeq \chi \alpha_j^{\pm 1}$ occurs in $\pi^{I_1}$ if and only if $j \in J_{\brho}$ and $\lambda_j \in \{x_j,p-3-x_j\}$ (resp.\ $\lambda_j \in \{x_j+2,p-1-x_j\}$) if the sign is positive (resp.\ negative), and in each such case Lemma~\ref{lem:buzzati-prop4.3-reformulation} applied to $\chi'$ gives $J'^\min = J' \mathbin\Delta \{j\}$ and $J'^\max = J' \sqcup J''$.
  In other words,
  \begin{align*}
    \JH(V_{\chi}) \cap W(\brho) &= \Big\{ \sigma \in W(\brho) : \exists \ K \subset J_{\brho}\setminus J'', |K| \le 1, J' \mathbin\Delta K \subset J_\sigma \subset (J' \mathbin\Delta K) \sqcup J''\Big\} \\
    &= \Big\{ \sigma \in W(\brho) : \exists \ K \subset J_{\brho}\setminus J'', |K| \le 1, J_\sigma \setminus J'' = J' \mathbin\Delta K \Big\} \\
    &= \Big\{ \sigma \in W(\brho) : \exists \ K \subset J_{\brho}\setminus J'', |K| \le 1, (J_\sigma \setminus J'') \mathbin\Delta J' = K \Big\},
  \end{align*}
  which is equivalent to~\eqref{eq:V-chi-cap-W-rho}.

  As $\brho$ is nonsplit, %
we may assume without loss of generality that $0 \notin J_{\brho}$ if $f$ is odd.
  Let
  \begin{equation*}
    \mu \defeq 
    \begin{cases}
      (x_0+1,p-2-x_1,x_2+1,p-2-x_3,\dots,p-2-x_{f-1}) & \text{if $f$ is even},\\
      (x_0,p-2-x_1,x_2+1,p-2-x_3,\dots,p-2-x_{f-2},x_{f-1}+1) & \text{if $f$ is odd},
    \end{cases}
  \end{equation*}
  so $\mu \in \P$ and for all $j \in J_{\brho}$ we have $\mu_j \in \{x_j+1,p-2-x_j\}$.
  Then \cite[Prop.\ 4.3]{breuil-buzzati} or Lemma~\ref{lem:buzzati-prop4.3-reformulation} imply that $\JH(\Ind_I^{\GL_2(\cO_K)} \chi_\mu) \supset W(\brho)$.
  Observe that if $\sigma \in W(\brho)$, then $\sigma$ is parametrized, in the notation of \cite{breuil-buzzati}, 
by the set
  \begin{equation*}
    J_{\sigma,\mu} \defeq  \{ \text{$j$ even} : j \in \delta(J_\sigma) \} \sqcup \{ \text{$j$ odd} : j\notin \delta(J_\sigma) \}.
  \end{equation*}
  (If $f$ is odd we take $0 \le j \le f-1$, not just $j \in \Z/f\Z$!)
  On the other hand, in the same situation, the minimal/maximal subsets in \cite[Prop.\ 4.3]{breuil-buzzati} equal
  \begin{align*}
    J^{\min} &= \delta(\{ \text{$j$ even} : j\notin J_{\brho}, \text{$j \ne 0$ if $f$ odd} \}) \\
    &= \{ \text{$j$ odd} : j\notin \delta(J_{\brho}) \}
  \end{align*}
  and 
  \begin{align*}
    J^{\max} &= \delta(\{ \text{$j$ even} : \text{$j \ne 0$ if $f$ odd} \} \sqcup \{ \text{$j$ odd} : j\in J_{\brho} \}) \\
    &= \{ \text{$j$ odd} \} \sqcup \{ \text{$j$ even} : j\in \delta(J_{\brho}) \}.
  \end{align*}
  In particular, we deduce that
  \begin{equation}\label{eq:Jsigma-mu}
    \begin{aligned}
      J_{\sigma,\mu} \setminus J^{\min} &= \{ \text{$j$ even} : j \in \delta(J_\sigma) \} \sqcup \{ \text{$j$ odd} : j\in \delta(J_{\brho} \setminus J_\sigma) \},\\
      J^{\max} \setminus J^{\min} &= \delta(J_{\brho}).
    \end{aligned}
  \end{equation}

  Let $\tau^0$ denote the lattice in a tame principal series type obtained by inducing the Teichm\"uller lift of $\chi_\mu$ from $I$ to $\GL_2(\cO_K)$, and let $\tau \defeq  \tau^0[1/p]$.
  By \cite[Thm.\ 7.2.1]{EGS} the corresponding fixed-determinant framed local deformation ring $R_{\brho}^{\tau,\psi,\square}$ is isomorphic to $\cO\bbra{x_j,y_j,z_m : j \in J^{\max} \setminus J^{\min}, 1 \le m \le f+3-\ell}/(x_j y_j : \text{all $j$})$ (of relative dimension $f+3$), where $\ell \defeq  |J^{\max} \setminus J^{\min}| = |J_{\brho}|$.
  The full fixed-determinant framed local deformation ring $R_{\brho}^{\psi,\square}$ is a power series ring in $3f+3$ variables.
  It is not hard to see that we can choose an isomorphism $R_{\brho}^{\psi,\square} \cong \cO\bbra{X_j,Y_j,Z_m : j \in J^{\max} \setminus J^{\min}, \ell < m \le 3f+3-\ell}$ such that $R_{\brho}^{\tau,\psi,\square} = R_{\brho}^{\psi,\square}/((X_j Y_j : \text{all $j$})+I_Z)$, where $I_Z \defeq  (Z_m : \ell < m \le 2f)$.

  Hence by \cite[Thm.\ 10.1.1]{EGS} we deduce that 
  \begin{align*}
    R_\infty &\cong \cO\bbra{X_j,Y_j,Z_m : j \in J^{\max} \setminus J^{\min}, \ell < m \le N}, \\
    M_\infty(\tau^0) &\cong R_\infty/((X_j Y_j : \text{all $j$})+I_Z),
  \end{align*}
  for some integer $N \ge 3f+3-\ell$.
  From \cite[Thm.\ 7.2.1(4), Lemma 10.1.12]{EGS} it follows that
  \begin{equation}\label{eq:Ann-serre-wt}
    \Ann_{\o R_\infty} M_\infty(\sigma) = (X_j : j \in J_{\sigma,\mu} \setminus J^{\min},\ Y_j : j \in J^{\max} \setminus J_{\sigma,\mu})+I_Z.
  \end{equation}
   {(In fact, to compare with the conventions of \cite{EGS} we have to replace $(J^{\min},J^{\max},J_{\sigma,\mu})$ by $((J^{\max})^c,(J^{\min})^c,J_{\sigma,\mu}^c)$, cf.~the proof of \cite[Lemma~7.4.1]{EGS}, but this amounts to interchanging $X_j$ and $Y_j$ for each $j$.)}

  By Lemma~\ref{lem:Ann} we have 
  \begin{align*}
    \Ann_{\o R_\infty} M_\infty(\Ind_I^{\GL_2(\cO_K)} \chi_\mu) &= \bigcap_{\sigma \in W(\brho)} \Ann_{\o R_\infty} M_\infty(\sigma), \\
    I_\chi \defeq  \Ann_{\o R_\infty} M_\infty(V_{\chi}) &= \bigcap_{\sigma \in \JH(V_{\chi}) \cap W(\brho)} \Ann_{\o R_\infty} M_\infty(\sigma).
  \end{align*}
  We will make several changes of variables, which will not affect the final result.
  Up to interchanging $X_j$'s and $Y_j$'s, we deduce from equations \eqref{eq:Jsigma-mu} and \eqref{eq:Ann-serre-wt} that 
  \begin{equation*}
    I_\chi \cong \bigcap_{\sigma \in \JH(V_{\chi}) \cap W(\brho)} \Big( (X_j : j \in \delta(J_\sigma),\ Y_j : j \in \delta(J_{\brho} \setminus J_\sigma))+I_Z \Big).
  \end{equation*}
  Shifting the indices in $X_j$ and $Y_j$ by one (to get rid of the $\delta(\cdot)$), and applying~\eqref{eq:V-chi-cap-W-rho} we get
  \begin{equation}\label{eq:I-chi}
    I_\chi \cong \bigcap_{J \subset J_{\brho},\; |(J\setminus J'') \mathbin\Delta J'| \le 1} \Big( (X_j : j \in J,\ Y_j : j \in J_{\brho} \setminus J)+I_Z \Big).
  \end{equation}
  As $J' \cap J'' = \emptyset$ we have $(J\setminus J'') \mathbin\Delta J' = (J \mathbin\Delta J')\setminus J''$.
  Interchanging variables $X_j$ and $Y_j$ for all $j \in J'$, which has the effect of replacing $J$ by $J \mathbin\Delta J'$ in \eqref{eq:I-chi} we get
  \begin{equation*}
    I_\chi \cong \bigcap_{J \subset J_{\brho},\; |J\setminus J''| \le 1} \Big( (X_j : j \in J,\ Y_j : j \in J_{\brho} \setminus J)+I_Z \Big).
  \end{equation*}
  By Lemma~\ref{lem:intersec} and \cite[Thm.\ 5.1.4]{BH93}, this intersection equals
  \begin{equation}\label{eq:I'-chi}
    (X_jY_j,\ Y_{j_1} Y_{j_2} : \text{all $j \in J_{\brho}$; $j_1 < j_2$ both contained in $J_{\brho} \setminus J''$})+I_Z.
  \end{equation}
  (The corresponding simplicial complex has facets $\{X_j : j \notin J, Y_j: j \in J\}$, hence minimal non-faces $\{Z_m\}$ for all $m$, $\{X_j, Y_j\}$ for all $j$, and $\{ Y_{j_1}, Y_{j_2} \}$ for all $j_1 < j_2$ both contained in $J_{\brho} \setminus J''$.)
  As $k_\lambda = |\{ j \in J_{\brho} : \lambda_j \notin \{x_j+1,p-2-x_j\}\}| = |J_{\brho} \setminus J''|$, we are done.
\end{proof}

\begin{prop}\label{prop:ext-i}
  {Assume that $\brho$ is $2$-generic}.
  Then for any $\lambda \in \P$ we have
  \begin{equation*}
    \dim_\F \Ext^i_{I/Z_1}(W_{\chi_\lambda,2},\pi) =
    \begin{cases}
      1 & \text{if $i = 0$},\\
      2f + \binom{k_\lambda}{2} & \text{if $i = 1$,}\\
      2f^2+(k_\lambda^2-k_\lambda-1)f-\binom{k_\lambda+1}3 & \text{if $i = 2$.}
    \end{cases}
  \end{equation*}
\end{prop}

\begin{proof}
  Let $\chi \defeq \chi_\lambda$ and $k \defeq  k_\lambda$ for short.

  By \cite[Lemma~\ref{bhhms4:lem:isom-Tor}]{BHHMS4}, $\Ext^i_{I/Z_1}(W_{\chi,2},\pi)$ is dual to $\Tor_i^{\o R_\infty}(\F,M_\infty(\Ind_I^{\GL_2(\cO_K)} W_{\chi,2}))$, hence by Pro\-position~\ref{prop:patched-module-V_chi-3} to
  $\Tor_i^{\o R_\infty}(\F,\o R_\infty/I_\infty)$, where
  \begin{align*}
    \o R_\infty &= \F\bbra{X_j, Y_j\ (1 \le j \le \ell), Z_m\ (\ell < m \le N)},\\
    I_\infty &= (X_jY_j\ (1 \le j \le \ell),\ Y_i Y_j\ (1 \le i < j \le k), \ Z_m (\ell < m \le 2f))\subset \o R_{\infty},
  \end{align*}
  where $\ell = |J_{\brho}|$.
  Let 
  \begin{align*}
    R &\defeq  \F[X_j, Y_j\ (1 \le j \le \ell), Z_m\ (\ell < m \le N)], \\
    I &\defeq  (X_jY_j\ (1 \le j \le \ell),\ Y_i Y_j\ (1 \le i < j \le k), \ Z_m (\ell < m \le 2f))\subset R.
  \end{align*}
As $\o R_{\infty}$ is flat over $R$ and $\o R_{\infty}/I_{\infty}=\o R_{\infty}\otimes_{R}R/I$, by considering minimal graded free resolutions we deduce an isomorphism
  \begin{equation*}
    \Tor_i^{R}(\F,R/I) \cong \Tor_i^{\o R_\infty}(\F,\o R_\infty/I_\infty).
  \end{equation*}

  It remains to compute $\Tor_i^R(\F,R/I)$ for $i \le 2$.
  We let
  \begin{align*}
    R^{(1)} &\defeq  \F[X_j, Y_j\ (1 \le j \le k)], & I^{(1)} &\defeq  (X_jY_j\ (1 \le j \le k),\ Y_i Y_j\ (1 \le i < j \le k)), \\
    R^{(2,j)} &\defeq  \F[X_j, Y_j], & I^{(2,j)} &\defeq  (X_jY_j), \\
    R^{(3,m)} &\defeq  \F[Z_m], & I^{(3,m)} &\defeq  (Z_m), \\
    R^{(4,n)} &\defeq  \F[Z_n], & I^{(4,n)} &\defeq  (0),
  \end{align*}
  so that
  \begin{gather*}
    R \cong R^{(1)} \otimes_\F \bigotimes_{k < j \le \ell} R^{(2,j)} \otimes_\F \bigotimes_{\ell < m \le 2f} R^{(3,m)} \otimes_\F \bigotimes_{n > 2f} R^{(4,n)},\\
    R/I \cong R^{(1)}/I^{(1)} \otimes_\F \bigotimes_{k < j \le \ell} R^{(2,j)}/I^{(2,j)} \otimes_\F \bigotimes_{\ell < m \le 2f} R^{(3,m)}/I^{(3,m)} \otimes_\F \bigotimes_{n > 2f} R^{(4,n)}/I^{(4,n)}.
  \end{gather*}
  By using the tensor product of a minimal graded free resolution of $R^{(1)}/I^{(1)}$ and of the minimal graded free resolutions
  \begin{align*}
    0 \to R^{(2,j)} \xrightarrow{X_jY_j} R^{(2,j)} \to &R^{(2,j)}/I^{(2,j)} \to 0 \\
    0 \to R^{(3,m)} \xrightarrow{Z_m} R^{(3,m)} \to &R^{(3,m)}/I^{(3,m)} \to 0 \\
    0 \to R^{(4,n)} \to &R^{(4,n)}/I^{(4,n)} \to 0
  \end{align*}
  we obtain that 
  \begin{align*}
    \dim_\F \Tor_i^R(\F,R/I) &= \sum_{j=0}^i \binom{2f-k}{i-j} \dim_\F \Tor_j^{R^{(1)}}(\F,R^{(1)}/I^{(1)}),\\
    &= \binom{2f-k}{i} + \sum_{j=1}^i j\cdot \binom{2f-k}{i-j} \binom{k+1}{j+1} 
  \end{align*}
  by Lemma~\ref{lem:Tor_i-stanley-reisner}.
  We conclude by a short calculation.
\end{proof}

\begin{cor}\label{cor:ext-i}
  {Assume that $\brho$ is $3$-generic}.
  Then for any $\lambda \in \P$ we have
  \begin{equation*}
    \dim_\F \Ext^1_{I/Z_1}(W_{\chi_\lambda,3},\pi) \ge 2f^2+f+\binom{k_\lambda+1}3.
  \end{equation*}
\end{cor}

\begin{rem}\label{rem:exi-i}
  We will see below (in the proof of Proposition~\ref{prop:degenerates}) that equality holds, {at least under a stronger genericity condition}.
  By the proof of this corollary, this implies in fact that the natural map $\Ext^2_{I/Z_1}(W_{\chi_\lambda,3},\pi) \to \Ext^2_{I/Z_1}(\chi_\lambda \otimes \m^2/\m^3,\pi)$ is injective.
\end{rem}

\begin{proof}
  Again let $\chi \defeq \chi_\lambda$.
  By \cite[Thm.~1.3]{HuWang2}, we have $\Hom_{I/Z_1}(W_{\chi,3},\pi) = \Hom_{I/Z_1}(\chi,\pi)$.
  The exact sequence $0 \to \chi \otimes \m^2/\m^3 \to W_{\chi,3} \to W_{\chi,2} \to 0$ thus gives rise to a long exact sequence
  \begin{equation}\label{eq:long-exact}
    \begin{aligned}
      0 \to \Hom_{I/Z_1}(\chi \otimes \m^2/\m^3,\pi) \to \Ext^1_{I/Z_1}(&W_{\chi,2},\pi) \to \Ext^1_{I/Z_1}(W_{\chi,3},\pi) \\
      &\to \Ext^1_{I/Z_1}(\chi \otimes \m^2/\m^3,\pi) \to \Ext^2_{I/Z_1}(W_{\chi,2},\pi).
    \end{aligned}
  \end{equation}
  Let $J \defeq  \{ 0 \le j \le f-1 : t_j \ne y_jz_j \}$ (where again the $t_j$ are defined in~(\ref{eq:id:al})).
  Let $\ve_j \defeq  -1$ if $t_j = y_j$, $\ve_j \defeq  +1$ if $t_j = z_j$.
  Note that, by \cite[Lemma~\ref{bhhms4:lem:compare-I1-invts}(ii)]{BHHMS4} (with $m=2$), $\JH(\chi \otimes \m^2/\m^3) \cap \JH(\pi^{I_1})$ consists of $\chi$ (occurring $2f$ times in $\chi \otimes \m^2/\m^3$) and all $\chi \alpha_i^{\ve_i} \alpha_j^{\ve_j}$ for $\{i < j\} \subset J$ (each occurring once in $\chi \otimes \m^2/\m^3$).
  Hence by assumption \ref{it:assum-iv} we deduce that
  \begin{equation*}
    \dim_\F \Ext^i_{I/Z_1}(\chi \otimes \m^2/\m^3,\pi) =
    \begin{cases}
      2f+\binom {k_\lambda}2 & \text{if $i = 0$},\\
      2f(2f+\binom {k_\lambda}2) & \text{if $i = 1$}.
    \end{cases}
  \end{equation*}
  By Proposition~\ref{prop:ext-i} we deduce that the first map in~\eqref{eq:long-exact} is an isomorphism, so
  \begin{align*}
    \dim_\F \Ext^1_{I/Z_1}(W_{\chi,3},\pi) &\ge \dim_\F \Ext^1_{I/Z_1}(\chi \otimes \m^2/\m^3,\pi) - \dim_\F \Ext^2_{I/Z_1}(W_{\chi,2},\pi) \\
    &= 2f\left(2f+\binom{k_\lambda}2\right) - \left(2f^2+(k_\lambda^2-k_\lambda-1)f-\binom{k_\lambda+1}3\right)\\
    &=2f^2+f+\binom{k_\lambda+1}3,
  \end{align*}
  where we used Proposition~\ref{prop:ext-i} again.
\end{proof}

\begin{lem}\label{lem:set-P}
  Assume that $\brho$ is {$(2n+1)$}-generic. 
  Suppose that $\chi : I \to \F\s$, $J, J' \subset \{0,1,\dots,f-1\}$, $i_j, i'_{j'} \in \Z \setminus \{0\}$ for all $j \in J$, $j' \in J'$ such that $\sum_{j\in J} |i_j| \le n$ and $\sum_{j\in J'} |i'_j| \le n$.
  If $\chi \prod_{j \in J} \alpha_j^{i_j} \in \JH(\pi^{I_1})$ and $\chi \prod_{j \in J'} \alpha_j^{i'_j} \in \JH(\pi^{I_1})$, then
  \begin{enumerate}
  \item $|i_j - i'_j| \le 1$ for all $j \in J \cap J'$;
  \item $\chi \prod_{j \in J''} \alpha_j^{i''_j} \in \JH(\pi^{I_1})$ for any $J \cap J' \subset J'' \subset J \cup J'$,
  where $i''_j = i_j$ if $j \in J$ and $i''_j = i'_j$ if $j \in J'' \setminus J$.
  \end{enumerate}
\end{lem}

\begin{proof}
  Let $\chi' \defeq  \chi \prod_{j \in J} \alpha_j^{i_j} \in \JH(\pi^{I_1})$ and $\chi'' \defeq  \chi \prod_{j \in J'} \alpha_j^{i'_j} \in \JH(\pi^{I_1})$.
  Write $\chi' = \chi_\lambda$ for some $\lambda \in \P$.
  Since
  \begin{equation*}
    \chi'' = \chi' \prod_{j \in J \cap J'} \alpha_j^{i'_j-i_j} \prod_{j \in J \setminus J'} \alpha_j^{-i_j} \prod_{j \in J' \setminus J} \alpha_j^{i'_j},
  \end{equation*}
  part (i) immediately follows from \cite[Lemma~\ref{bhhms4:lem:compare-I1-invts}(ii)]{BHHMS4} (with $m = 2n$).
  The same lemma implies part (ii) as well, by noting that 
  \begin{equation*}
    \chi \prod_{j \in J''} \alpha_j^{i''_j} = \chi' \prod_{j \in J \setminus J''} \alpha_j^{-i_j} \prod_{j \in J'' \setminus J} \alpha_j^{i'_j}
  \end{equation*}
  and since the assumptions imply that $J'' \setminus J \subset J' \setminus J$ and $J \setminus J'' \subset J \setminus J'$.
\end{proof}

\begin{proof}[Proof of Proposition \ref{prop:degenerates}]
  It suffices to establish a canonical isomorphism
  \[\Tor_1^{\gr(\Lambda)}(\gr(\Lambda)/\o\m^3,\gr_\m(\pi^{\vee}))\cong \gr(\Tor_1^{\Lambda}(\Lambda/\m^3,\pi^{\vee})).\] 
  Just as in the proof of \cite[Cor.~\ref{bhhms4:cor:Tor-tau}, Cor.~\ref{bhhms4:cor:Tor-pi}]{BHHMS4} it suffices to show that
  \[\dim_\F \Tor_1^{\gr(\Lambda)}(\gr(\Lambda)/\o\m^3,\gr_\m(\pi^{\vee})) \le \dim_\F \Tor_1^{\Lambda}(\Lambda/\m^3,\pi^{\vee}),\]
  and then equality has to hold.

  \textbf{Step 1.} We first show that
  \begin{multline*}
    \dim_\F \Tor_1^{\gr(\Lambda)}(\gr(\Lambda)/\o\m^3,\gr_\m(\pi^{\vee})) \\ = \sum_{\lambda \in \P} \left(4f^3 + (6-4k_\lambda)f^2 + (2k_\lambda^2-2k_\lambda+1)f - \frac 16 k_\lambda(k_\lambda-1)(2k_\lambda-1)\right).
  \end{multline*}
  From \cite[Thm.~\ref{bhhms4:thm:CMC}]{BHHMS4} we have %
  $\gr_\m(\pi^{\vee}) \cong \bigoplus_{\lambda\in\mathscr{P}}\chi_{\lambda}^{-1}\otimes R/\mathfrak{a}(\lambda)$.
  Fix $\lambda \in \P$ and let $J \defeq  \{0 \le j \le f-1 : t_j \ne y_jz_j \}$, $k \defeq  |J| = k_\lambda$.
  It will suffice to show that
  \begin{equation}\label{eq:gr-lambda-formula}
    \dim_\F \Tor_1^{\gr(\Lambda)}(\gr(\Lambda)/\o\m^3,R/\mathfrak{a}(\lambda)) = 4f^3 + (6-4k)f^2 + (2k^2-2k+1)f - \frac 16 k(k-1)(2k-1).
  \end{equation}

  We will compute this $\Tor_1$ using an explicit free resolution of $R/\mathfrak{a}(\lambda)$.
  
    Recall {from \cite[eq.~\eqref{bhhms4:eq:grj}]{BHHMS4}} that $\gr(\Lambda)$ (resp.\ $\gr(\Lambda)_j$) is the universal enveloping algebra of the Lie algebra $\bigoplus_{j=0}^{f-1} \mathfrak g_j$ (resp.\ $\mathfrak g_j$) over $\F$,
  where $\mathfrak g_j$ has $\F$-basis $y_j$, $z_j$, $h_j$, subject to $[y_j,z_j] = h_j$, $h_j$ is central, and $[\mathfrak g_j,\mathfrak g_{j'}] = 0$ for all $j \ne j'$.
  In the following we use the Poincar\'e--Birkhoff--Witt bases for these Lie algebras, for the ordering $y_0,\dots,y_{f-1},z_0,\dots,z_{f-1},h_0,\dots,h_{f-1}$.
  In particular, $\gr(\Lambda)/\o\m^3$ has $\F$-basis given by all ordered monomials whose degree is at least $-2$, where $y_j,z_j$ have degree $-1$ and $h_j$ has degree $-2$, and its dimension equals $2f^2+4f+1$.
  
  Note that $R/\mathfrak{a}(\lambda)$ is the tensor product of $\gr(\Lambda)_j/(t_j,h_j)$ over $\F$ for all $0 \le j \le f-1$.
  Recall from \cite[Lemma 9.6]{HuWang2} the minimal gr-free resolution of $\gr(\Lambda)_j/(t_j,h_j)$ as $\gr(\Lambda)_j$-module:
  \begin{equation*}
    G_\bullet^{(j)} : \quad 0 \to \gr(\Lambda)_j \xrightarrow{(-h_j,t_j)} \gr(\Lambda)_j \oplus \gr(\Lambda)_j \xrightarrow{\binom{t_j}{h_j}} \gr(\Lambda)_j \to 0,
  \end{equation*}
  where we ignore the grading and $H$-actions.
  (Compared to \cite[Lemma 9.6]{HuWang2} we applied the involution \ $(\alpha,\beta) \mapsto (-\alpha-\beta,\beta)$ \ to \ the \ middle \ term \ in \ case \ $t_j = y_jz_j$.)
  \ Then $\Tor_1^{\gr(\Lambda)}(\gr(\Lambda)/\o\m^3,R/\mathfrak{a}(\lambda))$ is obtained as the first homology of the complex
  $\gr(\Lambda)/\o\m^3 \otimes_{\gr(\Lambda)} G_\bullet$, where $G_\bullet$ is the tensor product complex of all $G_\bullet^{(j)}$, $0 \le j \le f-1$.
  Note that $G_0 \cong \gr(\Lambda)$, $G_1 \cong \bigoplus_{j=0}^{f-1} \gr(\Lambda)^{\oplus 2}$,
  $G_2 \cong \bigoplus_{j=0}^{f-1} \gr(\Lambda) \oplus \bigoplus_{0 \le i < j \le f-1} (\gr(\Lambda)^{\oplus 2}) \otimes_{\gr(\Lambda)} (\gr(\Lambda)^{\oplus 2})$.

  For the purpose of the calculation we may and it  will be convenient to assume that $t_j = z_j$ for all $j \in J$ (by interchanging $y_j$ and $z_j$, if necessary).

  The morphism $\partial_1 : G_1/\o\m^3 \to G_0/\o\m^3$ is given by $\binom{t_j}{h_j}$ in the $j$-th component, so its image in $\gr(\Lambda)/\o\m^3$ has $\F$-spanning vectors given by 
  \begin{equation*}
    \begin{cases}
      y_jz_j, h_j & \text{if $j \notin J$},\\
      t_j, y_i t_j, z_i t_j, h_j & \text{if $j \in J$},
    \end{cases}
  \end{equation*}
  where $0 \le i \le f-1$ is arbitrary.
  As the term $t_i t_j = t_j t_i$ gets counted twice for any $\{i < j\} \subset J$, we see that
  \begin{equation*}
    \dim_\F \im(\partial_1) = 2(f-k)+(2f+2)k-\binom k2 = 2f(k+1)-\binom k2.
  \end{equation*}
  Since $\dim_\F G_1/\o\m^3 = 4f^3+8f^2+2f$, we deduce that
  \begin{equation*}
    \dim_\F \ker(\partial_1) = 4f^3+8f^2-2kf+\binom k2.
  \end{equation*}

  The image of the morphism $\partial_2 : G_2/\o\m^3 \to G_1/\o\m^3$ is generated by $(-h_j,t_j)_j$ for all $j$ and $(t_j,0)_i-(t_i,0)_j$, $(h_j,0)_i-(0,t_i)_j$, $(0,h_j)_i-(0,h_i)_j$ for all $i\ne j$ as a $\gr(\Lambda)$-module.
  (Here the subscript $j$ denotes the $j$-th component of $G_1/\o\m^3\cong \bigoplus_{j=0}^{f-1} (\gr(\Lambda)/\o\m^3)^{\oplus 2}$.)
  As an $\F$-vector space we get spanning vectors
  \begin{align}
    (h_j,0)_i-(0,t_i)_j &\quad 0 \le i,j \le f-1, \notag \\
    (t_j,0)_i-(t_i,0)_j &\quad 0 \le i < j \le f-1, \notag \\
    (0,h_j)_i-(0,h_i)_j &\quad 0 \le i < j \le f-1, \notag \\
    \noalign{\noindent and}
    -(0,wt_i)_j &\quad \text{if $i \in J$, any $j$}, \label{eq:vec1}\\
    (wt_j,0)_i-(wt_i,0)_j &\quad \text{if $\{i < j\} \subset J$}, \label{eq:vec2} \\
    (wt_j,0)_i &\quad \text{if $i \notin J,\ j \in J$} \label{eq:vec3},
  \end{align}
  where $w \in \{y_0,\dots,y_{f-1},z_0,\dots,z_{f-1}\}$ is arbitrary.
  By the Poincar\'e--Birkhoff--Witt Theorem, the only linear relations occur in~\eqref{eq:vec1}, where $(0,t_j t_\ell)_i$ is listed twice for any $\{j < \ell\} \subset J$ and any $i$;
  in~\eqref{eq:vec2}, where for any $\{i < j < \ell\} \subset J$ the elements
  \begin{equation*}
    (t_\ell t_j,0)_i-(t_\ell t_i,0)_j,\     (t_i t_\ell,0)_j-(t_i t_j,0)_\ell,\     (t_j t_i,0)_\ell-(t_j t_\ell,0)_i,
  \end{equation*}
  add to zero, and in~\eqref{eq:vec3}, where $(t_j t_\ell,0)_i$ is listed twice for any $\{j < \ell\} \subset J$, $i \notin J$.
  Therefore,
  \begin{align*}
    \dim_\F \im(\partial_2) &= f^2+\binom f2+\binom f2+2f^2k+2f\binom k2+2fk(f-k)-f\binom k2-\binom k3-(f-k)\binom k2 \\
    &= 2f^2(2k+1)-f(2k^2+1)+2\binom{k+1}3.
  \end{align*}
  We finally check that $\dim_\F \ker(\partial_1)-\dim_\F \im(\partial_2)$ equals the right-hand side of \eqref{eq:gr-lambda-formula}, as desired.

  \textbf{Step 2.} We show that
  \begin{equation*}
    \dim_\F \Tor_1^{\Lambda}(\Lambda/\m^3,\pi^{\vee}) \ge \sum_{\lambda \in \P} \left(4f^3 + (6-4k_\lambda)f^2 + (2k_\lambda^2-2k_\lambda+1)f - \frac 16 k_\lambda(k_\lambda-1)(2k_\lambda-1)\right).
  \end{equation*}
  Note that
  \begin{equation}\label{eq:Tor1-direct-sum}
    \Tor_1^{\Lambda}(\Lambda/\m^3,\pi^{\vee}) \cong \Tor_1^{\F\bbra{I/Z_1}}(\F\bbra{I/Z_1} \otimes_\Lambda \Lambda/\m^3,\pi^{\vee}) \cong \bigoplus_{\chi : I \to \F\s} \Tor_1^{\F\bbra{I/Z_1}}(W_{\chi,3},\pi^{\vee}),
  \end{equation}
  and this is dual to $\bigoplus_{\chi : I \to \F\s} \Ext^1_{I/Z_1}(W_{\chi,3},\pi)$ by \cite[Lemma~\ref{bhhms4:lem:isom-Tor}]{BHHMS4}.

  By assumption \ref{it:assum-iv}, $\Ext^1_{I/Z_1}(W_{\chi,3},\pi) = 0$ if $\JH(W_{\chi,3}) \cap \JH(\pi^{I_1}) = \emptyset$.
  Assume that $\JH(W_{\chi,3}) \cap \JH(\pi^{I_1}) \ne \emptyset$ and let $0 \le i < 3$ be minimal such that $\JH(\chi \otimes \m^i/\m^{i+1}) \cap \JH(\pi^{I_1}) \neq \emptyset$.
  From Lemma~\ref{lem:set-P} (with $n = 2$) we deduce for this $i$ that $\JH(\chi \otimes \m^i/\m^{i+1}) \cap \JH(\pi^{I_1})$ is a singleton. %
  For any $\lambda \in \P$ and $0 \le i < 3$ let $X_{\lambda,i}$ be the set of all $\chi$ such that $\JH(\chi \otimes \m^i/\m^{i+1}) \cap \JH(\pi^{I_1}) = \{\chi_\lambda\}$ and $\JH(\chi \otimes \m^j/\m^{j+1}) \cap \JH(\pi^{I_1}) = \emptyset$ for all $0 \le j < i$.
  Let $X_\lambda \defeq  \bigsqcup_{0 \le i < 3} X_{\lambda,i}$.
  It will be sufficient to show that 
  \begin{equation*}
    \sum_{\chi \in X_{\lambda}} \dim_\F \Ext^1_{I/Z_1}(W_{\chi,3},\pi) \ge 4f^3 + (6-4k)f^2 + (2k^2-2k+1)f - \frac 16 k(k-1)(2k-1),
  \end{equation*}
  where $k \defeq  k_\lambda$ for short.

  If $\chi \in X_{\lambda,0}$, then $\chi = \chi_\lambda$ and
  \begin{equation*}
    \dim_\F \Ext^1_{I/Z_1}(W_{\chi,3},\pi) \ge 2f^2+f+\binom{k+1}3
  \end{equation*}
  by Corollary~\ref{cor:ext-i}.
  Moreover, $|X_{\lambda,0}| = 1$.

  Suppose that $\chi \in X_{\lambda,1}$.
  Then the unique (up to scalar) nonzero morphism $\Proj_{I/Z_1} \chi_\lambda \to W_{\chi,3}$ factors through a morphism $i: W_{\chi_\lambda,2} \to W_{\chi,3}$, as the image is contained in $\rad W_{\chi,3} = \m W_{\chi,3}$.
  Moreover, $i$ is injective by \cite[Lemma 6.1.2]{BHHMS1} and any Jordan--H\"older factor of $\coker(i)$ is not contained in $\JH(\pi^{I_1})$ by Lemma~\ref{lem:set-P} (with $n = 2$).
  (We know the constituents of $W_{\chi,3}$ and their multiplicities by \cite[(44)]{BHHMS1}.)
  Hence $i$ induces an isomorphism $\Ext^1_{I/Z_1}(W_{\chi,3},\pi) \congto \Ext^1_{I/Z_1}(W_{\chi_\lambda,2},\pi)$, which has dimension $2f+\binom k2$ by Proposition~\ref{prop:ext-i}.
  From %
  \cite[Lemma~\ref{bhhms4:lem:compare-I1-invts}(ii)]{BHHMS4} applied with $\sum_j |i_j| \le 1$ it follows that $|X_{\lambda,1}| = 2f-k$.

  If $\chi \in X_{\lambda,2}$, then $\JH(W_{\chi,3}) \cap \JH(\pi^{I_1}) = \{\chi_\lambda\}$ (with multiplicity one), so $\Ext^1_{I/Z_1}(W_{\chi,3},\pi) \xleftarrow{\sim} \Ext^1_{I/Z_1}(\chi_\lambda,\pi)$, which has dimension $2f$ by  assumption \ref{it:assum-iv}. 
  We claim that $|X_{\lambda,2}| = 2f^2-2kf+\binom{k+1}2$.
  Let again $J \defeq  \{ 0 \le j \le f-1 : t_j \ne y_jz_j \}$, which depends on $\lambda$.
  Let $\ve_j \defeq  -1$ if $t_j = y_j$, $\ve_j \defeq  +1$ if $t_j = z_j$, and $\ve_j \in \{\pm 1\}$ arbitrary for $j \notin J$.
  Note that for integers $i_j \in \Z$ ($0 \le j \le f-1$) such that $\sum_j |i_j| \le 2$ we have $\chi_\lambda \prod_j \alpha_j^{\ve_j i_j} \in \JH(\pi^{I_1})$ if and only if $i_j \in \{0,1\}$ if $j \in J$ and $i_j = 0$ if $j\notin J$, cf.\ \cite[Lemma~\ref{bhhms4:lem:compare-I1-invts}(ii)]{BHHMS4} (with $m = 2$).
  Using  Lemma~\ref{lem:set-P} (with $n=2$) we deduce %
  \begin{multline*}
    X_{\lambda,2} = \{ \chi_\lambda \alpha_j^{-\ve_j}\alpha_{j'}^{-\ve_{j'}} (\{j \le j'\} \subset J),\ 
    \chi_\lambda \alpha_j^{-\ve_j}\alpha_{j'}^{\pm 1} (j \in J, j' \notin J), \\
    \chi_\lambda \alpha_j^{\pm 2} (j \notin J),\
    \chi_\lambda \alpha_j^{\pm 1}\alpha_{j'}^{\pm 1} (\{j < j'\} \subset J^c) \},
  \end{multline*}
  which has cardinality
  \begin{equation*}
    \binom{k+1}2 + 2k(f-k) + 2(f-k) + 4\binom{f-k}2 = 2f^2-2kf+\binom{k+1}2.
  \end{equation*}

  We conclude by
  \begin{align*}
    \sum_{\chi \in X_{\lambda}} \dim_\F \Ext^1_{I/Z_1}(W_{\chi,3},\pi) &\ge \left(2f^2+f+\binom{k+1}3\right) + (2f-k)\left(2f+\binom k2\right) \\
    &\hspace{2cm}+ \left((2f^2-2kf+\binom{k+1}2\right)(2f) \\
    &= 4f^3 + (6-4k)f^2 + (2k^2-2k+1)f - \frac 16 k(k-1)(2k-1).\qedhere
  \end{align*}
\end{proof}

\bibliography{Biblio}
\bibliographystyle{amsalpha}

\end{document}